\documentclass[english,10pt]{elsarticle}
\usepackage[utf8]{inputenc}
\usepackage{geometry}
\usepackage{hyperref}
\usepackage{amsthm}
\usepackage{amsmath}
\usepackage{amssymb}
\usepackage{graphicx}
\usepackage{thmtools, thm-restate}
\makeatletter
\theoremstyle{plain}
\@ifundefined{date}{}{\date{31.08.19}}

\newtheorem{definition}{Definition}[section]

\newtheorem{remark}{Remark}[section]

\usepackage{babel}
\usepackage{tikz-cd}
\usepackage{blindtext}
\numberwithin{equation}{section}
\let\today\relax
\makeatletter
\def\ps@pprintTitle{%
	\let\@oddhead\@empty
	\let\@evenhead\@empty
	\def\@oddfoot{\footnotesize\itshape
		 \hfill\today}%
	\let\@evenfoot\@oddfoot
}
\makeatother

\begin{document}
	
\begin{frontmatter}
	
	
		
	\title{$\mathfrak{R}$-matrix for quantum superalgebra $\mathfrak{sl}(2|1)$ at roots of unity and its application to
centralizer algebras}
	
	\author{Alexander~Mazurenko}
	\address{e-mail: mazurencoal@gmail.com}
	\author{Vladimir A. Stukopin \footnote{Research of the author V. Stukopin was supported by funding from the European Research Council (ERC) under the European Union’s Horizon 2020 research and innovation program (QUASIFT grant agreement 677368). } }
	\address{1) Moscow Institute of Physics and Technology (MIPT), e-mail: stukopin@mail.ru}
	\address{2) SMI of VSC RAS (South Mathematical Institute of Vladikavkaz Scientific Center of Russian Academy of Sciences)}
	%\address{3) Interdisciplinary  Scientific Center J.-V. Poncelet (CNRS UMI 2615) and  Scoltech    Center for Advanced Study}

	\begin{abstract}
We consider fundamental facts from the theory of Hopf superalgebras. We use them to construct the quantum double of
the quantum superalgebra $sl(2|1)$ at roots of unity. Thus we obtain a multiplicative formula for universal $R$-matrix. Next we construct an $R$-matrix
to investigate parametrized family of centralizer algebras. We give multiplication laws in particular case and describe a
structure of such algebras in the general case.
	\end{abstract}
	
	
\end{frontmatter}

\section{Introduction}
\label{Int}

Braided Hopf algebra (or quasitriangular Hopf algebra (see \cite{Dr90}) is the most important example of quantum group. This object was introduced
by Vladimir Drinfeld in the middle of 80-th of last century (see \cite{Dr86}).  Braided Hopf algebra has a specific comultiplication structure $\Delta$, such
that a comultiplication and the opposite comultiplication are adjoint. The adjoint map which connects them is defined by a so-called universal
$R$-matrix. Drinfeld calls such Hopf algebras almost cocommutative Hopf algebras (see \cite{Dr90}). If a universal $R$-matrix
of an almost cocommutative Hopf algebra satisfies additional relations from which the quantum Yang-Baxter equation follows, the Hopf algebra is
called quasitriangular (or braided) Hopf algebra.  The problem of finding explicit (and  simply expressed through generators of an algebra)
formula for a universal $R$-matrix for various examples of quantum groups is posed by V. Drinfeld and has a long history.
In many papers the problem is solved for various classes of quantized universal envelopes of both simple and affine algebras.
We mention the following papers in which are used methods close to us \cite{LSS93}, \cite{KT91}, \cite{KirResh}.
The case when the quantization parameter is the root of unity has been less studied (see \cite{HS}, \cite{AbArBa}), but
has recently been in the focus of attention, along with the case of quantum superalgebras (see \cite{DFI}, \cite{AbArBa}, \cite{Yam94},
\cite{AisMak}, \cite{BGZ}, \cite{ZD}, \cite{ZSL}).

The quantum double construction builds a braided Hopf algebra out of any finite-dimensional Hopf algebra with invertible antipode.
These results can be applied to supercase. Some classes of Hopf superalgebras were investigated in \cite{AAY}, \cite{AEG}, \cite{BGZ},
\cite{GZB}, \cite{ZSL}. The number of articles devoted to applications in modern mathematical physics is huge. Here we mention only
the review paper \cite{Maj90} and the following papers which are more closely related to the subject of our work  \cite{BLHM}, \cite{DFI}.

One of the non-trivial applications of the explicit formula for a universal $R$-matrix is related to the Schur-Weyl duality, which
connects the representations of a quantum algebra and the Hecke algebra or braid group, as well as to a construction of invariants of knots
and links. A generalization of these constructions to the case of quantum Hopf superalgebras is, in our opinion, one of the most
important  and actual problems in the representation theory of quantum algebras which attracts attention of many researchers at last
years (see \cite{Mit}, \cite{RSW}, \cite{CAGE}, \cite{CAMA}, \cite{ZD}, \cite{MWLG}).

%Braided Hopf algebras provide solutions of the Yang-Baxter equation. The quantum double construction builds a braided Hopf
%algebra out of any finite-dimensional Hopf algebra with invertible antipode. These results can be applied to supercase. Some
%classes of Hopf superalgebras were investigated \cite{AAY}, \cite{AEG}, \cite{BGZ}, \cite{GZB}, \cite{ZSL}. For applications see
%for example \cite{BLHM}, \cite{DFI}.

In this paper we collect and prove supercase analogues of these results and show how to apply them to a finite-dimensional
quotient of the Hopf superalgebra $U_q(sl(2|1))$ at roots of unity. In particularly we obtain a new universal R-matrix. We use
this fact as an excuse to consider Hopf superalgebras and quantum double construction in supercase in systematic way. We adopt a fundamental approach presented in \cite{KAss}.

First we recall the basic notions about superalgebras. From now on let $\Bbbk$ denote a fixed field of characteristic $0$. All
vector spaces, linear maps and tensor products will be defined with respect to/taken over this field. We summarize definitions
of superbialgebras and Hopf superalgebras. For more details and notations we refer to \cite{AisMak}. A superspace $V$ is a
$\Bbbk$-vector space endowed with a $\mathbb{Z}_2$-grading $V = V_0 \oplus V_1$. If $v \in V_0$, we set $|v|=deg(v)=0$ and if $v
\in V_1$, $|v|=deg(v)=1$.

Let $V$, $W$ be two superspaces. The space of linear maps $\mathrm{Hom}_{\Bbbk}(V,W)$ has a natural $\mathbb{Z}_2$-grading given
by $f \in \mathrm{Hom}_{\Bbbk}(V,W)_{|j|}$ if $f(V_{|i|}) \in W_{|i|+|j|}$ for $|i|, |j| \in \mathbb{Z}_2$. A map $f: V \to W$
is a superspace morphism if $f(V_i) \subset W_i$ for $i \in \mathbb{Z}_2$.

The tensor product $V \otimes W$ of superspaces $V$ and $W$ is the tensor product of vector spaces $V$ and $W$, equipped with
the $\mathbb{Z}_2$-grading given by
\[ (V \otimes W)_{0} = ( V_{0} \otimes W_{0} ) \oplus ( V_{1} \otimes W_{1} ), \]
\[ (V \otimes W)_{1} = ( V_{0} \otimes W_{1} ) \oplus ( V_{1} \otimes W_{0} ). \]

We choose for basis of Lie superalgebra $sl(2|1)$ over $\Bbbk$ the following elements:
$ h_1=e_{1,1}-e_{2,2}, \quad h_2=e_{2,2}+e_{3,3}, $
$ e_1=e_{1,2}, \quad f_1=e_{2,1},$
$ e_2=e_{2,3}, \quad f_2=e_{3,2}, $
$ e_3=[e_1,e_2]=e_{1,3}, \quad f_3=[f_1,f_2]=-e_{3,1}, $
where $e_{i,j} \in M_3(\Bbbk)$ denotes matrix with $1$ at position $(i,j)$ and zeros elsewhere.
The elements $h_1, h_2, e_1, f_1$ are even and $e_2, f_2, e_3, f_3$ are odd.

We have $[h_i,h_j] = 0$, $[h_i, e_j] = a_{ij} e_j$, $[h_i,f_j]=-a_{ij} f_j$, $[e_i,f_j]=  \delta_{ij} h_i$,
$[e_2,e_2]=[f_2,f_2]=0$, $[e_1,[e_1,e_2]]=[f_1,[f_1,f_2]]=0$
with $(a_{ij})_{i,j}$ the Cartan matrix $$
A=\left(
\begin{array}{cc}
2 & -1 \\
-1 & 0 \\
\end{array}
\right).
$$

We denote by $H = \langle h_1,h_2 \rangle$ the Kartan subalgebra of $sl(2|1)$. Fix an element $a \in H$. Recall that
$a=diag(a_1,a_2,a_3)$ is a diagonal matrix, where $a_1,a_2,a_3 \in \Bbbk$. So we can introduce the following linear maps on
$H$:
\[ \epsilon_1(a)=a_1, \quad \epsilon_2(a)=a_2, \quad \delta_1(a)=a_3. \]
Then the set of roots $\Delta \subseteq H^{*}$ has the form $\Delta = \Delta_0 \oplus \Delta_1$, where
$ \Delta_0=\{\pm(\epsilon_1-\epsilon_2)\}, $
$ \Delta_1=\{ \pm(\epsilon_1-\delta_1),  \pm(\epsilon_2-\delta_1) \}. $
Fix the basis of $\Delta$: $\{ \epsilon_1-\epsilon_2, \epsilon_2-\delta_1 \}$. Select the order on the $\Delta$ in the following
way: $\epsilon_1-\epsilon_2 < \epsilon_2-\delta_1$. Then Kartan matrix has the form $A$.

Centralizer algebras of representations of quantum group $sl(2|1)$ at roots of unity are studied in \cite{CAMA}, \cite{MWLG},
see also \cite{Mit}, \cite{RSW}, \cite{ZD}. We consider a basis structure and multiplication laws for a parametrized family of
interwiner spaces corresponding to the tower of the tensor powers of certain representations of a finite-dimensional quotient of
quantum superalgebra $sl(2|1)$ at roots of unity.

This paper is split in five parts. In Section \ref{PrelRes} we prove (partially reprove may be well-known results for reader convenience)
important for us results for Hopf superalgebras. In Section \ref{QDUq} we describe how to construct the universal $R$-matrix for
a finite-dimensional factor-algebra of quantum $sl(2|1)$ at roots of unity. In Section \ref{PFCA} we consider centralizer algebras
which appears when we examine type A representations of quantum $sl(2|1)$ at roots of unity. All proofs are given in Sections \ref{PCR} and \ref{PAR}.

\section{Hopf Superalgebras Preliminaries}
\label{PrelRes}

In this section we prove fundamental theoretic results about superalgebras and related algebraic structures which we use in
Sections \ref{QDUq} and \ref{PFCA}.

\subsection{The Language of Hopf Superalgebras}

We describe basic constructions. By default we consider homogeneous elements when a degree appears. It is easy to prove

\begin{restatable}{proposition}{tensorhopfAlgebra}
	Let $H=(H,\mu_{H},\eta_{H},\Delta_{H},\epsilon_{H},S_{H})$ and $G=(G,\mu_{G},\eta_{G},\Delta_{G},\epsilon_{G},S_{G})$ be
Hopf superalgebras. Then the tensor product
	\[ H\otimes G = (H\otimes G, (\mu_{H}\otimes\mu_{G})\circ(id_{H}\otimes\tau_{G,H}\otimes id_{G}),
(\eta_{H}\otimes\eta_{G})\circ\nu_{k,k}^{-1}, (id_{H}\otimes\tau_{H,G}\otimes id_{G})\circ(\Delta_{H}\otimes\Delta_{G}),
\nu_{k,k} \circ(\epsilon_{H}\otimes\epsilon_{G}), S_{H}\otimes S_{G}) \]
	is a Hopf superalgebra.
\end{restatable}

\begin{remark}\label{rm:one}
	Define by $\tau_{G,H}: G \otimes H \to H \otimes G$ a superspace morphism $\tau_{G,H}(a \otimes b) = (-1)^{|a||b|} b \otimes
a$ for all $a \in G,  b \in H$.
	
	Define a superspace isomorphism $\upsilon_{\Bbbk,V}$ of vector superspaces $\Bbbk \otimes V$ and $V$, where
$\upsilon_{\Bbbk,V}: \Bbbk \otimes V \to V$, $\upsilon_{\Bbbk,V}(1_{\Bbbk} \otimes v) = 1_{\Bbbk}v=v$ for all $v \in V$.

	Analogically, define a superspace isomorphism $\upsilon_{V,\Bbbk}$ of vector superspaces $V \otimes \Bbbk$ and $V$, where
$\upsilon_{V,\Bbbk}: V \otimes \Bbbk \to V$, $\upsilon_{V,\Bbbk}(v \otimes 1_{\Bbbk}) = v$ for all $v \in V$.
\end{remark}

\begin{restatable}{lemma}{commutatorEps}
	Let $H=(H,\mu,\eta)$ be a superalgebra, $\Delta:H\to H\otimes H$ and $\epsilon:H\to \Bbbk$ - superalgebra morphisms. Suppose
that $H$ is multiplicatively generated by a subset $X$ such that for all $x \in X$
	\[(id_{H} \otimes \Delta) \circ \Delta(x)=(\Delta \otimes id_{H}) \circ \Delta(x),\]
	\[(id_{H} \otimes \epsilon) \circ \Delta(x) = (\epsilon \otimes id_{H}) \circ \Delta(x). \]
	Then $\Delta$ is a comultiplication and $\epsilon$ is a counit in $H$.
\end{restatable}

\begin{restatable}{proposition}{inverseAlgebra}\label{pr:invAlg}
	Let $H=(H,\mu,\eta,\Delta,\epsilon)$ be a superbialgebra. Then
	\[ H^{op}=(H,\mu^{op},\eta,\Delta,\epsilon), H^{cop}=(H,\mu,\eta,\Delta^{op},\epsilon) \]
	and $H^{op, cop}=(H,\mu^{op},\eta,\Delta^{op},\epsilon)$ are superbialgebras.
\end{restatable}

Let $V$ be a vector superspace. Then dual vector space $V^{*}$ is a vector superspace, where $(V^{*})_{i}=(V_{i})^{*}$ for $i
\in \{0,1\}$. Proof of the following result is almost identical to that in the case of vector spaces.

\begin{restatable}{proposition}{evaluationmap}\label{pr:evaluationmap}
	Let $M, M^{'}, N, N^{'}$ be finite-dimensional vector superspaces. Then a superspace morphism
	\[ \lambda: \mathrm{Hom}(M,M^{'}) \otimes \mathrm{Hom}(N,N^{'}) \to \mathrm{Hom}(M\otimes N,M^{'}\otimes N^{'}), \]
	\[ \lambda(f \otimes g)(m \otimes n) := (-1)^{|g||m|} f(m) \otimes g(n), \]
	for all $f \in \mathrm{Hom}(M,M^{'}),g \in \mathrm{Hom}(N,N^{'}), m \in M, n\in N$, is an isomorphism of finite-dimensional
vector superspaces.
\end{restatable}

\begin{restatable}{corollary}{evaluationmapcolone}\label{cl:evaluationmapcolone}
	Let $M, N$ be finite-dimensional vector superspaces. Then a superspace morphism
	\[ \lambda_{M,N}: M^{*} \otimes N^{*} \to ( M \otimes N )^{*}, \]
	\[ \lambda_{M,N}(f \otimes g)(m \otimes n) := (-1)^{|g||m|} f(m) g(n), \]
	for all $f \in M^{*},g \in N^{*}, m \in M, n \in N$, is an isomorphism of finite-dimensional vector superspaces.		
\end{restatable}

\begin{restatable}{corollary}{evaluationmapcoltwo}\label{cl:evaluationmapcoltwo}
	Let $\{M_h\}_{h \in [n]}$ be a family of finite-dimensional vector superspaces, where $[n]=\{1,2,...,n\}$. Then we have for
all $f_h \in M_h^{*}$ , $m_h \in M_h$, where $h \in [n]$, a superspace morphism
	\[ \lambda_{M_1,M_2,...M_n}: M_1^{*} \otimes M_2^{*} \otimes ... \otimes M_n^{*} \to (M_1 \otimes M_2 \otimes ... \otimes
M_n)^{*}, \]
	\[ \lambda_{M_1,M_2,...M_n} ( f_1 \otimes f_2 \otimes ... \otimes f_n ) ( m_1 \otimes m_2 \otimes ... \otimes m_n ) =
(-1)^{\sum_{i=2}^{n} \sum_{j=1}^{i-1} |f_{i}| |m_{j}|} \prod_{h=1}^{n} f_{h}(m_{h}) \]
	is an isomorphism of finite-dimensional vector superspaces.
\end{restatable}

\begin{remark}\label{rm:chiisom}
	Define an isomorphism of vector spaces:
	\[ \chi: \Bbbk \to \Bbbk^{*}, \]
	\[ \chi(a)(b) := ab, \]
	for all $a,b \in \Bbbk$.
	The inverse is
	\[ \chi^{-1}: \Bbbk^{*} \to \Bbbk, \]
	\[ \chi^{-1} ( f ) := f(1_{\Bbbk}), \]
	for all $f \in \Bbbk^{*}$.
\end{remark}

\begin{remark}
	Define a superspace morphism of evaluation $\rho_{V^{*},V}$ for vector superspaces $V^{*}$ and $V$, where $\rho_{V^{*},V}:
V^{*} \otimes V \to \Bbbk$, $\rho_{V^{*},V}(f \otimes x) = f(x)$, for all  $f \in V^{*}, \; x \in V$.
\end{remark}

\begin{restatable}{proposition}{HopfDualStruct}\label{pr:HopfDualStructl}
	Let $H=(H,\mu,\eta,\Delta,\epsilon,S)$ be a Hopf superalgebra with antipode $S$. Then the dual vector superspace
	\[ H^*=(H^{*},\Delta_{H}^{*} \circ \lambda_{H,H},\epsilon_{H}^{*} \circ \chi,\lambda_{H,H}^{-1} \circ \mu_{H}^{*},\chi^{-1}
\circ \eta^{*},S^{*}) \]
	is a Hopf superalgebra.
\end{restatable}

Endow a vector space $\mathrm{Hom}(A,B)$ with a structure of a vector superspace:
\[ \mathrm{Hom}(A,B)_{\alpha} = \{ f \in A | f(A_{\beta}) \subset B_{\alpha+\beta}, \beta \in \mathbb{Z}_{2} \} \]
for all $\alpha \in \mathbb{Z}_{2}$.

\begin{restatable}{proposition}{Alglinfunct}
	Let $A=(A,\mu_{A},\eta_{A},\Delta_{A},\epsilon_{A})$, $B=(B,\mu_{B},\eta_{B},\Delta_{B},\epsilon_{B})$ be superbialgebras.
Consider a vector superspace  $\mathrm{Hom}(A,B)$. Then $\mathrm{Hom}(A,B)$ is a superbialgebra
	\[ \mathrm{Hom}(A,B) = ( \mathrm{Hom}(A,B), *, \eta_{B} \circ \epsilon_{A}) \]
	for all $f,g \in \mathrm{Hom}(A,B)$
	\[ f * g = \mu_{B} \circ ( f \otimes g ) \circ \Delta_{A}. \]
\end{restatable}

\begin{restatable}{proposition}{HopfQuotient}
	\label{prHopfQuotient}
	Let $H=(H,\mu,\eta,\Delta,\epsilon,S)$ be a Hopf superalgebra, $I$ - a $\mathbb{Z}_{2}$-graded Hopf ideal in $H$, that is
$I$ is a $\mathbb{Z}_{2}$-graded ideal in superalgebra $H$, $\Delta(I) \subset I \otimes H + H \otimes I$, $\epsilon(I)=0$ and
$S(I) \subset I$. Then a quotient $H/I$ is a Hopf superalgebra. Moreover, a canonical map $\pi : H \to H/I$ is a Hopf
superalgebra morphism.
\end{restatable}

It is easy to prove that
\begin{restatable}{lemma}{ComultProp}
	Let $H=(H,\Delta,\epsilon)$ be a supercoalgebra. Then
	\[ \sum_{(a),(a^{'}),(a^{''})} (a^{'})^{'} \otimes (a^{'})^{''} \otimes (a^{''})^{'} \otimes (a^{''})^{''} =
\sum_{(a),(a^{''}),((a^{''})^{'})} a^{'} \otimes ((a^{''})^{'})^{'} \otimes ((a^{''})^{'})^{''} \otimes (a^{''})^{''} \]
	for all $a \in H$.
\end{restatable}

\begin{restatable}{proposition}{Sprop}\label{PropS}
	Let $H=(H,\mu,\eta,\Delta,\epsilon,S)$ be a Hopf superalgebra. The following statements hold:
	
	1.
	the antipode $S$ is unique;
	
	2. \label{it:Spropfirst}
	\[ S \circ \mu = \mu \circ \tau_{H,H} \circ( S \otimes S); \]
	
	3. \label{it:Spropsecond}
	\[ S \circ \eta = \eta; \]
	
	4. \label{it:Spropthird}
	\[ \epsilon \circ S = \epsilon; \]
	
	5. \label{it:Spropfourth}
	\[ \tau_{H,H} \circ (S \otimes S) \circ \Delta = \Delta \circ S; \]
	
	6.
	all finite-dimensional Hopf superalgebras have a bijective antipode.
\end{restatable}

Proof of the following result is almost identical to that in the case of vector spaces.

\begin{restatable}{proposition}{Sopcopisom}\label{Sopcopisoml}
	Let $H=(H,\mu,\eta,\Delta,\epsilon,S)$ be a Hopf superalgebra.
	
	$H^{op}$ is a Hopf superalgebra if and only if when $S$ is invertible as a superspace morphism.
	Analogically for $H^{cop}$. Moreover, $S_{H^{op}} = S_{H^{cop}} = S_{H}^{-1}$.
	
	\[ H^{op,cop}=(H,\mu^{op},\eta,\Delta^{op},\epsilon,S) \]
	is a Hopf superalgebra and $S:H \to H^{op,cop}$ is a Hopf superalgebra morphism. If $S$ is an superspace isomorphism with
the inverse $S^{-1}$, then $H$ and $H^{op,cop}$ are Hopf superalgebra isomorphisms. Moreover, they are isomorphic via $S$.
	
	If $S$ is an isomorphism of vector superspaces with inverse $S^{-1}$, then
	\[ H^{op}=(H,\mu^{op},\eta,\Delta,\epsilon,S^{-1}), H^{cop}=(H,\mu,\eta,\Delta^{op},\epsilon,S^{-1}) \]
	are isomorphic Hopf superalgebras, moreover, they are isomorphic via $S$.
\end{restatable}

\begin{restatable}{lemma}{sH}
	\label{lemma:sH}
	Let $H$ be a superbialgebra, $S:H \to H^{op}$ - a superalgebra morphism. Suppose that $H$ is multiplicatively generated by
elements of a subset $X$ such that
	\[ \sum_{(x)} x' S(x'') = \epsilon(x)1_{H} = \sum_{(x)} S(x') x'' \]
	for all $x \in X$. Then $S$ is the antipode in $H$.
\end{restatable}

\begin{restatable}{proposition}{HopfEqual}\label{HopfEquall}
	Let $H=(H,\mu,\eta,\Delta,\epsilon,S)$ be a Hopf superalgebra with an invertible as a superspace morphism antipode $S$. Then
Hopf superalgebras $(H^{op})^{*}$ and $(H^{*})^{cop}$ are equivalent, as well as $(H^{cop})^{*}$ and $(H^{*})^{op}$.
\end{restatable}

\begin{remark}
	Let $\{ e_{i} \}_{i \in I}$ be a basis for a finite-dimensional Hopf superalgebra $H$, where $e_{0} = 1_{H}$. Define for
$H^{*}$ a dual basis $\{ e_{i}^{*} \}_{i \in I}$, such that $ e_{i}^{*} ( e_{j}) = \delta_{ij} $, where $i, j \in I$.
	Then we have for all $f,g \in H^{*}$
	\[ \mu_{H^{*}} ( f \otimes g ) = \sum_{i \in I, (e_{i})} (-1)^{|g||e_{i}^{'}|} f(e_{i}^{'}) g(e_{i}^{''}) e_{i}^{*}. \]
	
	As $ \Delta_{H^{*}} (f) = \sum_{i,j \in I} a_{ij} e_{i}^{*} \otimes e_{j}^{*}$, we have
	\[ \lambda_{H,H}(\Delta_{H^{*}} (f)) (e_{i} \otimes e_{j}) = f(e_{i}e_{j}) = (-1)^{|e_{i}||e_{j}|} a_{ij} \Rightarrow a_{ij}
= (-1)^{|e_{i}||e_{j}|} f(e_{i}e_{j}) . \] Therefore,
	\[ \Delta_{H^{*}} (f) = \sum_{i,j \in I} (-1)^{|e_{i}||e_{j}|} f(e_{i}e_{j}) e_{i}^{*} \otimes e_{j}^{*}.  \]
	
	We have for $(H^{op})^{*}$:
	\[ (H^{op})^{*}=(H^{*},\Delta_{H}^{*} \circ \lambda_{H,H},\epsilon_{H}^{*} \circ \chi,\lambda_{H,H}^{-1} \circ (\mu_{H}
\circ \tau_{H,H})^{*},\chi^{-1} \circ \eta^{*},(S^{-1})^{*}), \]
	\[ \Delta_{H^{*}}(f) = \sum_{(f)} f^{'} \otimes f^{''}, \; \Delta_{(H^{op})^{*}}(f) = \sum_{i,j \in I} a_{ij}^{'} e_{i}^{*}
\otimes e_{j}^{*}, \]
	\[ \lambda_{H,H}(\Delta_{(H^{op})^{*}}(f)) (a \otimes b) = \sum_{(f)} (-1)^{|a||b| + |f^{''}||b|} f^{'}(b) f^{''}(a), \]
	\[ \lambda_{H,H}(\Delta_{(H^{op})^{*}}(f)) (e_{i} \otimes e_{j}) =   (-1)^{|e_{i}||e_{j}|} f(e_{j}e_{i}) =
(-1)^{|e_{i}||e_{j}|} a_{ij}^{'} \Rightarrow a_{ij}^{'} = f(e_{j}e_{i}), \]
	\[ \Delta_{(H^{op})^{*}} (f) = \sum_{i,j \in I} f(e_{j}e_{i}) e_{i}^{*} \otimes e_{j}^{*}.  \]
	
	We have for $(H^{cop})^{*}$:
	\[ (H^{cop})^{*}=(H^{*},(\tau_{H,H} \circ \Delta_{H})^{*} \circ \lambda_{H,H},\epsilon_{H}^{*} \circ \chi,\lambda_{H,H}^{-1}
\circ \mu_{H}^{*},\chi^{-1} \circ \eta^{*},(S^{-1})^{*}), \]
	\[ \mu_{(H^{cop})^{*}} ( f \otimes g ) ( a ) = \sum_{(a)} (-1)^{|a^{'}||a^{''}| + |g||a^{''}|} f(a^{''})g(a^{'}). \]
	\[ \mu_{(H^{cop})^{*}} ( f \otimes g ) = \sum_{i \in I, (e_{i})} (-1)^{|e_{i}^{'}||e_{i}^{''}| + |g||e_{i}^{''}|}
f(e_{i}^{''}) g(e_{i}^{'}) e_{i}^{*}. \]
	
	We have for $(H^{op,cop})^{*}$:
	\[ (H^{op,cop})^{*}=(H^{*},(\tau_{H,H} \circ \Delta_{H})^{*} \circ \lambda_{H,H},\epsilon_{H}^{*} \circ
\chi,\lambda_{H,H}^{-1} \circ (\mu_{H} \circ \tau_{H,H})^{*},\chi^{-1} \circ \eta^{*},S^{*}). \]
\end{remark}

\begin{remark}\label{rm:antipodeisom}
	Let $H$ be a Hopf superalgebra with an invertible as a superspace morphism antipode $S$. Then from Propositions
\ref{pr:HopfDualStructl}, \ref{Sopcopisoml} and \ref{HopfEquall} it follows that
	\[ S^{*}: (H^{cop})^{*} \to (H^{op})^{*}, \; S^{*}: (H^{*})^{op} \to (H^{*})^{cop}, \; (H^{*})^{op} \cong (H^{*})^{cop}. \]
	\[ (S^{-1})^{*}: (H^{op})^{*} \to (H^{cop})^{*}, \; (S^{-1})^{*}: (H^{*})^{cop} \to (H^{*})^{op}, \; (H^{*})^{cop} \cong
(H^{*})^{op}. \]
\end{remark}

\begin{restatable}{proposition}{Dualcopop}\label{pr:Dualcopop}
	Let $H=(H,\mu,\eta,\Delta,\epsilon,S)$ be a finite-dimensional Hopf superalgebra with an invertible as a superspace morphism
antipode $S$. Then
	\[ (((H^{cop})^{*})^{op})^{*} \cong H. \]
\end{restatable}

\subsection{Braided Hopf Superalgebras}

We introduce the conception of braided superbialgebras.

\begin{definition}
	Let $V$ be a vector superspace. A linear $\mathbb{Z}_2$-graded automorphism $c$ of vector superspace $V \otimes V$ is said
to be an $R$-matrix if it is a solution of the Yang-Baxter equation
	\[ (c \otimes id_V)(id_V \otimes c)(c \otimes id_V) = (id_V \otimes c)(c \otimes id_V)(id_V \otimes c), \]
	that holds in the automorphism group of vector superspace $V \otimes V \otimes V$.
\end{definition}

\begin{restatable}{proposition}{rmatrixgeneration}
	If $c \in Aut(V,V)$ is an $R$-matrix, then so are $\lambda c, \; c^{-1}$ and $\tau_{V,V} \circ c \circ \tau_{V,V}$ where
$\lambda$ is any non-zero element of $\Bbbk$.
\end{restatable}

\begin{definition}\label{quasicoco}
	Let $(H,\mu,\eta,\Delta,\epsilon)$ be a superbialgebra. We call it quasi-cocommutative if there exists an invertible even
element $R$, called universal $R$-matrix, of the superalgebra $H \otimes H$, such that we have for all $x \in H$
	\begin{equation}\label{eq:Runiversal}
	\Delta^{op}(x) = R \Delta(x) R^{-1}.
	\end{equation}
\end{definition}

If we set $R = \sum_{i \in I} s_i \otimes t_i$, where $I$ is an index set, then relation \ref{eq:Runiversal} can be expressed,
for all $x \in H$, by
\[ \sum_{(x),i \in I} (-1)^{|x'|(|x^{''}| + |s_i|)} x^{''} s_i \otimes x^{'} t_i = \sum_{(x),i \in I} (-1)^{|t_i||x'|} s_ix^{'}
\otimes t_i x^{''}. \]

\begin{definition}\label{twistedquasicoco}
	A quasi-cocommutative superbialgebra $(H,\mu,\eta,\Delta,\epsilon,R)$ or a quasi-cocommutative Hopf superalgebra
$(H,\mu,\eta,\Delta,\epsilon,S,S^{-1},R)$ is braided, if the universal $R$-matrix $R$ satisfies the two relations:
	\begin{equation}\label{eq:RuniversalBr1}
	(\Delta \otimes id_H)(R) = R_{13}R_{23},
	\end{equation}
	\begin{equation}\label{eq:RuniversalBr2}
	(id_H \otimes \Delta)(R) = R_{13}R_{12}.
	\end{equation}
\end{definition}

If $R = \sum_{i \in I} s_i \otimes t_i$, relations \ref{eq:RuniversalBr1} and \ref{eq:RuniversalBr2} can be expressed
respectively as
\[ \sum_{i \in I,(s_i)} (s_i)^{'} \otimes (s_i)^{''} \otimes t_i = \sum_{i,j \in I} (-1)^{|t_i||s_j|} s_i \otimes s_j \otimes
t_it_j, \]
\[ \sum_{i \in I,(t_i)} s_i \otimes (t_i)^{'} \otimes (t_i)^{''} = \sum_{i,j \in I} (-1)^{|t_i|(|s_j|+|t_j|)} s_is_j \otimes t_j
\otimes t_i = \sum_{i,j \in I} s_is_j \otimes t_j \otimes t_i. \]

\begin{restatable}{proposition}{newquasicom}\label{pr:newquasicom}
	1.
	If $(H,\mu,\eta,\Delta,\epsilon,S,S^{-1},R)$ is a quasi-cocommutative Hopf superalgebra whose antipode $S$ is bijective,
then so are
	\[ (H,\mu^{op},\eta,\Delta,\epsilon,S^{-1},S,R^{-1}), (H,\mu,\eta,\Delta^{op},\epsilon,S^{-1},S,R^{-1}), \]
	\[ (H,\mu,\eta,\Delta^{op},\epsilon,S^{-1},S,\tau_{H,H}(R)). \]
	
	2.
	If, furthermore, a Hopf superalgebra $(H,\mu,\eta,\Delta,\epsilon,S,S^{-1},R)$ is braided, then so is
	\[ (H,\mu,\eta,\Delta^{op},\epsilon,S^{-1},S,\tau_{H,H}(R)). \]
\end{restatable}

\begin{restatable}{proposition}{propuniversalmatrix}\label{Req}
	Let $(H,\mu,\eta,\Delta,\epsilon,R)$ be a braided superbialgebra.
	
	1.
	Then the universal $R$-matrix $R$ satisfies the equation
	\[ R_{12}R_{13}R_{23}=R_{23}R_{13}R_{12}, \]
	and we have
	\[ ((\eta \circ \epsilon) \otimes id_H)(R)= 1_{H} \otimes 1_{H} =(id_H \otimes (\eta \circ \epsilon))(R). \]
	
	2.
	If, moreover, $H$ has an invertible antipode, then
	\[ (S \otimes id_H)(R)=R^{-1}=(id_H \otimes S^{-1})(R), \]
	\[ (S \otimes S)(R)=R. \]
\end{restatable}

If $R = \sum_{i \in I} s_i \otimes t_i$, relations from Proposition \ref{Req} are equivalent to
\[ R_{12}R_{13}R_{23} = ( \sum_{i \in I} s_{i} \otimes t_{i} \otimes 1_H ) ( \sum_{j \in I} s_{j} \otimes 1_{H} \otimes t_{j} )
( \sum_{k \in I} 1_{H} \otimes s_{k} \otimes t_{k} ) = \]
\[ = \sum_{i,j,k \in I} (-1)^{|t_i||s_j|+|t_j||s_k|} s_is_j \otimes t_i s_k \otimes t_j t_k, \]
\[ R_{23}R_{13}R_{12} = ( \sum_{i \in I} 1_H \otimes s_{i} \otimes t_{i} ) ( \sum_{j \in I} s_{j} \otimes 1_{H} \otimes t_{j} )
( \sum_{k \in I} s_{k} \otimes t_{k} \otimes 1_{H} ) = \]
\[ = \sum_{i,j,k \in I} (-1)^{ |s_{k}| |s_{i}| } s_js_k \otimes s_i t_k \otimes t_i t_j. \]

\[ \sum_{i,j,k \in I} (-1)^{|t_i||s_j|+|t_j||s_k|} s_is_j \otimes t_i s_k \otimes t_j t_k = \sum_{i,j,k \in I} (-1)^{ |s_{k}|
|s_{i}| } s_js_k \otimes s_i t_k \otimes t_i t_j. \]

\[ \sum_{i \in I} \epsilon(s_i) 1_H \otimes t_i = \sum_{i \in I} s_i \otimes \epsilon(t_i) 1_H = 1_{H} \otimes 1_{H}, \]
\[ R^{-1} = \sum_{i \in I} S(s_i) \otimes t_i = \sum_{i \in I} s_i \otimes S^{-1} (t_i), \]
\[ \sum_{i \in I} S(s_i) \otimes S(t_i) = R. \]

\begin{definition}
	Let $A$ be a superalgebra. A left $A$-module is a vector superspace $M$ with a superspace morphism $\alpha_{A,M}: A \otimes
M \to M$, such that for all $a,b \in A, x,y \in M$
	\[ \alpha_{A,M}( a \otimes (x + y)) = \alpha_{A,M}(a \otimes x) + \alpha_{A,M}(a \otimes y), \]
	\[ \alpha_{A,M}(( a + b ) \otimes x) = \alpha_{A,M}(a \otimes x) + \alpha_{A,M}(b \otimes x), \]
	\[ \alpha_{A,M}((ab) \otimes x) = \alpha_{A,M}(a \otimes ( \alpha_{A,M}(b \otimes x) )), \]
	\[ \alpha_{A,M}(1_{A} \otimes x)=x. \]
	
	A right $A$-module is a vector superspace $M$  with a superspace morphism $\beta_{M,A}: M \otimes A \to M$, such that for
all $a,b \in A$, $x,y \in M$
	\[ \beta_{M,A}( (x + y) \otimes a) = \beta_{M,A}(x \otimes a) + \beta_{M,A}(y \otimes a), \]
	\[ \beta_{M,A} (x \otimes ( a + b )) = \beta_{M,A}(x \otimes a) + \beta_{M,A}(x \otimes b), \]
	\[ \beta_{M,A}(x \otimes (ab)) = \beta_{M,A}( \beta_{M,A}(x \otimes a) \otimes b), \]
	\[ \beta_{M,A}(x \otimes 1_{A})=x. \]
	
	Let $M, \; N$ be left (right) $A$-modules. A map of $A$-modules $\phi: M \to N$ is a superspace morphism such that
$\phi(ax)=a \phi(x)$ ($\phi(xa)=\phi(x) a$) for all $a \in A, \; x \in M$.
\end{definition}

\begin{remark}
	As a map $\alpha_{A,M}$ is even, it follows that $|\alpha_{A,M}(a \otimes x)| = |a| + |x|$.
	As a map $\beta_{M,A}$ is even, it follows that $|\beta_{M,A}(a \otimes x)| = |a| + |x|$.
\end{remark}

\begin{remark}
	Let $(H,\mu,\eta,\Delta,\epsilon)$ be a superbialgebra, $U, \; V$ be a left (right) $H$-modules. A superalgebra morphism
$\Delta:H \to H \otimes H$ allows to endow a vector superspace $U \otimes V$ with a structure of left (right) $H$-module.
	For left $H$-modules:
	\[ \alpha_{H,U \otimes V}( a \otimes (u \otimes v) ) = ((\alpha_{H,U} \otimes \alpha_{H,V}) \circ (id_{H} \otimes \tau_{H,U}
\otimes id_{V})) (\Delta(a) \otimes (u \otimes v)) = \]
	\[ = \sum_{(a)} (-1)^{|u||a^{''}|} \alpha_{H,U} (a^{'} \otimes u) \otimes \alpha_{H,V} (a^{''} \otimes v); \]
	
	for right $H$-modules:
	\[ \beta_{U \otimes V,H}((u\otimes v) \otimes a)= ((\beta_{U,H} \otimes \beta_{V,H}) \circ (id_{U} \otimes \tau_{V,H}
\otimes id_{H})) ((u\otimes v) \otimes \Delta(a))= \]
	\[ = \sum_{(a)}(-1)^{|v||a^{'}|} (\beta_{U,H} (u \otimes a^{'}) \otimes (\beta_{V,H}(v \otimes a^{''}) \]
	for all $a \in H, \; u \in U, \; v \in V$.
\end{remark}

\begin{remark}
	Let $(H,\mu,\eta,\Delta,\epsilon)$ be a superbialgebra. A counit endows a vector superspace $V$ with structure of trivial
left $H$-module by
	\[ \alpha_{H,V}(a \otimes x) = \epsilon(a)x \]
	for all $a \in H, x \in V$.	
	Also a counit endows a vector superspace $V$ with structure of trivial right $H$-module by
	\[ \beta_{V,H}(x \otimes a) = x\epsilon(a) \]
	for all $a \in H, x \in V$. It is obvious, $\alpha_{H,V} = \beta_{V,H}$.
\end{remark}

Let $(H,\mu,\eta,\Delta,\epsilon,S,R)$ be a braided Hopf superalgebra. Consider vector superspaces $V$ and $W$ which are left
$H$-modules. If $R = \sum_{i \in I} s_i \otimes t_i$, then we have for all $v \in V, w \in W$ define a superspace morphism
\[ c_{V,W}^R: V \otimes W \to W \otimes V, \]
by
\[ c_{V,W}^R(v \otimes w) = \tau_{V,W}( ((\alpha_{H,V} \otimes \alpha_{H,W}) \circ (id_H \otimes \tau_{H,V} \otimes id_{W})) (R
\otimes (v \otimes w)) ) = \]
\[ = \sum_{i \in I} (-1)^{|s_i| ( |t_i| + |w| ) + |w||v|} \alpha_{H,W} (t_i \otimes w) \otimes \alpha_{H,V} (s_i \otimes v). \]

\begin{restatable}{lemma}{isomctouniveralmatrix}\label{lm:isomctouniveralmatrix}
	$c_{V,W}^R$ is an isomorphism of vector superspaces. It's inverse is a supespace morphism
	\[ (c_{V,W}^R)^{-1}: W \otimes V \to V \otimes W, \]
	\[ (c_{V,W}^R)^{-1}(w \otimes v) = ((\alpha_{H,V} \otimes \alpha_{H,W}) \circ (id_H \otimes \tau_{H,V} \otimes id_{W}))
(R^{-1} \otimes (\tau_{W,V}(w \otimes v))) = \]
	\[ = \sum_{i \in I} (-1)^{|w||v| + |t_i||v|} \alpha_{H,V} (S(s_i) \otimes v) \otimes \alpha_{H,W} (t_i \otimes w) = \]
	\[ = \sum_{i \in I} (-1)^{|w||v| + |t_i||v|} \alpha_{H,V} (s_i \otimes v) \otimes \alpha_{H,W} (S^{-1}(t_i) \otimes w). \]
	The last equation holds only when $H$ has an invertible antipode $S$.
\end{restatable}

Now we can proof

\begin{restatable}{proposition}{isompropuniverslamatrix}\label{pr:isompropuniverslamatrix}
	Under the previous hypotheses,
	
	1.
	the map $c_{V,W}^R$ is an superspace isomorphism of left $H$-modules, and
	
	2.
	for any triple $(U,V,W)$ of left $H$-modules, we have
	\[ c_{U \otimes V,W}^{R} = (c_{U,W}^{R} \otimes id_V)(id_{U} \otimes c_{V,W}^{R}),
	c_{U,V \otimes W}^{R} = (id_{V} \otimes c_{U,W}^{R})(c_{U,V}^{R} \otimes id_{W}) \]
	and
	\[ (c_{V,W}^{R} \otimes id_{U})(id_{V} \otimes c_{U,W}^{R})(c_{U,V}^{R} \otimes id_{W} ) = (id_{W} \otimes
c_{U,V}^{R})(c_{U,W}^{R} \otimes id_{V})(id_{U} \otimes c_{V,W}^{R}). \]
\end{restatable}

Setting $U=V=W$ in part 2 of Proposition \ref{pr:isompropuniverslamatrix}, we conclude that $c_{V,V}^{R}$ is a solution of the
Yang-Baxter equation for any left $H$-module $V$.

\subsection{Drinfeld's Quantum Double}

We show how to build a braided Hopf superalgebra out of any finite-dimensional Hopf superalgebra with invertible antipode.

\begin{definition}
	Let $H$ be a superbialgebra and $C$ be a supercoalgebra. We say that $C$ is a module-supercoalgebra over $H$ if there exists
a supercoalgebra morphism $H \otimes C \to C \; ( C \otimes H \to C ) $ inducing a left (right) $H$-module structure on $C$.
\end{definition}

\begin{definition}\label{df:twistedB}
	A pair $(X,A)$ of superbialgebras is matched if there exist superspace morphisms $\alpha: A \otimes X \to X$ and $\beta: A
\otimes X \to A$ turning $X$ into a left module-supercoalgebra over $A$, and turning $A$ into a right module-supercoalgebra over
$X$, such that, if we set
	\[ \alpha( a \otimes x ) = a \cdot x, \beta( a \otimes x ) = a^{x}, \]
	the following conditions are satisfied:
	\begin{subequations}
		\renewcommand{\theequation}{\theparentequation.\arabic{equation}}
		\begin{align}
		\label{eq:twistBi1}
		a \cdot (xy) & = \sum_{(a),(x)} (-1)^{|x^{'}||a^{''}|} (a^{'} \cdot x^{'}) ((a^{''})^{x^{''}} \cdot y), \\
		\label{eq:twistBi2}
		a \cdot 1_{X} & = \epsilon_{A}(a) 1_{X}, \\
		\label{eq:twistBi3}
		(ab) ^ {x}  & = \sum_{(b),(x)} (-1)^{|b^{''}||x^{'}|} (a^{b^{'} \cdot x^{'}}) (b^{''})^{ x^{''}}, \\
		\label{eq:twistBi4}
		(1_{A})^{x} & = \epsilon_{X}(x) 1_{A}, \\
		\label{eq:twistBi5}
		\sum_{(a),(x)} (-1)^{|a^{''}||x^{'}|} (a^{'})^{x^{'}} \otimes a^{''} \cdot x^{''} & = \sum_{(a),(x)}
(-1)^{|a^{'}||a^{''}|+|x^{'}||x^{''}|+|a^{'}||x^{''}|} (a^{''})^{x^{''}} \otimes a^{'} \cdot x^{'}
		\end{align}
	\end{subequations}
	for all $a,b \in A$ and $x,y \in X$.
\end{definition}

\begin{remark}
	We can rewrite Definition \ref{df:twistedB} in the following way
	
	1.	
	\[ \alpha \circ ( id_{A} \otimes \mu_{X} ) ( a \otimes x \otimes y ) = \mu_{X} \circ ( \alpha \otimes \alpha ) \circ ( id_{A
\otimes X} \otimes \beta \otimes id_{X} ) \circ ( \Delta_{A \otimes X} \otimes id_{X} ) ( a \otimes x \otimes y ). \]
	
	2.	
	\[ \alpha(a \otimes 1_{X}) = \mu_{X} \circ ( \eta_{X} \otimes id_{X} ) \circ ( \epsilon_{A} \otimes id_{X} ) ( a \otimes
1_{X} ). \]
	
	3.	
	\[ \beta \circ ( \mu_{A} \otimes id_{X} ) ( a \otimes b \otimes x ) = \mu_{A} \circ ( \beta \otimes \beta ) \circ ( id_{A}
\otimes \alpha \otimes id_{A \otimes X} ) \circ ( id_{A} \otimes \Delta_{A \otimes X} ) ( a \otimes b \otimes x ). \]
	
	4.	
	\[ \beta( 1_{A} \otimes x ) = \mu_{A} \circ ( id_{A} \otimes \eta_{A} ) \circ ( id_{A} \otimes \epsilon_{X} ) ( 1_{A}
\otimes x ). \]
	
	5.	
	\[ ( \beta \otimes \alpha ) \circ \Delta_{A \otimes X} (a \otimes x) = ( \beta \otimes \alpha ) \circ ( id_{A} \otimes
\tau_{A,X} \otimes id_{X} ) \circ ( id_{A} \otimes id_{A} \otimes \tau_{X,X} ) \circ ( \tau_{A,A} \otimes id_{X} \otimes id_{X}
) ( \Delta_{A} \otimes \Delta_{X} ) ( a \otimes x ). \]
	
\end{remark}

\begin{remark}
	As maps $\alpha$ and $\beta$ are supercoalgebras morphisms, we have by definition
	
	1.
	\[ (\alpha \otimes \alpha) \circ \Delta_{A \otimes X} = \Delta_{X} \circ \alpha, \]
	\[ \Delta_{X} ( a \cdot x ) = \sum_{(a),(x)} (-1)^{|x^{'}||a^{''}|} a^{'} \cdot x^{'} \otimes a^{''} \cdot x^{''}, \]
	
	2.
	\[ \epsilon_{X} \circ \alpha = \epsilon_{A \otimes X}, \]
	\[ \epsilon_{X} ( a \cdot x ) = \epsilon_{A}(a) \epsilon_{X}(x). \]
	
	3.
	\[ (\beta \otimes \beta) \circ \Delta_{A \otimes X} = \Delta_{A} \circ \beta, \]
	\[ \Delta_{A} ( a^{x} ) = \sum_{(a),(x)} (-1)^{|x^{'}||a^{''}|} (a^{'})^{x^{'}} \otimes (a^{''})^{x^{''}}, \]
	
	4.
	\[ \epsilon_{A} \circ \beta = \epsilon_{A \otimes X}, \]
	\[ \epsilon_{A} ( a^{x} ) = \epsilon_{A}(a) \epsilon_{X}(x). \]
\end{remark}

\begin{restatable}{theorem}{twistedBialgebras}
	\label{theorem:twistedBialgebras}
	Let $(X,A)$ be a matched pair of superbialgebras. There exists a superbialgebra structure on the vector superspace $X
\otimes A$, with unit equal to $1_{X} \otimes 1_{A}$, such that its product is given by
	\[ \mu_{X \bowtie A} = ( \mu_{X} \otimes \mu_{A} ) \circ ( id_{X} \otimes \alpha \otimes \beta \otimes id_{A} ) \circ (
id_{X} \otimes \Delta_{A \otimes X} \otimes id_{A} ), \]
	\[ ( x \otimes a ) ( y \otimes b ) = \sum_{(a),(y)} (-1)^{ |y^{'}||a^{''}| } x ( a^{'} \cdot y^{'} ) \otimes
(a^{''})^{y^{''}} b, \]
	its coproduct by
	\[ \Delta_{X \bowtie A}(x \otimes a) = \sum_{(a),(x)} (-1)^{ |a^{'}| |x^{''}| } (x^{'} \otimes a^{'}) \otimes ( x^{''}
\otimes a^{''} ), \]
	and its counit by
	\[ \epsilon_{X \bowtie A}(x \otimes a) = \epsilon_{X}(x) \epsilon_{A}(a) \]
	for all $x,y \in X$ and $a,b \in A$. Equipped with this superbialgebra structure, $X \otimes A$ is called the bicrossed
product of $X$ and $A$ and denoted $X \bowtie A$. Furthermore, the injective maps $i_{X}(x) = x \otimes 1_{A}$ and $i_{A}(a) =
1_{X} \otimes a$ from $X$ and from $A$ into $X \bowtie A$ are superbialgebra morphisms. We also have
	\[ x \otimes a = \mu_{X \bowtie A} ( ( x \otimes 1_{A} ) \otimes ( 1_{X} \otimes a ) ) \]
	for all $x \in X$ and $a \in A$.
	
	If the superbialgebras $X$ and $A$ have antipodes, respectively denoted by $S_{X}$ and $S_{A}$, then the bicrossed product
is a Hopf superalgebra with antipode $S_{X \bowtie A}$ given by
	\[ S_{X \bowtie A} = ( \alpha \otimes \beta ) \circ ( \tau_{X,A} \circ \tau_{X,A} ) \circ ( id_{X} \otimes \tau_{X,A}
\otimes id_{A} ) \circ ( \tau_{X,X} \otimes \tau_{A,A} ) \circ ( S_{X} \otimes S_{X} \otimes S_{A} \otimes S_{A} ) \circ (
\Delta_{X} \otimes \Delta_{A} ), \]
	\[ S_{X \bowtie A}(x \otimes a) = \sum_{(x),(a)} (-1)^{|x^{'}||x^{''}| + |a^{'}||a^{''}| + |x^{'}||a^{''}| +
|x^{''}||a^{''}| + |x^{'}||a^{'}|} S_{A}(a^{''}) \cdot S_{X}(x^{''}) \otimes ( S_{A}(a^{'}) )^{S_{X}(x^{'})} \]
	for all $x \in X$ and $a \in A$.
\end{restatable}

\begin{definition}
	Let $H$ be a superbialgebra and $A$ be a superalgebra. $A$ is a left (right) module-superalgebra over $H$ if
	
	1.
	as a vector superspace $A$ is a left (right) $H$-module and
	
	2.
	multiplication $\mu_{A}: A \otimes A \to A$ and unit $\eta_{A}: k \to A$ are maps of left (right) $H$-modules.
	
	For left $H$-module
	\begin{equation}
	\label{eq:tensormodule}
	\alpha_{H,A} (x \otimes \mu ( a \otimes b )) = \mu ( \alpha_{H,A \otimes A} (x \otimes ( a \otimes b )) ) = \sum_{(x)}
(-1)^{|a||x^{''}|} \alpha_{H,A} (x^{'} \otimes a) \alpha_{H,A} (x^{''} \otimes b),
	\end{equation}
	\begin{equation}\label{eq:epsmodule}
	\alpha_{H,A} (x \otimes 1_{A}) = \alpha_{H,A} (x \otimes \eta_{A} (1_{k})) = \eta_{A} ( \alpha_{H,k} (x \otimes 1_{k}) ) =
\eta_{A} ( \epsilon_{H} (x) 1_{k} ) = \epsilon_{H} (x) 1_{A}
	\end{equation}
	for all $x \in H, a,b \in A$.
	
	For right $H$-module:
	\begin{equation}\label{eq:tensormodule_1}
	\beta_{A,H} (\mu ( a \otimes b ) \otimes x) = \mu ( \beta_{A \otimes A,H} (( a \otimes b ) \otimes x) ) = \sum_{(x)}
(-1)^{|b||x^{'}|} \beta_{A,H} (a \otimes x^{'}) \beta_{A,H} (b \otimes x^{''}),
	\end{equation}
	\begin{equation}\label{eq:epsmodule_1}
	\beta_{A,H} (1_{A} \otimes x) = \beta_{A,H} (\eta_{A} (1_{k}) \otimes x) = \eta_{A} ( \beta_{k,H} (1_{k} \otimes x) ) =
\eta_{A} ( \epsilon_{H} (x) 1_{k} ) = \epsilon_{H} (x) 1_{A}
	\end{equation}
	for all $x \in H, a,b \in A$.
\end{definition}

\begin{restatable}{lemma}{modularalgebra}
	Let $H=(H,\mu,\eta,\Delta,\epsilon)$ be a superbialgebra, $A$ be a superalgebra, endowed with a structure of left (right)
$H$-module, such that the relations \ref{eq:epsmodule} ( \ref{eq:epsmodule_1} ) are satisfied. Suppose that $H$ is
multiplicatively generated by a subset $X$, which elements satisfy conditions \ref{eq:tensormodule} ( \ref{eq:tensormodule_1} )
for all $a,b$ in $A$. Then $A$ is a left (right) module-superalgebra over $H$.
\end{restatable}

Let $H=(H,\mu,\eta,\Delta,\epsilon,S)$ be a Hopf superalgebra. Define two superspace morphisms for all $a,x \in H$ by
\begin{equation}\label{eq:leftmult}
\gamma(a \otimes x) = \mu \circ ( \mu \otimes id ) \circ ( id \otimes \tau_{H,H} ) \circ ( id \otimes S \otimes id ) \circ (
\Delta \otimes id ) ( a \otimes x ) = \sum_{(a)} (-1)^{|x||a^{''}|} a^{'} x S(a^{''}),
\end{equation}

\begin{equation}\label{eq:rightmult}
\delta(x \otimes a) = \mu \circ ( \mu \otimes id ) \circ ( \tau_{H,H} \otimes id ) \circ ( id \otimes S \otimes id ) \circ ( id
\otimes \Delta ) ( x \otimes a ) = \sum_{(a)} (-1)^{|x||a^{'}|} S(a^{'}) x a^{''}.
\end{equation}

\begin{restatable}{proposition}{moduleAlgebra}\label{pr:moduleAlgebra}
	The superspace morphism $\gamma:H \otimes H \to H$ endows $H$ with the structure of a left module-superalgebra on the
superbialgebra $H$. Similarly, the superspace morphism $\delta:H \otimes H \to H$ endows $H$ with the structure of a right
module-superalgebra on the superbialgebra $H$.
\end{restatable}

\begin{restatable}{lemma}{RightLeftaction}\label{cl:action}
	Consider a Hopf superalgebra $H$ with invertible antipode $S$ and a superalgebra $A$ that is a left (right)
module-superalgebra over $H$. Let us put on the dual vector superspace $A^{*}$ the left (right) $H$-module structure given by
	\begin{equation}\label{eq:modulesupercoalgebraleft}
	\alpha(x \otimes f)(a) = (-1)^{|x||f|} f( \alpha_{H,A} (S^{-1}(x) \otimes a)),
	\end{equation}
	\begin{equation}\label{eq:modulesupercoalgebraright}
	\beta(f \otimes x)(a)=(-1)^{|a||x|}f( \beta_{A,H} (a \otimes S^{-1}(x)))
	\end{equation}
	for all $a \in A, x \in H$ and $f \in A^{*}$. If superalgebra $A$ is finite-dimensional, then the supercoalgebra
$(A^{op})^{*}$ is a left (right) module-supercoalgebra over $H$.
\end{restatable}

\begin{restatable}{corollary}{hondualop}\label{cl:caction}
	Let $H=(H,\mu,\eta,\Delta,\epsilon,S,S^{-1})$ be a finite-dimensional Hopf superalgebra with invertible antipode $S$. There
is a left (right) $H$-module-supercoalgebra structure on the Hopf superalgebra $(H^{op})^{*}=(H^{*},\Delta_{H}^{*} \circ
\lambda_{H,H},\epsilon_{H}^{*} \circ \chi,\lambda_{H,H}^{-1} \circ (\mu_{H} \circ \tau_{H,H})^{*},\chi^{-1} \circ
\eta^{*},(S^{-1})^{*},S^{*})$, given for all $a,x \in H$ and $f \in (H^{op})^{*}$ by
	\[ \alpha(x \otimes f)(a) = \sum_{(x)} (-1)^{|x||f|+|x^{'}||x^{''}|+|a||x^{'}|} f(S^{-1}(x^{''}) a x^{'}), \]
	\[ \beta( f \otimes x )(a) = \sum_{(x)} (-1)^{|x||a|+|x^{'}||x^{''}|+|a||x^{''}|} f(x^{''}aS^{-1}(x^{'})).\]
\end{restatable}

\begin{restatable}{corollary}{rightAction}\label{cl:rightAction}
	Under the hypotheses of Lemma \ref{cl:action} there exists a right $(H^{op})^{*}$-module-supercoalgebra structure on $H$,
given for all $a \in H$ and $f \in (H^{op})^{*}$ by
	\[ \beta(a \otimes f) =  \sum_{(a),(a^{''})} (-1)^{|a||f| + |(a^{''})^{''}|(|(a^{''})^{'}| + |a^{'}|)}
f(S^{-1}((a^{''})^{''}) a^{'}) (a^{''})^{'}. \]	
\end{restatable}

\begin{restatable}{theorem}{QBT}
	Let $(H,\mu,\eta,\Delta,\epsilon,S,S^{-1})$ be a finite-dimensional Hopf superalgebra with invertible antipode $S$. Consider
the Hopf superalgebra
	\[ X = (H^{op})^{*}=(H^{*},\Delta_{H}^{*} \circ \lambda_{H,H},\epsilon_{H}^{*} \circ \chi,\lambda_{H,H}^{-1} \circ (\mu_{H}
\circ \tau_{H,H})^{*},\chi^{-1} \circ \eta^{*},(S^{-1})^{*}). \]
	Let $\alpha:H \otimes X \to X$ and $\beta: H \otimes X \to H$ be superspace morphisms given by
	\[ \alpha(a \otimes f) = a \cdot f = \sum_{(a)} (-1)^{|a||f|+|a^{'}||a^{''}|+|?||a^{'}|} f(S^{-1}(a^{''}) ? a^{'}), \]
	\[ \beta(a \otimes f) = a^f = \sum_{(a),(a^{''})} (-1)^{|a||f| + |(a^{''})^{''}|(|(a^{''})^{'}| + |a^{'}|)}
f(S^{-1}((a^{''})^{''}) a^{'}) (a^{''})^{'}, \]
	for all $a \in H, \; f \in X$. Then the pair $(X,H)$ of Hopf superalgebras is matched in the sense of Defenition
\ref{df:twistedB}.
\end{restatable}

\begin{definition}\label{QDDef}
	The quantum double $D(H)$ of the Hopf superalgebra $H$ is the bicrossed product of $X=(H^{op})^{*}$ and of $H$:
	\[ D(H) = X \bowtie H = (H^{op})^{*} \bowtie H. \]
\end{definition}

\begin{restatable}{lemma}{MultQD}\label{lm:MultQDr}
	The multiplication in $D(H)$ is given by
	\[ ( f \otimes a ) ( g \otimes b ) = \sum_{(a),(a^{'})} (-1)^{ |g||a| + |a^{''}||a^{'}| + |?||(a^{'})^{'}| } f
g(S^{-1}(a^{''}) ? (a^{'})^{'} ) \otimes (a^{'})^{''} b, \]
	for all $f,g \in X$ and $a,b \in H$.
\end{restatable}

\begin{restatable}{proposition}{DHprop}
	The unit in $D(H)$ is equal to $1_{X} \otimes 1_{H} = \epsilon_{H} \otimes 1_{H} $. The counit and comultiplication are
given by
	\[ \epsilon_{X \bowtie H} ( f \otimes a ) = \epsilon_{X}(f) \epsilon_{H} (a) = f(1_{H}) \epsilon_{H} (a), \]
	\[ \Delta_{X \bowtie H}( f \otimes a ) = \sum_{(a),(f)} (-1)^{|f^{'}||f^{''}| + |a^{'}||f^{'}|} ( f^{''} \otimes a^{'} )
\otimes ( f^{'} \otimes a^{''} ), \]
	for all $f \in X, a \in H$, furthemore, $\Delta_{H^*} (f) = \sum_{(f)} f^{'} \otimes f^{''}$.
	
	The antipode $S_{X \bowtie H}$ is given by
	\[ S_{X \bowtie A}(f \otimes a) = \sum_{(a),(a^{'})} (-1)^{ |a^{''}||a^{'}| + |(a^{'})^{'}||?| } f(a^{''} S^{-1}(?)
S^{-1}((a^{'})^{'})) \otimes S( (a^{'})^{''} ) \]
	for all $f \in X$ and $a \in H$.
	
	The quantum double $D(H)$ contains $H$ and $X$ as Hopf subsuperalgebras where the inclusion maps $i_{H}$ and $i_{X}$ are
given by
	\[ i_{H}(a) = 1_{X} \otimes a \; , \; i_{X}(f) = f \otimes 1_{H}. \]
	for all $f \in X$ and $a \in H$.
	Moreover,
	\[ i_{X}(f)i_{H}(a) = f \otimes a,\]
	\[ i_{H}(a)i_{X}(f) = \sum_{(a),(a^{'})} (-1)^{|f||a| + |a^{''}||a^{'}| + |?||(a^{'})^{'}|} f(S^{-1}( a^{''} ) ?
(a^{'})^{'}) \otimes (a^{'})^{''} \]
	for all $f \in X$ and $a \in H$.
\end{restatable}

Proof of following two results is almost identical to that in the case of vector spaces.

\begin{restatable}{lemma}{tensormapproplinear}\label{lm:tensormapproplinear}
	Let $M_1, M_2, ... , M_k, N$ be super vector spaces, with $k > 2$, and suppose
	\[ \eta: M_1 \times M_2 \times ... \times M_{k-1} \times M_k \to N \]
	is a function that is bilinear in $M_1$ and $M_2$ when other coordinates are fixed. There is a unique function
	\[ \omega: (M_1 \otimes M_2) \times ... \times M_{k-1} \times M_k \to N \]
	that is linear in $M_1 \otimes M_2$ when the other coordinates are fixed and satisfies
	\[ \omega(m_1 \otimes m_2,...,m_k) = \eta(m_1,m_2,...,m_k) \]
	for all $m_1 \in M_1,m_2 \in M_2,...,m_k \in M_k$.
	If $\eta$ is multilinear in $M_1, M_2, ..., M_k$, then $\omega$ is multilinear in $(M_1 \otimes M_2), M_3, ..., M_k$.
\end{restatable}

Lemma \ref{lm:tensormapproplinear} is not specific to functions that are bilinear in the first two coordinates: any two
coordinates can be used when the function is bilinear in those two coordinates.

\begin{restatable}{theorem}{mapSuperIsom}\label{cl:mapIsom}
	Let $V,U$ be finite-dimensional vector superspaces. A superspace morphism $\xi_{V,U}: U \otimes V^{*} \to
\mathrm{Hom}(V,U)$, defined for all $u \in U, \alpha \in V^{*}, v \in V$ by
	\begin{equation}\label{eq:clambda}
	\xi_{V,U}(u \otimes \alpha)(v)=(-1)^{|u|( |\alpha| + |v| )} \alpha(v)u,
	\end{equation}
	is an superspace isomorphism. In particular, if $V$ is a finite-dimensional vector superspace the map $\xi_{V,V}$ is the
superspace isomorphism
	\[ V \otimes V^* \cong End(V). \]
\end{restatable}

Let us consider the map $\xi_{H,H}:H \otimes X \to End(H)$ defined in Theorem \ref{cl:mapIsom} for $a,b\in H$ and $f\in X$ by
$\xi_{H,H}(a\otimes f)(b)= (-1)^{|a|(|f| + |b|)} f(b)a$. Since the map  $\xi_{H,H}$ is the isomorphism of vector superspaces, we
have
\[ \rho=\xi_{H,H}^{-1}(id_{H})\in H\otimes X. \]

We define the universal $R$-matrix of the quantum double as the element
\[ R=(i_{H}\otimes i_{X})(\rho) \in D(H)\otimes D(H). \]
We get a more explicit formula for $R$ by choosing a basis $\{e_{i}\}_{i\in I}$ of the vector superspace $H$ together with its
dual basis $\{e^{i}\}_{i\in I}$ in $X$, where $I$ is a index set. Then
\[ \rho=\sum_{i\in I}e_{i}\otimes e^{i}\;\text{and} \; R=\sum_{i\in I}(1_{X}\otimes e_{i})\otimes(e^{i}\otimes1_{H}). \]

Indeed, $\xi_{H,H}(\rho)=id_{H}$. Fix any element $b\in H:b=\sum_{j\in I}\alpha_{j}e_{j},\alpha_{j}\in   \Bbbk$. Then
\[ \xi_{H,H}(\sum_{i\in I}e_{i}\otimes e^{i})(b)=\xi_{H,H}(\sum_{i\in I}e_{i}\otimes e^{i})(\sum_{j\in
I}\alpha_{j}e_{j})=\sum_{i,j\in I} (-1)^{ |e_{i}|(|e_{j}| + |e^{i}|) } \alpha_{j}e^{i}(e_{j})e_{i}= \]
\[ = \sum_{i\in I} (-1)^{ |e_{i}|(|e_{i}| + |e^{i}|) }\alpha_{i}e_{i} = \sum_{i\in I} \alpha_{i}e_{i} =b=id_{H}(b). \]

Therefore, $\rho=\sum_{i\in I}e_{i}\otimes e^{i}$.

\begin{restatable}{theorem}{RQD} \label{RQDth}
	Under the previous hypotheses, the Hopf superalgebra $D(H)$ equipped with the element $R=\sum_{i\in I}(1_{X}\otimes
e_{i})\otimes(e^{i}\otimes1_{H}) \in D(H)\otimes D(H)$ is braided.
\end{restatable}

\section{Construction of Universal $R$-matrix in Supercase and its Multiplicative Formula}
\label{QDUq}

In this section we describe the quantum double of a finite-dimensional Hopf superalgebra $\bar{U}_q$ \ref{Uqdef}. First, we
recall and proof some important results about $U_q(sl(2|1))$. After that we construct bases for Hopf superalgebras
$U_q(sl(2|1))$ and $\bar{U}_q$. We consider the subalgebra $B_q^{+}$ \ref{PDBSdef} and its dual in order to construct the
quantum double $D(B_q^{+})$ of $\bar{U}_q$. We use $D(B_q^{+})$ to construct an epimorphism \ref{th:specialisom} of Hoph
superalgebras and get the universal $R$-matrix of $\bar{U}_q$. We also describe a multiplicative formula for the obtained
universal $R$-matrix.

\subsection{Hopf Superagelbra $U_q(sl(2|1))$}

\begin{definition}\label{def:u_q} Fix an algebraically independent and invertible element $q$ over $\Bbbk$. Let $F$ be the field
of fractions of the integral domain $\Bbbk[q,q^{-1}]$.
$U_q(sl(2|1)) \; (\text{we also use notation } U_q) $ is the associative unital algebra over $F$, generated by elements $k_1,
k_2, e_1, e_2, f_1, f_2$, which satisfy the following relations:
	
	\[ k_1k_1^{-1}=k_1^{-1}k_1=1, \; k_2k_2^{-1}=k_2^{-1}k_2=1, \; k_2k_1=k_1k_2, \]
	\[ e_1k_1=q^{-2}k_1e_1, \; e_2k_1=qk_1e_2, \; e_1k_2=qk_2e_1, \; e_2k_2=k_2e_2, \]
	\[ k_1f_1=q^{-2}f_1k_1, \; k_1f_2=qf_2k_1, \; k_2f_1=qf_1k_2, \; k_2f_2=f_2k_2, \]
	\[ e_1f_1= f_1e_1 + \frac{k_1-k_1^{-1}}{q-q^{-1}}, \; e_2f_2= -f_2e_2 + \frac{k_2-k_2^{-1}}{q-q^{-1}}, \]
	\[e_1f_2=f_2e_1, \; e_2f_1=f_1e_2, \; e_2^2=f_2^2=0,\]
	\[e_1^2e_2-(q+q^{-1})e_1e_2e_1+e_2e_1^2=0,\]
	\[f_1^2f_2-(q+q^{-1})f_1f_2f_1+f_2f_1^2=0.\]
\end{definition}

Additionally define elemenets:
\[ e_3 := e_1e_2 - q^{-1}e_2e_1, f_3 := f_2f_1 - qf_1f_2. \]
Then two last equations can be expressed in the form:
\begin{equation} \label{e3f3eq}
	e_3 e_1 = q^{-1} e_1 e_3, \; f_3 f_1 = q^{-1} f_1 f_3.
\end{equation}

We can introduce the superalgebra structure on $U_q$:
\[deg(k_1)=deg(k_2)=deg(k_1^{-1})=deg(k_2^{-1})=deg(e_1)=deg(f_1)=0,\]
\[deg(e_2)=deg(f_2)=deg(e_3)=deg(f_3)=1.\]

$U_q$ is also the Hopf superalgebra. Define the comultiplication $\Delta$ on multiplicative generators:
\[ \Delta:U_q \rightarrow U_q \otimes U_q, \]
\[ \Delta(k_i)=k_i \otimes k_i, \; \Delta(k_i^{-1})=k_i^{-1} \otimes k_i^{-1}, \; \Delta(e_i)=e_i \otimes 1 + k_i \otimes e_i,
\; \Delta(f_i)=f_i \otimes k_i^{-1} + 1 {\otimes} f_i,\]
\[ \Delta(e_3) = (q-q^{-1}) k_2 e_1 \otimes e_2 + e_3 \otimes 1 + k_1 k_2 \otimes e_3, \]
\[ \Delta(f_3) = ( q^{-1} - q ) f_2 \otimes f_1 k_2^{-1} + 1 \otimes f_3 + f_3 \otimes k_1^{-1} k_2^{-1}, \]
where $i \in \{1,2\}$.

Define the counit $\epsilon$ on multiplicative generators:
\[ \epsilon:U_q \rightarrow F, \]
\[ \epsilon(1_{U_q})=1_{F}, \; \epsilon(k_i)=1_{F}, \; \epsilon(k_i^{-1})=1_{F}, \; \epsilon(e_i)=0, \; \epsilon(f_i)=0, \;
\epsilon(e_3)=0, \; \epsilon(f_3)=0,\]
where $i \in \{1,2\}$.

Define the antipode $S$ on multiplicative generators:
\[ S:U_q \rightarrow U_q^{op}, \]
\[ S(1)=1, \; S(k_i)=k_i^{-1}, \; S(k_i^{-1})=k_i, \; S(e_i)=-k_i^{-1}e_i, \; S(f_i)=-f_ik_i, \]
\[ S(e_3) = (1-q^{-2}) k_1^{-1} k_2^{-1} e_1 e_2 - k_1^{-1} k_2^{-1} e_3, \; S(f_3) = (q - q^3) f_1 f_2 k_1 k_2 - q^2 f_3 k_1
k_2, \]
where $i \in \{1,2\}$.

Define the inverse $S^{-1}$ of the antipode on multiplicative generators of $U_q^{op}$:
\[ S^{-1}:U_q^{op} \rightarrow U_q, \]
\[ S^{-1}(1) = 1, S^{-1} (k_i) = k_i^{-1}, \; S^{-1}(k_i^{-1})=k_i, \; S^{-1} (e_i) = - q^{2[i=1]} k_i^{-1} e_i, S^{-1} (f_i) =
- q^{-2[i=1]} f_i k_i, \]
\[ S^{-1}(e_3) = (q^2-1) k_1^{-1} k_2^{-1} e_1 e_2 - q^2 k_1^{-1} k_2^{-1} e_3, \;  S^{-1}(f_3) = (q^{-1} - q) f_1 f_2 k_1 k_2 -
f_3 k_1 k_2,  \]
where $i \in \{1,2\}$.

We introduce some notations.
We have for all $n \in \mathbb{Z}$
\[[n]:=\frac{q^n-q^{-n}}{q-q^{-1}}.\]
We have for all $m \in \mathbb{Z}_+$
\[[0]!:=1, \; [m]!:=[1][2]...[m].\]
We have for all $n,m \in \mathbb{Z}_+$, $n \ge m$
\[{n \brack m}:=\frac{[n]!}{[m]![n-m]!}.\]	
We have for all $n \in \mathbb{Z}$ and an invertible variable $x$
\[ [x;n] := \frac{x q^n - x^{-1} q^{-n} }{q - q^{-1}}. \]
If non-commuting variables $x$ and $y$ satisfy the equation $yx=q^2xy$, then we have for $n>0$
\[ (x+y)^n = \sum_{k=0}^{n} q^{k(n-k)} {n \brack k} x^k y^{n-k}. \]	

\begin{restatable}{lemma}{eq}\label{lm:relonbasis}
	\label{lemma:eq}
	We have in $U_q$
	\[ f_3^sf_1^w=q^{-sw}f_1^wf_3^s, \; f_2^l f_1^w = q^{wl} f_1^w f_2^l + [l=1] [w] f_1^{w-1} f_3, \]
	\[ k_1^{i}f_1^{w}=q^{-2iw}f_1^wk_1^i, \; k_2^jf_1^w=q^{jw}f_1^wk_2^j,\]
	\[ e_1^r f_1^w = f_1^w e_1^r + [w,r > 0] \sum_{u=1}^{min(r,w)} \frac{[r]! [w]!}{[u]! [r-u]! [w-u]!}  f_1^{w-u} \times \]
	\[ \times [k_1; \; 2u-r-w] [k_1; \; 2u-r-w-1] ... [k_1; \; u-r-w+1] e_1^{r-u}, \]
	\[ e_3^h f_1^w = f_1^w e_3^h - q^{w-2} [h=1] [w] f_1^{w-1} k_1^{-1} e_2, \;  e_2^tf_1^w=f_1^we_2^t, \]
	\[ f_3^2 = 0, \; f_2^lf_3^s=(-q)^{sl}f_3^sf_2^l, \; k_1^if_3^s=q^{-is}f_3^sk_1^i, \; k_2^jf_3^s=q^{js}f_3^sk_2^j, \]
	\[ e_1^r f_3^s = f_3^s e_1^r - q^{2-r} [s=1] [r] f_2 k_1 e_1^{r-1}, e_3^h f_3^s = (-1)^{sh} f_3^s e_3^h + [s,h>0]
\frac{k_1k_2-k_1^{-1}k_2^{-1}}{q-q^{-1}}, \]
	\[ e_2^t f_3^s = (-1)^{st} f_3^s e_2^t + [s,t > 0] f_1k_2^{-1}, \]
	\[ f_2^2=0, k_1^if_2^l=q^{il}f_2^lk_1^i, \; k_2^jf_2^l=f_2^lk_2^j, \; e_1^rf_2^l=f_2^le_1^r, \; e_3^h f_2^l = (-1)^{lh}
f_2^l e_3^h + [h,l > 0] k_2 e_1, \]
	\[ e_2^t f_2^l = (-1)^{lt} f_2^l e_2^t + [l,t > 0] [k_2;0], \]
	\[ k_2^jk_1^i=k_1^ik_2^j, \; e_1^rk_1^i=q^{-2ir}k_1^{i}e_1^{r}, \; e_3^hk_1^i=q^{-ih}k_1^ie_3^h, \;
e_2^tk_1^i=q^{it}k_1^ie_2^t, \]
	\[ e_1^rk_2^j=q^{jr}k_2^je_1^r, \; e_3^hk_2^j=q^{jh}k_2^je_3^h, \; e_2^tk_2^j=k_2^je_2^t, \]
	\[ e_3^he_1^r=q^{-hr}e_1^re_3^h, \; e_2^t e_1^r = q^{rt} e_1^r e_2^t - q [t=1] [r] e_1^{r-1} e_3, \]
	\[ e_3^2 = 0, e_2^te_3^h= (-q)^{ht} e_3^he_2^t, e_2^2=0, \]
	where $i,j \in \mathbb{Z}, w,r \in \mathbb{N}, l,s,t,h \in \{0,1\} $.
\end{restatable}

\begin{restatable}{theorem}{basisU}
	\label{theorem:basisU}
	The elements of
	\[ G = \{ f_1^wf_3^sf_2^lk_1^ik_2^je_1^re_3^he_2^t | \; i,j \in \mathbb{Z}, \; w,r \in \mathbb{Z}_+, \; l,s,t,h \in \{0,1\}
\} \]
	form an additive basis of $U_q$.
\end{restatable}

\subsection{Hopf Superagelbra $\bar{U}_q$ at roots of unity}

Next we suppose that $q$ is a root unity of odd order $d$. Then $ F = \Bbbk (q) = \Bbbk [q] / ( \Phi_d (q))$, where $\Phi_d$ is
the $d$-cyclotomic polynomial. In this case we additionally consider the equation $q^d 1_{U_q} = 1_{U_q}$. It is obviously from
the proof of Theorem \ref{theorem:basisU} that the basis of $U_q$ remains the same.

\begin{restatable}{lemma}{cent}\label{lm:cent}
	\label{lemma:cent}
	Elements $f_1^d, k_1^d - 1, k_2^d - 1, e_1^d$ are in the center of the Hopf superalgebra $U_q$.
\end{restatable}

We now can give

\begin{definition} \label{Uqdef}
Consider in $U_q$ the two-sided $\mathbb{Z}_2$-graded ideal
\[I=(f_1^d, k_1^d - 1, k_2^d - 1, e_1^d).\]
Denote by $\bar{U}_q$ the quotien algebra of the Hopf superalgebra $U_q$ by the ideal $I=(f_1^d, k_1^d - 1, k_2^d - 1, e_1^d)$.
\end{definition}

Next we proof

\begin{restatable}{proposition}{strHUU}
	\label{proposition:strHUU}
	The superalgebra $\bar{U}_q$ has a Hopf superalgebra structure such that the canonical projection from $U_q$ to $\bar{U}_q$
is a Hopf superalgebra morphism.
\end{restatable}

\begin{restatable}{theorem}{basisUU}
	\label{theorem:basisUU}
	Elements of the set
	\[ \{f_1^wf_3^sf_2^lk_1^ik_2^je_1^re_3^he_2^t | \; 0 \le w,i,j,r \le d-1, \; l,s,t,h \in \{0,1\} \} \]
	additively generate a basis of $\bar{U}_q$.
\end{restatable}

\subsection{Negative Definite Borel Subalgebra of $\bar{U}_q$ and its Dual}

\begin{definition} Negative definite Borel subalgebra $B_q^-$ is a subalgebra in superalgebra $\bar{U}_q$, which is linearly
generated by elements of set
	\[ \{ f_1^wf_3^sf_2^lk_1^ik_2^j | \; 0 \le w,i,j \le d-1, \; s,l \in \{0,1\} \}. \]
\end{definition}

\begin{definition} \label{PDBSdef} Positive definite Borel subalgebra $B_q^+$ is a subalgebra in superalgebra $\bar{U}_q$, which
is linearly generated by elements of set
	\[ \{ k_1^ik_2^je_1^re_3^he_2^t | \; 0 \le i,j,r < d, \; t,h \in \{0,1\} \}. \]
\end{definition}

\begin{restatable}{proposition}{subH}
	\label{proposition:subH}
	Negative and positive definite Borel subalgebras $B_q^-$ and $B_q^+$ are a Hopf subsuperalgebras in $\bar{U}_q$.
\end{restatable}

We know from Section \ref{PrelRes} that $X=((B_q^+)^{op})^*$ is a Hopf superalgebra, where
\[ X = ((B_q^{+})^{*},\Delta_{B_q^{+}}^{*} \circ \lambda_{B_q^{+},B_q^{+}},\epsilon_{B_q^{+}}^{*} \circ
\chi,\lambda_{B_q^{+},B_q^{+}}^{-1} \circ (\mu_{B_q^{+}} \circ \tau_{B_q^{+},B_q^{+}})^{*},\chi^{-1} \circ
\eta_{B_q^{+}}^{*},(S_{B_q^{+}}^{-1})^{*},S_{B_q^{+}}^{*}).  \]

We consider the dual basis of $X$ to the basis of $B_q^{+}$:
\[ \{ (k_1^ik_2^je_1^re_3^he_2^t)^{*} | \; 0 \le i,j,r < d, \; t,h \in \{0,1\} \}. \]
Therefore we have
\[ (k_1^ik_2^je_1^re_3^he_2^t)^{*}(k_1^{i'}k_2^{j'}e_1^{r'}e_3^{h'}e_2^{t'}) = [i=i',j=j',r=r',h=h',t=t'] 1_F. \]

$\epsilon_{B_q^+}$ is an unit in $X$ and, moreover,
\[ \epsilon_{B_q^+} = \sum_{0 \le v,p \le d-1} (k_1^v k_2^p)^*. \]
Define elements in $X$:
\[ \alpha_{k_1} = \sum_{0 \le v,p \le d-1} q^{-2v+p} (k_1^v k_2^p)^*, \; \alpha_{k_2} = \sum_{0 \le v,p \le d-1} q^{v} (k_1^v
k_2^p)^*, \]
\[ \alpha_{e_1} = -(q-q^{-1})^{-1} \sum_{0 \le v,p \le d-1} q^{2v-p} (k_1^v k_2^p e_1)^{*}, \; \alpha_{e_2} = (q-q^{-1})^{-1}
\sum_{0 \le v,p \le d-1} q^{-v} (k_1^v k_2^p e_2)^{*}, \]
\[ \alpha_{e_3} = (q - q^{-1})^{-1} \sum_{0 \le v_1,v_2,p_1,p_2 \le d-1} (q^{v_1-p_1} (k_1^{v_1} k_2^{p_1} e_1 e_2)^{*} +
q^{v_2-p_2} (k_1^{v_2} k_2^{p_2} e_3)^{*}). \]

\begin{restatable}{lemma}{Xmultip}\label{lm:mulX}
	We have for all $0 \le i_1,j_1,r_1 < d, \; t_1,h_1 \in \{0,1\}$
	\[ \alpha_{e_1}^{r_1} \alpha_{e_3}^{h_1} \alpha_{e_2}^{t_1} \alpha_{k_1}^{i_1} \alpha_{k_2}^{j_1} = \]
	\[ = \sum_{0 \le v,p \le d-1} (-1)^{r_1} (q-q^{-1})^{-r_1} q^{(r_1-i_1)(2v-p) + \frac{r_1(r_1-1)}{2} + j_1v} [r_1]! (k_1^v
k_2^p e_1^{r_1})^{*} + \]
	\[ + (-1)^{r_1} (q-q^{-1})^{-r_1-1} q^{r_1(2v-p+1) + \frac{r_1(r_1-1)}{2} + 2i_1(-2v+p) + j_1(2v+1) -i_1 + v - p } [r_1+1]!
(k_1^v k_2^p e_1^{r_1+1} e_2 )^{*} + \]
	\[ + (-1)^{r_1} (q-q^{-1})^{-r_1-1} q^{r_1(2v-p+2) + \frac{r_1(r_1-1)}{2} + 2i_1(-2v+p) + j_1(2v+1) -i_1 + v - p} [r_1]!
(k_1^v k_2^p e_1^{r_1} e_3 )^{*} + \]
	\[ + (-1)^{r_1} (q-q^{-1})^{-r_1-1} q^{r_1(2v-p)+\frac{r_1(r_1-1)}{2} + i_1(-2v+p) + j_1v -v} [r_1]! (k_1^{v} k_2^{p}
e_1^{r_1} e_2 )^{*} + \]
	\[ + (q-q^{-1}) (-1)^{r_1} (q-q^{-1})^{-r_1-1} q^{r_1(2v-p)+\frac{r_1(r_1-1)}{2}+ i_1(-2v+p) + j_1v -v -1} [r_1]! (k_1^{v}
k_2^{p} e_1^{r_1-1} e_3 )^{*} + \]
	\[ +  (-1)^{r_1+1} (q-q^{-1})^{-r_1-2} q^{r_1(2v-p) + \frac{r_1(r_1-1)}{2} + i_1(-2v+p)+j_1v - v - 2} [r_1]! (k_1^{v}
k_2^{p} e_1^{r_1} e_3 e_2 )^{*}. \]
\end{restatable}

\begin{restatable}{proposition}{subBF}
	\label{pr:baisX}
	The following relations hold in the Hopf superalgebra $X$:
	\[ \alpha_{e_1}^d = 0, \; \alpha_{e_3} \alpha_{e_1} = q^{-1} \alpha_{e_1} \alpha_{e_3}, \]
	\[ \alpha_{e_2} \alpha_{e_1} = q \alpha_{e_1} \alpha_{e_2} + \alpha_{e_3}, \; \alpha_{k_1} \alpha_{e_1} = q^{-2}
\alpha_{e_1} \alpha_{k_1}, \]
	\[ \alpha_{k_2} \alpha_{e_1} = q \alpha_{e_1} \alpha_{k_2}, \; \alpha_{e_3}^2 = 0, \]
	\[ \alpha_{e_2} \alpha_{e_3} = - q \alpha_{e_3} \alpha_{e_2}, \; \alpha_{k_1} \alpha_{e_3} = q^{-1} \alpha_{e_3}
\alpha_{k_1}, \]
	\[ \alpha_{k_2} \alpha_{e_3} = q \alpha_{e_3} \alpha_{k_2}, \; \alpha_{e_2}^2 = 0, \; \alpha_{k_1} \alpha_{e_2} = q
\alpha_{e_2} \alpha_{k_1}, \; \alpha_{k_2} \alpha_{e_2} = \alpha_{e_2} \alpha_{k_2}, \]
	\[ \alpha_{k_1}^d=1_X, \; \alpha_{k_2}^d=1_X, \; \alpha_{k_2} \alpha_{k_1} = \alpha_{k_1} \alpha_{k_2}, \]
	\[ \Delta_{X} ( \alpha_{k_1} ) = \alpha_{k_1} \otimes \alpha_{k_1}, \; \Delta_{X} ( \alpha_{k_2} ) = \alpha_{k_2} \otimes
\alpha_{k_2}, \]
	\[  \Delta_{X} ( \alpha_{e_1} ) = \alpha_{e_1} \otimes \alpha_{k_1}^{-1} + 1_X \otimes \alpha_{e_1}, \; \Delta_{X} (
\alpha_{e_2} ) = \alpha_{e_2} \otimes \alpha_{k_2}^{-1} + 1_X \otimes \alpha_{e_2} , \]
	\[ \Delta_{X} ( \alpha_{e_3} ) = 1_X \otimes \alpha_{e_3} + \alpha_{e_3} \otimes \alpha_{k_1}^{-1} \alpha_{k_2}^{-1} +
(q^{-1}-q) \alpha_{e_2} \otimes \alpha_{e_1} \alpha_{k_2}^{-1}, \]
	\[ \epsilon_{X}(\alpha_{k_1}) = \alpha_{k_1}(1_{B^+_q}) = 1_F, \; \epsilon_{X}(\alpha_{k_2}) = \alpha_{k_1}(1_{B^+_q}) =
1_F, \]
	\[ \epsilon_{X}(\alpha_{e_1}) = \alpha_{e_1}(1_{B_q^{+}}) = 0, \; \epsilon_X(\alpha_{e_2}) = \alpha_{e_2}(1_{B_q^{+}}) = 0,
\; \epsilon_{X}(\alpha_{e_3}) = \alpha_{e_3}(1_{B_q^{+}}) = 0, \]
	\[ S_{X}(\alpha_{k_1}) = \alpha_{k_1}^{d-1}, \; S_{X}(\alpha_{k_2}) = \alpha_{k_2}^{d-1}, \; S_{X}(\alpha_{e_1}) = -
\alpha_{e_1} \alpha_{k_1}, \]
	\[S_X(\alpha_{e_2}) = - \alpha_{e_2} \alpha_{k_2}, \; S_X(\alpha_{e_3}) = (q-q^3) \alpha_{e_1} \alpha_{e_2} \alpha_{k_1}
\alpha_{k_2} - q^{2} \alpha_{e_3} \alpha_{k_1} \alpha_{k_2}. \]
	Moreover, the set $\{ \alpha_{e_1}^{w} \alpha_{e_3}^{s} \alpha_{e_2}^{l} \alpha_{k_1}^{i} \alpha_{k_2}^{j} | \; 0 \le w,i,j
\le d-1, \; s,l \in \{0,1\} \}$ forms a basis of $X$.
\end{restatable}

\begin{restatable}{lemma}{Xeq} \label{lm:Xeq}
	The following relations hold in the Hopf superalgebra $X$ for all $v \in \{1,2\}$:
	\[ \alpha_{k_v}(k_1^{-1} ? k_1) = \alpha_{k_v}, \; \alpha_{e_1}(k_1^{-1} ? k_1) = q^{-2} \alpha_{e_1}, \;
\alpha_{e_2}(k_1^{-1} ? k_1) = q \alpha_{e_2}, \; \alpha_{e_3}(k_1^{-1} ? k_1) = q^{-1} \alpha_{e_3}, \]
	\[ \alpha_{k_v}(k_2^{-1} ? k_2) = \alpha_{k_v},\; \alpha_{e_1}(k_2^{-1} ? k_2) = q \alpha_{e_1}, \; \alpha_{e_2}(k_2^{-1} ?
k_2) = \alpha_{e_2}, \; \alpha_{e_3}(k_2^{-1} ? k_2) = q \alpha_{e_3}, \]
	\[ \alpha_{k_v}(?e_1) = 0, \; \alpha_{e_1}(?e_1) = -(q-q^{-1})^{-1} \alpha_{k_1}^{-1}, \; \alpha_{e_2}(?e_1) = 0, \;
\alpha_{e_3}(?e_1) = 0, \]
	\[ \alpha_{k_v}(?k_1) = ([v=1]q^{-2} + [v=2]q) \alpha_{k_v}, \; \alpha_{e_1}(?k_1) = \alpha_{e_1}, \; \alpha_{e_2}(?k_1) =
\alpha_{e_2}, \; \alpha_{e_3}(?k_1) = \alpha_{e_3}, \]
	\[ \alpha_{k_v}(k_1^{-1} e_1 ? k_1) = 0, \alpha_{e_1}(k_1^{-1} e_1 ? k_1) = - (q-q^{-1})^{-1} q^{-2} 1_X, \;
\alpha_{e_2}(k_1^{-1} e_1 ? k_1) = 0, \alpha_{e_3}(k_1^{-1} e_1 ? k_1) = q^{-1} \alpha_{e_2}, \]
	\[ \alpha_{k_v}(?e_2) = 0, \; \alpha_{e_1}(?e_2) = 0, \; \alpha_{e_2}(?e_2) = (q-q^{-1})^{-1} \alpha_{k_2}^{-1}, \;
\alpha_{e_3}(?e_2) = - \alpha_{e_1} \alpha_{k_2}^{-1}, \]
	\[ \alpha_{k_v}(?k_2) = q^{[v=1]} \alpha_{k_v}, \; \alpha_{e_1}(?k_2) = \alpha_{e_1}, \; \alpha_{e_2}(?k_2) = \alpha_{e_2},
\; \alpha_{e_3}(?k_2) = \alpha_{e_3}, \]
	\[ \alpha_{k_v}(k_2^{-1} e_2 ? k_2) = 0, \; \alpha_{e_1}(k_2^{-1} e_2 ? k_2) = 0, \; \alpha_{e_2}(k_2^{-1} e_2 ? k_2) =
(q-q^{-1})^{-1} 1_X, \; \alpha_{e_3}(k_2^{-1} e_2 ? k_2) = 0, \]
	\[ \alpha_{k_v}(k_2^{-1}e_2 ?k_2e_1) = 0, \; \alpha_{e_1}(k_2^{-1}e_2 ?k_2e_1) = 0, \; \alpha_{e_2}(k_2^{-1}e_2 ?k_2e_1) =
0, \; \alpha_{e_3}(k_2^{-1}e_2 ?k_2e_1) = 0, \]
	\[ \alpha_{k_v}(k_2^{-1}e_2 ?k_1k_2) = 0, \; \alpha_{e_1}(k_2^{-1}e_2 ?k_1k_2) = 0, \; \alpha_{e_2}(k_2^{-1}e_2 ?k_1k_2) =
(q-q^{-1})^{-1} 1_X, \; \alpha_{e_3}(k_2^{-1}e_2 ?k_1k_2) = 0, \]
	\[ \alpha_{k_v}(?k_2e_1) = 0, \; \alpha_{e_1}(?k_2e_1) = - (q-q^{-1})^{-1} q^{-1} \alpha_{k_1}^{-1}, \;
\alpha_{e_2}(?k_2e_1) = 0, \; \alpha_{e_3}(?k_2e_1) = 0, \]
	\[ \alpha_{k_v}(?e_3) = 0, \; \alpha_{e_1}(?e_3) = 0, \; \alpha_{e_2}(?e_3) = 0, \; \alpha_{e_3}(?e_3) = (q-q^{-1})^{-1}
\alpha_{k_1}^{-1} \alpha_{k_2}^{-1}, \]
	\[ \alpha_{k_v}(?k_1k_2) = ([v=1] q^{-1} + [v=2] q) \alpha_{k_v}, \; \alpha_{e_1}(?k_1k_2) = \alpha_{e_1}, \;
\alpha_{e_2}(?k_1k_2) = \alpha_{e_2}, \; \alpha_{e_3}(?k_1k_2) = \alpha_{e_3}, \]
	\[ \alpha_{k_v}(k_1^{-1} k_2^{-1} e_1 e_2?k_1k_2) = 0, \; \alpha_{e_1}(k_1^{-1} k_2^{-1} e_1 e_2?k_1k_2) = 0, \]
	\[ \alpha_{e_2}(k_1^{-1} k_2^{-1} e_1 e_2?k_1k_2) = 0, \; \alpha_{e_3}(k_1^{-1} k_2^{-1} e_1 e_2?k_1k_2) = (q-q^{-1})^{-1}
1_X, \]
	\[ \alpha_{k_v}(k_1^{-1} k_2^{-1} e_3?k_1k_2) = 0, \; \alpha_{e_1}(k_1^{-1} k_2^{-1} e_3?k_1k_2) = 0, \]
	\[ \alpha_{e_2}(k_1^{-1} k_2^{-1} e_3?k_1k_2) = 0, \; \alpha_{e_3}(k_1^{-1} k_2^{-1} e_3?k_1k_2) = (q-q^{-1})^{-1} 1_X. \]
\end{restatable}

\subsection{Quantum Double of $\bar{U}_q$}

We now construct the quantum double $D:=D(B_q^{+})$. By definition \ref{QDDef}, the set
\[ \{ \alpha_{e_1}^{w} \alpha_{e_3}^{s} \alpha_{e_2}^{l} \alpha_{k_1}^{i_1} \alpha_{k_2}^{j_1} \otimes
k_1^{i_2}k_2^{j_2}e_1^{r}e_3^{h}e_2^{t} |  \; 0 \le w,i_1,i_2,j_1,j_2,r \le d-1, \; l,s,t,h \in \{0,1\} \} \]
is a basis of $D$.

\begin{restatable}{proposition}{DF}
	The following relations hold in $D$:
	\[(1_X \otimes k_v) (\alpha_{k_1} \otimes 1_{B_q^+}) = (\alpha_{k_1} \otimes 1_{B_q^+}) (1_X \otimes k_v), \;
	(1_X \otimes e_1) (\alpha_{k_1} \otimes 1_{B_q^+}) = q^{-2} (\alpha_{k_1} \otimes 1_{B_q^+}) (1_X \otimes e_1),\]
	\[(1_X \otimes e_2) (\alpha_{k_1} \otimes 1_{B_q^+}) =  (\alpha_{k_1} \otimes 1_{B_q^+}) (1_X \otimes e_2),\; (1_X \otimes
e_3) (\alpha_{k_1} \otimes 1_{B_q^+}) = q^{-1} (\alpha_{k_1} \otimes 1_{B_q^+}) (1_X \otimes e_3), \]
	\[ (1_X \otimes k_v) (\alpha_{k_2} \otimes 1_{B_q^+}) =(\alpha_{k_2} \otimes 1_{B_q^+})(1_X \otimes k_v), \; (1_X \otimes
e_1) (\alpha_{k_2} \otimes 1_{B_q^+}) = q (\alpha_{k_2} \otimes 1_{B_q^+}) (1_X \otimes e_1),\]
	\[(1_X \otimes e_2) (\alpha_{k_2} \otimes 1_{B_q^+}) = (\alpha_{k_2} \otimes 1_{B_q^+})(1_X \otimes e_2), \; (1_X \otimes
e_3) (\alpha_{k_2} \otimes 1_{B_q^+}) = q (\alpha_{k_2} \otimes 1_{B_q^+}) (1_X \otimes e_3),\]
	\[(1_X \otimes k_1) (\alpha_{e_1} \otimes 1_{B_q^+}) = q^{-2} \alpha_{e_1} \otimes k_1, \; (1_X \otimes k_2) (\alpha_{e_1}
\otimes 1_{B_q^+}) = q \alpha_{e_1} \otimes k_2,\]
	\[(1_X \otimes e_1) (\alpha_{e_1} \otimes 1_{B_q^+}) = - (q-q^{-1})^{-1} \alpha_{k_1}^{-1} \otimes 1_{B_q^+} + \alpha_{e_1}
\otimes e_1 + (q-q^{-1})^{-1} 1_{X} \otimes k_1,\]
	\[(1_X \otimes e_2) (\alpha_{e_1} \otimes 1_{B_q^+}) = \alpha_{e_1} \otimes e_2, \; (1_X \otimes e_3) (\alpha_{e_1} \otimes
1_{B_q^+}) = - q^{-1} \alpha_{k_1}^{-1} \otimes e_2 + \alpha_{e_1} \otimes e_3,\]
	\[(1_X \otimes k_1) (\alpha_{e_2} \otimes 1_{B^{+}_q}) = q \alpha_{e_2} \otimes k_1, \; (1_X \otimes k_2)(\alpha_{e_2}
\otimes 1_{B_q^{+}}) = \alpha_{e_2} \otimes e_1,\]
	\[(1_X \otimes e_2)(\alpha_{e_2} \otimes 1_{B_q^{+}}) = - (q-q^{-1})^{-1} \alpha_{k_2}^{-1} \otimes 1_{B_q^{+}} -
\alpha_{e_2} \otimes e_2 + (q-q^{-1})^{-1} 1_{X} \otimes k_2,\]
	\[(1_X \otimes e_3)(\alpha_{e_2} \otimes 1_{B_q^{+}}) = 1_{X} \otimes k_2 e_1 - \alpha_{e_2} \otimes e_3, \; (1_X \otimes
k_1) (\alpha_{e_3} \otimes 1_{B_q^{+}}) = q^{-1} \alpha_{e_3} \otimes k_1,\]
	\[ (1_X \otimes k_2) (\alpha_{e_3} \otimes 1_{B_q^{+}}) = q \alpha_{e_3} \otimes k_2, \; (1_X \otimes e_1) (\alpha_{e_3}
\otimes 1_{B_q^{+}}) = \alpha_{e_3} \otimes e_1 - q \alpha_{e_2} \otimes k_1,\]
	\[ (1_X \otimes e_2) (\alpha_{e_3} \otimes 1_{B_q^{+}}) = \alpha_{e_1} \alpha_{k_2}^{-1} \otimes 1_{B_q^{+}} - \alpha_{e_3}
\otimes e_2,\]	
	\[ (1_X \otimes e_3) (\alpha_{e_3} \otimes 1_{B_q^+}) = - (q-q^{-1})^{-1} \alpha_{k_1}^{-1} \alpha_{k_2}^{-1} \otimes
1_{B_q^{+}} - \alpha_{e_3} \otimes e_3 + (q-q^{-1})^{-1} 1_X  \otimes k_1 k_2. \]
\end{restatable}

\begin{restatable}{theorem}{thisom} \label{th:specialisom}
	Let $\chi: D \to \bar{U}_q$ be a superspace morphism determined by
	\[ \chi( \alpha_{e_1}^{w} \alpha_{e_3}^{s} \alpha_{e_2}^{l} \alpha_{k_1}^{i_1} \alpha_{k_2}^{j_1} \otimes
k_1^{i_2}k_2^{j_2}e_1^{r}e_3^{h}e_2^{t} ) = f_1^wf_3^sf_2^lk_1^{i_1+i_2} k_2^{j_1+j_2} e_1^re_3^he_2^t, \]
	where $0 \le w,i_1,i_2,j_1,j_2,r \le d-1, \; l,s,t,h \in \{0,1\}$. Then $\chi$ is the epimorphism of Hopf superalgebras.
\end{restatable}

\begin{restatable}{corollary}{braided} \label{Bruq}
	The Hopf superalgebra $\bar{U}_q$ is braided.
\end{restatable}

\subsection{Multiplicative formula for universal $R$-matrix of $\bar{U}_q$}

Denote by
\[ exp_{q}(x) := \sum_{n=0}^{\infty} \frac{x^{n}}{(n)_{q}!}, \]
where for all $k \in \mathbb{N}$ we set $(k)_{q} := \frac{q^{k}-1}{q-1}$ and $(0)_{q}!:=1$, $(n)_{q}!:=(1)_{q} (2)_{q} ...
(n)_{q},$ if $n \in \mathbb{Z}_{+}$.
Consider the following expressions:
\[ \tilde{R} = exp_{q^2}( (q-q^{-1}) e_3 \otimes f_3 ) exp_{q^2} ((q-q^{-1}) e_2 \otimes f_2 ) exp_{q^2}((-1) (q-q^{-1}) e_1
\otimes f_1 ) \times \]
\[ \times exp_{q^2}( (-1) (q^{2}-1) (q-q^{-1})^{2} e_3 e_2 \otimes f_3 f_2 ), \]
\[ K = d^{-2} \sum_{0 \le i_1, j_1, i_2, j_2 \le d-1} q^{i_1(2i_2-j_2)-j_1i_2} k_1^{i_2} k_2^{j_2} \otimes k_1^{i_1} k_2^{j_1}.
\]

\begin{restatable}{theorem}{univeralR}\label{URBUq}
	The multiplicative formula of universal $R$-matrix of $\bar{U}_q$ is given by
	\[ \bar{R} = \tilde{R} K. \]
\end{restatable}

\section{Parameterized Family of Centralizer Algebras}
\label{PFCA}

In this section we use the universal $R$-matrix of $\bar{U}_q$ to investigate a parameterized family of centralizer algebras
$L_{n,\mu}$ of $\bar{U}_q$ type A representations \cite{AbArBa}. In particular, we construct a bases for algebras $L_{n,\mu}$,
where $n \le 4$. We also give multiplication laws for $L_{3,\mu}$. We examine structure of centralizer algebras in general
case.

We consider typical type $A$ representations described in \cite{AbArBa}, \cite{CAGE}. In particular, we use the simple highest
weight module $V_{\mu}:=V(1,0,\mu)$, where $\mu \in \Bbbk/d\mathbb{Z}$, with a highest weight vector $w_{0,0,0}$ such that
\[ e_1 w_{0,0,0} = 0, \; e_2 w_{0,0,0} = 0, \]
\[ k_1 w_{0,0,0} = w_{0,0,0}, \; k_2 w_{0,0,0} = q^{\mu} w_{0,0,0}. \]
As $V_{\mu}$ is simple we require $[\mu][1+\mu] \ne 0$.
$V_{\mu}$ has the basis
$ \{ w_{0,\sigma,\rho} | \sigma, \rho \in \{0,1\} \} $
whose action of $U_q$ is given by:
\[ k_1 w_{0,\sigma,\rho} = q^{\rho-\sigma} w_{0,\sigma,\rho}, \; k_2 w_{0,\sigma,\rho} = q^{\mu+\sigma} w_{0,\sigma,\rho}, \]
\[ f_1 w_{0,\sigma,\rho} = [\sigma=0, \rho=1] (-1) q^{-1} w_{0,1,0}, \; f_2 w_{0,\sigma,\rho} = [\rho=0] w_{0,\sigma,1}, \]
\[ f_3 w_{0,\sigma,\rho} = [\sigma=0] (-1)^{\rho} q^{-\rho} w_{0,1,\rho}, \; f_3 f_2 w_{0,\sigma,\rho} = [\sigma=\rho=0] (-1)
q^{-1} w_{0,1,1}, \]
\[ e_1 w_{0,\sigma,\rho} = [\sigma=1, \rho=0] (-1) q w_{0,0,1}, \; e_2 w_{0,\sigma,\rho} = [\rho=1] [\mu+\sigma] w_{0,\sigma,0},
\]
\[ e_3 w_{0,\sigma,\rho} = [\sigma=1] (-1)^{\rho} q^{\rho} [\mu+\rho] w_{0,0,\rho}, \; e_3 e_2 w_{0,\sigma,\rho} =
[\sigma=\rho=1] [\mu+1] [\mu] w_{0,0,0}. \]

Notice that simple $U_{q}$-module $M$ is also a simple $\bar{U}_q$-module if central elements
$k_1^{d}-1,k_2^{d}-1,e_1^{d},f_1^{d}$ act trivially on $M$. Let us check that central elements
$k_1^{d}-1,k_2^{d}-1,e_1^{d},f_1^{d}$ act trivially on $V_{\mu}$:
\[ (k_1^{d} - 1) w_{0, \sigma, \rho} = (q^{-d \sigma +d \rho} - 1) w_{0, \sigma, \rho} = (1-1) w_{0, \sigma, \rho} = 0, \]
\[ (k_2^{d} - 1) w_{0, \sigma, \rho} = (q^{d \sigma + d \mu} - 1) w_{0, \sigma, \rho} = 0, \]
\[ f_1^{d} w_{p, \sigma, \rho} = 0, \; e_1^{d} w_{p, \sigma, \rho} = 0. \]

Consider the dictionary ordered basis for $V_{\mu} \otimes V_{\mu}$: $\{ w_{0,i_1,j_1} \otimes w_{0,i_2,j_2} | 0 \le
i_1,j_1,i_2,j_2 \le 1 \}$ where $(0,0) < (1,0) < (0,1) < (1,1)$. The basis for $V_{\mu} \otimes V_{\mu} \otimes V_{\mu}$ is
ordered in the same manner.

We use Theorem \ref{URBUq} to find image of the universal $R$-matix $\bar{R}$ in $V_{\mu} \otimes V_{\mu}$:
\[ \bar{R}(w_{0,\sigma_{1},\rho_1} \otimes w_{0,\sigma_{2},\rho_2}) = q^{ 2 \mu^2 + \sigma_1 (\rho_2+ \mu  ) + \rho_1
(\sigma_2+\mu) + \mu (\sigma_{2} + \rho_2 ) } ( w_{0,\sigma_{1},\rho_1} \otimes w_{0,\sigma_{2},\rho_2} + \]
\[ + [\sigma_1=\rho_1=1,\sigma_2=\rho_2=0] q (q-q^{-1})^{2} [\mu+1] [\mu] w_{0,0,0} \otimes w_{0,1,1} + \]
\[ + [\sigma_1=1, \rho_1=0,\sigma_2=0, \rho_2=1] (-1) (q-q^{-1}) w_{0,0,1} \otimes w_{0,1,0} + \]
\[ + [\rho_1=1,\rho_2=0] (-1)^{\sigma_{1}+1} (q-q^{-1}) [\mu+\sigma_1] w_{0,\sigma_{1},0} \otimes  w_{0,\sigma_{2},1} + \]
\[ + [\sigma_1=1, \rho_1=0,\sigma_2=0, \rho_2=1] (q-q^{-1})^{2} [\mu] w_{0,0,0} \otimes w_{0,1,1} + \]
\[ + [\sigma_1=1,\sigma_2=0] (-1)^{1+\rho_2} (q-q^{-1}) q^{\rho_1-\rho_2} [\mu+\rho_1] w_{0,0,\rho_1} \otimes w_{0,1,\rho_2} ).
\]

Now we can construct the $R$-matrix
\[ c^{\bar{R}}_{V_{\mu},V_{\mu}}: V_{\mu} \otimes V_{\mu} \to V_{\mu} \otimes V_{\mu}, \]
where $c^{\bar{R}}_{V_{\mu},V_{\mu}}=q^{- 2 \mu^2} \tau_{V_{\mu},V_{\mu}} \circ \bar{R}$.

\[ c^{\bar{R}}_{V_{\mu},V_{\mu}}(w_{0,0,0} \otimes w_{0,0,0}) = w_{0,0,0} \otimes w_{0,0,0}, \]
\[ c^{\bar{R}}_{V_{\mu},V_{\mu}}(w_{0,1,0} \otimes w_{0,0,0}) = q^{\mu} ( w_{0,0,0} \otimes w_{0,1,0} + (-1) (q-q^{-1}) [\mu]
w_{0,1,0} \otimes w_{0,0,0} ), \]
\[ c^{\bar{R}}_{V_{\mu},V_{\mu}}(w_{0,0,1} \otimes w_{0,0,0}) = q^{\mu} ( w_{0,0,0} \otimes w_{0,0,1} + (-1) (q-q^{-1}) [\mu]
w_{0,0,1} \otimes w_{0,0,0} ), \]
\[ c^{\bar{R}}_{V_{\mu},V_{\mu}}(w_{0,1,1} \otimes w_{0,0,0}) = q^{2\mu} ( w_{0,0,0} \otimes w_{0,1,1} + q (q-q^{-1})^{2}
[\mu+1] [\mu] w_{0,1,1} \otimes w_{0,0,0} + \]
\[ + (-1) (q-q^{-1}) [\mu+1] w_{0,0,1} \otimes w_{0,1,0} + q (q-q^{-1}) [\mu+1] w_{0,1,0} \otimes w_{0,0,1} ), \]

\[ c^{\bar{R}}_{V_{\mu},V_{\mu}}(w_{0,0,0} \otimes w_{0,1,0}) = q^{\mu} w_{0,1,0} \otimes w_{0,0,0} , \]
\[ c^{\bar{R}}_{V_{\mu},V_{\mu}}(w_{0,1,0} \otimes w_{0,1,0}) = (-1) q^{2\mu} w_{0,1,0} \otimes w_{0,1,0} , \]
\[ c^{\bar{R}}_{V_{\mu},V_{\mu}}(w_{0,0,1} \otimes w_{0,1,0}) = q^{2\mu+1} ( (-1) w_{0,1,0} \otimes w_{0,0,1} + (-1) (q-q^{-1})
[\mu] w_{0,1,1} \otimes w_{0,0,0} ),\]
\[ c^{\bar{R}}_{V_{\mu},V_{\mu}}(w_{0,1,1} \otimes w_{0,1,0}) = q^{3\mu+1} ( w_{0,1,0} \otimes w_{0,1,1} + (q-q^{-1}) [\mu+1]
w_{0,1,1} \otimes w_{0,1,0} ),\]

\[ c^{\bar{R}}_{V_{\mu},V_{\mu}}(w_{0,0,0} \otimes w_{0,0,1}) = q^{\mu} w_{0,0,1} \otimes w_{0,0,0} , \]
\[ c^{\bar{R}}_{V_{\mu},V_{\mu}}(w_{0,1,0} \otimes w_{0,0,1}) = q^{2\mu+1} ( (-1) w_{0,0,1} \otimes w_{0,1,0} +  q (q-q^{-1})
[\mu] w_{0,1,1} \otimes w_{0,0,0} + (q-q^{-1}) w_{0,1,0} \otimes w_{0,0,1} ), \]
\[ c^{\bar{R}}_{V_{\mu},V_{\mu}}(w_{0,0,1} \otimes w_{0,0,1}) = (-1) q^{2\mu} w_{0,0,1} \otimes w_{0,0,1} , \]
\[ c^{\bar{R}}_{V_{\mu},V_{\mu}}(w_{0,1,1} \otimes w_{0,0,1}) = q^{3\mu+1} ( w_{0,0,1} \otimes w_{0,1,1} + (q-q^{-1}) [\mu+1]
w_{0,1,1} \otimes w_{0,0,1} ), \]

\[ c^{\bar{R}}_{V_{\mu},V_{\mu}}(w_{0,0,0} \otimes w_{0,1,1}) = q^{2\mu} w_{0,1,1} \otimes w_{0,0,0}, \]
\[ c^{\bar{R}}_{V_{\mu},V_{\mu}}(w_{0,1,0} \otimes w_{0,1,1}) = q^{3\mu+1} w_{0,1,1} \otimes w_{0,1,0}, \]
\[ c^{\bar{R}}_{V_{\mu},V_{\mu}}(w_{0,0,1} \otimes w_{0,1,1}) = q^{3\mu+1} w_{0,1,1} \otimes w_{0,0,1}, \]
\[ c^{\bar{R}}_{V_{\mu},V_{\mu}}(w_{0,1,1} \otimes w_{0,1,1}) = q^{4\mu+2} w_{0,1,1} \otimes w_{0,1,1}. \]

We define the family of elements $c^{\bar{R}}_{V_{\mu}^i,V_{\mu}^{i+1}} \in End(V_{\mu}^{\otimes n})$:
\[ c^{\bar{R}}_{V_{\mu}^i,V_{\mu}^{i+1}} = id_{V_{\mu}}^{\otimes i-1 } \otimes c^{\bar{R}}_{V_{\mu},V_{\mu}} \otimes
id_{V_{\mu}}^{\otimes n-i-1 }, \]
where $n > 1, \; i= \{ 1,2,...,n-1 \}$.

Define the representation $\rho_{n,\mu}: \bar{U}_q \to V_{\mu}^{\otimes n}$ given by $\Delta_{\bar{U}_q}$:
\[ \rho_{n,\mu}(a) v = \Delta_{\bar{U}_q}^{n-1}(a) v, \]
where $a \in \bar{U}_q$, $v \in V_{\mu}^{\otimes n}$.

We consider the family of unital associative algebras $L_{n,\mu}$ for which representation $\gamma_{n,\mu}: L_{n,\mu} \to
End(V_{\mu}^{\otimes n})$ is faithful and, moreover, $\gamma_{n,\mu}(L_{n,\mu})=End_{\rho_{n,\mu}(U_q)}(V_{\mu}^{\otimes n})$.
$L_{n,\mu}$ is generated by elements $g_1, g_2, ..., g_{n-1}$ such that $\gamma_{n,\mu}(g_i):=
c^{\bar{R}}_{V_{\mu}^i,V_{\mu}^{i+1}}$. According to $R$-matrix properties the generation elements satisfy equations
\[ g_{j} g_{i} g_{j} = g_{i} g_{j} g_{i}, \]
for $|i-j| = 1$,
\[ g_{j} g_{i} = g_{i} g_{j}, \]
where $|i-j| \ge 2, \; i,j \in \{ 1,2,...,n-1 \}$.

We set
\[ L_{0,\mu} = L_{1,\mu} = k*1. \]

Algebra $L_{2,\mu}$ is generated by element $g_1$. Generating relation for $L_{2,\mu}$ is
\[ (g_1-1)(g_1+q^{2 \mu})(g_1-q^{4 \mu + 2}) = 0. \]
From this equation we get
\[ g_1^{-1} = - q^{-6 \mu - 2} g_1^2 + ( q^{-2 \mu} - q^{-4 \mu - 2} + q^{-6 \mu - 2} ) g_1 + ( 1 - q^{-2 \mu} + q^{-4 \mu - 2}
) 1. \]	
The basis is $ B_{2,\mu} = \{1, g_1, g_1^2\}$.

To generate bases for $L_{3,\mu}$ and $L_{4,\mu}$ we use the following method. Let us introduce the degree-lexicographic order
$\le_{deglex}$ on monomials formed by generator elements using relations: $1 <_{deglex} g_1 <_{deglex} g_2 <_{deglex} g_3$. We
exclude all monomials $a$ which can be reduced using generation relations to linear combination of elements $b$ such that $b
<_{deglex} a$. After that we add to basis elements starting from $1$ in ascending order according to $<_{deglex}$ and verify
linear independence to exclude linear dependent elements. It is possible as we investigate a faithful representation.
For $L_{4,\mu}$ we additionally require to exclude elements that contain subwords $g_1 g_3 g_2 g_2 g_3, \; g_3 g_2 g_2 g_1 g_3,
\; g_3 g_2 g_2 g_1 g_1 g_3$ in order to proof Theorem \ref{th:inductionAlgebra}.

$L_{3,\mu}$ is generated by elements $g_1$ and $g_2$. We describe here all generating relations for $L_{3,\mu}$:
\[ g_1 g_2 g_1 = g_2 g_1 g_2, \]
\[ (g_i-1)(g_i+q^{2 \mu})(g_i-q^{4 \mu + 2}) = 0, \]

\[ (g_1 + q^{2 \mu}) (g_2 + q^{4 \mu}) (g_1 - q^{4 \mu + 2} ) g_1 ( g_2 + q^{2 \mu}) = (g_1 + q^{2 \mu}) (g_2 - q^{4 \mu +2})
g_2 (g_1 + q^{4 \mu}) ( g_2 + q^{2 \mu}), \]

\[ (g_1 + q^{2 \mu} ) (g_2 + q^{2 \mu} - q^{4 \mu + 2} + q^{2 \mu + 2}) (g_2 - 1) (g_1 + q^{2 \mu} ) (g_1 + q^{-2 \mu}) = \]
\[ = (g_1 + q^{2 \mu} ) (g_1 + q^{-2 \mu}) (g_2 + q^{2 \mu} - q^{4 \mu + 2} + q^{2 \mu + 2}) (g_2 - 1) (g_1 + q^{2 \mu} ), \]

where $i=1,2$.

Now we give the basis for $L_{3,\mu}$:
\[ B_{3,\mu} = \{ 1, g_1, g_2, g_1^2,  g_1 g_2, g_2 g_1, g_2^2, g_1^2 g_2, g_1 g_2 g_1, g_1 g_2^2, g_2 g_1^2, g_2^2 g_1, \]
\[ g_1^2 g_2 g_1, g_1^2 g_2^2, g_1 g_2 g_1^2, g_1 g_2^2 g_1, g_2 g_1^2 g_2, g_2^2 g_1^2, g_1^2 g_2 g_1^2, g_1^2 g_2^2 g_1 \}.
\]

In order to describe multiplication laws we add additional equations. All of them can be deduced using generating relations
mentioned above.

\[ g_2 g_1 g_1 g_2 g_1 = g_1 g_2 g_1 g_1 g_2, \]
\[ g_2 g_1^2 g_2^2 =
(q^{2 \mu} - q^{4 \mu} q^2 + q^{6 \mu} q^2) g_2 g_1^2 +
(q^{4 \mu} q^2 - q^{2 \mu} - q^{6 \mu} q^2) g_2^2 g_1 + \]
\[ + (q^{2 \mu} - q^{4 \mu} q^2 - 1) g_1 g_2^2 g_1 +
(q^{4 \mu} q^2 - q^{2 \mu} + 1) g_2 g_1^2 g_2 +
g_1^2 g_2^2 g_1, \]
\[ g_2^2 g_1^2 g_2 =
(q^{6 \mu} - q^{4 \mu} + q^{6 \mu + 2} - q^{8 \mu + 2} ) g_1 g_2 +
(q^{4 \mu} - q^{6 \mu} - q^{6 \mu + 2} + q^{8 \mu + 2} ) g_2 g_1 +
q^{4 \mu} g_1^2 g_2 + \]
\[ + (q^{4 \mu} - q^{2 \mu} + q^{4 \mu + 2} - q^{6 \mu + 2}) g_1 g_2^2 +
(q^{2 \mu} - q^{4 \mu} - q^{4 \mu + 2} + q^{6 \mu + 2}) g_2 g_1^2 + \]
\[ + (-q^{4 \mu}) g_2^2 g_1 +
(q^{2 \mu} + q^{2 \mu + 2} - q^{4 \mu + 2} - 1) g_1^2 g_2 g_1 +
q^{2 \mu} g_1^2 g_2^2 +
(q^{4 \mu + 2} - q^{2 \mu + 2} - q^{2 \mu} + 1) g_1 g_2 g_1^2 + \]
\[ + (q^{2 \mu} - q^{4 \mu + 2} - 1) g_1 g_2^2 g_1 +
(q^{4 \mu + 2} - q^{2 \mu} + 1) g_2 g_1^2 g_2 +
(-q^{2 \mu}) g_2^2 g_1^2 +
g_1^2 g_2^2 g_1.
\]

Note that
\begin{equation}\label{eq:l3decomp}
L_{3,\mu} = \sum_{i=0}^{2} L_{2,\mu} g_2^{i} L_{2,\mu} + \Bbbk g_2 g_1^{2} g_2.
\end{equation}

Now we give the basis for $L_{4,\mu}$:
\[ B_{4,\mu} = B_{3,\mu} \cup \{ g_3 g_2^{2} g_3 \} \cup \{ a g_3^{i}, a g_3 g_2 | a \in B_{3,\mu}, i \in \{1,2\} \} \cup \]
\[ \cup \{ a g_3 g_3 g_2, a g_3 g_2^2 | a \in B_{3,\mu}, a < g_2^2 g_1^2 \} \cup \{ g_3^2 g_2^2, g_1 g_3^2 g_2^2, g_1^2 g_3^2
g_2^2, g_2 g_1 g_3^2 g_2^2 \} \cup \]
\[ \cup \{ a g_3 g_2 g_1 | a \in B_{3,\mu}, a < g_2 g_1^2 g_2 \} \cup \{ a g_3 g_3 g_2 g_1 | a \in B_{3,\mu}, a < g_2^2 g_1, a
\ne g_1 g_2^2 \} \cup \]
\[ \cup \{ a g_3 g_2 g_2 g_1 | a \in B_{3,\mu}, a < g_2^2 g_1, a \notin \{ g_1^2 g_2, g_1 g_2^2 \} \} \cup \{ g_3 g_3 g_2 g_2
g_1, g_1 g_3 g_3 g_2 g_2 g_1 \} \cup \]
\[ \cup \{ a g_3 g_2 g_1 g_1 | a \in B_{3,\mu}, a < g_1 g_2^2 \} \cup \{ g_3 g_3 g_2 g_1 g_1,	g_1 g_3 g_3 g_2 g_1 g_1, g_2 g_3
g_3 g_2 g_1 g_1, g_1 g_1 g_3 g_3 g_2 g_1 g_1 \} \cup \]
\[ \cup \{ g_3 g_2 g_1 g_1 g_2, g_2 g_3 g_2 g_1 g_1 g_2, g_3 g_3 g_2 g_1 g_1 g_2 \} \cup \{ g_3 g_2 g_2 g_1 g_1, g_2 g_3 g_2 g_2
g_1 g_1, g_3 g_3 g_2 g_2 g_1 g_1 \}. \]

Note that
\begin{equation}\label{eq:l4decomp}
L_{4,\mu} = \sum_{i=0}^{2} L_{3,\mu} g_3^{i} L_{3,\mu} + L_{1,\mu} g_3 g_2^{2} g_3.
\end{equation}

\begin{restatable}{lemma}{lmbimo} \label{lm:bimolln}
	For $n \ge 3$ we have
	\[ L_{n-1,\mu}  g_{n-1} g_{n-2}^{2} g_{n-1} L_{n-1,\mu} \subset \sum_{i=0}^{2} L_{n-1,\mu} g_{n-1}^{i} L_{n-1,\mu} +
L_{n-3,\mu} g_{n-1} g_{n-2}^{2} g_{n-1}. \]
\end{restatable}

\begin{restatable}{theorem}{thindAlg} \label{th:inductionAlgebra}
	For $n \ge 3$ we have
	\[ L_{n,\mu} = \sum_{i=0}^{2} L_{n-1,\mu} g_{n-1}^{i} L_{n-1,\mu} + L_{n-3,\mu} g_{n-1} g_{n-2}^{2} g_{n-1}. \]
\end{restatable}

\section{Proofs of Central Results}
\label{PCR}

\basisU*
\begin{proof}		
	Let $L$ be a vector superspace over field $F$ such that $L=L_0 \oplus L_1$, where $L_0=<f_1,k_1,k_2,e_1>$,
$L_1=<f_2,f_3,e_2,e_3>$. Consider a pair $(T(L),i)$ where $T(L)$ is the tensor superalgebra of the vector superspace $L$ and $i$
is the canonical inclusion of $L$ in $T(L)$. Consider equations in Definition \ref{def:u_q} and replace two last relations by
formulas \ref{e3f3eq}. Rewrite them in $T(L)$ in the following way:
	\[ a \otimes b - (-1)^{|a||b|} q^{\delta(a,b)} b \otimes a = [a,b], \]
	where $a,b \in i(X), \; [a,b] \in T(L), \; \delta: i(X) \times i(X) \to \{ -2,-1,0,1,2 \}$. Notice that for all $a,b \in
i(X)$
	\[ \delta(b,a) = - \delta(a,b) \Rightarrow q^{\delta(b,a)} = (q^{\delta(a,b)})^{-1}. \]
	Denote by $J$ a $\mathbb{Z}_{2}$-graded two-sided ideal in $T(L)$ generated by the following relations
	\[ a\otimes b-(-1)^{|a||b|}q^{\delta(a,b)}b\otimes a-[a,b], \]
	where $a,b\in i(X)$. Notice that $U_{q} \cong T(L)/J$.
	Next we identify $X$ and $i(X)$. Order the set $X$ in the following way
	\[ f_1<f_3<f_2<k_1<k_2<e_1<e_3<e_2. \]
	
	Define the index of a monomial $x_{i_1} \otimes x_{i_2} \otimes ... \otimes x_{i_n} \in T(L)$ by
	\[ ind(x_{i_1} \otimes x_{i_2} \otimes ... \otimes x_{i_n})=\sum_{j<k} \pi_{jk},\]
	where
	\[\pi_{jk} = \begin{cases} 0, & \mbox{if } x_{i_j} \leq x_{i_k}, \\ 1, & \mbox{if } x_{i_j}>x_{i_k} \end{cases}.\]
	We call a monomial standart if it has the index equal to $0$. It is possible if and only if
	\[x_{i_1} \leq x_{i_2} \leq ... \leq x_{i_n}.\]
	
	We adopt in a natural way the definition of the index on elements of $U_q$. Next we show that each element in $U_q$ is a
$F$-linear combination of unit and standart monomials. By applying the canonical superalgebra morphism $T(L) \to U_q$ we notice
that elements of form $y_1 \otimes y_2 \otimes ... \otimes y_n + J$, where $y_i \in X$ span $U_q$. Therefore it is sufficient to
prove that every monomial is a linear combination of elements of the set $G$.
	
	Let us order monomials by degrees. Monomials of equal degree we order by the indeces. The statement is trivially true for
monomials of degree one. The statement is true for monomials of degree two by Lemma \ref{lm:relonbasis}. Consider a non-standart
monomial $y_{i_1} \otimes y_{i_2} \otimes ... \otimes y_{i_n} + J$ and fix it degree $n$ and index $m$. We assume that the
statement is true for monomials of degree less than $n$ and for monomials of the same degree but of index less than $m$. Also we
assume for clarity that $y_{i_k}>y_{i_{k+1}}$. Then
	\[ y_{i_1} \otimes y_{i_2} \otimes ... \otimes y_{i_k} \otimes y_{i_{k+1}} \otimes ... \otimes y_{i_n} + J =
(-1)^{|y_{i_k}||y_{i_{k+1}}|} q^{\delta(y_{i_k},y_{i_{k+1}})} y_{i_1} \otimes y_{i_2} \otimes...y_{i_{k+1}}\otimes
y_{i_k}\otimes...\otimes y_{i_n} + \]
	\[ + y_{i_1}\otimes y_{i_2}\otimes...\otimes [y_{i_k},y_{i_{k+1}}] \otimes...\otimes y_{i_n} + J. \]
	Notice that $[y_{i_k},y_{i_{k+1}}]$ is a linear combination of standart monomials of degree one or two. Thus the first
monomial in the right side of equality has  index less than $m$, and the second monomial is a linear combination of elements of
degree less than $n$ or index less than $m$. The required result then follows by mathematical induction.
	
	We now show that elements from the set $G$ are linear independent in $U_q$. Consider a polynomial ring
$R=F[z_1,z_2,z_3,z_4,z_5,z_6,z_7,z_8]$. We endow $R$ with the structure of a superalgebra in the following way:
$deg(z_1)=deg(z_4)=deg(z_5)=deg(z_6)=0$, $deg(z_2)=deg(z_3)=deg(z_7)=deg(z_8)=1$.
	
	Let us proof that there is a superspace morphism $\theta:T(L) \to R$ which satisfies the following relations
	\[ \theta(1) = 1, \]
	\[ \theta(x_{i_1} \otimes x_{i_2} \otimes ... \otimes x_{i_n}) = z_{i_1} z_{i_2} ... z_{i_n}, \; \text{if } i_1 \le i_2 \le
... \le i_n, \]
	\[ \theta( x_{i_1} \otimes x_{i_2} \otimes ... \otimes  x_{i_k} \otimes x_{i_{k+1}} \otimes ... \otimes x_{i_n}) -
(-1)^{|x_{i_k}||x_{i_{k+1}}|} q^{\delta(x_{i_k},x_{i_{k+1}})} \theta( x_{i_1} \otimes x_{i_2} \otimes ... \otimes  x_{i_{k+1}}
\otimes x_{i_k} \otimes ... \otimes x_{i_n}) = \]
	\[ = \theta( x_{i_1} \otimes x_{i_2} \otimes ... \otimes  [x_{i_k},x_{i_{k+1}}] \otimes ... \otimes x_{i_n} ) \]
	for all $x_{i_1},x_{i_2},...,x_{i_n} \in X$ and for each $k:1 \le k < n$.
	
	Recall that $T^{0}(L) = F 1$, $T^{n}(L)= \bigotimes_{i=1}^{n} L$ for all $n \in \mathbb{N}$. Denote by $T^{n,j}(L)$ a linear
subspace $T^{n}(L)$ spanned by all monomials $x_{i_1} \otimes x_{i_2} \otimes ... \otimes x_{i_n}$, which have index less or
equal to $j$. Thus,
	\[ T^{n,0}(L) \subset T^{n,1}(L) \subset ... \subset T^{n}(L). \]
	We define $\theta:T^{0}(L) \to R$ by $\theta(1)=1$. Suppose inductively that $\theta:T^{0}(L) \oplus T^{1}(L) ... \oplus
T^{n-1}(L) \to R$ has already been defined satisfying the required conditions. We will show that $\theta$ can be extended to
$\theta:T^{0}(L) \oplus T^{1}(L) ... \oplus T^{n}(L) \to R$. We define $\theta:T^{n,0}(L) \to R$ by
	\[ \theta(x_{i_1} \otimes x_{i_2} \otimes ... \otimes x_{i_n}) = z_{i_1} z_{i_2}...z_{i_n} \]
	for standart monomials of degree $n$. We suppose $\theta:T^{n,i-1} \to R$ has already been defined, thus giving a superspace
morphism from $\theta:T^{0}(L) \oplus T^{1}(L) ... \oplus T^{n-1}(L) \oplus T^{n,i-1}(L) \to R$ satisfying the required
conditions. We wish to define $\theta:T^{n,i}(L) \to R$.
	
	Assume that the monomial $x_{i_{1}}\otimes x_{i_{2}}\otimes...\otimes x_{i_{n}}$ has the index $i\ge1$ and let $x_{i_{k}}\ge
x_{i_{k+1}}$. Then define
	\begin{equation}\label{eq:basisULinearMap}
	\theta(x_{i_{1}}\otimes...\otimes x_{i_{k}}\otimes x_{i_{k+1}}\otimes...\otimes
x_{i_{n}})=\theta(x_{i_{1}}\otimes...\otimes[x_{i_{k}},x_{i_{k+1}}]\otimes...\otimes x_{i_{n}}) +
	\end{equation}
	\[ +(-1)^{|x_{i_{k}}||x_{i_{k+1}}|}q^{\delta(x_{i_{k}},x_{i_{k+1}})}\theta(x_{i_{1}}\otimes...\otimes x_{i_{k+1}}\otimes
x_{i_{k}}\otimes...\otimes x_{i_{n}}). \]
	
	This definition is correct as both terms on the right side of the equation belong to a vector superspace
$T^{0}(L)+T^{1}(L)+...+T^{n-1}(L)+T^{n,i-1}(L)$. We show that definition \ref{eq:basisULinearMap} doesn't depend on the choise
of the pair $(x_{i_{k}},x_{i_{k+1}}), x_{i_{k}}>x_{i_{k+1}}$. Let $(x_{i_{j}},x_{i_{j+1}})$ be another pair, where
$x_{i_{j}}>x_{i_{j+1}}$. There are two different possible situations: 1. $x_{i_{j}}>x_{i_{k+1}}$, 2. $x_{i_{j}}=x_{i_{k+1}}$.
	
	1. Suppose that $x_{i_{k}}=a,x_{i_{k+1}}=b,x_{i_{j}}=c,x_{i_{j+1}}=d$. Then we can rewrite by inductive assumption the right
side of the equation \ref{eq:basisULinearMap}, starting from the pair $(a,b)$, in the following way:
	\[ \theta( x_{i_{1}} \otimes ... \otimes a \otimes b \otimes...\otimes c \otimes d \otimes ... \otimes x_{i_{n}} ) =
(-1)^{|a||b|} q^{\delta(a,b)} \theta( x_{i_{1}} \otimes ... \otimes b \otimes a \otimes ... \otimes c \otimes d \otimes ...
\otimes x_{i_{n}} ) + \]
	\[ + \theta( x_{i_{1}} \otimes ... \otimes [a,b] \otimes ... \otimes c \otimes d \otimes ... \otimes x_{i_{n}} ) =
(-1)^{|a||b|+|c||d|} q^{\delta(a,b)+\delta(c,d)} \theta( x_{i_{1}} \otimes ... \otimes b \otimes a \otimes ... \otimes d \otimes
c \otimes ... \otimes x_{i_{n}} ) + \]
	\[ + (-1)^{|a||b|} q^{\delta(a,b)} \theta( x_{i_{1}} \otimes ... \otimes b \otimes a \otimes ... \otimes [c,d] \otimes ...
\otimes x_{i_{n}} ) + \]
	\[ + (-1)^{|c||d|} q^{\delta(c,d)} \theta( x_{i_{1}} \otimes ... \otimes [a,b] \otimes ... \otimes d \otimes c \otimes ...
\otimes x_{i_{n}} ) + \theta( x_{i_{1}} \otimes ... \otimes [a,b] \otimes ... \otimes [c,d] \otimes ... \otimes x_{i_{n}} ) .
\]
	
	If we start from the pair $(c,d)$, we get
	\[ \theta( x_{i_{1}} \otimes ... \otimes a \otimes b \otimes ... \otimes c \otimes d \otimes ... x_{i_{n}} ) = (-1)^{|c||d|}
q^{\delta(c,d)} \theta( x_{i_{1}} \otimes ... \otimes a \otimes b \otimes ... \otimes d \otimes c \otimes ... \otimes x_{i_{n}}
) + \]
	\[ + \theta( x_{i_{1}} \otimes ... \otimes a \otimes b \otimes ... \otimes [c,d] \otimes ... \otimes x_{i_{n}} ) =
(-1)^{|a||b|+|c||d|} q^{\delta(a,b)+\delta(c,d)} \theta( x_{i_{1}} \otimes ... \otimes b \otimes a \otimes ... \otimes d \otimes
c \otimes ... \otimes x_{i_{n}} ) + \]
	\[ + (-1)^{|c||d|} q^{\delta(c,d)} \theta( x_{i_{1}} \otimes ... \otimes [a,b] \otimes ... \otimes d \otimes c \otimes ...
\otimes x_{i_{n}} ) + \]
	\[ + (-1)^{|a||b|} q^{\delta(a,b)} \theta( x_{i_{1}} \otimes ... \otimes b \otimes a \otimes ... \otimes [c,d] \otimes ...
\otimes x_{i_{n}} ) + \theta( x_{i_{1}} \otimes ... \otimes [a,b] \otimes ... \otimes [c,d] \otimes ... \otimes x_{i_{n}} ). \]
	
	We see that the result is the same as in the first case.
	
	2. Suppose that $x_{i_{k}}=a,x_{i_{k+1}}=b,x_{i_{k+2}}=c$. Then we can rewrite by inductive assumption the right side of the
equation \ref{eq:basisULinearMap}, starting from the pair $(a,b)$, in the following way:	
	\[ \theta( x_{i_{1}} \otimes ... \otimes a \otimes b \otimes c \otimes ... \otimes x_{i_{n}} ) = (-1)^{|a||b|}
q^{\delta(a,b)} \theta( x_{i_{1}} \otimes ... \otimes b \otimes a \otimes c \otimes ... \otimes x_{i_{n}} ) + \theta( x_{i_{1}}
\otimes ... \otimes [a,b] \otimes c \otimes ... \otimes x_{i_{n}} )  = \]
	\[ = (-1)^{|a||b|+|a||c|} q^{\delta(a,b) + \delta(a,c)} \theta ( x_{i_{1}} \otimes ... \otimes b \otimes c \otimes a \otimes
... \otimes x_{i_{n}} )  + \]
	\[ + (-1)^{|a||b|} q^{\delta(a,b)} \theta( x_{i_{1}} \otimes ... \otimes b \otimes [a,c] \otimes ... \otimes x_{i_{n}} ) +
\theta (x_{i_{1}} \otimes ... \otimes [a,b] \otimes c \otimes ... \otimes x_{i_{n}}) = \]
	\[ = (-1)^{|a||b|+|a||c| + |b||c|} q^{\delta(a,b) + \delta(a,c) + \delta(b,c)} \theta (x_{i_{1}} \otimes ... \otimes c
\otimes b \otimes a \otimes ... \otimes x_{i_{n}}) + \]
	\[ + (-1)^{|a||b|+|a||c|} q^{\delta(a,b) + \delta(a,c)} \theta( x_{i_{1}} \otimes ... \otimes [b,c] \otimes a \otimes ...
\otimes x_{i_{n}} ) + \]
	\[ + (-1)^{|a||b|} q^{\delta(a,b)} \theta( x_{i_{1}} \otimes ... \otimes b \otimes [a,c] \otimes ... \otimes x_{i_{n}})  +
\theta( x_{i_{1}} \otimes ... \otimes [a,b] \otimes c \otimes ... \otimes x_{i_{n}}) . \]
	
	If we start from the pair $(b,c)$, we get
	\[ \theta( x_{i_{1}} \otimes ... \otimes a \otimes b \otimes c \otimes ... \otimes x_{i_{n}} ) = (-1)^{|b||c|}
q^{\delta(b,c)} \theta( x_{i_{1}} \otimes ... \otimes a \otimes c \otimes b \otimes ... \otimes x_{i_{n}} ) + \theta( x_{i_{1}}
\otimes ... \otimes a \otimes [b,c] \otimes ... \otimes x_{i_{n}} ) = \]
	\[ = (-1)^{|a||c|+|b||c|} q^{\delta(a,c) + \delta(b,c)} \theta ( x_{i_{1}} \otimes ... \otimes c \otimes a \otimes b \otimes
... \otimes x_{i_{n}} ) + (-1)^{|b||c|} q^{\delta(b,c)} \theta ( x_{i_{1}} \otimes ... \otimes [a,c] \otimes b \otimes ...
\otimes x_{i_{n}} ) + \]
	\[ + \theta( x_{i_{1}} \otimes ... \otimes a \otimes [b,c] \otimes ... \otimes x_{i_{n}} ) = \]
	\[ = (-1)^{|a||b|+|a||c| + |b||c|} q^{\delta(a,b) + \delta(a,c) + \delta(b,c)} \theta ( x_{i_{1}} \otimes ... \otimes c
\otimes b \otimes a \otimes ... \otimes x_{i_{n}} ) + \]
	\[ + (-1)^{|a||c|+|b||c|} q^{\delta(a,c) + \delta(b,c)} \theta ( x_{i_{1}} \otimes ... \otimes c \otimes [a,b] \otimes ...
\otimes x_{i_{n}} ) + (-1)^{|b||c|} q^{\delta(b,c)} \theta ( x_{i_{1}} \otimes ... \otimes [a,c] \otimes b \otimes ... \otimes
x_{i_{n}} ) + \]
	\[ + \theta( x_{i_{1}} \otimes ... \otimes a \otimes [b,c] \otimes ... \otimes x_{i_{n}}). \]
	
	Thus we have to show that $\theta$ annihilates the following element of the vector superspace
$T^{0}(L)+T^{1}(L)+...+T^{n-1}(L)+T^{n,i-1}(L)$:
	\[ (-1)^{|a||b|+|a||c|} q^{\delta(a,b) + \delta(a,c)} x_{i_{1}} \otimes ... \otimes [b,c] \otimes a \otimes ... \otimes
x_{i_{n}} - x_{i_{1}} \otimes ... \otimes a \otimes [b,c] \otimes ... \otimes x_{i_{n}} + \]
	\[ + (-1)^{|a||b|} q^{\delta(a,b)} x_{i_{1}} \otimes ... \otimes b \otimes [a,c] \otimes ... \otimes x_{i_{n}} -
(-1)^{|b||c|} q^{\delta(b,c)} x_{i_{1}} \otimes ... \otimes [a,c] \otimes b \otimes ... \otimes x_{i_{n}}  + \]
	\[ + x_{i_{1}} \otimes ... \otimes [a,b] \otimes c \otimes ... \otimes x_{i_{n}} - (-1)^{|a||c|+|b||c|} q^{\delta(a,c) +
\delta(b,c)} x_{i_{1}} \otimes ... \otimes c \otimes [a,b] \otimes ... \otimes x_{i_{n}}. \]
	
	Note that
	\[ \theta ( x_{i_{1}} \otimes ... \otimes a \otimes [b,c] \otimes ... \otimes x_{i_{n}} ) = \theta ( x_{i_{1}} \otimes ...
\otimes a \otimes b \otimes c \otimes ... \otimes x_{i_{n}} ) - (-1)^{|b||c|} q^{\delta(b,c)} \theta( x_{i_{1}} \otimes ...
\otimes a \otimes c \otimes b \otimes ... \otimes x_{i_{n}}) = \]
	\[ = (-1)^{|a||b|} q^{\delta(a,b)} \theta ( x_{i_{1}} \otimes ... \otimes b \otimes a \otimes c \otimes ... \otimes
x_{i_{n}} ) + \theta ( x_{i_{1}} \otimes ... \otimes [a,b] \otimes c \otimes ... \otimes x_{i_{n}} ) - \]
	\[ - (-1)^{|b||c| + |a||c|} q^{\delta(b,c) + \delta(a,c)} \theta ( x_{i_{1}} \otimes ... \otimes c \otimes a \otimes b
\otimes ... \otimes x_{i_{n}} ) - (-1)^{|b||c|} q^{\delta(b,c)} \theta ( x_{i_{1}} \otimes ... \otimes [a,c] \otimes b \otimes
... \otimes x_{i_{n}} ) = \]
	\[ = (-1)^{|a||b| + |a||c|} q^{\delta(a,b) + \delta(a,c)} \theta( x_{i_{1}} \otimes ... \otimes b \otimes c \otimes a
\otimes ... \otimes x_{i_{n}} ) + (-1)^{|a||b|} q^{\delta(a,b)} \theta( x_{i_{1}} \otimes ... \otimes b \otimes [a,c] \otimes
... \otimes x_{i_{n}}) +  \]
	\[ + \theta ( x_{i_{1}} \otimes ... \otimes [a,b] \otimes c \otimes ... \otimes x_{i_{n}} ) - (-1)^{|b||c| + |a||c| +
|a||b|} q^{\delta(b,c) + \delta(a,c) + \delta(a,b)} \theta ( x_{i_{1}} \otimes ... \otimes c \otimes b \otimes a \otimes ...
\otimes x_{i_{n}} ) - \]
	\[ - (-1)^{|b||c| + |a||c|} q^{\delta(b,c) + \delta(a,c)} \theta( x_{i_{1}} \otimes ... \otimes c \otimes [a,b] \otimes ...
\otimes x_{i_{n}} ) - (-1)^{|b||c|} q^{\delta(b,c)} \theta ( x_{i_{1}} \otimes ... \otimes [a,c] \otimes b \otimes ... \otimes
x_{i_{n}} ) = \]
	\[ = (-1)^{|a||b| + |a||c|} q^{\delta(a,b) + \delta(a,c)} \theta( x_{i_{1}} \otimes ... \otimes [b,c] \otimes a \otimes ...
\otimes x_{i_{n}} ) + \]
	\[ +  (-1)^{|a||b|} q^{\delta(a,b)} \theta( x_{i_{1}} \otimes ... \otimes b \otimes [a,c] \otimes ... \otimes x_{i_{n}} ) +
\theta( x_{i_{1}} \otimes ... \otimes [a,b] \otimes c \otimes ... \otimes x_{i_{n}}) - \]
	\[ - (-1)^{|b||c| + |a||c|} q^{\delta(b,c) + \delta(a,c)} \theta( x_{i_{1}} \otimes ... \otimes c \otimes [a,b] \otimes ...
\otimes x_{i_{n}} ) - (-1)^{|b||c|} q^{\delta(b,c)} \theta( x_{i_{1}} \otimes ... \otimes [a,c] \otimes b \otimes ... \otimes
x_{i_{n}}) . \]
	
	Hence
	\[ (-1)^{|a||b|+|a||c|} q^{\delta(a,b) + \delta(a,c)} \theta ( x_{i_{1}} \otimes ... \otimes [b,c] \otimes a \otimes ...
\otimes x_{i_{n}}) - (-1)^{|a||b| + |a||c|} q^{\delta(a,b) + \delta(a,c)} \theta ( x_{i_{1}} \otimes ... \otimes [b,c] \otimes a
\otimes ... \otimes x_{i_{n}} ) - \]
	\[ - (-1)^{|a||b|} q^{\delta(a,b)} \theta ( x_{i_{1}} \otimes ... \otimes b \otimes [a,c] \otimes ... \otimes x_{i_{n}} ) -
\theta ( x_{i_{1}} \otimes ... \otimes [a,b] \otimes c \otimes ... \otimes x_{i_{n}} ) + \]
	\[ + (-1)^{|b||c| + |a||c|} q^{\delta(b,c) + \delta(a,c)} \theta ( x_{i_{1}} \otimes ... \otimes c \otimes [a,b] \otimes ...
\otimes x_{i_{n}} ) + (-1)^{|b||c|} q^{\delta(b,c)} \theta ( x_{i_{1}} \otimes ... \otimes [a,c] \otimes b \otimes ... \otimes
x_{i_{n}} ) + \]
	\[ + (-1)^{|a||b|} q^{\delta(a,b)} \theta ( x_{i_{1}} \otimes ... \otimes b \otimes [a,c] \otimes ... \otimes x_{i_{n}} ) -
(-1)^{|b||c|} q^{\delta(b,c)} \theta ( x_{i_{1}} \otimes ... \otimes [a,c] \otimes b \otimes ... \otimes x_{i_{n}} ) + \]
	\[ + \theta (x_{i_{1}} \otimes ... \otimes [a,b] \otimes c \otimes ... \otimes x_{i_{n}}) - (-1)^{|a||c|+|b||c|}
q^{\delta(a,c) + \delta(b,c)} \theta ( x_{i_{1}} \otimes ... \otimes c \otimes [a,b] \otimes ... \otimes x_{i_{n}} ) = 0. \]
	
	Consequently, even in this case the right side of the equation \ref{eq:basisULinearMap} is correctly defined.
	
	Thus we have defined a map $\theta:T^{n,i}(L) \to R$. A linear extension of this map gives us $\theta:\sum_{j=0}^{n-1}
T^{j}(L) \oplus T^{n,i}(L) \to R$, which satisfies the required conditions. Since $T^{n}= T^{n,r}$ for sufficiently large $r$,
we can consider a map $\theta: \sum_{j=0}^{n} T^j(L) \to R$. Since $ T(L)=T^{0} \oplus \sum_{\substack{i \in \mathbb{N}}} T^i(L)
$, we get a map $\theta: T(L) \to R$, which satisfies the required conditions.
	
	We have $U_{q}=T(L)/J$, and each element in the ideal $J$ is a linear combination of elements
	\[ \theta( x_{i_{1}}\otimes...\otimes x_{i_{k}}\otimes x_{i_{k+1}}\otimes...\otimes x_{i_{n}})-
(-1)^{|x_{i_{k}}||x_{i_{k+1}}|}q^{\delta(x_{i_{k}},x_{i_{k+1}})} \theta (x_{i_{1}}\otimes...\otimes x_{i_{k+1}}\otimes
x_{i_{k}}\otimes...\otimes x_{i_{n}})-\]
	\[-\theta(x_{i_{1}}\otimes...\otimes[x_{i_{k}},x_{i_{k+1}}]\otimes...\otimes x_{i_{n}}). \]
	
	Thus the linear $\mathbb{Z}_2$-graded $\theta:T(L)\to R$ annihilates all elements in $J$, inducing a superspace morphism
$\bar{\theta}:T(L)/J\to R$, that is $\bar{\theta}:U_{q}\to R$. A monomial $f_1^wf_3^sf_2^lk_1^ik_2^je_1^re_3^he_2^t$ is mapped
by $\bar{\theta}$ in $z_{1}^wz_{2}^sz_{3}^lz_{4}^iz_{5}^jz_{6}^rz_{7}^hz_{8}^t \in R$, where $0 \le w,i,j,r \le d-1, \; l,s,t,h
\in \{0,1\}$. Since elements $z_{1}^wz_{2}^sz_{3}^lz_{4}^iz_{5}^jz_{6}^rz_{7}^hz_{8}^t$ are linearly independent in the
polynomial ring $R$, it follows that elements $f_1^wf_3^sf_2^lk_1^ik_2^je_1^re_3^he_2^t$ must be linearly independent in $U_{q}$
for all $0 \le w,i,j,r \le d-1, \; l,s,t,h \in \{0,1\}$.
\end{proof}

\basisUU*
\begin{proof}
	It follows from Theorem \ref{theorem:basisU} that elements of a set  $\{ f_1^wf_3^sf_2^lk_1^ik_2^je_1^re_3^he_2^t \}$, for
all $0 \le w,i,j,r \le d-1, \; l,s,t,h \in \{0,1\}$, span $\bar{U}_q$. We only have to check a linear independence of these
elements. We consider intermediate factor-superalgebras until we construct a basis in $\bar{U}_q$.
	
	We consider in Theorem \ref{theorem:basisU} the superspace morphism $\bar{\theta}:U_{q}\to R$ between vector superspaces
$U_q$ and $R=F[z_{1},z_{2},...,z_{8}]$, which is the monomorphism on the set of basis elements in $U_q$. Recall that we denote
by $d$ he order of the root of unity $q$, where $F=\Bbbk(q)$. Consider a $\mathbb{Z}_2$-graded two-sided ideal
$R(I)=(z_1^{d},z_6^{d}) \subset R$, where $\bar{\theta}(f_1) = z_1, \bar{\theta}(e_1)=z_6$. Then $R(I)$ is a linear combination
of the following basis elements in $R$:
	\[ R(I) = <z_{1}^{d+w}z_2^sz_3^lz_4^iz_5^jz_6^rz_7^hz_8^t, z_1^wz_2^sz_3^lz_4^iz_5^jz_6^{d+r}z_7^hz_8^t>, \]
	where $i,j \in \mathbb{Z}, w,r \in \mathbb{Z}_+, l,s,t,h \in \{0,1\}$.
	
	Let $ \phi: R \to R/R( I ) $ be the canonical superspace morphism. A basis of $R/R( I )$ is composed by elements of a set
	\[ \{z_{1}^{w}z_2^sz_3^lz_4^iz_5^jz_6^rz_7^hz_8^t + R(I)\}, \]
	for all $i,j \in \mathbb{Z}, 0 \le w,r \le d-1, l,s,t,h \in \{0,1\}$. Thus we can consider a superspace morphism $\phi \circ
\bar{\theta}:U_{q}\to R/R(I) $. Denote by $ I = (f_1^d,e_1^d) $ a two-sided $\mathbb{Z}_2$-graded ideal in $U_q$. It follows
from Lemma \ref{lm:cent} that a basis of $ I$ is composed by elements of a set
	\[ \{ f_1^{d+w}f_3^sf_2^lk_1^ik_2^je_1^re_3^he_2^t, f_1^wf_3^sf_2^lk_1^ik_2^je_1^{d+r}e_3^he_2^t \}, \]
	where $i,j \in \mathbb{Z}, w,r \in \mathbb{Z}_+, l,s,t,h \in \{0,1\} $. Thus $I \subset ker(\phi \circ \bar{\theta})$.
Consequently as it is well-known there exists a unique superspace morphism $\psi : U_q / I \to R/R( I ) $, such that $\psi \circ
\nu = \phi \circ \bar{\theta}$.
	
	\[
	\begin{tikzpicture}
	\matrix (m) [matrix of math nodes,row sep=3em,column sep=4em,minimum width=2em]
	{
		U_q & R \\
		U_q/I	& R/R( I ) \\};
	\path[-stealth]
	(m-1-1) edge[commutative diagrams/dashed] node [right] {$\nu$} (m-2-1)
	(m-1-1) edge node [above] {$\bar{\theta}$} (m-1-2)
	(m-1-2) edge[commutative diagrams/dashed] node [right] {$\phi$} (m-2-2)
	(m-2-1) edge node [above] {$\psi$} (m-2-2);
	\end{tikzpicture}
	\]
	
	We prove that elements of a set $\{ f_1^wf_3^sf_2^lk_1^ik_2^je_1^re_3^he_2^t + I \}$, for all $ 0 \leq w,r \leq d-1, \; i,j
\in \mathbb{Z}, \; l,s,t,h \in \{0,1\} $, form a basis for $U_q / I$. Indeed, it follows from Theorem \ref{theorem:basisU} that
these elements span $U_q / I$.
	Consider the following relation in $U_q / I$
	\[ \sum_{0 \leq w,r \leq d-1, \; i,j \in \mathbb{Z}, \; l,s,t,h \in \{0,1\} } \alpha_{wslijrht}
f_1^wf_3^sf_2^lk_1^ik_2^je_1^re_3^he_2^t + I = 0, \]
	where $\alpha_{wslijrht} \in F$ for all $0 \leq w,r \leq d-1, \; i,j \in \mathbb{Z}, \; l,s,t,h \in \{0,1\}$.
	Then a linear independence follows from the following relations
	\[ \psi( \sum_{0 \leq w,r \leq d-1, \; i,j \in \mathbb{Z}, \; l,s,t,h \in \{0,1\} } \alpha_{wslijrht}
f_1^wf_3^sf_2^lk_1^ik_2^je_1^re_3^he_2^t + I) = \]
	\[  = \psi \circ \nu ( \sum_{0 \leq w,r \leq d-1, \; i,j \in \mathbb{Z}, \; l,s,t,h \in \{0,1\} } \alpha_{wslijrht}
f_1^wf_3^sf_2^lk_1^ik_2^je_1^re_3^he_2^t )= \]
	\[ = \phi \circ \bar{\theta} ( \sum_{0 \leq w,r \leq d-1, \; i,j \in \mathbb{Z}, \; l,s,t,h \in \{0,1\} } \alpha_{wslijrht}
f_1^wf_3^sf_2^lk_1^ik_2^je_1^re_3^he_2^t ) = \]
	\[ = \sum_{0 \leq w,r \leq d-1, \; i,j \in \mathbb{Z}, \; l,s,t,h \in \{0,1\} } \alpha_{wslijrht} (\phi \circ
\bar{\theta})(f_1^wf_3^sf_2^lk_1^ik_2^je_1^re_3^he_2^t) = \]
	\[ = \sum_{0 \leq w,r \leq d-1, \; i,j \in \mathbb{Z}, \; l,s,t,h \in \{0,1\} } \alpha_{wslijrht} \phi
(z_1^wz_2^sz_3^lz_4^iz_5^jz_6^rz_7^hz_8^t) = \]
	\[ = \sum_{0 \leq w,r \leq d-1, \; i,j \in \mathbb{Z}, \; l,s,t,h \in \{0,1\} } \alpha_{wslijrht}
z_1^wz_2^sz_3^lz_4^iz_5^jz_6^rz_7^hz_8^t + R(I)  =0. \]
	Consequently all $\alpha_{wslijrht}=0$ for all $0 \leq w,r \leq d-1, \; i,j \in \mathbb{Z}, \; l,s,t,h \in \{0,1\}$.
	
	Denote by $deg_{k_1}(f)$ (respectively by $deg_{k_1^{-1}} (f)$ ) a degree of non-zero element $f \in U_q / I$, which is
written in the mentioned above basis as a polynomial in $k_1$ (respectively in $k_1^{-1}$).
	
	Consider in $U_q / (f_1^d, k_1^d - 1, e_1^d)$ a linear equation
	\begin{equation}\label{eq:ideal1}
	\sum_{ 0 \leq w,i,r \leq d-1; j \in \mathbb{Z}; 0 \leq l,s,t,h \leq 1  } \alpha_{wslijrht}
f_1^wf_3^sf_2^lk_1^ik_2^je_1^re_3^he_2^t + (f_1^d, k_1^d - 1, e_1^d) = 0.
	\end{equation}
	
	Examine a superalgebra morphism $\mu_{1}: U_q / I \to U_q / (f_1^d, k_1^d - 1, e_1^d)$. Then we can consider the element
$Z_1=\sum_{ 0 \leq w,i,r \leq d-1; j \in \mathbb{Z}; 0 \leq l,s,t,h \leq 1  } \alpha_{wslijrht}
f_1^wf_3^sf_2^lk_1^ik_2^je_1^re_3^he_2^t + I$ in the superalgebra $U_q / I$, which is a preimage of the linear combination in
the left side of the equation \ref{eq:ideal1}. Moreover,  $Z_1 \in ker(\mu_{1})$. Let us rewrite $Z_1$ as a polynomial in
$k_1$:
	\[ Z_1 = \sum_{ 0 \leq w,i,r \leq d-1; j \in \mathbb{Z}; 0 \leq l,s,t,h \leq 1  } \alpha_{wslijrht}
f_1^wf_3^sf_2^lk_1^ik_2^je_1^re_3^he_2^t + I = \]
	\[ = \sum_{ 0 \leq w,i,r \leq d-1; j \in \mathbb{Z}; 0 \leq l,s,t,h \leq 1  } q^{i(2r-t+h)} \alpha_{wslijrht}
f_1^wf_3^sf_2^lk_2^je_1^re_3^he_2^tk_1^i + I \in P_1[k_1] + I, \]
	where $P_1 = <f_1^wf_3^sf_2^lk_2^je_1^re_3^he_2^t>$. Consequently $deg_{k_1}(Z_1) \le d-1, \; deg_{k_1^{-1}}(Z_1)=0$, that
is
	\begin{equation}\label{eq:idealdeg2}
	0 \leq deg_{k_1^{-1}}(Z_1) \leq deg_{k_1}(Z_1) < d.
	\end{equation}
	
	From the third theorem of isomorphism it follows that
	\[(U_q / I) / ((I + (k_1^d - 1)) / I) \cong U_q / (f_1^d, k_1^d - 1, e_1^d).\]
	Thus $Z_1$ is in a $\mathbb{Z}_2$-graded two-sided ideal of the superalgebra $U_q / I$, genererated by element $k_1^d - 1 $,
that is $Z_1 \in (I + (k_1^d - 1)) / I$, as $ker(\mu_{1}) = (I + (k_1^d - 1)) / I$.
	Thus $Z_1=X_1((k_1^d - 1) + I)$, where
	\[X_1 = \sum_{ 0 \leq w,r \leq d-1; i,j \in \mathbb{Z}; 0 \leq l,s,t,h \leq 1  } \beta_{wslijrht}
f_1^wf_3^sf_2^lk_1^ik_2^je_1^re_3^he_2^t + I. \]
	
	As the element $k_1^d$ is central, we have
	\begin{equation}\label{eq:ideal2}
	Z_1 = \sum_{ 0 \leq w,r \leq d-1; i,j \in \mathbb{Z}; 0 \leq l,s,t,h \leq 1  } (\beta_{wslijrht}
f_1^wf_3^sf_2^lk_1^{i+d}k_2^je_1^re_3^he_2^t + I) -
	\end{equation}
	\[ - \sum_{ 0 \leq w,r \leq d-1; i,j \in \mathbb{Z}; 0 \leq l,s,t,h \leq 1  } (\beta_{wslijrht}
f_1^wf_3^sf_2^lk_1^ik_2^je_1^re_3^he_2^t + I) = \]
	\[ = \sum_{ 0 \leq w,r \leq d-1; i,j \in \mathbb{Z}; 0 \leq l,s,t,h \leq 1  } (q^{(i+d)(2r-t+h)} \beta_{wslijrht}
f_1^wf_3^sf_2^lk_2^je_1^re_3^he_2^t k_1^{i+d} + I) - \]
	\[ - \sum_{ 0 \leq w,r \leq d-1; i,j \in \mathbb{Z}; 0 \leq l,s,t,h \leq 1  } (q^{i(2r-t+h)} \beta_{wslijrht}
f_1^wf_3^sf_2^lk_2^je_1^re_3^he_2^t k_1^i  + I). \]
	Suppose that $Z_1 \neq 0$, which means that $X_1 \neq 0$. From the equality \ref{eq:ideal2} it follows that
	\begin{equation}\label{eq:idealdeg1}
	deg_{k_1}(Z_1) = deg_{k_1}(X_1) + d, deg_{k_1^{-1}}(Z_1) = deg_{k_1^{-1}}(X_1).
	\end{equation}
	
	Considering \ref{eq:idealdeg1}, \ref{eq:idealdeg2}, we get $ deg_{k_1}(X_1) < 0 \leq deg_{k_1^{-1}}(Z_1) =
deg_{k_1^{-1}}(X_1)  $. It is impossible. Consequently $Z_1=0$. Therefore all $ \beta_{wslijrht} = 0 $ in \ref{eq:ideal2}. Thus
all $ \alpha_{wslijrht} = 0 $ in \ref{eq:ideal1}.
	
	Denote by $deg_{k_2}(f)$ (respectively by $deg_{k_2^{-1}} (f)$ ) a degree of non-zero element $f \in U_q / (f_1^d, k_1^d -
1, e_1^d)$, which is written in the mentioned above basis as a polynomial in $k_2$ (respectively in $k_2^{-1}$).
	
	Consider in $U_q / (f_1^d, k_1^d - 1, k_2^d - 1, e_1^d)$ a linear equation
	\begin{equation}\label{eq:ideal3}
	\sum_{ 0 \leq w,i,j,r \leq d-1; 0 \leq l,s,t,h \leq 1  } \alpha_{wslijrht} f_1^wf_3^sf_2^lk_1^ik_2^je_1^re_3^he_2^t +
(f_1^d, k_1^d - 1, k_2^d - 1, e_1^d) = 0.
	\end{equation}
	
	Examine a superalgebra morphism
	\[\mu_{2}: U_q / (f_1^d, k_1^d - 1, e_1^d) \to U_q / (f_1^d, k_1^d - 1, k_2^d - 1, e_1^d).\]
	Then we can consider the element
	\[Z_2=\sum_{ 0 \leq w,i,j,r \leq d-1; 0 \leq l,s,t,h \leq 1  } \alpha_{wslijrht} f_1^wf_3^sf_2^lk_1^ik_2^je_1^re_3^he_2^t +
(f_1^d, k_1^d - 1, e_1^d)\]
	in superalgebra $U_q / (f_1^d, k_1^d - 1, e_1^d)$, which is a preimage of the linear combination in the left side of the
equation \ref{eq:ideal3}. Moreover, $Z_2 \in ker(\mu_{2})$. Let us rewrite $Z_2$ as a polynomial in $k_2$:
	\[ Z_2 = \sum_{ 0 \leq w,i,j,r \leq d-1; 0 \leq l,s,t,h \leq 1  } \alpha_{wslijrht} f_1^wf_3^sf_2^lk_1^ik_2^je_1^re_3^he_2^t
+ (f_1^d, k_1^d - 1, e_1^d) = \]
	\[ = \sum_{ 0 \leq w,i,j,r \leq d-1; 0 \leq l,s,t,h \leq 1  } q^{j(-r-h)} \alpha_{wslijrht}
f_1^wf_3^sf_2^lk_1^ie_1^re_3^he_2^tk_2^j + (f_1^d, k_1^d - 1, e_1^d) \in P_2[k_2] + (f_1^d, k_1^d - 1, e_1^d), \]
	where $P_2 = <f_1^wf_3^sf_2^lk_1^ie_1^re_3^he_2^t>$. Consequently $deg_{k_2}(Z_2) \le d-1, \; deg_{k_2^{-1}}(Z_2)=0$, that
is
	\begin{equation}\label{eq:idealdeg4}
	0 \leq deg_{k_2^{-1}}(Z_2) \leq deg_{k_2}(Z_2) < d.
	\end{equation}
	
	From the third theorem of isomorphism it follows that
	\[(U_q / (f_1^d, k_1^d - 1, e_1^d)) / ((f_1^d, k_1^d - 1, k_2^d - 1, e_1^d) / (f_1^d, k_1^d - 1, e_1^d)) \cong U_q / (f_1^d,
k_1^d - 1, k_2^d - 1, e_1^d).\]
	Thus $Z_2$ is in a $\mathbb{Z}_2$-graded two-sided ideal of superalgebra $U_q / (f_1^d, k_1^d - 1, e_1^d)$, genererated by
element $k_2^d - 1 $, that is $Z_2 \in ((f_1^d, k_1^d - 1, e_1^d) + (k_2^d - 1)) / (f_1^d, k_1^d - 1, e_1^d)$, as $ker(\mu_{2})
= (f_1^d, k_1^d - 1, k_2^d - 1, e_1^d) / (f_1^d, k_1^d - 1, e_1^d)$.
	Thus $Z_2=X_2((k_2^d - 1) +(f_1^d, k_1^d - 1, e_1^d))$, where
	\[X_2 = \sum_{ 0 \leq w,i,r \leq d-1; j \in \mathbb{Z}; 0 \leq l,s,t,h \leq 1  } \beta_{wslijrht}
f_1^wf_3^sf_2^lk_1^ik_2^je_1^re_3^he_2^t + (f_1^d, k_1^d - 1, e_1^d). \]
	
	As the element $k_2^d$ is central, we have
	\begin{equation}\label{eq:ideal4}
	Z_2 = \sum_{ 0 \leq w,i,r \leq d-1; j \in \mathbb{Z}; 0 \leq l,s,t,h \leq 1  } (\beta_{wslijrht}
f_1^wf_3^sf_2^lk_1^ik_2^{j+d}e_1^re_3^he_2^t + (f_1^d, k_1^d - 1, e_1^d)) -
	\end{equation}
	\[ - \sum_{ 0 \leq w,i,r \leq d-1; j \in \mathbb{Z}; 0 \leq l,s,t,h \leq 1  } (\beta_{wslijrht}
f_1^wf_3^sf_2^lk_1^ik_2^je_1^re_3^he_2^t + (f_1^d, k_1^d - 1, e_1^d)) = \]
	\[ = \sum_{ 0 \leq w,i,r \leq d-1; j \in \mathbb{Z}; 0 \leq l,s,t,h \leq 1  } (q^{(j+d)(-r-h)} \beta_{wslijrht}
f_1^wf_3^sf_2^lk_1^ie_1^re_3^he_2^tk_2^{j+d} + (f_1^d, k_1^d - 1, e_1^d)) - \]
	\[ - \sum_{ 0 \leq w,i,r \leq d-1; j \in \mathbb{Z}; 0 \leq l,s,t,h \leq 1  } (q^{j(-r-h)} \beta_{wslijrht}
f_1^wf_3^sf_2^lk_1^ie_1^re_3^he_2^tk_2^j + (f_1^d, k_1^d - 1, e_1^d)). \]
	Suppose that $Z_2 \neq 0$, which means that $X_2 \neq 0$. From the equality \ref{eq:ideal4} it follows that
	\begin{equation}\label{eq:idealdeg3}
	deg_{k_2}(Z_2) = deg_{k_2}(X_2) + d, deg_{k_2^{-1}}(Z_2) = deg_{k_2^{-1}}(X_2).
	\end{equation}
	
	Considering \ref{eq:idealdeg3}, \ref{eq:idealdeg4}, we get $ deg_{k_2}(X_2) < 0 \leq deg_{k_2^{-1}}(Z_2) =
deg_{k_2^{-1}}(X_2)  $. It is impossible. Consequently $Z_2=0$. Therefore all $ \beta_{wslijrht} = 0 $ in \ref{eq:ideal4}. Thus
all $ \alpha_{wslijrht} = 0 $ in \ref{eq:ideal3}.
\end{proof}

\thisom*
\begin{proof}		
	The surjectivity of $\chi$ follows from the fact, that the image of the basis $\{\alpha_{e_1}^{w} \alpha_{e_3}^{s}
\alpha_{e_2}^{l} \alpha_{k_1}^{i_1} \alpha_{k_2}^{j_1} \otimes k_1^{i_2}k_2^{j_2}e_1^{r}e_3^{h}e_2^{t}\}$ additively generates
$\bar{U}_q$.
	Indeed,
	\[ \chi(\alpha_{k_1} \otimes 1_{B_q^{+}}) = k_1, \; \chi(\alpha_{k_2} \otimes 1_{B_q^{+}}) = k_2, \]
	\[ \chi(\alpha_{e_1} \otimes 1_{B_q^{+}}) = f_1, \; \chi(\alpha_{e_2} \otimes 1_{B_q^{+}}) = f_2, \]
	\[ \chi(\alpha_{e_3} \otimes 1_{B_q^{+}}) = f_3, \]
	\[ \chi( 1_X \otimes k_1) = k_1, \; \chi( 1_X \otimes k_2) = k_2, \]
	\[ \chi( 1_X \otimes e_1) = e_1, \; \chi( 1_X \otimes e_2) = e_2, \]
	\[ \chi( 1_X \otimes e_3) = e_3. \]
	
	We check that $\chi$ is a superalgebra morphism.
	We have for $\alpha_{k_1}$:
	\[ \mu_{\bar{U}_q} \circ (\chi \otimes \chi) ( (1_X \otimes k_v) \otimes (\alpha_{k_1} \otimes 1_{B^{+}_q}) ) = k_v k_1 =
k_1 k_v  = \chi ( \alpha_{k_1} \otimes k_v ) = \]
	\[ = \chi ( \alpha_{k_1}( k_v^{-1} ? k_v ) \otimes k_v ) = \chi \circ \mu_{D}  ( (1_X \otimes k_v) \otimes (\alpha_{k_1}
\otimes 1_{B^{+}_q}) ), \]
	\[ \mu_{\bar{U}_q} \circ ( \chi \otimes \chi )( (1_X \otimes e_1) \otimes (\alpha_{k_1} \otimes 1_{B^{+}_q}) ) = e_1 k_1 =
q^{-2} k_1 e_1 = \]
	\[ = \chi( q^{-2} (\alpha_{k_1} \otimes 1_{B_q^+}) \otimes (1_X \otimes e_1) ) = \chi \circ \mu_{D}  ( (1_X \otimes e_1)
\otimes (\alpha_{k_1} \otimes 1_{B^{+}_q}) ), \]
	\[ \mu_{\bar{U}_q} \circ ( \chi \otimes \chi )( (1_X \otimes e_2) \otimes (\alpha_{k_1} \otimes 1_{B^{+}_q}) ) = e_2 k_1 = q
k_1 e_2 = \]
	\[ = \chi( q (\alpha_{k_1} \otimes 1_{B_q^+}) \otimes (1_X \otimes e_2) ) = \chi \circ \mu_{D}  ( (1_X \otimes e_2) \otimes
(\alpha_{k_1} \otimes 1_{B^{+}_q}) ), \]
	\[ \mu_{\bar{U}_q} \circ ( \chi \otimes \chi ) ( (1_X \otimes e_3) \otimes (\alpha_{k_1} \otimes 1_{B^{+}_q}) ) = e_3 k_1 =
q^{-1} k_1 e_3 = \]
	\[ = \chi ( q^{-1} (\alpha_{k_1} \otimes 1_{B_q^+}) \otimes (1_X \otimes e_3) ) = \chi \circ \mu_{D}( (1_X \otimes e_3)
\otimes (\alpha_{k_1} \otimes 1_{B^{+}_q}) ). \]
	
	We have for $\alpha_{k_2}$:
	\[ \mu_{\bar{U}_q} \circ (\chi \otimes \chi) ( (1_X \otimes k_v) \otimes (\alpha_{k_2} \otimes 1_{B^{+}_q}) ) = \]
	\[ = k_v k_2 = k_2 k_v = \chi ( \alpha_{k_2} \otimes k_v ) = \chi \circ \mu_{D}  ( (1_X \otimes k_v) \otimes (\alpha_{k_2}
\otimes 1_{B^{+}_q}) ), \]
	\[ \mu_{\bar{U}_q} \circ ( \chi \otimes \chi ) ( ( 1_X \otimes e_1) \otimes (\alpha_{k_2} \otimes 1_{B^{+}_q}) ) = e_1 k_2 =
q k_2 e_1 = \]
	\[ = \chi( q (\alpha_{k_2} \otimes 1_{B_q^+}) \otimes (1_X \otimes e_1) ) = \chi \circ \mu_{D}( (1_X \otimes e_1) \otimes
(\alpha_{k_2} \otimes 1_{B^{+}_q}) ), \]
	\[ \mu_{\bar{U}_q} \circ ( \chi \otimes \chi )( (1_X \otimes e_2) \otimes (\alpha_{k_2} \otimes 1_{B^{+}_q}) ) = e_2 k_2 =
k_2 e_2 = \]
	\[ = \chi( (\alpha_{k_2} \otimes 1_{B_q^+}) \otimes (1_X \otimes e_2) ) = \chi \circ \mu_{D}  ( (1_X \otimes e_2) \otimes
(\alpha_{k_2} \otimes 1_{B^{+}_q}) ), \]
	\[ \mu_{\bar{U}_q} \circ ( \chi \otimes \chi ) ( (1_X \otimes e_3) \otimes (\alpha_{k_2} \otimes 1_{B_q^+}) ) = e_3 k_2 = q
k_2 e_3 = \]
	\[ = \chi( q (\alpha_{k_2} \otimes 1_{B_q^+}) \otimes (1_X \otimes e_3) ) = \chi \circ \mu_{D} ( (1_X \otimes e_3) \otimes
(\alpha_{k_2} \otimes 1_{B_q^+}) ). \]
	
	We have for $\alpha_{e_1}$:
	\[ \mu_{\bar{U}_q} \circ ( \chi \otimes \chi )( (1_X \otimes k_1) \otimes (\alpha_{e_1} \otimes 1_{B^{+}_q}) ) = k_1 f_1 =
q^{-2} f_1 k_1 = \]
	\[ = \chi(q^{-2} \alpha_{e_1} \otimes k_1) = \chi \circ \mu_{D}  ( (1_X \otimes k_1) \otimes (\alpha_{e_1} \otimes
1_{B^{+}_q}) ), \]
	\[ \mu_{\bar{U}_q} \circ ( \chi \otimes \chi ) ((1_X \otimes k_2) \otimes (\alpha_{e_1} \otimes 1_{B_q^+})) = k_2 f_1 = q
f_1 k_2 = \]
	\[ = \chi(q \alpha_{e_1} \otimes k_2) = \chi \circ \mu_{D} ((1_X \otimes k_2) \otimes (\alpha_{e_1} \otimes 1_{B_q^+})), \]
	\[ \mu_{\bar{U}_q} \circ (\chi \otimes \chi) ( (1_X \otimes e_1) \otimes (\alpha_{e_1} \otimes 1_{B^{+}_q}) ) = e_1 f_1 =
f_1e_1 + \frac{k_1-k_1^{-1}}{q-q^{-1}} = \]
	\[ = \chi( - (q-q^{-1})^{-1} \alpha_{k_1}^{-1} \otimes 1_{B_q^+} + \alpha_{e_1} \otimes e_1 + (q-q^{-1})^{-1} 1_{X} \otimes
k_1) =  \]
	\[ = \chi \circ \mu_{D}  ( (1_X \otimes e_1) \otimes (\alpha_{e_1} \otimes 1_{B^{+}_q}) ), \]
	\[ \mu_{\bar{U}_q} \circ ( \chi \otimes \chi ) ( (1_X \otimes e_2) \otimes (\alpha_{e_1} \otimes 1_{B^{+}_q}) ) = e_2 f_1 =
f_1 e_2 = \]
	\[ = \chi( \alpha_{e_1} \otimes e_2 ) = \chi \circ \mu_{D}( (1_X \otimes e_2) \otimes (\alpha_{e_1} \otimes 1_{B^{+}_q}) ),
\]
	\[ \mu_{\bar{U}_q} \circ ( \chi \otimes \chi ) ( (1_X \otimes e_3) \otimes (\alpha_{e_1} \otimes 1_{B_q^+}) ) = e_3 f_1 =
f_1 e_3 - q^{-1} k_1^{-1} e_2 = \]
	\[ = \chi( - q^{-1} \alpha_{k_1}^{-1} \otimes e_2 + \alpha_{e_1} \otimes e_3 ) =  \chi \circ \mu_{D} ( (1_X \otimes e_3)
\otimes (\alpha_{e_1} \otimes 1_{B_q^+}) ). \]
	
	We have for $\alpha_{e_2}$:
	\[ \mu_{\bar{U}_q} \circ ( \chi \otimes \chi ) ( (1_X \otimes k_1) \otimes (\alpha_{e_2} \otimes 1_{B^{+}_q}) ) = k_1 f_2 =
q f_2 k_1 = \]
	\[ = \chi( q \alpha_{e_2} \otimes k_1 ) = \chi \circ \mu_{D}  ( (1_X \otimes k_1) \otimes (\alpha_{e_2} \otimes 1_{B^{+}_q})
), \]
	\[ \mu_{\bar{U}_q} \circ (\chi \otimes \chi)( (1_X \otimes k_2) \otimes (\alpha_{e_2} \otimes 1_{B_q^{+}}) ) = k_2 f_2 = f_2
k_2 = \]
	\[ = \chi( \alpha_{e_2} \otimes k_2 ) = \chi \circ \mu_{D} ( (1_X \otimes k_2) \otimes (\alpha_{e_2} \otimes 1_{B_q^{+}}) ),
\]
	\[ \mu_{\bar{U}_q} \circ (\chi \otimes \chi)( (1_X \otimes e_1) \otimes (\alpha_{e_2} \otimes 1_{B_q^{+}}) ) = e_1 f_2 = f_2
e_1 = \]
	\[ = \chi( \alpha_{e_2} \otimes e_1 ) = \chi \circ \mu_{D} ( (1_X \otimes e_1) \otimes (\alpha_{e_2} \otimes 1_{B_q^{+}}) ),
\]
	\[ \mu_{\bar{U}_q} \circ (\chi \otimes \chi) ( (1_X \otimes e_2) \otimes (\alpha_{e_2} \otimes 1_{B_q^{+}}) ) = e_2 f_2 =
-f_2e_2 + \frac{k_2-k_2^{-1}}{q-q^{-1}} = \]
	\[ = \chi( - (q-q^{-1})^{-1} \alpha_{k_2}^{-1} \otimes 1_{B_q^{+}} - \alpha_{e_2} \otimes e_2 + (q-q^{-1})^{-1} 1_{X}
\otimes k_2 ) = \]
	\[ = \chi \circ \mu_{D} ( (1_X \otimes e_2) \otimes (\alpha_{e_2} \otimes 1_{B_q^{+}}) ), \]
	\[ \mu_{\bar{U}_q} \circ (\chi \otimes \chi)( (1_X \otimes e_3) \otimes (\alpha_{e_2} \otimes 1_{B_q^{+}}) ) = e_3 f_2 = -
f_2 e_3 + k_2 e_1 = \]
	\[ = \chi( 1_{X} \otimes k_2 e_1 - \alpha_{e_2} \otimes e_3 ) = \chi \circ \mu_{D}( (1_X \otimes e_3) \otimes (\alpha_{e_2}
\otimes 1_{B_q^{+}}) ). \]
	
	We have for $\alpha_{e_3}$:
	\[ \mu_{\bar{U}_q} \circ (\chi \otimes \chi)( (1_X \otimes k_1) \otimes (\alpha_{e_3} \otimes 1_{B_q^{+}}) ) = k_1 f_3 =
q^{-1} f_3 k_1 = \chi( q^{-1} \alpha_{e_3} \otimes k_1 ) = \chi \circ \mu_{D}( (1_X \otimes k_1) \otimes (\alpha_{e_3} \otimes
1_{B_q^{+}}) ), \]
	\[ \mu_{\bar{U}_q} \circ (\chi \otimes \chi)( (1_X \otimes k_2) \otimes (\alpha_{e_3} \otimes 1_{B_q^{+}}) ) = \]
	\[ = k_2 f_3 = q f_3 k_2 = \chi( q \alpha_{e_3} \otimes k_2 ) = \chi \circ \mu_{D}( (1_X \otimes k_2) \otimes (\alpha_{e_3}
\otimes 1_{B_q^{+}}) ), \]
	\[ \mu_{\bar{U}_q} \circ (\chi \otimes \chi)( (1_X \otimes e_1) \otimes (\alpha_{e_3} \otimes 1_{B_q^{+}}) ) = e_1 f_3 = f_3
e_1 - q f_2 k_1 = \]
	\[ = \chi( \alpha_{e_3} \otimes e_1 - q \alpha_{e_2} \otimes k_1 ) = \chi \circ \mu_{D}( (1_X \otimes e_1) \otimes
(\alpha_{e_3} \otimes 1_{B_q^{+}}) ), \]
	\[ \mu_{\bar{U}_q} \circ (\chi \otimes \chi)( (1_X \otimes e_2) \otimes (\alpha_{e_3} \otimes 1_{B_q^{+}}) ) = e_2 f_3 =
-f_3 e_2 + f_1 k_2^{-1} = \]
	\[ = \chi( \alpha_{e_1} \alpha_{k_2}^{-1} \otimes 1_{B_q^{+}} - \alpha_{e_3} \otimes e_2 ) = \chi \circ \mu_{D}( (1_X
\otimes e_2) \otimes (\alpha_{e_3} \otimes 1_{B_q^{+}}) ), \]
	\[ \mu_{\bar{U}_q} \circ ( \chi \otimes \chi ) ( (1_X \otimes e_3) \otimes (\alpha_{e_3} \otimes 1_{B_q^+}) ) = e_3 f_3 = -
f_3 e_3 + \frac{k_1k_2-k_1^{-1}k_2^{-1}}{q-q^{-1}} = \]
	\[ = \chi( - (q-q^{-1})^{-1} \alpha_{k_1}^{-1} \alpha_{k_2}^{-1} \otimes 1_{B_q^{+}} - \alpha_{e_3} \otimes e_3 +
(q-q^{-1})^{-1} 1_X  \otimes k_1 k_2 ) = \]
	\[ = \chi \circ \mu_{D} ( (1_X \otimes e_3) \otimes (\alpha_{e_3} \otimes 1_{B_q^+}) ). \]
	
	We check that $\chi$ respects the comultiplication.
	\[ \Delta_{\bar{U}_q} \circ \chi( \alpha_{k_v} \otimes 1_{B^{+}_q} ) = k_v \otimes k_v =  \]
	\[ = (\chi \otimes \chi) ( (\alpha_{k_v} \otimes 1_{B^{+}_q}) \otimes (\alpha_{k_v} \otimes 1_{B^{+}_q}) )  = (\chi \otimes
\chi) \circ \Delta_{D} ( \alpha_{k_v} \otimes 1_{B^{+}_q} ), \]
	\[ \Delta_{\bar{U}_q} \circ \chi( \alpha_{e_1} \otimes 1_{B^{+}_q} ) = \Delta_{\bar{U}_q}(f_1) =  f_1 \otimes k_1^{-1} +
1_{\bar{U}_q} \otimes f_1 = \]
	\[ = (\chi \otimes \chi) ( \alpha_{e_1} \otimes \alpha_{k_1}^{-1} + 1_X \otimes \alpha_{e_1} ) = (\chi \otimes \chi) \circ
\Delta_{D} ( \alpha_{e_1} \otimes 1_{B^{+}_q}), \]
	\[\Delta_{\bar{U}_q} \circ \chi (\alpha_{e_2} \otimes 1_{B_q^{+}}) = \Delta_{\bar{U}_q}(f_2) = f_2 \otimes k_2^{-1} +
1_{\bar{U}_q} \otimes f_2 = \]
	\[ = (\chi \otimes \chi)( \alpha_{e_2} \otimes \alpha_{k_2}^{-1} + 1_X \otimes \alpha_{e_2} ) = (\chi \otimes \chi) \circ
\Delta_{D} ( \alpha_{e_2} \otimes 1_{B_q^{+}} ), \]
	\[ \Delta_{\bar{U}_q} \circ \chi (\alpha_{e_3} \otimes 1_{B_q^{+}}) = \Delta_{\bar{U}_q}(f_3) =  \]
	\[ = 1 \otimes f_3 + f_3 \otimes k_1^{-1} k_2^{-1} +  ( q^{-1} - q ) f_2 \otimes f_1 k_2^{-1} = \]
	\[ = (\chi \otimes \chi)( 1_X \otimes \alpha_{e_3} + \alpha_{e_3} \otimes \alpha_{k_1}^{-1} \alpha_{k_2}^{-1} + (q^{-1}-q)
\alpha_{e_2} \otimes \alpha_{e_1} \alpha_{k_2}^{-1} ) = (\chi \otimes \chi) \circ \Delta_{D} ( \alpha_{e_3} \otimes 1_{B_q^{+}}
). \]
	
	We check that $\chi$ respects the counit.
	\[ \epsilon_{\bar{U}_q} \circ \chi(\alpha_{k_v} \otimes 1_{B^{+}_q} ) =  \epsilon_{\bar{U}_q}(k_v) = 1_F = \epsilon_{D} (
\alpha_{k_v} \otimes 1_{B^{+}_q}), \]
	\[ \epsilon_{\bar{U}_q} \circ \chi(\alpha_{e_1} \otimes 1_{B^{+}_q} ) = \epsilon_{\bar{U}_q}(f_1) = 0 =  \epsilon_{D} (
\alpha_{e_1} \otimes 1_{B^{+}_q}), \]
	\[\epsilon_{\bar{U}_q} \circ \chi( \alpha_{e_2} \otimes 1_{B_q^{+}}) = \epsilon_{\bar{U}_q}(f_2) = 0 =
\epsilon_{D}(\alpha_{e_2} \otimes 1_{B_q^{+}}), \]
	\[ \epsilon_{\bar{U}_q} \circ \chi( \alpha_{e_3} \otimes 1_{B_q^{+}}) =  \epsilon_{\bar{U}_q}(f_3) = 0 =
\epsilon_{D}(\alpha_{e_3} \otimes 1_{B_q^{+}}). \]
	
	We check that $\chi$ respects the antipode.
	\[ S_{\bar{U}_q}(\chi(\alpha_{k_v} \otimes 1_{B^{+}_q})) = S_{\bar{U}_q}(k_v) = k_v^{-1} = \chi ( \alpha_{k_v}^{d-1} \otimes
1_{B^{+}_q} ) = \chi ( S_{X} (\alpha_{k_v}) \otimes 1_{B^{+}_q} ) = \chi(S_{D}( \alpha_{k_v} \otimes 1_{B^{+}_q})), \]
	\[ S_{\bar{U}_q} ( \chi(\alpha_{e_1} \otimes 1_{B^{+}_q})) = S_{\bar{U}_q}(f_1 ) = - f_1 k_1 = \chi( - \alpha_{e_1}
\alpha_{k_1} \otimes 1_{B^{+}_q} ) = \]
	\[ = \chi ( S_{X} (\alpha_{e_1}) \otimes 1_{B^{+}_q} ) = \chi(S_{D}( \alpha_{e_1} \otimes 1_{B^{+}_q})), \]
	\[S_{\bar{U}_q}(\chi(\alpha_{e_2} \otimes 1_{B_q^{+}})) = S_{\bar{U}_q}(f_2) = - f_2 k_2 = \chi(- \alpha_{e_2} \alpha_{k_2}
\otimes 1_{B_q^{+}}) = \]
	\[= \chi( S_{X} (\alpha_{e_2}) \otimes 1_{B_q^{+}}) = \chi( S_{D} ( \alpha_{e_2} \otimes 1_{B_q^{+}} ) ), \]
	\[ S_{\bar{U}_q}(\chi(\alpha_{e_3} \otimes 1_{B_q^{+}})) = S_{\bar{U}_q}(f_3) = (q - q^3) f_1 f_2 k_1 k_2 - q^2 f_3 k_1 k_2
= \]
	\[ = \chi( (q-q^3) \alpha_{e_1} \alpha_{e_2} \alpha_{k_1} \alpha_{k_2} \otimes 1_{B_q^{+}} - q^{2} \alpha_{e_3} \alpha_{k_1}
\alpha_{k_2} \otimes 1_{B_q^{+}} ) = \chi(S_X(\alpha_{e_3}) \otimes 1_{B_q^{+}}) = \chi(S_{D}(\alpha_{e_3} \otimes
1_{B_q^{+}})). \]	
\end{proof}

\univeralR*
\begin{proof}
	We use notations from the proof of Corollary \ref{Bruq}. According to Theorem \ref{RQDth}, we have $R_{D} = \sum_{i \in I}
(1_X \otimes e_i) \otimes (e^{i} \otimes 1_{B_{q}^{+}})$ where $\{ e_i \}_{i \in I}$ is a basis of the supervector space
$B_{q}^{+}$ and $\{e^{i}\}_{i \in I}$ is its dual. Consequently,
	\[ \bar{R} = \sum_{i \in I} \chi(1_X \otimes e_i) \otimes \chi(e^{i} \otimes 1_{B_{q}^{+}} ). \]
	We take the set $\{ k_1^ik_2^je_1^re_3^he_2^t | \; 0 \le i,j,r < d, \; h,t \in \{0,1\} \}$ as a basis of $B_{q}^{+}$. Denote
by $\{(k_1^{i} k_2^{j} e_1^{r} e_3^{h} e_2^{t})^{*} | \; 0 \le i,j,r < d, \; h,t \in \{0,1\}\}$ the dual basis. By Lemma
\ref{lm:mulX} we know there exist scalars $\{ \mu^{i_{2}j_{2}r_{2}h_{2}t_{2}}_{i_{1}j_{1}r_{1}h_{1}t_{1}} | 0 \le
i_1,i_2,j_1,j_2,r_1,r_2 < d, \; h_1,h_2,t_1,t_2 \in \{0,1\} \}$ such that
	\[ (k_1^{i_2} k_2^{j_2} e_1^{r_2} e_3^{h_2} e_2^{t_2})^{*} = \sum_{0 \le i_1,j_1,r_1 \le d-1, 0 \le h_1, t_1 \le 1}
\mu^{i_{2}j_{2}r_{2}h_{2}t_{2}}_{i_{1}j_{1}r_{1}h_{1}t_{1}} \alpha_{e_1}^{r_1} \alpha_{e_3}^{h_1} \alpha_{e_2}^{t_1}
\alpha_{k_1}^{i_1} \alpha_{k_2}^{j_1} = \]
	\[ = \sum_{0 \le i_1,j_1 \le d-1} [h_2=t_2=0] \mu^{i_{2}j_{2}r_{2},0,0}_{i_{1}j_{1}r_{2},0,0} \alpha_{e_1}^{r_2}
\alpha_{k_1}^{i_1} \alpha_{k_2}^{j_1} + [h_2=t_2=1]  \mu^{i_{2}j_{2}r_{2},1,1}_{i_{1}j_{1}r_{2},1,1} \alpha_{e_1}^{r_2}
\alpha_{e_3} \alpha_{e_2} \alpha_{k_1}^{i_1} \alpha_{k_2}^{j_1} + \]
	\[ + [h_2=1,t_2=0] (\mu^{i_{2},j_{2},r_{2},1,0}_{i_{1},j_{1},r_{2},1,0} \alpha_{e_1}^{r_2} \alpha_{e_3} \alpha_{k_1}^{i_1}
\alpha_{k_2}^{j_1} + \mu^{i_{2},j_{2},r_{2},1,0}_{i_{1},j_{1},r_{2}+1,0,1} \alpha_{e_1}^{r_2+1} \alpha_{e_2} \alpha_{k_1}^{i_1}
\alpha_{k_2}^{j_1}) +  \]
	\[ + [h_2=0,t_2=1] (\mu^{i_{2}j_{2}r_{2},0,1}_{i_{1}j_{1},r_{2}-1,1,0} \alpha_{e_1}^{r_2-1} \alpha_{e_3} \alpha_{k_1}^{i_1}
\alpha_{k_2}^{j_1} + \mu^{i_{2}j_{2}r_{2},0,1}_{i_{1}j_{1},r_{2},0,1} \alpha_{e_2} \alpha_{k_1}^{i_1} \alpha_{k_2}^{j_1}). \]
	
	Then
	\[ \bar{R} = \sum_{\substack{0 \le i_1,j_1,i_2,j_2,r_1,r_2 \le d-1, \\ h_1,h_2,t_1,t_2 \in \{0,1\} }}
\mu^{i_{2}j_{2}r_{2}h_{2}t_{2}}_{i_{1}j_{1}r_{1}h_{1}t_{1}} k_1^{i_2}k_2^{j_2}e_1^{r_2}e_3^{h_2}e_2^{t_2} \otimes f_1^{r_1}
f_3^{h_1} f_2^{t_1} k_1^{i_1} k_2^{j_1} = \]
	\[ = \sum_{0 \le i_1,j_1,i_2,j_2,r \le d-1} ( \mu^{i_{2},j_{2},r,0,0}_{i_{1},j_{1},r,0,0} k_1^{i_2}k_2^{j_2}e_1^{r} \otimes
f_1^{r} k_1^{i_1} k_2^{j_1} + \mu^{i_{2},j_{2},r,1,1}_{i_{1},j_{1},r,1,1} k_1^{i_2}k_2^{j_2}e_1^{r} e_3 e_2  \otimes f_1^{r} f_3
f_2 k_1^{i_1} k_2^{j_1} + \]
	\[ + \mu^{i_{2},j_{2},r,1,0}_{i_{1},j_{1},r,1,0} k_1^{i_2}k_2^{j_2}e_1^{r}e_3 \otimes f_1^{r} f_3 k_1^{i_1} k_2^{j_1} +
\mu^{i_{2},j_{2},r,1,0}_{i_{1},j_{1},r+1,0,1} k_1^{i_2}k_2^{j_2}e_1^{r}e_3 \otimes f_1^{r+1} f_2 k_1^{i_1} k_2^{j_1} + \]
	\[ + \mu^{i_{2},j_{2},r,0,1}_{i_{1},j_{1},r-1,1,0} k_1^{i_2}k_2^{j_2}e_1^{r}e_2 \otimes f_1^{r-1} f_3 k_1^{i_1} k_2^{j_1} +
\mu^{i_{2},j_{2},r,0,1}_{i_{1},j_{1},r,0,1} k_1^{i_2}k_2^{j_2}e_1^{r} e_2 \otimes f_1^{r} f_2 k_1^{i_1} k_2^{j_1} ), \]
	where	
	\[ \mu^{i_2,j_2,r,0,0}_{i_1,j_1,r,0,0} = (-1)^{r} \frac{(q-q^{-1})^{r}}{d^2[r]!} q^{-r(2i_2-j_2) - \frac{r(r-1)}{2} +
i_1(2i_2-j_2) - j_1i_2}, \]
	\[ \mu^{i_2,j_2,r,1,1}_{i_1,j_1,r,1,1} = (-1)^{r+1} \frac{(q-q^{-1})^{r+2}}{d^2[r]!} q^{-r(2i_2-j_2) - \frac{r(r-1)}{2} +
i_1(2i_2-j_2)-j_1i_2 + j_2 + 2}, \]
	\[ \mu^{i_2,j_2,r,1,0}_{i_1,j_1,r,1,0} = (-1)^{r} \frac{(q-q^{-1})^{r+1}}{d^2[r]!} q^{-r(2i_2-j_2) - \frac{r(r-1)}{2} +
i_1(2i_2-j_2) - j_1i_2 - i_2 + j_2+2}, \]
	\[ \mu^{i_2,j_2,r,1,0}_{i_1,j_1,r+1,0,1} = (-1)^{r} \frac{(q-q^{-1})^{r+2}}{d^2[r]!} q^{-r(2i_2-j_2) - \frac{r(r-1)}{2} +
i_1(2i_2-j_2) - j_1i_2 - i_2 + j_2+2}, \]
	\[ \mu^{i_2,j_2,r,0,1}_{i_1,j_1,r,0,1} = (-1)^{r} \frac{(q-q^{-1})^{r+1}}{d^2[r]!} q^{-r(2i_2-j_2-2) - \frac{r(r-1)}{2} +
i_1(2i_2-j_2) - j_1i_2 + i_2}, \]
	\[ \mu^{i_2,j_2,r,0,1}_{i_1,j_1,r-1,1,0} = (-1)^{r} \frac{(q-q^{-1})^{r+1}}{d^2[r-1]!} q^{-r(2i_2-j_2-1) - \frac{r(r-1)}{2}
+ i_1(2i_2-j_2) - j_1i_2 + i_2}. \]
	
	We can represent $\bar{R}$ in a multiplicative form. Let us note, that
	\[ \tilde{R} = exp_{q^2}( (q-q^{-1}) e_3 \otimes f_3 ) exp_{q^2} ((q-q^{-1}) e_2 \otimes f_2 ) exp_{q^2}((-1) (q-q^{-1}) e_1
\otimes f_1 ) \times \]
	\[ \times exp_{q^2}( (-1) (q^{2}-1) (q-q^{-1})^{2} e_3 e_2 \otimes f_3 f_2 ) = \]
	\[ = \sum_{r=0}^{d-1} \frac{(-1)^{r} (q-q^{-1})^{r} q^{-\frac{r(r-1)}{2}} }{[r]!} e_1^{r} \otimes f_1^{r} + \sum_{r=0}^{d-1}
\frac{(-1)^{r+1} (q-q^{-1})^{r+2} q^{-\frac{r(r-1)}{2}+2} }{[r]!} e_1^{r} e_{3} e_{2} \otimes f_1^{r} f_{3} f_{2} +  \]
	\[ + \sum_{r=0}^{d-1} \frac{(-1)^{r} (q-q^{-1})^{r+1} q^{-\frac{r(r-1)}{2}+2r} }{[r]!} e_1^{r} e_{2} \otimes f_1^{r} f_{2} +
\sum_{r=1}^{d-1} \frac{(-1)^{r} (q-q^{-1})^{r+1} q^{-\frac{r(r-1)}{2}+r}}{[r-1]!} e_1^{r} e_{2} \otimes f_1^{r-1} f_{3} + \]
	\[ + \sum_{r=0}^{d-1} \frac{(-1)^{r} (q-q^{-1})^{r+2} q^{-\frac{r(r-1)}{2}+2} }{[r]!} e_1^{r} e_{3} \otimes f_1^{r+1} f_{2}
+ \sum_{r=0}^{d-1} \frac{(-1)^{r} (q-q^{-1})^{r+1} q^{-\frac{r(r-1)}{2}+2} }{[r]!} e_1^{r} e_{3} \otimes f_1^{r} f_{3}. \]
	
	We set
	\[ K = d^{-2} \sum_{0 \le i_1, j_1, i_2, j_2 \le d-1} q^{i_1(2i_2-j_2)-j_1i_2} k_1^{i_2} k_2^{j_2} \otimes k_1^{i_1}
k_2^{j_1}. \]
	
	Therefore, we have
	\[ \bar{R} = \tilde{R} K. \]
	
\end{proof}

\thindAlg*
\begin{proof}
	We give a proof by induction on $n$. The cases $n=3$ and $n=4$ are already considered. Let $W = \sum_{i=0}^{2} L_{n-1,\mu}
g_{n-1}^{i} L_{n-1,\mu} + L_{n-3,\mu} g_{n-1} g_{n-2}^{2} g_{n-1}$. Since $1 \in W$, we need to prove that $ g_{i}^{j} W \subset
W, \; W g_{i}^{j} \subset W$ for $i \in \{ n-3,n-2,n-1 \}, \; j \in \{1,2\}$. These relations are correct for $g_{n-3}$ and
$g_{n-2}$, as it follows from Lemma \ref{lm:bimolln}. Recall that $L_{n-3,\mu}$ commutes with $g_{n-1} g_{n-2}^{2} g_{n-1}$. So
$g_{n-1}^{i_1} g_{n-1} g_{n-2}^{2} g_{n-1} g_{n-1}^{i_2} \in \sum_{i=0}^{2} L_{n-1,\mu} g_{n-1}^{i} L_{n-1,\mu} + \Bbbk g_{n-1}
g_{n-2}^{2} g_{n-1}$, where $i_1, i_2 \in \{ 1,2 \}$, by basis construction for $L_{3,\mu}$.
	
	We use the induction assumption $L_{n-1,\mu} = \sum_{i=0}^{2} L_{n-2,\mu} g_{n-2}^{i} L_{n-2,\mu} + L_{n-4,\mu} g_{n-2}
g_{n-3}^{2} g_{n-2}$. Hence, for $a \in \{ 1,2 \}$
	\[ g_{n-1}^{a} \sum_{i=0}^{2} L_{n-1,\mu} g_{n-1}^{i} L_{n-1,\mu} \subset g_{n-1}^{a} \sum_{i=0}^{2} ( \sum_{j=0}^{2}
L_{n-2,\mu} g_{n-2}^{j} L_{n-2,\mu} + L_{n-4,\mu} g_{n-2} g_{n-3}^{2} g_{n-2} ) g_{n-1}^{i} L_{n-1,\mu} = \]
	\[ = \sum_{i,j=0}^{2} L_{n-2,\mu} g_{n-1}^{a} g_{n-2}^{i} g_{n-1}^{j} L_{n-2,\mu} L_{n-1,\mu} + L_{n-4,\mu} \sum_{i=0}^{2}
g_{n-1}^{a} g_{n-2} g_{n-3}^{2} g_{n-2} g_{n-1}^{i} L_{n-1,\mu} \subset \]
	\[ \subset \sum_{i=0}^{2} L_{n-1,\mu} g_{n-1}^{i} L_{n-1,\mu} + L_{n-1,\mu} g_{n-1} g_{n-2}^{2} g_{n-1} L_{n-1,\mu}  \subset
\sum_{i=0}^{2} L_{n-1,\mu} g_{n-1}^{i} L_{n-1,\mu} + L_{n-3,\mu} g_{n-1} g_{n-2}^{2} g_{n-1}.  \]
	We use bases constructions for $L_{3,\mu}$ and $L_{4,\mu}$ on step $4$. We use Lemma \ref{lm:bimolln} on step $5$.
	Analogically, we prove
	\[ \sum_{i=0}^{2} L_{n-1,\mu} g_{n-1}^{i} L_{n-1,\mu} g_{n-1}^{a} \subset L_{n,\mu}.\]
\end{proof}

\section{Proofs of Auxiliary Results}
\label{PAR}

\commutatorEps*
\begin{proof}
	Since we consider linear mappings it is enough to verify that if the statement is true for $x$ and $y$ it will be true for
$xy$.
	
	We have for $\Delta$ for all $x,y \in H$:
	\[(id_{H}\otimes\Delta)\circ\Delta(xy)=(id_{H}\otimes\Delta)(\sum_{(x),(y)}(-1)^{|x^{''}||y^{'}|}x^{'}y^{'}\otimes
x^{''}y^{''})=\]
	
\[=\sum_{(x),(y),(x^{''}),(y^{''})}(-1)^{|x^{''}||y^{'}|+|(x^{''})^{''}||(y^{''})^{'}|}x^{'}y^{'}\otimes(x^{''})^{'}(y^{''})^{'}\otimes(x^{''})^{''}(y^{''})^{''}.\]
	
	\[ \mu_{H \otimes H \otimes H} ((id_{H}\otimes\Delta)\circ\Delta(x) \otimes (id_{H}\otimes\Delta)\circ\Delta(y))=\]
	\[ = \mu_{H \otimes H \otimes H} ((\sum_{(x),(x^{''})}x^{'}\otimes(x^{''})^{'}\otimes(x^{''})^{''}) \otimes
(\sum_{(y),(y^{''})}y^{'}\otimes(y^{''})^{'}\otimes(y^{''})^{''}))=\]
	
\[=\sum_{(x),(y),(x^{''}),(y^{''})}(-1)^{|x^{''}||y^{'}|+|(x^{''})^{''}||(y^{''})^{'}|}x^{'}y^{'}\otimes(x^{''})^{'}(y^{''})^{'}\otimes(x^{''})^{''}(y^{''})^{''}.\]
	
	\[(\Delta\otimes id_{H})\circ\Delta(xy)=(\Delta\otimes id_{H})(\sum_{(x),(y)}(-1)^{|x^{''}||y^{'}|}x^{'}y^{'}\otimes
x^{''}y^{''})=\]
	
\[=\sum_{(x),(y),(x^{'}),(y^{'})}(-1)^{|x^{''}||y^{'}|+|(y^{'})^{'}||(x^{'})^{''}|}(x^{'})^{'}(y^{'})^{'}\otimes(x^{'})^{''}(y^{'})^{''}\otimes
x^{''}y^{''}.\]
	
	\[\mu_{H \otimes H \otimes H}((\Delta\otimes id_{H})\circ\Delta(x)\otimes(\Delta\otimes id_{H})\circ\Delta(y))=\]
	\[=\mu_{H \otimes H \otimes H}((\sum_{(x),(x^{'})}(x^{'})^{'}\otimes(x^{'})^{''}\otimes
x^{''})\otimes(\sum_{(y),(y^{'})}(y^{'})^{'}\otimes(y^{'})^{''}\otimes y^{''}))=\]
	
\[=\sum_{(x),(y),(x^{'}),(y^{'})}(-1)^{|x^{''}||y^{'}|+|(y^{'})^{'}||(x^{'})^{''}|}(x^{'})^{'}(y^{'})^{'}\otimes(x^{'})^{''}(y^{'})^{''}\otimes
x^{''}y^{''}.\]
	
	Thus
	\[(id_{H}\otimes\Delta)\circ\Delta(xy)=\mu_{H \otimes H \otimes
H}((id_{H}\otimes\Delta)\circ\Delta(x)\otimes(id_{H}\otimes\Delta)\circ\Delta(y))=\]
	\[=\mu_{H \otimes H \otimes H}((\Delta\otimes id_{H})\circ\Delta(x)\otimes(\Delta\otimes
id_{H})\circ\Delta(y))=(\Delta\otimes id_{H})\circ\Delta(xy).\]
	
	We have for $\epsilon$ for all $x,y \in H$:
	\[(id_{H}\otimes\epsilon)\circ\Delta(xy)=(id_{H}\otimes\epsilon)(\sum_{(x),(y)}(-1)^{|x^{''}||y^{'}|}x^{'}y^{'}\otimes
x^{''}y^{''})=\sum_{(x),(y)}x^{'}y^{'}\otimes\epsilon(x^{''})\epsilon(y^{''})=xy\otimes1_{k},\]
	
	\[(\epsilon\otimes id_{H})\circ\Delta(xy)=(\epsilon\otimes id_{H})(\sum_{(x),(y)}(-1)^{|x^{''}||y^{'}|}x^{'}y^{'}\otimes
x^{''}y^{''})=\sum_{(x),(y)}\epsilon(x^{'})\epsilon(y^{'})\otimes x^{''}y^{''}=1_{k}\otimes xy.\]
	
	Thus using superspace isomorphisms defined in Remark \ref{rm:one}, we have
	\[ \nu_{H,k}((id_{H}\otimes\epsilon)\circ\Delta(xy)) = \nu_{H,k}(xy\otimes1_{k}) =  xy = id_{H}(xy), \]
	\[ \nu_{k,H} ((\epsilon\otimes id_{H})\circ\Delta(xy)) = \nu_{k,H} (1_{k}\otimes xy) = xy = id_{H}(xy). \]
	
	Thus
	\[ (id_{H}\otimes\epsilon)\circ\Delta = id_H = (\epsilon\otimes id_{H})\circ\Delta. \]
\end{proof}

\inverseAlgebra*
\begin{proof}
	
	1. 	$\mu^{op}$ is a superspace morphism.
	
	1.1
	Verify that the associativity axiom holds in $H^{op}$:
	\[ \mu^{op} \circ (\mu^{op} \otimes id_{ H^{op} } ) = \mu^{op} \circ ( id_{H^{op}} \otimes \mu^{op} ). \]
	\[ \mu^{op} \circ (\mu^{op} \otimes id_{ H^{op} } ) ( a \otimes b \otimes c ) = \mu^{op} ( (-1)^{|a||b|} ba \otimes c ) =
(-1)^{|a||b| + |c||b| + |c||a|} cba, \]
	\[ \mu^{op} \circ ( id_{H^{op}} \otimes \mu^{op} ) ( a \otimes b \otimes c ) = \mu^{op} ( (-1)^{|b||c|} a \otimes cb ) =
(-1)^{|b||c|+|c||a|+|b||a|} cba \]
	for all $a, b, c \in H^{op}$.

	1.2
	Verify that the unity axiom holds in $H^{op}$:
	\[ \mu^{op} \circ (\eta \otimes id_{H^{op}}) = \mu^{op} \circ (id_{H^{op}} \otimes \eta). \]
	\[ \mu^{op} \circ (\eta \otimes id_{H^{op}}) ( 1_{k} \otimes a ) = \mu^{op} (\eta(1_{k}) \otimes a) = \mu^{op} (1_{H^{op}}
\otimes a) = (-1)^{|1_{H^{op}}||a|} a 1_{H^{op}} = a, \]
	\[ \mu^{op} \circ (id_{H^{op}} \otimes \eta) ( a \otimes 1_{k} ) = \mu^{op} ( a \otimes \eta(1_k) ) = \mu^{op} ( a \otimes
1_{H^{op}} ) = (-1)^{|1_{H^{op}}||a|} 1_{H^{op}} a = a \]
	for all $a \in H^{op}$.
	
	1.3
	Verify that $\Delta$ is a superalgebra morphism in $H^{op}$:
	\[ \Delta \circ \mu^{op} = ( \mu^{op} \otimes \mu^{op} ) \circ ( id_{H^{op}} \otimes \tau_{H^{op},H^{op}} \otimes
id_{H^{op}} ) \circ ( \Delta \otimes \Delta ), \]
	\[ \Delta \circ \mu^{op} (a \otimes b) = (-1)^{|a||b|} \Delta ( ba ) = (-1)^{|a||b|} \Delta (b) \Delta(a) =   \sum_{(a),(b)}
(-1)^{|a||b|+|b^{''}||a^{'}|} b^{'}a^{'} \otimes b^{''}a^{''} = \]
	\[ = \sum_{(a),(b)} (-1)^{|a^{''}||b^{'}|+|a^{'}||b^{'}|+|a^{''}||b^{''}|} b^{'}a^{'} \otimes b^{''}a^{''}, \]
	\[  ( \mu^{op} \otimes \mu^{op} ) \circ ( id_{H^{op}} \otimes \tau_{H^{op},H^{op}} \otimes id_{H^{op}} ) \circ ( \Delta
\otimes \Delta ) (a \otimes b) = \]
	\[ = ( \mu^{op} \otimes \mu^{op} ) ( \sum_{(a),(b)} (-1)^{|a^{''}||b^{'}|} a^{'} \otimes b^{'} \otimes a^{''} \otimes b^{''}
) = \sum_{(a),(b)} (-1)^{|a^{''}||b^{'}|+|a^{'}||b^{'}|+|a^{''}||b^{''}|} b^{'}a^{'} \otimes b^{''}a^{''} \]
	for all $a,b \in H^{op}$.
	
	Remark. We use the following relation for degrees
	\[ |a||b|+|b^{''}||a^{'}| = (|a^{'}|+|a^{''}|)(|b^{'}|+|b^{''}|) +|b^{''}||a^{'}| =
|a^{''}||b^{'}|+|a^{'}||b^{'}|+|a^{''}||b^{''}|. \qed \]
	
	1.4
	Verify that $\epsilon$ is a superalgebra morphism in $H^{op}$:
	\[ \epsilon \circ \mu^{op} = \mu_{k} \circ (\epsilon \otimes \epsilon ). \]
	\[ \epsilon \circ \mu^{op} (a \otimes b) = (-1)^{|a||b|} \epsilon ( ba ) = (-1)^{|a||b|} \epsilon (a) \epsilon (b) =
\epsilon (a) \epsilon (b), \]
	\[ \mu_{k} \circ (\epsilon \otimes \epsilon ) (a \otimes b) = \epsilon (a) \epsilon (b) \]
	for all $a,b \in H^{op}$.
	
	2. $\Delta^{op}$ is a superspace morphism.
	
	2.1
	Verify that the coassociativity axiom holds in $H^{cop}$:
	\[ (\Delta^{op} \otimes id_{H^{cop}}) \circ \Delta^{op} = ( id_{H^{cop}} \otimes \Delta^{op} ) \circ \Delta^{op}. \]
	\[ (\Delta^{op} \otimes id_{H^{cop}}) \circ \Delta^{op} ( a ) = (\Delta^{op} \otimes id_{H^{cop}}) ( \sum_{(a)}
(-1)^{|a^{'}||a^{''}|} a^{''} \otimes a^{'} ) = \]
	\[ = \sum_{(a),(a^{''})} (-1)^{|a^{'}||a^{''}|+|(a^{''})^{'}||(a^{''})^{''}|} (a^{''})^{''} \otimes (a^{''})^{'} \otimes
a^{'}, \]
	
	\[ ( id_{H^{cop}} \otimes \Delta^{op} ) \circ \Delta^{op} (a) = ( id_{H^{cop}} \otimes \Delta^{op} ) ( \sum_{(a)}
(-1)^{|a^{'}||a^{''}|} a^{''} \otimes a^{'} ) = \]
	\[ = \sum_{(a)} (-1)^{|a^{'}||a^{''}|+|(a^{'})^{'}||(a^{'})^{''}|} a^{''} \otimes (a^{'})^{''} \otimes (a^{'})^{'} \]
	for all $a \in H^{cop}$.
	
	Note that
	\[ (\Delta^{op} \otimes id_{H^{cop}}) \circ \Delta^{op} (a) = (id_{H} \otimes \tau_{H,H}) \circ \tau_{H \otimes H,H} \circ
(( id_{H} \otimes \Delta_H ) \circ \Delta_H) (a) = \]
	\[ = (id_{H} \otimes \tau_{H,H}) \circ \tau_{H \otimes H,H} \circ (( \Delta_H \otimes id_{H} ) \circ \Delta_H) (a) = (
id_{H^{cop}} \otimes \Delta^{op} ) \circ \Delta^{op} (a). \]

	2.2
	Verify that the counit axiom holds in $H^{cop}$:
	\[ (\epsilon \otimes id_{H^{cop}}) \circ \Delta^{op} = ( id_{H^{cop}} \otimes \epsilon ) \circ \Delta^{op}. \]
	\[ ( id_{H^{cop}} \otimes \epsilon ) \circ \Delta^{op} ( a ) = \sum_{(a)} (-1)^{|a^{'}||a^{''}|} a^{''} \otimes
\epsilon(a^{'}) = \sum_{(a)} a^{''} \otimes  \epsilon(a^{'}) = a \otimes 1_{k}, \]
	\[ (\epsilon \otimes id_{H^{cop}}) \circ \Delta^{op} ( a ) = \sum_{(a)} (-1)^{|a^{'}||a^{''}|} \epsilon (a^{''}) \otimes
a^{'} = \sum_{(a)} \epsilon (a^{''}) \otimes a^{'} = 1_{k} \otimes a \]
	for all $a \in H^{cop}$.
	
	Thus using superspace isomorphisms defined in Remark \ref{rm:one}, we have
	\[ \nu_{H^{cop},k}((id_{H^{cop}}\otimes\epsilon)\circ\Delta^{op}(a)) = \nu_{H^{cop},k}(a\otimes1_{k}) =  a =
id_{H^{cop}}(a), \]
	\[ \nu_{k,H^{cop}} ((\epsilon\otimes id_{H^{cop}})\circ\Delta^{op}(a)) = \nu_{k,H^{cop}} (1_{k}\otimes a) = a =
id_{H^{cop}}(a). \]
	
	Consequently
	\[ (id_{H^{cop}}\otimes\epsilon)\circ\Delta^{op} = id_{H^{cop}} = (\epsilon\otimes id_{H^{cop}})\circ\Delta^{op}. \]
	
	2.3
	Verify that $\Delta^{op}$ is a superalgebra morphism in $H^{cop}$. We have for all $a,b \in H^{cop}$
	\[\Delta^{op} \circ \mu = ( \mu \otimes \mu ) \circ ( id_{H^{cop}} \otimes \tau_{H^{cop},H^{cop}} \otimes id_{H^{cop}} )
\circ ( \Delta^{op} \otimes \Delta^{op} ). \]
	\[ (\Delta^{op} \circ \mu) ( a \otimes b ) = \sum_{(a),(b)} (-1)^{|a^{''}||b^{'}|+|a^{'}b^{'}||a^{''}b^{''}|} a^{''}b^{''}
\otimes a^{'}b^{'} = \]
	\[ = \sum_{(a),(b)} (-1)^{|a^{'}||a^{''}|+|a^{'}||b^{''}|+|b^{'}||b^{''}|} a^{''}b^{''} \otimes a^{'}b^{'}, \]
	\[ ( \mu \otimes \mu ) \circ ( id_{H^{cop}} \otimes \tau_{H^{cop},H^{cop}} \otimes id_{H^{cop}} ) \circ ( \Delta^{op}
\otimes \Delta^{op} ) ( a \otimes b ) =  \]
	\[ = ( \mu \otimes \mu ) \circ ( id_{H^{cop}} \otimes \tau_{H^{cop},H^{cop}} \otimes id_{H^{cop}} ) ( \sum_{(a), (b) }
(-1)^{|a^{'}||a^{''}|+|b^{'}||b^{''}|} a^{''} \otimes a^{'} \otimes b^{''} \otimes b^{'} ) = \]
	\[ = \sum_{(a),(b)} (-1)^{|a^{'}||a^{''}|+|a^{'}||b^{''}|+|b^{'}||b^{''}|} a^{''}b^{''} \otimes a^{'}b^{'}. \]

	2.4
	\[ \Delta^{op} \circ \eta = \eta \otimes \eta. \]
	\[ \Delta^{op} \circ \eta (1_{k}) = 1_{H^{cop}} \otimes 1_{H^{cop}}, \]
	\[ (\eta \otimes \eta) \circ \Delta_{k} (1_k) = (\eta \otimes \eta) ( 1_{k} \otimes  1_{k} ) = 1_{H^{cop}} \otimes
1_{H^{cop}}. \]

	3.
	Verify that $\Delta^{op}$ is a superalgebra morphism in $H^{op,cop}$. We have for all $a,b \in H^{op,cop}$
	\[\Delta^{op} \circ \mu^{op} = ( \mu^{op} \otimes \mu^{op} ) \circ ( id_{H^{op,cop}} \otimes \tau_{H^{op,cop},H^{op,cop}}
\otimes id_{H^{op,cop}} ) \circ ( \Delta^{op} \otimes \Delta^{op} ). \]
	\[ \Delta^{op} \circ \mu^{op} ( a \otimes b ) = \sum_{(a),(b)} (-1)^{|a||b|+|b^{''}||a^{'}|+|b^{'}a^{'}||b^{''}a^{''}|}
b^{''}a^{''} \otimes b^{'}a^{'} = \]
	\[ = \sum_{(a),(b)} (-1)^{|a^{'}||a^{''}| + |b^{'}|| b^{''}| + |a^{'}||b^{''}| + |a^{''}||b^{''}| + |a^{'}||b^{'}|}
b^{''}a^{''} \otimes b^{'}a^{'}, \]
	
	\[ ( \mu^{op} \otimes \mu^{op} ) \circ ( id_{H^{op,cop}} \otimes \tau_{H^{op,cop},H^{op,cop}} \otimes id_{H^{op,cop}} )
\circ ( \Delta^{op} \otimes \Delta^{op} ) ( a \otimes b ) = \]
	\[ = ( \mu^{op} \otimes \mu^{op} ) \circ ( id_{H^{op,cop}} \otimes \tau_{H^{op,cop},H^{op,cop}} \otimes id_{H^{op,cop}} ) (
\sum_{(a),(b)} (-1)^{|a^{'}||a^{''}| + | b^{'}|| b^{''}|} a^{''} \otimes a^{'} \otimes b^{''} \otimes b^{'} ) = \]
	\[ = ( \mu^{op} \otimes \mu^{op} ) ( \sum_{(a),(b)} (-1)^{|a^{'}||a^{''}| + | b^{'}|| b^{''}| + |a^{'}||b^{''}|} a^{''}
\otimes b^{''} \otimes a^{'} \otimes b^{'} ) = \]
	\[ = \sum_{(a),(b)} (-1)^{|a^{'}||a^{''}| + |b^{'}|| b^{''}| + |a^{'}||b^{''}| + |a^{''}||b^{''}| + |a^{'}||b^{'}|}
b^{''}a^{''} \otimes b^{'}a^{'}. \]
	
	Remark. We use the following relation for degrees
	\[ |a||b|+|b^{''}||a^{'}|+|b^{'}a^{'}||b^{''}a^{''}| = |a^{'}||b^{'}| + |a^{'}||b^{''}| + |a^{''}||b^{'}| + |a^{''}||b^{''}|
+  \]
	\[ + |b^{''}||a^{'}|  + |b^{'}||b^{''}| + |b^{'}||a^{''}| + |a^{'}||b^{''}| + |a^{'}||a^{''}| = |a^{'}||b^{'}| +
|a^{'}||b^{''}| + |a^{''}||b^{''}| + |b^{'}||b^{''}| + |a^{'}||a^{''}|. \]
\end{proof}

\evaluationmapcolone*
\begin{proof}
	If we set $M^{'}=N^{'}=k$ in Proposition \ref{pr:evaluationmap} we shall get a superspace morphism
	\[\lambda:\mathrm{Hom}(M,k)\otimes \mathrm{Hom}(N,k)\to \mathrm{Hom}(M\otimes N,k\otimes k),\]
	\[\lambda(f\otimes g)(m\otimes n)=(-1)^{|g||m|}f(m)\otimes g(n).\]
	
	Using the superspace isomorphism $\nu_{k,k}:k\otimes k\to k$ defined in Remark \ref{rm:one}, we get an isomorphism of vector
superspaces
	\[\lambda_{M,N}=\nu_{k,k}\circ\lambda:\mathrm{Hom}(M,k)\otimes \mathrm{Hom}(N,k)\to \mathrm{Hom}(M\otimes N,k),\]
	\[\lambda_{M,N}(f\otimes g)(m\otimes n)=(-1)^{|g||m|}f(m)g(n).\]
\end{proof}

\evaluationmapcoltwo*
\begin{proof}
	We know from Corollary \ref{cl:evaluationmapcolone}, that the statement holds for $n=2$. Suppose that the result is true for
all $k<n$. We prove that the statement is true for $k=n$. Indeed, from Proposition \ref{pr:evaluationmap}, Corollary
\ref{cl:evaluationmapcolone} and the induction hypothesis it follows that
	\[ \nu_{k,k} \circ \lambda \circ (\lambda_{M_{1},M_{2},...,M_{n-1}} \otimes id_{M_{n}^{*}}) : M_{1}^{*} \otimes M_{2}^{*}
\otimes ... \otimes M_{n}^{*} \to (M_{1} \otimes M_{2} \otimes ... \otimes M_{n})^{*}, \]
	\[ \nu_{k,k} \circ \lambda \circ (\lambda_{M_{1},M_{2},...,M_{n-1}} \otimes id_{M_{n}^{*}}) (\bigotimes_{h=1}^{n}f_{h})
(\bigotimes_{h=1}^{n}m_{h}) =  \]
	\[ = (-1)^{|f_{n}|\sum_{j=1}^{n-1}|m_{j}|}\ \nu_{k,k} ((-1)^{\sum_{i=2}^{n-1} \sum_{j=1}^{i-1} |f_{i}||m_{j}|}
\prod_{h=1}^{n-1}f_{h}(m_{h}) \otimes f_{n}(m_{n})) = \]
	\[ = (-1)^{\sum_{i=2}^{n} \sum_{j=1}^{i-1} |f_{i}||m_{j}|} \prod_{h=1}^{n}f_{h}(m_{h}). \]
	
	Denote by $\lambda_{M_{1},M_{2},...,M_{n}} := \nu_{k,k} \circ \lambda \circ (\lambda_{M_{1},M_{2},...,M_{n-1}} \otimes
id_{M_{n}^{*}}) $. Note that $\lambda_{M_{1},M_{2},...,M_{n}}$ is an isomorphism of vector superspaces by definition.
\end{proof}

\HopfDualStruct*
\begin{proof}	
	Define a multiplication as a superspace morphism
	\[ \mu_{H^{*}} := \Delta_{H}^{*} \circ \lambda_{H,H}, \]
	\[ \mu_{H^{*}} ( f \otimes g ) (a) = \sum_{(a)} (-1)^{|a^{'}||g|} f(a^{'}) g(a^{''}) \]
	for all $f,g \in H^{*}, a \in H$.
	
	Verify that the associativity axiom holds
	\[ ( \mu_{H^{*}} \circ ( \mu_{H^{*}} \otimes id_{H^{*}} ) ( f \otimes g \otimes h ) ) ( a ) =   \sum_{(a)} (-1)^{|h||a^{'}|}
\mu_{H^{*}} ( f \otimes g )(a^{'}) h(a^{''}) = \]
	\[ \sum_{(a),(a^{'})} (-1)^{|h||a^{'}| + |g||(a^{'})^{'}|}  f((a^{'})^{'}) g((a^{'})^{''}) h(a^{''}) =  ( f \otimes g
\otimes h ) ( \sum_{(a),(a^{'})} (a^{'})^{'} \otimes (a^{'})^{''} \otimes a^{''} ), \]
	\[ ( \mu_{H^{*}} \circ ( id_{H^{*}} \otimes \mu_{H^{*}} ) ( f \otimes g \otimes h ) ) ( a ) =  \sum_{(a)} (-1)^{|g||a^{'}| +
|h||a^{'}|} f(a^{'}) \mu_{H^{*}} ( g \otimes h )(a^{''}) = \]
	\[ = \sum_{(a),(a^{''})} (-1)^{|g||a^{'}| + |h||a^{'}| + |h||(a^{''})^{'}|} f(a^{'}) g((a^{''})^{'}) h((a^{''})^{''}) = ( f
\otimes g \otimes h ) ( \sum_{(a),(a^{''})} a^{'} \otimes (a^{''})^{'} \otimes (a^{''})^{''} ), \]
	for all $f,g,h \in H^{*},a \in H$.
	
	Recall that from the coassociativity axiom in $H$ it follows that for all $a \in H$
	\[ \sum_{(a),(a^{'})} (a^{'})^{'} \otimes (a^{'})^{''} \otimes a^{''} = \sum_{(a),(a^{''})} a^{'} \otimes (a^{''})^{'}
\otimes (a^{''})^{''}. \]
	Thus the statement holds.
	
	Define a unit as a superspace morphism
	\[ \eta_{H^{*}} := \epsilon_{H}^{*} \circ \chi, \]
	\[ \eta_{H^{*}} ( 1_{k} ) = \chi ( 1_{k} ) \epsilon_{H} = \epsilon_{H}, \]
	where $\chi$ is the superspace isomorphism defined in Remark \ref{rm:chiisom}.
	
	Verify that the unit axiom holds
	\[ (\mu_{H^{*}} \circ (\eta_{H^{*}} \otimes id_{H^{*}}) ( 1_k \otimes f )) (a) = \mu_{H^{*}} ( \epsilon_{H} \otimes f ) (a)
= \sum_{(a)} (-1)^{|f| |a^{'}|} \epsilon_{H}(a^{'}) f(a^{''}) =  f( \sum_{(a)} \epsilon_{H}(a^{'}) a^{''} ) = f(a), \]
	\[ (\mu_{H^{*}} \circ (id_{H^{*}} \otimes \eta_{H^{*}}) ( f \otimes 1_k )) (a) = \mu_{H^{*}} ( f \otimes \epsilon_{H} ) (a)
= \sum_{(a)} (-1)^{|\epsilon_{H}| |a^{'}|} f(a^{'}) \epsilon_{H}(a^{''}) =  f( \sum_{(a)} a^{'} \epsilon_{H}(a^{''}) ) = f(a)
\]
	for all $f \in H^{*}, a \in H$.
	
	Define a comultiplication as a superspace morphism
	\[ \Delta_{H^{*}} := \lambda_{H,H}^{-1} \circ \mu_{H}^{*}, \]
	\[ \Delta_{H^{*}}(f) = \sum_{(f)} f^{'} \otimes f^{''}, \]
	\[ f(ab) = \sum_{(a),(b)} (-1)^{|a||f^{''}|} f^{'}(a) f^{''} (b) \]
	for all $f \in H^{*}, a,b \in H$.
	
	Recall that from Corollary \ref{cl:evaluationmapcoltwo} it follows that there exists a superspace morphism
	\[ \lambda_{H,H,H}: H^* \otimes H^* \otimes H^* \to (H \otimes H \otimes H)^*, \]
	\[ \lambda_{H,H,H}( f_1 \otimes f_2 \otimes f_3 ) ( h_1 \otimes h_2 \otimes h_3 ) = (-1)^{|f_2||h_1| + |f_3|( |h_1| + |h_2|
)} f_1(h_1) f_2 (h_2) f_3 (h_3) \]
	for all $f_1, f_2, f_3 \in H^{*}$, $h_1, h_2, h_3 \in H$.
	
	Verify that the coassociativity axiom holds
	\[ ( (\Delta_{H^{*}} \otimes id_{H^{*}}) \circ \Delta_{H^{*}} (f) ) (a \otimes b \otimes c) = ( \sum_{(f)} \Delta_{H^{*}}
(f^{'}) \otimes f^{''} ) (a \otimes b \otimes c) = \]
	\[ = \sum_{(f)} (-1)^{|f^{''}|( |a| + |b| )} \Delta_{H^{*}} (f^{'}) ( a \otimes b ) \otimes f^{''} ( c ) =  \sum_{(f),
(f^{'})} (-1)^{|f^{''}|( |a| + |b| ) + |a||(f^{'})^{''}|} (f^{'})^{'}(a) \otimes (f^{'})^{''}(b) \otimes f^{''}( c ), \]
	
	\[ ( (id_{H^{*}} \otimes \Delta_{H^{*}}) \circ \Delta_{H^{*}} (f) ) (a \otimes b \otimes c) = ( \sum_{(f)} f^{'} \otimes
\Delta_{H^{*}}(f^{''}) ) (a \otimes b \otimes c) = \]
	\[ = \sum_{(f)} (-1)^{|f^{''}||a|} f^{'} (a) \otimes \Delta_{H^{*}} (f^{''})( b \otimes c ) =  \sum_{(f), (f^{''})}
(-1)^{|f^{''}||a| + |(f^{''})^{''}| |b|} f^{'} (a) \otimes (f^{''})^{'} (b) \otimes (f^{''})^{''} (c), \]
	
	\[ \lambda_{H,H,H}( \sum_{(f), (f^{'})} (f^{'})^{'} \otimes (f^{'})^{''} \otimes f^{''} ) (a \otimes b \otimes c) =  \]
	\[ = \sum_{(f), (f^{'})} (-1)^{|f^{''}|( |a| + |b| ) + |a||(f^{'})^{''}|} (f^{'})^{'}(a) (f^{'})^{''}(b) f^{''}( c ) =
\sum_{(f)} (-1)^{|f^{''}|( |a| + |b| )} f^{'}(ab) f^{''}( c ) = f(abc), \]
	
	\[ \lambda_{H,H,H} ( \sum_{(f), (f^{''})} f^{'} \otimes (f^{''})^{'} \otimes (f^{''})^{''} ) (a \otimes b \otimes c) = \]
	\[ = \sum_{(f), (f^{''})} (-1)^{|f^{''}||a| + |(f^{''})^{''}| |b|} f^{'} (a) (f^{''})^{'} (b) (f^{''})^{''} (c) = \sum_{(f)}
(-1)^{|f^{''}||a|} f^{'} (a) f^{''} (bc) = f(abc) \]
	for all $f \in H^{*}, a,b,c \in H$.
	
	Since $\lambda_{H,H,H}$ is the isomorphism of vector superspaces
	\[ \sum_{(f), (f^{'})}  (f^{'})^{'} \otimes (f^{'})^{''} \otimes f^{''} = \sum_{(f), (f^{''})} f^{'}  \otimes (f^{''})^{'}
\otimes (f^{''})^{''} \Rightarrow  (\Delta_{H^{*}} \otimes id_{H^{*}}) \circ \Delta_{H^{*}} = (id_{H^{*}} \otimes
\Delta_{H^{*}}) \circ \Delta_{H^{*}}. \]
	
	Define a counit as a superspace morphism
	\[ \epsilon_{H^{*}} := \chi^{-1} \circ \eta^{*}, \]
	\[ \epsilon_{H^{*}} ( f ) = f(1_{H}) \]
	for all $f \in H^{*}$.
	
	Verify that the counit axiom holds
	\[ (id_{H^{*}} \otimes \epsilon_{H^{*}}) \circ \Delta_{H^{*}} ( f ) = \sum_{(f)} f^{'} \otimes \epsilon_{H^{*}}( f^{''} ) =
\sum_{(f)} f^{'} \otimes f^{''}(1_{H}) =  \sum_{(f)} f^{'} f^{''}(1_{H}) \otimes 1_{k} = f(?1_{H}) \otimes 1_{k} = f \otimes
1_{k}, \]
	\[ (\epsilon_{H^{*}} \otimes id_{H^{*}}) \circ \Delta_{H^{*}} ( f ) = \sum_{(f)} \epsilon_{H^{*}}( f^{'} ) \otimes f^{''} =
\sum_{(f)} f^{'}(1_{H}) \otimes f^{''} =  1_{k} \otimes \sum_{(f)} f^{'}(1_{H}) f^{''} = 1_{k} \otimes f(1_{H}?) = 1_k \otimes f
\]
	for all $f \in H^{*}$, $?$ is a numb variable.
	
	Thus using superspace isomorphisms defined in Remark \ref{rm:one}, we get
	\[ \nu_{H^{*},k}((id_{H^{*}}\otimes\epsilon_{H^{*}})\circ\Delta_{H^{*}}(f)) = \nu_{H^{*},k}(f\otimes1_{k}) = f =
id_{H^{*}}(f), \]
	\[ \nu_{k,H^{*}} ((\epsilon_{H^{*}}\otimes id_{H^{*}})\circ\Delta_{H^{*}}(f)) = \nu_{k,H^{*}} (1_{k}\otimes f) = f =
id_{H^{*}}(f). \]
	
	Thus
	\[ (id_{H^{*}}\otimes\epsilon_{H^{*}})\circ\Delta_{H^{*}} = id_{H^{*}} = (\epsilon_{H^{*}}\otimes
id_{H^{*}})\circ\Delta_{H^{*}}. \]
	
	Verify that the comultiplication is a superalgebra morphism
	\[ (\lambda_{H,H} \circ \Delta_{H^{*}} \circ \mu_{H^{*}} ( f \otimes g )) ( a \otimes b ) = \mu_{H^{*}} ( f \otimes g ) ( ab
) = \]
	\[ = \sum_{(a),(b)} (-1)^{|a^{''}||b^{'}| + |g||a^{'}| + |g||b^{'}|}  f(a^{'}b^{'})  g(a^{''}b^{''}) \Leftrightarrow \]
	\[ \Leftrightarrow (\Delta_{H^{*}} \circ \mu_{H^{*}} ( f \otimes g )) ( a \otimes b ) = \sum_{(a),(b)} (-1)^{|a^{''}||b^{'}|
+ |g||a^{'}| + |g||b^{'}|}  f(a^{'}b^{'}) \otimes g(a^{''}b^{''}), \]
	\[ ( (\mu_{H^{*}} \otimes \mu_{H^{*}}) \circ (id_{H^{*}} \otimes \tau_{H^{*},H^{*}} \otimes id_{H^{*}}) \circ
(\Delta_{H^{*}} \otimes \Delta_{H^{*}}) ( f \otimes g )) ( a \otimes b ) = \]
	\[ = (\sum_{(f),(g)} (-1)^{|f^{''}||g^{'}|} \mu_{H^{*}} (f^{'} \otimes g^{'}) \otimes \mu_{H^{*}} (f^{''} \otimes g^{''})) (
a \otimes b ) = \]
	\[ = \sum_{(f),(g),(a),(b)} (-1)^{|f^{''}||g^{'}| + |f^{''}||a| + |g^{''}||a| + |g^{'}||a^{'}| + |g^{''}||b^{'}|}
f^{'}(a^{'}) g^{'}(a^{''}) \otimes f^{''}(b^{'}) g^{''}(b^{''}) = \]
	\[ = \sum_{(a),(b)} (-1)^{|a^{''}||b^{'}| + |g||a^{'}| + |g||b^{'}|} (\sum_{(f)} (-1)^{|f^{''}||a^{'}|} f^{'}(a^{'})
f^{''}(b^{'})) \otimes (\sum_{(g)} (-1)^{|g^{''}||a^{''}|} g^{'}(a^{''}) g^{''}(b^{''})) = \]
	\[ = \sum_{(a),(b)} (-1)^{|a^{''}||b^{'}| + |g||a^{'}| + |g||b^{'}|} f(a^{'} b^{'}) \otimes g(a^{''} b^{''})\]
	for all $f,g \in H^{*}, a,b \in H$.
	
	Note that
	\[ |f^{''}||g^{'}| + |f^{''}||a| + |g^{''}||a| + |g^{'}||a^{'}| + |g^{''}||b^{'}| = |a^{''}||b^{'}| + |b^{'}||a| +
|b^{''}||a| + |a^{''}||a^{'}| + |b^{''}||b^{'}| =  \]
	\[ = |a^{''}||b^{'}| + ( |b^{''}| + |a^{''}| ) |a^{'}| + |b^{''}||a^{''}| + ( |a^{''}| + |b^{''}| ) |b^{'}| + |b^{'}||a^{'}|
= \]
	\[ = |a^{''}||b^{'}| + |g||a^{'}| + |g||b^{'}| + |g^{''}||a^{''}| + |f^{''}||a^{'}|. \]
	
	Note that
	\[ \epsilon_{H} ( ab ) = \epsilon_{H} (a) \epsilon_{H} (b) = \sum_{(\epsilon_{H})} (\epsilon_{H})^{'}(a)
(\epsilon_{H})^{''}(b) \Leftrightarrow \Delta_{H^{*}} ( \epsilon_{H} ) = \epsilon_{H} \otimes \epsilon_{H} \]
	for all $a,b \in H$.
	
	Therefore
	\[ (\Delta_{H^{*}} \circ \eta_{H^*}) (1_{k}) = \Delta_{H^{*}} ( \epsilon_{H} )  = \epsilon_{H} \otimes \epsilon_{H}, \]
	\[ (\eta_{H^*} \otimes \eta_{H^*}) \circ \Delta_{k} (1_k) = (\eta_{H^*} \otimes \eta_{H^*}) ( 1_{k} \otimes 1_{k} ) =
\epsilon_{H} \otimes \epsilon_{H}. \]
	
	Verify that the counit is a superalgebra morphism
	\[ (\epsilon_{H^{*}} \circ \mu_{H^{*}}) ( f \otimes g ) = \mu_{H^{*}} ( f \otimes g ) (1_{H}) = f(1_{H}) g(1_{H}), \]
	\[ (\mu_{k} \circ ( \epsilon_{H^{*}} \otimes \epsilon_{H^{*}} )) ( f \otimes g ) = \mu_{k} ( f(1_{H}) \otimes g(1_{H} )) =
f(1_{H}) g(1_{H}), \]
	for all $f,g \in H^{*}$.
	
	\[ \epsilon_{H^{*}} \circ \eta_{H^{*}} (1_{k}) = \epsilon_{H^{*}} ( \epsilon_{H} ) = \epsilon_{H}(1_{H}) = 1_{k} = id_{k}
(1_{k}).  \]
	
	Define the antipode as a superspace morphism
	\[ S_{H^{*}} = S^{*}, \]
	\[ S^{*}( f ) = f \circ S, \]
	for all $f \in H^{*}$.
	
	\[ (\eta_{H^{*}} \circ \epsilon_{H^{*}}) (f) = \eta_{H^{*}} ( f(1_{H}) ) = f(1_{H}) \eta_{H^{*}} ( 1_{k} ) = f(1_{H})
\epsilon_{H}, \]
	\[ (\mu_{H^{*}} \circ ( S_{H^{*}} \otimes id_{H^{*}} ) \circ \Delta_{H^{*}} ( f )) (a) = \mu_{H^{*}} ( \sum_{(f)}
S_{H^{*}}(f^{'}) \otimes f^{''} ) (a) = \]
	\[ = \sum_{(a),(f)} (-1)^{|f^{''}||a^{'}|} S_{H^{*}}(f^{'})(a^{'}) f^{''} ( a^{''} ) = \sum_{(a),(f)} (-1)^{|f^{''}||a^{'}|}
f^{'}(S(a^{'})) f^{''} ( a^{''} ) = \]
	\[ = \sum_{(a)} f( S(a^{'}) a^{''} ) = f ( \sum_{(a)} S(a^{'}) a^{''} ) = f(1_{H}) \epsilon_{H}(a) = ((\eta_{H^{*}} \circ
\epsilon_{H^{*}}) (f) )(a), \]
	
	\[ (\mu_{H^{*}} \circ ( id_{H^{*}} \otimes S_{H^{*}} ) \circ \Delta_{H^{*}} ( f )) (a) = \mu_{H^{*}} ( \sum_{(f)} f^{'}
\otimes S_{H^{*}} (f^{''} )) (a) = \]
	\[ = \sum_{(a),(f)} (-1)^{|f^{''}| |a^{'}|} f^{'} (a^{'}) S_{H^{*}} (f^{''} ) (a^{''}) = \sum_{(a),(f)} (-1)^{|f^{''}|
|a^{'}|} f^{'} (a^{'}) f^{''} (S(a^{''})) = \]
	\[ = \sum_{(a)} f( a^{'} S(a^{''}) ) = f( \sum_{(a)} a^{'} S(a^{''}) ) = f(1_{H}) \epsilon_{H}(a) = ((\eta_{H^{*}} \circ
\epsilon_{H^{*}}) (f) )(a) \]
	
	for all $f \in H^{*}, a \in H$.
	
\end{proof}

\Alglinfunct*
\begin{proof}
	Note that
	\[ | ( f * g ) (a)| = |f * g| + |a| = |f(a^{'})| + |g(a^{''})| = |f| + |a^{'}| + |g| + |a^{''}| = |f| + |g| + |a|
\Rightarrow |f * g| = |f| + |g|. \]
	Thus the multiplication is a superspace morphism.
	
	Verify that the associativity axiom holds
	\[ ((f * g) * h) (a) = \sum_{(a)} (-1)^{|h||a^{'}|} (f * g) (a^{'}) h(a^{''}) = \sum_{(a),(a^{'})} (-1)^{|h||a^{'}| +
|g||(a^{'})^{'}|} f((a^{'})^{'}) g((a^{'})^{''}) h(a^{''}) = \]
	\[ = \mu_{B} \circ (\mu_{B} \otimes id_{B} ) \circ (f \otimes g \otimes h) ( \sum_{(a),(a^{'})} (a^{'})^{'} \otimes
(a^{'})^{''} \otimes a^{''} ) =  \mu_{B} \circ (\mu_{B} \otimes id_{B} ) \circ (f \otimes g \otimes h) ( \sum_{(a),(a^{''})}
a^{'} \otimes (a^{''})^{'} \otimes (a^{''})^{''} ) = \]
	\[ = \sum_{(a),(a^{''})} (-1)^{|a^{'}| ( |g| + |h| ) + |(a^{''})^{'}||h|} f(a^{'}) g((a^{''})^{'}) h((a^{''})^{''}) =
\sum_{(a)} (-1)^{|a^{'}| ( |g| + |h| )} f(a^{'}) (g * h) (a^{''}) = (f * ( g * h )) (a) \]
	for all $f,g,h \in \mathrm{Hom}(A,B), a \in A$.
	
	It is obvious that the unity is a superspace morphism. Verify that the unity axiom holds
	\[ ((\eta_{B} \circ \epsilon_{A}) * f) (a) = \sum_{(a)} (-1)^{|f||a^{'}|} \epsilon_{A}(a^{'}) f(a^{''}) = f( \sum_{(a)}
\epsilon_{A}(a^{'}) a^{''} ) = f(a), \]
	\[ (f * (\eta_{B} \circ \epsilon_{A})) (a) = \sum_{(a)} (-1)^{|\eta_{B} \circ \epsilon_{A}||a^{'}|} f(a^{'})
\epsilon_{A}(a^{''}) = f( \sum_{(a)} a^{'} \epsilon_{A}(a^{''} ) ) = f(a) \]
	for all $f \in \mathrm{Hom}(A,B), \; a \in A$.
\end{proof}

\HopfQuotient*
\begin{proof}
	Notice that
	\[ H / I = (H_{0} + I) / I \oplus (H_{1} + I) / I. \]
	Define a multiplication in $H/I$ by
	\[ \mu_{H/I}: H/I \otimes H/I \to H/I, \]
	\[ \mu_{H/I} ( (a + I) \otimes (b + I) ) = ab + I \]
	for all $a, b \in H$. It is clear that the multiplication is a superspace morphism which respects the associativity axiom.
Moreover, we have for all $\alpha, \beta \in \mathbb{Z}_{2}$
	\[ \mu_{H/I} ( (H_{\alpha} + I) / I \otimes (H_{\beta} + I) / I ) = H_{\alpha + \beta} + I. \]
	
	Respectively a unit is defined by
	\[ \eta_{H/I} : k \to H/I, \]
	\[ \eta_{H/I} (1_{k}) = 1_{H} + I. \]
	The unit is a superspace morphism that respects the unity axiom.
	
	Before we define a comultiplication we investigate intermediate constructions.
	$J = H \otimes I + I \otimes H$ is a $\mathbb{Z}_2$-graded ideal in $H \otimes H$. Indeed,
	\[ H \otimes I + I \otimes H = (H_{0} + H_{1}) \otimes (H_{0} \cap I + H_{1} \cap I) +  (H_{0} \cap I + H_{1} \cap I)
\otimes (H_{0} + H_{1}) = \]
	\[ = (H_{0} \otimes (H_{0} \cap I) + (H_{0} \cap I) \otimes H_{0} + H_{1} \otimes (H_{1} \cap I) + (H_{1} \cap I) \otimes
H_{1}) + \]
	\[ + (H_{1} \otimes (H_{0} \cap I) + (H_{1} \cap I) \otimes H_{0} + H_{0} \otimes (H_{1} \cap I) + (H_{0} \cap I) \otimes
H_{1}) = \]
	\[ = (H_{0} \otimes H_{0} + H_{1} \otimes H_{1}) \cap J + (H_{1} \otimes H_{0} + H_{0} \otimes H_{1}) \cap J. \]
	
	Define the canonical superspace morphism
	\[ \gamma: H \otimes H \to H \otimes H / ( H \otimes I + I \otimes H ), \]
	\[ \gamma(a \otimes b) = a \otimes b + J \]
	for all $a,b \in H$.
	
	Consider a superspace morphism
	\[ \phi^{'}: H \otimes H \to H/I \otimes H/I, \]
	\[ \phi^{'} ( a \otimes b ) = (a + I) \otimes (b + I) \]
	for all $a,b \in H$. It is clear that $im(\phi^{'}) = H/I \otimes H/I$, $ker(\phi^{'}) = J$. According to the first
isomorphism theorem we can construct a superspace isomorphism
	\[ \phi : H \otimes H / ( H \otimes I + I \otimes H ) \to H/I \otimes H/I, \]
	\[ \phi ( a \otimes b + J ) = (a + I) \otimes (b + I) \]
	for all $a,b \in H$.
	
	Define a comultiplication by
	\[ \Delta_{H/I}: H/I \to H/I \otimes H/I, \]
	\[ \Delta_{H/I} ( a + I ) = \phi \circ \gamma \circ \Delta_{H} (a) =  \sum_{(a)} (a^{'} + I) \otimes (a^{''} + I) \]
	for all $a \in H$. The mapping is correctly defined. Indeed, $a + I = b + I \Leftrightarrow a - b \in I \Rightarrow
\Delta_{H} (a) - \Delta_{H} (b) \in H \otimes I + I \otimes H \Leftrightarrow \Delta_{H} (a) + J = \Delta_{H} (b) + J
\Leftrightarrow \phi ( \Delta_{H} (a) + J ) = \phi ( \Delta_{H} (b) + J ) \Leftrightarrow \Delta_{H/I} ( a + I ) = \Delta_{H/I}
( b + I ) $ for all $a, b \in H$. It is clear that $\Delta_{H/I}$ is a superspace morphism.
	Verify that the coassociativity axiom holds
	\[ ((\Delta_{H/I} \otimes id_{H/I}) \circ \Delta_{H/I}) ( a + I ) = ((\Delta_{H/I} \otimes id_{H/I})) ( \sum_{(a)} (a^{'} +
I) \otimes (a^{''} + I) ) = \]
	\[ = \sum_{(a),(a^{'})} ((a^{'})^{'} + I) \otimes ((a^{'})^{''} + I) \otimes (a^{''} + I) = (\pi \otimes \pi \otimes \pi) (
\sum_{(a),(a^{'})} (a^{'})^{'} \otimes (a^{'})^{''} \otimes a^{''}  ), \]
	\[ ((id_{H/I} \otimes \Delta_{H/I}) \circ \Delta_{H/I}) ( a + I ) = ((id_{H/I} \otimes \Delta_{H/I})) ( \sum_{(a)} (a^{'} +
I) \otimes (a^{''} + I) ) = \]	
	\[ = \sum_{(a),(a^{''})} (a^{'} + I) \otimes ((a^{''})^{'} + I) \otimes ((a^{''})^{''} + I) =  (\pi \otimes \pi \otimes \pi)
( \sum_{(a),(a^{''})} a^{'} \otimes ((a^{''})^{'} \otimes ((a^{''})^{''} ) \]
	for all $a \in H$.
	The result follows from the coassociativity axiom in $H$.
	
	Define a counit by
	\[ \epsilon_{H/I}: H/I \to k, \]
	\[ \epsilon_{H/I}(a + I) = \epsilon_{H} (a) \]
	for all $a \in H$.
	The mapping is correctly defined. Indeed, $a + I = b + I \Leftrightarrow a - b \in I \Leftrightarrow \epsilon_{H} (a) -
\epsilon_{H}(b) \in \epsilon_{H}(I) \Rightarrow \epsilon_{H} (a) = \epsilon_{H}(b)$ for all $a, b \in H$. It is clear that
$\epsilon_{H/I}$ is a superspace morphism.
	Verify that the counit axiom holds
	\[ (( id_{H/I} \otimes \epsilon_{H/I} ) \circ \Delta_{H/I}) (a + I) = ( id_{H/I} \otimes \epsilon_{H/I} ) ( \sum_{(a)}
(a^{'} + I) \otimes (a^{''} + I) ) = \]
	\[ = \sum_{(a)} (a^{'} + I) \otimes \epsilon_{H}(a^{''}) = ( ( \sum_{(a)} a^{'} \epsilon_{H}(a^{''}) ) + I) \otimes 1_{k} =
(a + I) \otimes 1_{k}, \]
	\[ (( \epsilon_{H/I} \otimes id_{H/I} ) \circ \Delta_{H/I}) (a + I) = ( \epsilon_{H/I} \otimes id_{H/I} ) ( \sum_{(a)}
(a^{'} + I) \otimes (a^{''} + I) ) = \]
	\[ = \sum_{(a)} \epsilon_{H}(a^{'}) \otimes (a^{''} + I) = 1_{k} \otimes ((\sum_{(a)} \epsilon_{H}(a^{'}) a^{''} ) + I) =
1_{k} \otimes (a + I) \]
	for all $a \in H$.
	
	Thus using superspace isomorphisms defined in Remark \ref{rm:one}, we have
	\[ \nu_{H/I,k}((id_{H/I}\otimes\epsilon_{H/I})\circ\Delta_{H/I}(a + I)) = \nu_{H/I,k}((a + I)\otimes1_{k}) = a + I =
id_{H/I}(a + I), \]
	\[ \nu_{k,H/I} ((\epsilon_{H/I}\otimes id_{H/I})\circ\Delta_{H/I}(a + I)) = \nu_{k,H/I} (1_{k}\otimes (a + I)) = a + I =
id_{H/I}(a + I). \]
	
	Consequently,
	\[ (id_{H/I}\otimes\epsilon_{H/I})\circ\Delta_{H/I} = id_{H/I} = (\epsilon_{H/I}\otimes id_{H/I})\circ\Delta_{H/I}. \]
	
	We verify that the comultiplication is a superalgebra morphism
	\[ (( \mu_{H/I} \otimes \mu_{H/I} ) \circ ( id_{H/I} \otimes \tau_{H/I,H/I} \otimes id_{H/I} ) \circ ( \Delta_{H/I} \otimes
\Delta_{H/I} )) ( (a + I) \otimes (b + I) ) =  \]
	\[ = ( \mu_{H/I} \otimes \mu_{H/I} ) ( \sum_{(a),(b)} (-1)^{|b^{'}||a^{''}|} (a^{'} + I) \otimes (b^{'} + I) \otimes (a^{''}
+ I) \otimes (b^{''} + I) ) = \]
	\[ = \sum_{(a),(b)} (-1)^{|b^{'}||a^{''}|} (a^{'}b^{'} + I) \otimes (a^{''}b^{''} + I), \]
	
	\[ (\Delta_{H/I} \circ \mu_{H/I}) ( (a + I) \otimes (b + I) ) = \Delta_{H/I} ( ab + I ) = (\phi \circ \gamma) (\Delta_{H}
(ab)  ) = \]
	\[ = (\phi \circ \gamma) (\Delta_{H}(a) \Delta_{H}(b) ) = (\phi \circ \gamma) ( \sum_{(a),(b)} (-1)^{|b^{'}||a^{''}|}
a^{'}b^{'} \otimes a^{''}b^{''} ) = \]
	\[ = \sum_{(a),(b)} (-1)^{|b^{'}||a^{''}|} (\phi \circ \gamma) (a^{'}b^{'} \otimes a^{''}b^{''}) = \sum_{(a),(b)}
(-1)^{|b^{'}||a^{''}|} (a^{'} b^{'} + I) \otimes (a^{''} b^{''} + I) \]
	for all $a, b \in H$.
	
	\[ (\Delta_{H/I} \circ \eta_{H/I}) ( 1_{k} ) = \Delta_{H/I} ( 1_{H} + I ) = (1_{H} + I) \otimes (1_{H} + I), \]
	\[ ((\eta_{H/I} \otimes \eta_{H/I}) \circ \Delta_{k}) (1_{k}) = (\eta_{H/I} \otimes \eta_{H/I}) ( 1_{k} \otimes 1_{k} ) =
(1_{H} + I) \otimes (1_{H} + I). \]
	
	We verify that the counit is a superalgebra morphism
	\[ (\epsilon_{H/I} \circ \mu_{H/I}) ( ( a + I ) \otimes ( b + I ) ) = \epsilon_{H/I} ( ab + I ) = \epsilon_{H} (ab) =
\epsilon_{H}(a) \epsilon_{H}(b), \]
	\[ (\mu_{k} \circ ( \epsilon_{H/I} \otimes \epsilon_{H/I} )) ( ( a + I ) \otimes ( b + I ) ) = \epsilon_{H}(a)
\epsilon_{H}(b)\]
	for all $a,b \in H$.
	
	\[ (\epsilon_{H/I} \circ \eta_{H/I}) (1_{k}) = \epsilon_{H/I} ( 1_H + I ) = \epsilon_{H} (1_{H}) = 1_{k} = id_{k} ( 1_{k} ).
\]
	Define a antipode by
	\[ S_{H/I} : H/I \to H/I, \]
	\[ S_{H/I} ( a + I ) = S_{H} (a) + I \]
	for all $a \in H$.
	The mapping is correctly defined. Indeed, $a + I = b + I \Leftrightarrow a - b \in I \Rightarrow S_{H}(a) - S_{H}(b) \in
S_{H}(I) \Leftrightarrow  S_{H}(a) - S_{H}(b) \in I \Leftrightarrow S_{H}(a) + I = S_{H}(b) + I$ for all $a,b \in H$. It is
clear that $S_{H/I}$ is a superspace morphism. Moreover,
	\[ \mu_{H/I} \circ ( S_{H/I} \otimes id_{H/I} ) \circ \Delta_{H/I} (a + I) = \mu_{H/I} ( \sum_{(a)} (S_{H} (a^{'}) + I)
\otimes (a^{''} + I) ) = \]
	\[ = ( \sum_{(a)} (S_{H} (a^{'}) a^{''} ) + I = \epsilon_{H}(a) (1_{H} + I) = \eta_{H/I} ( \epsilon_{H}(a) ) = (\eta_{H/I}
\circ \epsilon_{H/I}) (a + I), \]
	
	\[ \mu_{H/I} \circ ( id_{H/I} \otimes S_{H/I} ) \circ \Delta_{H/I} (a + I) = \mu_{H/I} ( \sum_{(a)} (a^{'} + I) \otimes
(S_{H} (a^{''}) + I) ) = \]
	\[ = ( \sum_{(a)} (a^{'} S_{H} (a^{''}) ) + I = \epsilon_{H}(a) (1_{H} + I) = \eta_{H/I} ( \epsilon_{H}(a) ) = (\eta_{H/I}
\circ \epsilon_{H/I}) (a + I) \]
	for all $a \in H$.
	
	We verify that the canonical superspace morphism
	\[ \pi: H \to H/I, \]
	\[ \pi(a) = a + I \]
	for all $a \in H$, is a Hopf superalgebra morphism.
	We show that $\pi$ is a superalgebra morphism
	\[ (\pi \circ \mu_{H}) ( a \otimes b ) = ab + I, \]
	\[ (\mu_{H/I} \circ ( \pi \otimes \pi )) ( a \otimes b ) = \mu_{H/I} ( (a + I) \otimes (b + I) ) = ab + I \]
	for all $a,b \in H$. Moreover,
	
	\[ (\pi \circ \eta_{H}) (1_{k}) = 1_{H} + I = \eta_{H/I} (1_{k}). \]
	
	We show that $\pi$ is a supercoalgebra morphism
	\[ ((\pi \otimes \pi) \circ \Delta_{H}) (a) = \sum_{(a)} (a^{'} + I) \otimes (a^{''} + I) = \Delta_{H/I} ( a + I ) = (
\Delta_{H/I} \circ \pi ) ( a ), \]
	\[ (\epsilon_{H/I} \circ \pi) (a) = \epsilon_{H/I} ( a + I ) = \epsilon_{H}(a) \]
	for all $a \in H$.
	
	We verify that $\pi$ commutes with antipodes
	\[ (S_{H/I} \circ \pi) (a) = S_{H} (a) + I = \pi (S_{H} (a)) = (\pi \circ S_{H}) (a) \]
	for all $a \in H$.
\end{proof}

\Sprop*
\begin{proof}
	\begin{enumerate}
		1.
		Let $S$ and $S_{1}$ be antipodes of Hopf superalgebra $H$. Then
		\[ S_{1}(h) = \sum_{(h)} S_{1}(h^{'}) \epsilon(h^{''}) = \sum_{(h),(h^{''})} S_{1}(h^{'}) (h^{''})^{'}
S((h^{''})^{''}) = \lambda_{H,H,H} ( S_{1} \otimes id_{H} \otimes S ) ( \sum_{(h),(h^{''})} h^{'} \otimes (h^{''})^{'} \otimes
(h^{''})^{''} ) =  \]
		\[ = \lambda_{H,H,H} ( S_{1} \otimes id_{H} \otimes S ) ( \sum_{(h),(h^{'})} (h^{'})^{'} \otimes (h^{'})^{''} \otimes
h^{''} ) = \sum_{(h),(h^{'})} S_{1} ((h^{'})^{'}) (h^{'})^{''} S(h^{''}) =  \sum_{(h)} \epsilon(h^{'}) S(h^{''}) = S(h) \]
		for all $h \in H$.
		
		2.
		Consider a superalgebra $\mathrm{Hom}(H \otimes H, H) = ( \mathrm{Hom}(H \otimes H, H), *, \eta_{H} \circ \epsilon_{H
\otimes H} )$, where
		\[ f * g = \mu_{H} \circ ( f \otimes g ) \circ \Delta_{H \otimes H} \]
		for all $f,g \in \mathrm{Hom}(H \otimes H, H)$.
		
		We prove that a superspace morphism $S \circ \mu_{H}$ is a left inverse to $\mu_{H}$ in superalgebra $\mathrm{Hom}(H
\otimes H, H)$:
		\[ ((S \circ \mu_{H}) * \mu_{H}) (a \otimes b) =  \mu_{H} \circ ((S \circ \mu_{H}) \otimes \mu_{H}) \circ \Delta_{H
\otimes H} (a \otimes b) = \sum_{(a),(b)} (-1)^{|a^{''}||b^{'}|} S(a^{'}b^{'}) a^{''}b^{''} =  \]
		\[ = \mu_{H} \circ (S \otimes id_{H}) ( \mu_{H \otimes H} (\Delta_{H} (a) \otimes \Delta_{H} (b)) ) = \mu_{H} \circ (S
\otimes id_{H}) ( \Delta_{H} (ab) ) = \]
		\[ = \sum_{(ab)} S((ab)^{'}) (ab)^{''} = \epsilon_{H} (ab) 1_{H} = \eta_{H} \circ \epsilon_{H \otimes H} ( a \otimes b
) \]
		for all $a,b \in H$.
		
		We prove that a superspace morphism $\mu_{H} \circ \tau_{H,H} \circ ( S \otimes S )$ is a right inverse to $\mu_{H}$
in superalgebra $\mathrm{Hom}(H \otimes H, H)$:
		\[ (\mu_{H} * (\mu_{H} \circ \tau_{H,H} \circ ( S \otimes S ))) ( a \otimes b ) =  \mu_{H} \circ (\mu_{H} \otimes
(\mu_{H} \circ \tau_{H,H} \circ ( S \otimes S ))) \circ \Delta_{H \otimes H} (a \otimes b) = \]
		\[ = \sum_{(a),(b)} (-1)^{|a^{''}||b^{'}| + |a^{''}| |b^{''}|} a^{'} b^{'} S(b^{''}) S(a^{''}) = \sum_{(a),(b)}
(-1)^{|a^{''}||b|} a^{'} b^{'} S(b^{''}) S(a^{''}) = \]
		\[ = \sum_{(a),(b)} (-1)^{|a^{''}||b|} a^{'} \epsilon_{H}(b) S(a^{''}) = \epsilon_{H}(b) \sum_{(a)} a^{'} S(a^{''}) =
\epsilon_{H}(a) \epsilon_{H}(b) 1_{H} = \epsilon_{H} (ab) 1_{H} = \eta_{H} \circ \epsilon_{H \otimes H} ( a \otimes b ) \]
		for all $a,b \in H$.
		
		The result follows from the fact that if an element in associative monoid has both a left and a right inverse, then
the left and right inverse are equal, and from the fact that the inverse is unique.
		
		3.
		\[ S * id_H (1_H) = \mu \circ ( S \otimes id_{H} ) \circ  \Delta(1_{H}) = \mu ( S(1_{H}) \otimes 1_{H}) = S(1_{H}) 1_H
= S \circ \eta(1_k) = \epsilon(1_{H}) 1_H = \eta(1_{k}). \]
		
		4.
		\[ \epsilon( S(h) ) = \epsilon( \sum_{(h)} S(h^{'}) \epsilon(h^{''}) ) = \sum_{(h)} \epsilon( S(h^{'}) )
\epsilon(h^{''}) = \]
		\[ = \epsilon( \sum_{(h)} S(h^{'}) h^{''} ) = \epsilon( \epsilon(h) ) = \epsilon(h) \]
		for all $h \in H$.
		
		5.
		Consider a superalgebra $\mathrm{Hom}(H,H \otimes H) = ( \mathrm{Hom}(H,H \otimes H), *, \eta_{H \otimes H} \circ
\epsilon_{H} )$, where
		\[ f * g = \mu_{H \otimes H} \circ ( f \otimes g ) \circ \Delta_{H} \]
		for all $f,g \in \mathrm{Hom}(H,H \otimes H)$.
		
		We prove that a superspace morphism $\Delta_{H} \circ S$ is a left inverse to $\Delta_{H}$ in the superalgebra
$\mathrm{Hom}(H,H \otimes H)$:
		\[ ((\Delta_{H} \circ S) * \Delta_{H}) (a) =  \mu_{H \otimes H} \circ ( (\Delta_{H} \circ S) \otimes \Delta_{H} )
\circ \Delta_{H} (a) = \]
		\[ = \sum_{(a)} \mu_{H \otimes H} ( \Delta_{H} (S(a^{'})) \otimes \Delta_{H} (a^{''} ) ) =  \sum_{(a)} \Delta_{H} (
S(a^{'}) a^{''} ) = \Delta_{H} ( \sum_{(a)} S(a^{'}) a^{''} ) = \]
		\[ = \Delta_{H}( \epsilon_{H}(a) ) = \epsilon_{H}(a) 1_{H} \otimes 1_{H} = \eta_{H \otimes H} \circ \epsilon_{H} (a)
\]
		for all $a \in H$.
		
		We prove that a superspace morphism $\tau_{H,H} \circ (S \otimes S) \circ \Delta_{H}$ is a right inverse to
$\Delta_{H}$ in the superalgebra $\mathrm{Hom}(H,H \otimes H)$:
		\[ ( \Delta_{H} * ( \tau_{H,H} \circ (S \otimes S) \circ \Delta_{H} ) ) (a) =  \mu_{H \otimes H} \circ ( \Delta_{H}
\otimes ( \tau_{H,H} \circ (S \otimes S) \circ \Delta_{H} ) ) \circ \Delta_{H} (a) = \]
		\[ = \sum_{(a),(a^{'}),(a^{''})} (-1)^{ |(a^{''})^{'}| |(a^{''})^{''}| + |(a^{'})^{''}| |(a^{''})^{''}| } (a^{'})^{'}
S((a^{''})^{''}) \otimes (a^{'})^{''} S((a^{''})^{'}) = \]
		\[ = \mu_{H \otimes H} \circ ( id \otimes id \otimes S \otimes S ) \circ ( id \otimes id \otimes \tau_{H,H} ) (
\sum_{(a),(a^{'}),(a^{''})} (a^{'})^{'} \otimes (a^{'})^{''} \otimes (a^{''})^{'} \otimes (a^{''})^{''} ) = \]
		\[ = \mu_{H \otimes H} \circ ( id \otimes id \otimes S \otimes S ) \circ ( id \otimes id \otimes \tau_{H,H} ) (
\sum_{(a),(a^{''}),((a^{''})^{'})} a^{'} \otimes ((a^{''})^{'})^{'} \otimes ((a^{''})^{'})^{''} \otimes (a^{''})^{''} ) = \]
		\[ = \sum_{(a),(a^{''}),((a^{''})^{'})} (-1)^{ |((a^{''})^{'})^{''}| |(a^{''})^{''}| + |((a^{''})^{'})^{'}|
|(a^{''})^{''}| } a^{'} S( (a^{''})^{''} ) \otimes ((a^{''})^{'})^{'} S(((a^{''})^{'})^{''}  )  = \]
		\[ = \sum_{(a),(a^{''}),((a^{''})^{'})} (-1)^{ |(a^{''})^{'}| |(a^{''})^{''}| } a^{'} S( (a^{''})^{''} ) \otimes
((a^{''})^{'})^{'} S(((a^{''})^{'})^{''}  ) = \]
		\[ = (id \otimes (\mu_{H} \circ (id \otimes S) \circ \Delta_{H} ) ) (\sum_{(a),(a^{''}),((a^{''})^{'})} (-1)^{
|(a^{''})^{'}| |(a^{''})^{''}| } a^{'} S( (a^{''})^{''} ) \otimes (a^{''})^{'} ) = \]
		\[ = (id \otimes ( \eta_{H} \circ \epsilon_{H} ) ) (\sum_{(a),(a^{''})} (-1)^{ |(a^{''})^{'}| |(a^{''})^{''}| } a^{'}
S( (a^{''})^{''} ) \otimes (a^{''})^{'} ) =  \sum_{(a),(a^{''})} (-1)^{ |(a^{''})^{'}| |(a^{''})^{''}| } a^{'} S( \epsilon(
(a^{''})^{'} ) (a^{''})^{''} ) \otimes 1_{H} = \]
		\[ = ( (\mu_{H} \circ (id \otimes S) \circ (id \otimes ( \nu_{k,H} \circ ( \epsilon_{H} \otimes id ) \circ
\Delta_{H})) \otimes id ) (\sum_{(a)} a^{'} \otimes a^{''} \otimes 1_{H}) = \]
		\[ = ( (\mu_{H} \circ (id \otimes S) \circ (id \otimes id) ) \otimes id ) (\sum_{(a)} a^{'} \otimes a^{''} \otimes
1_{H}) =  \sum_{(a)} a^{'} S(a^{''}) \otimes 1_{H} = \epsilon_{H}(a) 1_{H} \otimes 1_{H} = \eta_{H \otimes H} \circ
\epsilon_{H} (a) \]
		for all $a \in H$.
		
		The result follows from the fact that if an element in associative monoid has both a left and a right inverse, then
the left and right inverse are equal, and from the fact that the inverse is unique.
		
		6.
		We prove by induction on a dimension $dim(H)$ of a vector superspace $H$. Let $dim(H)=1$, that is $H=k1_{H}$. We
define a superspace morphism $S:H\to H,S(1_{H})=1_{H}$. In this case $S^{-1}:H\to H, S^{-1}(1_{H})=1_{H}$.
		
		Suppose that $dim(H)>1$ and antipodes of finite-dimensional Hopf superalgebras of lower dimension are superspace
automorphisms.		
		We prove that $ker(S)$ is an ideal in Hopf superalgebra $H$
		
		The kernel of $S:H\to H$ is a $\mathbb{Z}_{2}$-graded subsuperspace. Indeed, for each $a \in H$ it is possible to
write a decomposition $a=a_{0}+a_{1}$, where $a_{i}\in H_{i}$, $i\in \mathbb{Z}_{2}$. Suppose that $a\in ker(S)$. Then
$S(a)=S(a_{0})+S(a_{1})=0_H$, $S(a_{i})\in H_{i}$, $i\in \mathbb{Z}_{2}$. Consequently $S(a_{0_H})=S(a_{1})=0_H$.
		
		We have for all $a,b\in ker(S),\lambda_{1},\lambda_{2}\in k$
		\[ S(\lambda_{1}a+\lambda_{2}b)=\lambda_{1}S(a)+\lambda_{2}S(b)=0_{H} \in ker(S), \]
		\[ S(ab)=(-1)^{|a||b|}S(b)S(a)=0_{H} \in ker(S), \]
		\[ \tau_{H,H}\circ(S\otimes S)\circ\Delta(a)=\sum_{(a)}(-1)^{|a^{'}||a^{''}|}S(a^{''})\otimes
S(a^{'})=\Delta(S(a))=\Delta(0_{H})=0_{H}\otimes0_{H}\Rightarrow \]
		\[ \Rightarrow\Delta(ker(S))\subseteqq ker(S)\otimes H+H\otimes ker(S), \]
		\[ \epsilon(a)=\epsilon(S(a))=\epsilon(0_{H})=0_{k}, \]
		\[ S(ker(S))=0_{H}\subset ker(S). \]
		
		We prove that $im(S)$ is a Hopf subsuperalgebra in $H$. Indeed,
		\[ S(H_{i})\subseteqq H_{i}\Rightarrow im(S)_{i}=S(H_{i})\Rightarrow im(S)_{i}=H_{i}\cap im(S) \]
		for all $i\in \mathbb{Z}_{2}$.
		
		We have for all $a,b\in im(S)\Rightarrow\exists a_{1},b_{1}\in H:a=S(a_{1}),b=S(a_{1}),\lambda_{1},\lambda_{2}\in k$
		\[ \lambda_{1}a+\lambda_{2}b=\lambda_{1}S(a_{1})+\lambda_{2}S(b_{1})=S(\lambda_{1}a_{1}+\lambda_{2}b_{1})\in im(S),
\]
		\[ \mu_{H}(a\otimes b)=ab=S(a_{1})S(b_{1})=(-1)^{|a_{1}||b_{1}|}S(b_{1}a_{1})=S((-1)^{|a_{1}||b_{1}|}b_{1}a_{1})\in
im(S), \]
		\[ \Delta(a)=\Delta(S(a_{1}))=\sum_{(a_{1})}(-1)^{|(a_{1})^{'}||(a_{1})^{''}|}S((a_{1})^{''})\otimes
S((a_{1})^{'})\Rightarrow\Delta(im(S)) \subseteqq im(S)\otimes im(S), \]
		\[ S(1_{H})=1_{H}\in im(S). \]
		Let $c\in S(im(S))$. Then $\exists c_{1}\in H$:
		\[ c=S^{2}(c_{1})=S(S(c_{1}))=S(c_{2})\in im(S) \]
		where $c_{2}=S(c_{1})$. Consequently $S(im(S)) \subseteqq im(S)$.		
		
		Let $im(S)=H$, then $ker(S)=\{ 0 \}$. The result follows.
		
		Let $im(S)\ne H$. Then according to inductive assumption $S|_{im(S)}$ is a superspace automorphism of $im(S)$.		
		Since $H=im(S)\oplus ker(S)$, we can consider a superspace morphism $\pi:H\to im(S)$, $\pi(a)=[a\in im(S)]a$ for all
$a\in H$. Since $ker(S)=ker(\pi)$, we have $\Delta(ker(S))\subset ker(S)\otimes H+H\otimes ker(S)=ker(\pi)\otimes H+H\otimes
ker(\pi)$. Condiser a superalgebra $\mathrm{Hom}(H,H)$. Then
		\[ (\pi*S)(a)=\mu\circ(\pi\otimes
S)\circ\Delta(a)=\sum_{(a)}\pi(a^{'})S(a^{''})=0=\epsilon(a)1_{H}=(\eta\circ\epsilon)(a) \]
		for all $a\in ker(S)$.
		\[ (\pi*S)(a)=\mu\circ(\pi\otimes
S)\circ\Delta(a)=\sum_{(a)}\pi(a^{'})S(a^{''})=\sum_{(a)}a^{'}S(a^{''})=\epsilon(a)1_{H}=(\eta\circ\epsilon)(a)
		\]
		for all $a\in im(S)$.
		
		Next we have
		
\[(S*\pi)(a)=\mu\circ(S\otimes\pi)\circ\Delta(a)=\sum_{(a)}S(a^{'})\pi(a^{''})=0=\epsilon(a)1_{H}=(\eta\circ\epsilon)(a) \]
		for all $a\in ker(S)$. Moreover,
		
\[(S*\pi)(a)=\mu\circ(S\otimes\pi)\circ\Delta(a)=\sum_{(a)}S(a^{'})\pi(a^{''})=\sum_{(a)}S(a^{'})a^{''}=\epsilon(a)1_{H}=(\eta\circ\epsilon)(a)
\]
		for all $a\in im(S)$.
		
		We see that $\pi$ is the inverse of $S$ in the superalgebra $\mathrm{Hom}(H,H)$. Since the inverse is unique
$\pi=id_{H}$. Thus we get a contradiction as $ker(S_H)=ker(id_H)={0_H}$. Consequently $im(S)=H$.
	\end{enumerate}
\end{proof}

\sH*
\begin{proof}
	Since we consider $\mathbb{Z}_2$-graded linear mappings it is sufficient to prove that if the statement is true for $x\in X$
and $y \in X$ it will be true for $xy$.
	\[ \mu \circ ( S \otimes id ) \circ \Delta (xy) = \mu \circ ( S \otimes id ) ( \Delta(x) \Delta(y) ) =  \mu \circ ( S
\otimes id ) ( \sum_{(x),(y)} (-1)^{|x^{''}||y^{'}|} x^{'} y^{'} \otimes x^{''} y^{''}  ) = \]
	\[ = \sum_{(x),(y)} (-1)^{|x^{''}||y^{'}| + |x^{'}| |y^{'}|} S(y^{'}) S(x^{'}) x^{''} y^{''} =  \sum_{(y)} (-1)^{|x||y^{'}|}
S(y^{'}) ( \sum_{(x)} S(x^{'}) x^{''} ) y^{''} = \]
	\[ = \sum_{(y)} (-1)^{|x||y^{'}|} \epsilon(x) S(y^{'}) y^{''} =  \epsilon(x) \sum_{(y)} S(y^{'}) y^{''} = \epsilon(x)
\epsilon(y) 1_{H} = \eta \circ \epsilon (xy), \]
	
	\[ \mu \circ ( id \otimes S ) \circ \Delta (xy) = \mu \circ ( id \otimes S ) ( \Delta(x) \Delta(y) ) =  \mu \circ ( id
\otimes S ) ( \sum_{(x),(y)} (-1)^{|x^{''}||y^{'}|} x^{'} y^{'} \otimes x^{''} y^{''}  ) = \]
	\[ = \sum_{(x)(y)} (-1)^{|x''|(|y'|+|y''|)} x'y' S(y'')S(x'') =  \sum_{(x)} (-1)^{|x''||y|} x' (\sum_{(y)} y' S(y'')) S(x'')
= \]
	\[ = \sum_{(x)} (-1)^{|x''||y|} \epsilon (y) x' S(x'') =  \epsilon (y) \sum_{(x)} x' S(x'') = \epsilon (y) \epsilon (x)
1_{H} = \eta \circ \epsilon (xy). \]
\end{proof}

\HopfEqual*
\begin{proof}
	Consider
	\[ (H^{op})^{*} = (H^{*},\Delta_{H}^{*} \circ \lambda_{H,H},\epsilon_{H}^{*} \circ \chi,\lambda_{H,H}^{-1} \circ (\mu_{H}
\circ \tau_{H,H})^{*},\chi^{-1} \circ \eta^{*},(S^{-1})^{*}), \]
	\[ (H^{*})^{cop} = (H^{*},\Delta_{H}^{*} \circ \lambda_{H,H},\epsilon_{H}^{*} \circ \chi,\tau_{H^{*},H^{*}} \circ
\lambda_{H,H}^{-1} \circ \mu_{H}^{*} = \tau_{H^{*},H^{*}} \circ \Delta_{H^{*}},\chi^{-1} \circ \eta^{*},(S^{*})^{-1}). \]
	We have for all $f \in H^{*}, a,b \in H$
	\[ \Delta_{H^{*}} (f) = \sum_{(f)} f^{'} \otimes f^{''}, \]
	\[ \Delta_{(H^{op})^{*}} (f) = \lambda_{H,H}^{-1} \circ (\mu_{H} \circ \tau_{H,H})^{*} \circ f \Leftrightarrow \]
	\[ \Leftrightarrow \lambda_{H,H} (\Delta_{(H^{op})^{*}} (f)) ( a \otimes b ) = (-1)^{|a||b|} f(ba) = \sum_{(f)} (-1)^{|a||b|
+ |f^{''}||b|} f^{'} (b)  f^{''} (a) =  \sum_{(f),f^{'} (b) \neq 0,f^{''} (a) \neq 0} f^{'} (b) f^{''} (a). \]
	
	\[ \Delta_{(H^{*})^{cop}} (f) = (\tau_{H^{*},H^{*}} \circ \lambda_{H,H}^{-1} \circ \mu_{H}^{*}) (f) = \tau_{H^{*},H^{*}}
\circ (\sum_{(f)} f^{'} \otimes f^{''}) \Leftrightarrow \]
	\[ \Leftrightarrow \lambda_{H,H} (\Delta_{(H^{*})^{cop}} (f)) ( a \otimes b )  = \sum_{(f)} (-1)^{|f^{'}||f^{''}| + |f^{'}|
|a|} f^{''} (a)  f^{'} (b) = \sum_{(f), f^{'} (b) \neq 0,f^{''} (a) \neq 0}  f^{'} (b) f^{''} (a). \]	
	Since $\lambda_{H,H}$ is the superspace isomorphism, we have $\Delta_{(H^{op})^{*}} = \Delta_{(H^{*})^{cop}}$.
	
	We prove that for a superspace morphism $S^{*}: H^{*} \to H^{*}$: $(S^{*})^{-1} = (S^{-1})^{*}$.
	Indeed,
	\[ (S^{*} \circ (S^{-1})^{*}) (f) = f \circ S^{-1} \circ S = f, \]
	\[ ((S^{-1})^{*} \circ S^{*}) (f) = f \circ S \circ S^{-1} = f \]
	for all $f \in H^{*}$.
	
	Consider
	\[ (H^{cop})^{*}=(H^{*},(\tau_{H,H} \circ \Delta_{H})^{*} \circ \lambda_{H,H},\epsilon_{H}^{*} \circ \chi,\lambda_{H,H}^{-1}
\circ \mu_{H}^{*},\chi^{-1} \circ \eta^{*},(S^{-1})^{*}), \]
	\[ (H^{*})^{op} = (H^{*},\Delta_{H}^{*} \circ \lambda_{H,H} \circ \tau_{H^{*},H^{*}} = \mu_{H^{*}} \circ
\tau_{H^{*},H^{*}},\epsilon_{H}^{*} \circ \chi, \lambda_{H,H}^{-1} \circ \mu_{H}^{*},\chi^{-1} \circ \eta^{*},(S^{*})^{-1}). \]
	We have for all $f,g \in H^{*}, a \in H$
	\[ \mu_{(H^{cop})^{*}} ( f \otimes g ) ( a ) = \sum_{(a)} (-1)^{|a^{'}||a^{''}| + |g||a^{''}|} f(a^{''})g(a^{'}) = \]
	\[ = \sum_{(a),f(a^{''}) \neq 0, g(a^{'}) \neq 0} (-1)^{|a^{'}||a^{''}| + |a^{'}||a^{''}|} f(a^{''})g(a^{'}) =
\sum_{(a),f(a^{''}) \neq 0, g(a^{'}) \neq 0} f(a^{''})g(a^{'}). \]
	\[ \mu_{(H^{*})^{op}} ( f \otimes g ) (a) = \sum_{(a)} (-1)^{|f||g| + |f||a^{'}|} g(a^{'}) f(a^{''}) = \]
	\[ = \sum_{(a),f(a^{''}) \neq 0, g(a^{'}) \neq 0} (-1)^{|f||g| + |f||g|} g(a^{'}) f(a^{''}) = \sum_{(a),f(a^{''}) \neq 0,
g(a^{'}) \neq 0} g(a^{'}) f(a^{''}). \]
	
	We have already proved that $(S^{*})^{-1} = (S^{-1})^{*}$.
\end{proof}

\Dualcopop*
\begin{proof}
	Fix a vector $x$ in a finite-dimensional vector superspace $V$. Define a mapping $E_x : V^{*} \to k$, $E_{x}(f) =
(-1)^{|f||x|} f(x)$ for all $f \in V^{*}$. We verify that $E_{x} \in V^{*}$. Indeed,
	\[ E_{x}( f + g ) = ((-1)^{|x||f|}f + (-1)^{|x||g|} g)(x) = (-1)^{|x||f|}f(x) + (-1)^{|x||g|} g(x) = E_x(f) + E_x(g) \]
	for all $f,g \in V^{*}$.
	\[ E_x (af) = (-1)^{|x|(|a| + |f|)} (af)(x) = (-1)^{|x||f|} (af)(x) = a ((-1)^{|x||f|} f(x)) = a (E_x (f))  \]
	for all $a \in k, \; f \in V^{*}$.
	Thus $E_{x} \in V^{*}$. Also $E_x$ is a $\mathbb{Z}_2$-graded mapping and, moreover, $|E_x| = |x|$.
	
	Define a mapping $ev_{V}:V \to V^{**}$, $ev_{V} (x) = E_{x}$ for all $x \in V$. We verify that $ev_{V}$ is a linear map.
Indeed, for all $x,y \in V, \; f \in V^{*}$
	\[ ev_{V}(x + y)(f) = E_{x+y}(f) = f( (-1)^{|f||x|} x + (-1)^{|f||y|} y ) = (-1)^{|f||x|} f(x) + (-1)^{|f||y|} f(y) = \]
	\[ = E_x(f) + E_y(f) = ev_V(x)(f) + ev_V(y)(f). \]
	Furthermore, we have for all $a \in k, \; x \in V, \; f \in V^{*}$
	\[ ev_V(ax)(f) = E_{ax}(f) = (-1)^{|f|(|a| + |x|)} f(ax) = \]
	\[ = (-1)^{|f||x|} f(ax) = a ((-1)^{|f||x|} f(x)) = a (E_x (f)) = a (ev_V(x)(f)). \]
	Notice that $ev_V$ is an even mapping.
	
	For a subset $S \subseteq V^{*}$ define $Z(S)=\{ x \in V | f(x) = 0_{k} \; \text{for all} \; f \in S \}$. We prove that
$Z(V^{*})=\{0_{V}\}$. Notice that $0_V \in Z(V^{*})$, as all linear mappings map $0_V$ in $0_k$. Fix a vector $v \ne 0_V$ in
$V$, and verify that $v \notin Z(V^{*})$. We find a linear mapping $f \in V^{*}$, such that $f(v) \ne 0_k$. Thus extend a set
$\{v\}$ to a basis $B$ of the vector superspace $V$. The first function in a dual basis to $B$ is the element in $V^{*}$ that
maps $v$ in $1_k \ne 0_k$.
	
	Since $ev_V$ is a superspace morphism we prove its injectivity by showing that a condition $ev_V(x) = 0_{V^{**}}$ implies
that $x = 0_V$. We give a proof by assuming the opposite. Fix a non-zero vector $x \in V$. Since $Z(V^{*})=\{0_{V}\}$ it follows
that there exists $f \in V^{*}$ such that $f(x) \ne 0_{k}$. Then $ev_V(x)(f) \ne 0_k$, that is $ev_{V}(x)$ is a non-zero mapping
from $V^{*}$ to $k$. Thus $ev_{V} \ne 0_{V^{**}}$. Since $V$ is a finite-dimensional vector superspace it follows that
$dim(V^{**})=dim(V)$. Consequently the injective supaerspace morphism $ev_V$ is surjective since the domain and codomain have
the same finite dimension.
	
	Thus for a finite-dimensional vector superspace $V$ exists a natural superspace isomorphism $ev_{V}:V \to V^{**}$, $ev_{V}
(x) = E_{x}, E_{x}(f) = (-1)^{|f||x|} f(x)$ for all $x \in V, f \in V^{*}$. Consequently we need only to prove that the mapping
$ev_{H}$ is a Hopf superalgebra morphism.
	
	Notice that
	\[ H^{cop} = (H, \mu, \eta, \tau_{H,H} \circ \Delta, \epsilon, S^{-1}), \]
	\[ (H^{cop})^{*}=(H^{*},(\tau_{H,H} \circ \Delta_{H})^{*} \circ \lambda_{H,H},\epsilon_{H}^{*} \circ \chi,\lambda_{H,H}^{-1}
\circ \mu_{H}^{*},\chi^{-1} \circ \eta^{*},(S^{-1})^{*}), \]
	\[ ((H^{cop})^{*})^{op}=(H^{*},(\tau_{H,H} \circ \Delta_{H})^{*} \circ \lambda_{H,H} \circ
\tau_{H^{*},H^{*}},\epsilon_{H}^{*} \circ \chi,\lambda_{H,H}^{-1} \circ \mu_{H}^{*},\chi^{-1} \circ
\eta^{*},((S^{-1})^{*})^{-1}), \]
	\[ (((H^{cop})^{*})^{op})^{*} = ( H^{**}, (\lambda_{H,H}^{-1} \circ \mu_{H}^{*})^{*} \circ \lambda_{H^{*},H^{*}}, (\chi^{-1}
\circ \eta^{*})^{*} \circ \chi, \lambda_{H^{*},H^{*}}^{-1} \circ ((\tau_{H,H} \circ \Delta_{H})^{*} \circ \lambda_{H,H} \circ
\tau_{H^{*},H^{*}})^{*}, \]
	\[ \chi^{-1} \circ (\epsilon_{H}^{*} \circ \chi)^{*},(((S^{-1})^{*})^{-1})^{*} ). \]
	
	We have for all $f \in H^{*}, x,y \in H$
	\[  (\lambda_{H,H}^{-1} \circ \mu_{H}^{*})^{*} \circ \lambda_{H^{*},H^{*}} : H^{**} \otimes H^{**} \to H^{**}, \]
	\[ ((\mu_{(((H^{cop})^{*})^{op})^{*}} \circ ( ev_{H} \otimes ev_{H} )) ( x \otimes y )) (f) = \]
	\[ (\lambda_{H,H}^{-1} \circ \mu_{H}^{*})^{*} ( \lambda_{H^{*},H^{*}} ( E_{x} \otimes E_{y} ) ) ( f ) =
(\lambda_{H^{*},H^{*}} ( E_{x} \otimes E_{y} )) ( \lambda_{H,H}^{-1} \circ \mu_{H}^{*} ) ( f ) = \]
	\[ = \sum_{(f)} (-1)^{|E_{y}||f^{'}|} E_{x} (f^{'}) E_{y} ( f^{''} ) = \sum_{(f)} (-1)^{|y||f^{'}| + |f^{'}||x| +
|f^{''}||y|} f^{'} (x) f^{''} ( y ) =  \]
	\[ = \sum_{(f)} (-1)^{|x||f^{''}| + |f|(|x| + |y|)} f^{'} (x) f^{''} ( y ) = (-1)^{|f|(|x| + |y|)} f(xy) = E_{xy} (f) =
((ev_{H} \circ \mu_{H}) ( x \otimes y )) (f) . \]
	
	We have for all $f \in H^{*}$
	\[ (\chi^{-1} \circ \eta^{*})^{*} \circ \chi : k \to H^{**}, \]
	\[ (\eta_{(((H^{cop})^{*})^{op})^{*}} ( 1_{k} )) (f) = ( (\chi^{-1} \circ \eta^{*})^{*}  (\chi (1_{k}) ) ) (f) = \]
	\[ = \chi (1_{k}) ((\chi^{-1} \circ \eta^{*}) (f)) = \chi (1_{k}) (f(1_{H})) = f(1_{H}) = E_{1_{H}} (f) = ev_{H} ( 1_{H} )
(f) = (ev_{H} \circ \eta_{H} (1_k)) (f). \]
	
	We have for all $f,g \in H^{*}, x \in H$
	\[ \lambda_{H^{*},H^{*}}^{-1} \circ ((\tau_{H,H} \circ \Delta_{H})^{*} \circ \lambda_{H,H} \circ \tau_{H^{*},H^{*}})^{*} :
H^{**} \to H^{**} \otimes H^{**}, \]
	
	\[ ( ( \lambda_{H^{*},H^{*}} \circ \Delta_{(((H^{cop})^{*})^{op})^{*}} \circ ev_{H}) (x) ) ( f \otimes g ) = \]
	\[ = ( \lambda_{H^{*},H^{*}} \circ \lambda_{H^{*},H^{*}}^{-1} \circ  ( (\tau_{H,H} \circ \Delta_{H})^{*} \circ \lambda_{H,H}
\circ \tau_{H^{*},H^{*}})^{*}  ( E_{x} )) ( f \otimes g ) = \]
	\[ = E_{x} ((\tau_{H,H} \circ \Delta_{H})^{*} \circ \lambda_{H,H} \circ \tau_{H^{*},H^{*}}  ( f \otimes g ))  =
(-1)^{|x|(|f| + |g|)} ((\tau_{H,H} \circ \Delta_{H})^{*} \circ \lambda_{H,H} \circ \tau_{H^{*},H^{*}} ( f \otimes g )) (x)  =
\]
	\[ = (-1)^{|x|(|f| + |g|) + |f||g|} \lambda_{H,H} ( g \otimes f ) ( \sum_{(x)} (-1)^{|x^{'}||x^{''}|} x^{''} \otimes x^{'} )
=   \sum_{(x)} (-1)^{|x|(|f| + |g|) + |f||g| + |x^{'}| |x^{''}| + |f||x^{''}|} g(x^{''}) f(x^{'}) = \]
	\[ = \sum_{(x)} (-1)^{|f||x^{'}| + |g||x^{''}| + |f||x^{''}|} g(x^{''}) f(x^{'}) =  \sum_{(x)} (-1)^{|f||E_{x^{''}}|}
E_{x^{'}} (f) E_{x^{''}} (g) = \lambda_{H^{*},H^{*}} (\sum_{(x)} E_{x^{'}} \otimes E_{x^{''}}) ( f \otimes g ) = \]
	\[ = (\lambda_{H^{*},H^{*}} \circ ((ev_{H} \otimes ev_{H}) \circ \Delta_{H}) (x)) ( f \otimes g ). \]
	
	Thus
	\[ \Delta_{(((H^{cop})^{*})^{op})^{*}} \circ ev_{H} = (ev_{H} \otimes ev_{H}) \circ \Delta_{H}. \]
	
	We have for all $ x \in H $
	\[ \chi^{-1} \circ (\epsilon_{H}^{*} \circ \chi)^{*} : H^{**} \to k, \]
	\[ (\epsilon_{(((H^{cop})^{*})^{op})^{*}} \circ ev_{H}) (x) = (\chi^{-1} \circ (\epsilon_{H}^{*} \circ \chi)^{*}) ( E_{x} )
= \]
	\[ = ((\epsilon_{H}^{*} \circ \chi)^{*} ( E_{x} )) (1_{k}) = E_{x} ( (\epsilon_{H}^{*} \circ \chi ) (1_{k}) ) = E_{x} (
\epsilon_{H} ) = \epsilon_{H}(x). \]
	
	We have for all $ x \in H, \; f \in H^{*} $
	\[ (((S^{-1})^{*})^{-1})^{*} : H^{**} \to H^{**}, \]
	\[ (S_{(((H^{cop})^{*})^{op})^{*}} \circ ev_{H} (x)) (f) = ((((S^{-1})^{*})^{-1})^{*} \circ ev_{H} (x)) (f) =
((((S^{-1})^{*})^{-1})^{*} (E_{x})) (f) = \]
	\[ = (E_{x} \circ ((S^{-1})^{*})^{-1}) (f) = (E_{x} \circ S^{*}) (f) = E_{x} ( f \circ S ) = (-1)^{|x||f|} f( S(x) ) =
E_{S(x)} (f) = (ev_{H} \circ S (x)) (f). \]
	
	Remark.
	We have for all $ f \in H^{*} $
	\[ ((S^{-1})^{*})^{-1} : H^{*} \to H^{*}, \]
	\[ ((S^{-1})^{*})^{-1} = S^{*}. \]
	Indeed,
	\[ (S^{-1})^{*} \circ S^{*} (f) = f \circ S \circ S^{-1} = f, \]
	\[ S^{*} \circ (S^{-1})^{*} (f) = f \circ S^{-1} \circ S = f. \]
\end{proof}

\rmatrixgeneration*
\begin{proof}
	1. Notice that $id_{H} \otimes (\lambda c) = \lambda (id_{H} \otimes c)$, $(\lambda c) \otimes id_{H} = \lambda (c \otimes
id_{H})$.
	\[ ((\lambda c) \otimes id_{V}) (id_{V} \otimes (\lambda c)) ((\lambda c) \otimes id_{V}) = \lambda^{3} ((c \otimes id_{V})
(id_{V} \otimes c) (c \otimes id_{V})) = \]
	\[ = \lambda^{3} ((id_{V} \otimes c) (c \otimes id_{V}) (id_{V} \otimes c)) = (id_{V} \otimes (\lambda c)) ((\lambda c)
\otimes id_{V}) (id_{V} \otimes (\lambda c)). \]
	
	2. Since $c^{-1} \otimes id_{V} = (c \otimes id_{V})^{-1}$, $id_{V} \otimes c^{-1} = (id_{V} \otimes c)^{-1}$, we have
	\[ id_{V} \otimes id_{V} \otimes id_{V} = (c \otimes id_{V})^{-1} (id_{V} \otimes c)^{-1} (c \otimes id_{V})^{-1} (c \otimes
id_{V}) (id_{V} \otimes c) (c \otimes id_{V}) = \]
	\[ = ( c^{-1} \otimes id_{V} ) (id_{V} \otimes c^{-1}) ( c^{-1} \otimes id_{V} ) (id_{V} \otimes c) (c \otimes id_{V})
(id_{V} \otimes c) \iff \]
	\[ \iff (id_{V} \otimes c^{-1})(c^{-1} \otimes id_{V})(id_{V} \otimes c^{-1}) =  \]
	\[ = ( c^{-1} \otimes id_{V} ) (id_{V} \otimes c^{-1}) ( c^{-1} \otimes id_{V} ) (id_{V} \otimes c) (c \otimes id_{V})
(id_{V} \otimes c) (id_{V} \otimes c^{-1})(c^{-1} \otimes id_{V})(id_{V} \otimes c^{-1}) = \]
	\[ = ( c^{-1} \otimes id_{V} ) (id_{V} \otimes c^{-1}) ( c^{-1} \otimes id_{V} ) (id_{V} \otimes c) (c \otimes id_{V})
(id_{V} \otimes c) (id_{V} \otimes c)^{-1} (c \otimes id_{V})^{-1} (id_{V} \otimes c)^{-1} = \]
	\[ = ( c^{-1} \otimes id_{V} ) (id_{V} \otimes c^{-1}) ( c^{-1} \otimes id_{V} ). \]
	
	3. Let $\{e_{i}\}_{i \in I}$ be a basis in $V$. The automorphism $c$ is defined by a set of scalars
$\{c_{ij}^{kl}\}_{i,j,k,l \in I}$ defined from equation
	\[ c( e_i \otimes e_j ) = \sum_{k,l \in I} c_{ij}^{kl} e_k \otimes e_l \]
	for all $i,j \in I$, where $I$ is a set of indexes.
	
	Let $\sigma = (\tau_{V,V} \otimes id_{V}) (id_{V} \otimes \tau_{V,V})   (\tau_{V,V} \otimes id_{V}) $. Then $\sigma^{-1} =
(id_{V} \otimes \tau_{V,V}) (\tau_{V,V} \otimes id_{V}) (id_{V} \otimes \tau_{V,V}) $.
	We prove that
	\[ (\tau_{V,V} \circ c \circ \tau_{V,V}) \otimes id_{V} = \sigma ( id_{V} \otimes c ) \sigma^{-1} .\]
	
	We have for all vectors $v_1,v_2,v_3 \in V$
	\[ ((\tau_{V,V} \circ c \circ \tau_{V,V}) \otimes id_{V}) ( v_1 \otimes v_2 \otimes v_3 ) =  (-1)^{|v_1||v_2| } \tau_{V,V}
(c( v_2 \otimes v_1 )) \otimes v_3 =  \]
	\[ = (-1)^{|v_1||v_2| } \tau_{V,V}  ( \sum_{i,j,k,l \in I} v_{2i} v_{1j} c_{ij}^{kl} e_{k} \otimes e_{l} ) \otimes v_3 =
\sum_{i,j,k,l \in I} (-1)^{|v_1||v_2| + |e_{l}||e_{k}| } v_{1j} v_{2i} c_{ij}^{kl} e_{l} \otimes e_{k} \otimes v_3. \]
	
	\[ \sigma ( id_{V} \otimes c ) \sigma^{-1} ( v_1 \otimes v_2 \otimes v_3 ) = (-1)^{ |v_2||v_3| + |v_1||v_3| + |v_1||v_2| }
\sigma ( id_{V} \otimes c ) (v_3 \otimes v_2 \otimes v_1) = \]
	\[ = (-1)^{ |v_2||v_3| + |v_1||v_3| + |v_1||v_2| + |c||v_3| } \sigma (v_3 \otimes ( \sum_{i,j,k,l \in I} v_{2i} v_{1j}
c_{ij}^{kl} e_{k} \otimes e_{l} )) = \]
	\[ = \sum_{i,j,k,l \in I} (-1)^{ |v_2||v_3| + |v_1||v_3| + |v_1||v_2| + |c||v_3| + |e_k||e_l| + |v_3||e_l| + |v_3||e_k| }
v_{1j} v_{2i} c_{ij}^{kl} e_{l} \otimes e_{k} \otimes v_3 = \]
	\[ = \sum_{i,j,k,l \in I} (-1)^{ |v_1||v_2| + |e_k||e_l| + |v_3|(|v_2| + |v_1|) + |v_3|(|c| + |e_l| + |e_k|) } v_{1j} v_{2i}
c_{ij}^{kl} e_{l} \otimes e_{k} \otimes v_3 = \]
	\[ = \sum_{i,j,k,l \in I} (-1)^{ |v_1||v_2| + |e_k||e_l| + |v_3|(|v_2| + |v_1|) + |v_3|(|v_2| + |v_1|) } v_{1j} v_{2i}
c_{ij}^{kl} e_{l} \otimes e_{k} \otimes v_3 = \]
	\[ = \sum_{i,j,k,l \in I} (-1)^{ |v_1||v_2| + |e_k||e_l| } v_{1j} v_{2i} c_{ij}^{kl} e_{l} \otimes e_{k} \otimes v_3. \]
	
	We show that
	\[ id_{V} \otimes (\tau_{V,V} \circ c \circ \tau_{V,V}) =  \sigma ( c \otimes id_{V} ) \sigma^{-1} .\]
	
	We have for all vectors $v_1,v_2,v_3 \in V$
	\[ (id_{V} \otimes (\tau_{V,V} \circ c \circ \tau_{V,V})) ( v_1 \otimes v_2 \otimes v_3 ) = (-1)^{|c||v_1| + |v_2||v_3|} v_1
\otimes \tau_{V,V}( c( v_3 \otimes v_2 ) ) = \]
	\[ = (-1)^{|c||v_1| + |v_2||v_3|} v_1 \otimes \tau_{V,V} ( \sum_{i,j,k,l \in I} v_{2j} v_{3i} c_{ij}^{kl} e_{k} \otimes
e_{l} ) = \sum_{i,j,k,l \in I} (-1)^{|c||v_1| + |v_2||v_3| + |e_k||e_l|} v_{2j} v_{3i} c_{ij}^{kl} v_1 \otimes e_{l} \otimes
e_{k}. \]
	
	\[ \sigma ( c \otimes id_{V} ) \sigma^{-1} ( v_1 \otimes v_2 \otimes v_3 ) = (-1)^{ |v_2||v_3| + |v_1||v_3| + |v_1||v_2| }
\sigma ( c \otimes id_{V} ) (v_3 \otimes v_2 \otimes v_1) = \]
	\[ = (-1)^{ |v_2||v_3| + |v_1||v_3| + |v_1||v_2| } \sigma ( ( \sum_{i,j,k,l \in I} v_{2j} v_{3i} c_{ij}^{kl} e_{k} \otimes
e_{l} ) \otimes v_1 ) = \]
	\[ = \sum_{i,j,k,l \in I} (-1)^{ |v_2||v_3| + |v_1|(|v_3| + |v_2|) + |v_1|( |e_k| + |e_l| ) + |e_k||e_l| } v_{2j} v_{3i}
c_{ij}^{kl} v_1 \otimes e_{l} \otimes e_{k} = \]
	\[ = \sum_{i,j,k,l \in I} (-1)^{ |v_2||v_3| + |v_1|(|v_3| + |v_2|) + |v_1|( |v_3| + |v_2| + |c| ) + |e_k||e_l| } v_{2j}
v_{3i} c_{ij}^{kl} v_1 \otimes e_{l} \otimes e_{k} = \]	
	\[ = \sum_{i,j,k,l \in I} (-1)^{|c||v_1| + |v_2||v_3| + |e_k||e_l|} v_{2j} v_{3i} c_{ij}^{kl} v_1 \otimes e_{l} \otimes
e_{k}. \]	
	
	Therefore
	\[ ((\tau_{V,V} \circ c \circ \tau_{V,V}) \otimes id_{V}) (id_{V} \otimes (\tau_{V,V} \circ c \circ \tau_{V,V}))
((\tau_{V,V} \circ c \circ \tau_{V,V}) \otimes id_{V}) = \]
	\[ = \sigma ( id_{V} \otimes c ) \sigma^{-1} \sigma ( c \otimes id_{V} ) \sigma^{-1} \sigma ( id_{V} \otimes c ) \sigma^{-1}
= \sigma ( id_{V} \otimes c ) ( c \otimes id_{V} ) ( id_{V} \otimes c ) \sigma^{-1} = \]
	\[ = \sigma ( c \otimes id_{V} ) ( id_{V} \otimes c ) ( c \otimes id_{V} ) \sigma^{-1} =  \sigma ( c \otimes id_{V} )
\sigma^{-1} \sigma ( id_{V} \otimes c ) \sigma^{-1} \sigma ( c \otimes id_{V} ) \sigma^{-1} = \]
	\[ = (id_{V} \otimes (\tau_{V,V} \circ c \circ \tau_{V,V})) ((\tau_{V,V} \circ c \circ \tau_{V,V}) \otimes id_{V}) (id_{V}
\otimes (\tau_{V,V} \circ c \circ \tau_{V,V})). \]
\end{proof}

\newquasicom*
\begin{proof}
	Let $R = \sum_{i \in I} s_i \otimes t_i$, $R^{-1}=\sum_{j \in J} r_j \otimes d_j$, where $I,J$ are sets of indexes.
	We rewrite the equality \ref{eq:Runiversal} for each $x \in H$ in the following form
	\[ \tau_{H,H} \circ \Delta_{H} (x) = \mu_{H,H} \circ ( id_{H \otimes H} \otimes \mu_{H \otimes H} ) \circ ( id_{H \otimes H}
\otimes \Delta_{H} \otimes id_{H \otimes H} ) ( R \otimes x \otimes R^{-1} ). \]
	
	1.
	1.1 Consider $H^{op}$. We have for each $x \in H$
	\[ \mu_{H^{op},H^{op}} \circ ( id_{H^{op} \otimes H^{op}} \otimes \mu_{H^{op} \otimes H^{op}} ) \circ ( id_{H^{op} \otimes
H^{op}} \otimes \Delta_{H^{op}} \otimes id_{H^{op} \otimes H^{op}} ) ( R^{-1} \otimes x \otimes R ) = \ \]
	\[ = \mu_{H^{op},H^{op}} ( R^{-1} \otimes (\sum_{(x), i \in I} (-1)^{|x^{''}||s_{i}| + |s_{i}||x^{'}| + |t_{i}||x^{''}| }
s_{i} x^{'} \otimes t_{i} x^{''}) ) = \]
	\[ = \sum_{(x), i \in I, j \in J} (-1)^{|x^{''}||s_{i}| + |s_{i}||x^{'}| + |t_{i}||x^{''}| + |d_j|( |s_{i}| + |x^{'}| ) +
|r_j|( |s_{i}| + |x^{'}| ) + |d_j|( |t_{i}| + |x^{''}| ) } s_{i} x^{'} r_j \otimes t_{i} x^{''} d_j = \]
	\[ = \sum_{(x), i \in I, j \in J} (-1)^{|x|(|s_{i}| + |t_{i}| + |r_j| + |d_j|) + (|s_{i}| + |t_{i}|)( |r_j| + |d_j| ) +
|x^{''}||r_j| + |t_i|( |x^{'}| + |r_j| ) } s_{i} x^{'} r_j \otimes t_{i} x^{''} d_j = \]
	\[ = \sum_{(x),i \in I,j \in J} (-1)^{ |x^{''}||r_j| + |t_i|( |x^{'}| + |r_j| ) } s_i x^{'} r_j \otimes t_i x^{''} d_j =
\mu_{H,H} ( R \otimes \sum_{(x),j \in J} (-1)^{|x^{''}||r_j|} x^{'} r_j \otimes x^{''} d_j  ) =  \]	
	\[ = \mu_{H,H} \circ ( id_{H \otimes H} \otimes \mu_{H \otimes H} ) \circ ( id_{H \otimes H} \otimes \Delta_{H} \otimes
id_{H \otimes H} ) ( R \otimes x \otimes R^{-1} ) =  \tau_{H,H} \circ \Delta_{H} (x) = \tau_{H^{op},H^{op}} \circ
\Delta_{H^{op}} (x). \]
	
	1.2 Consider $H^{cop}$. It follows from the equality \ref{eq:Runiversal} that we have
$R^{-1}(\tau_{H,H}\circ\Delta_{H})(x)R=\Delta(x)$ for each $x\in H$, that is
	\[\Delta(x)=\sum_{(x),i\in I,j\in J}(-1)^{|x^{'}||x^{''}|+|x^{'}||s_{i}|+|d_{j}|(|x^{''}|+|s_{i}|)}r_{j}x^{''}s_{i}\otimes
d_{j}x^{'}t_{i}.\]
	
	We have for each $x \in H$
	\[\mu_{H^{cop}\otimes H^{cop}}\circ(id_{H^{cop}}\otimes\mu_{H^{cop}\otimes H^{cop}})\circ(id_{H^{cop}\otimes
H^{cop}}\otimes\Delta_{H^{cop}}\otimes id_{H^{cop}\otimes H^{cop}})(R^{-1}\otimes x\otimes R)=\]
	\[=\mu_{H^{cop}\otimes H^{cop}}(R^{-1}\otimes(\sum_{(x),i\in I}(-1)^{|x^{'}||x^{''}|+|x^{'}||s_{i}|}x^{''}s_{i}\otimes
x^{'}t_{i}))=\]
	\[=\sum_{(x),i\in I,j\in J}(-1)^{|x^{'}||x^{''}|+|x^{'}||s_{i}|+|d_{j}|(|x^{''}|+|s_{i}|)}r_{j}x^{''}s_{i}\otimes
d_{j}x^{'}t_{i}=\Delta(x)=\tau_{H^{cop},H^{cop}}\circ\Delta_{H^{cop}}(x).\]
	
	1.3 Consider $H^{cop}$. We have for each $x \in H$
	\[\tau_{H^{cop},H^{cop}}\circ\Delta_{H^{cop}}(x)=\Delta_{H}(x)=\tau_{H,H}\circ\mu_{H\otimes H}\circ(id_{H\otimes
H}\otimes\mu_{H\otimes H})\circ(id_{H\otimes H}\otimes\Delta_{H}\otimes id_{H\otimes H})(R\otimes x\otimes R^{-1})=\]
	\[=\tau_{H,H}(\sum_{(x),i\in I,j\in J}(-1)^{|x^{''}||r_{j}|+|t_{i}|(|x^{'}|+|r_{j}|)}s_{i}x^{'}r_{j}\otimes
t_{i}x^{''}d_{j})=\]
	\[=\sum_{(x),i\in I,j\in
J}(-1)^{|x^{''}||r_{j}|+|t_{i}|(|x^{'}|+|r_{j}|)+(|t_{i}|+|x^{''}|+|d_{j}|)(|s_{i}|+|x^{'}|+|r_{j}|)}t_{i}x^{''}d_{j}\otimes
s_{i}x^{'}r_{j}=\]
	\[=\mu_{H^{cop}\otimes H^{cop}}(\tau_{H,H}(R)\otimes(\sum_{(x),j\in
J}(-1)^{|x^{'}||x^{''}|+|r_{j}||d_{j}|+|x^{'}||d_{j}|}x^{''}d_{j}\otimes x^{'}r_{j}))=\]
	\[=\mu_{H^{cop}\otimes H^{cop}}\circ(id_{H^{cop}}\otimes\mu_{H^{cop}\otimes H^{cop}})\circ(id_{H^{cop}\otimes
H^{cop}}\otimes\Delta_{H^{cop}}\otimes id_{H^{cop}\otimes H^{cop}})(\tau_{H,H}(R)\otimes x\otimes\tau_{H,H}(R^{-1}))=\]
	\[=\mu_{H^{cop}\otimes H^{cop}}\circ(id_{H^{cop}}\otimes\mu_{H^{cop}\otimes H^{cop}})\circ(id_{H^{cop}\otimes
H^{cop}}\otimes\Delta_{H^{cop}}\otimes id_{H^{cop}\otimes H^{cop}})(\tau_{H,H}(R)\otimes x\otimes\tau_{H,H}(R)^{-1}).\]
	
	Remark.
	We notice that $\tau_{H,H}(R)$ is the invertible element. Indeed,
	\[\tau_{H,H}(R)\tau_{H,H}(R^{-1})=(\sum_{i\in I}(-1)^{|s_{i}||t_{i}|}t_{i}\otimes s_{i})(\sum_{j\in
J}(-1)^{|r_{j}||d_{j}|}d_{j}\otimes r_{j})=\]
	\[=\sum_{i\in I,j\in J}(-1)^{|s_{i}||t_{i}|+|r_{j}||d_{j}|+|s_{i}||d_{j}|}t_{i}d_{j}\otimes s_{i}r_{j}=\tau_{H,H}(\sum_{i\in
I,j\in J}(-1)^{|r_{j}||t_{i}|}s_{i}r_{j}\otimes t_{i}d_{j})=\tau_{H,H}(RR^{-1})=1_{H}\otimes1_{H}.\]
	\[\tau_{H,H}(R^{-1})\tau_{H,H}(R)=(\sum_{j\in J}(-1)^{|r_{j}||d_{j}|}d_{j}\otimes r_{j})(\sum_{i\in
I}(-1)^{|s_{i}||t_{i}|}t_{i}\otimes s_{i})=\]
	\[=\sum_{i\in I,j\in J}(-1)^{|s_{i}||t_{i}|+|r_{j}||d_{j}|+|r_{j}||t_{i}|}d_{j}t_{i}\otimes r_{j}s_{i}=\tau_{H,H}(\sum_{i\in
I,j\in J}(-1)^{|s_{i}||d_{j}|}r_{j}s_{i}\otimes d_{j}t_{i})=\tau_{H,H}(R^{-1}R)=1_{H}\otimes1_{H}.\]
	
	Thus
	\[\tau_{H,H}(R^{-1})=\tau_{H,H}(R)^{-1}. \qed \]
	
	2. It follows from 1 that a Hopf superalgebra $(H,\mu,\eta,\Delta^{op},\epsilon,S^{-1},S,\tau_{H,H}(R))$ is
quasi-cocommutative.
	
	2.1 We verify that the realtion \ref{eq:RuniversalBr2} holds.
	\[ (\tau_{H,H} \otimes id_{H}) ( \Delta_{H} \otimes id_{H} )(R) = (\tau_{H,H} \otimes  id_{H}) (R_{13} R_{23}), \]
	\[ \sum_{(s_i),i \in I} (-1)^{|s_i^{'}||s_i^{''}|} s_i^{''} \otimes s_i^{'} \otimes t_i = (\tau_{H,H} \otimes  id_{H}) ((
\sum_{i \in I} s_{i} \otimes 1_{H} \otimes t_{i} ) ( \sum_{j \in J} 1_{H} \otimes s_{j} \otimes t_{j} )) = \]
	\[ = \sum_{i \in I, j \in J} (-1)^{|t_{i}||s_{j}| + |s_j||s_i|} s_j \otimes s_i \otimes t_i t_j, \]
	\[ (\tau_{H,H} \otimes id_{H}) \circ ( id_{H} \otimes \tau_{H,H} ) ( \sum_{(s_i),i \in I} (-1)^{|s_i^{'}||s_i^{''}|}
s_i^{''} \otimes s_i^{'} \otimes t_i ) = \]
	\[ = (\tau_{H,H} \otimes id_{H}) \circ ( id_{H} \otimes \tau_{H,H} ) (\sum_{i \in I, j \in J} (-1)^{|t_{i}||s_{j}| +
|s_j||s_i|} s_j \otimes s_i \otimes t_i t_j), \]
	\[ \sum_{(s_i),i \in I} (-1)^{|s_i^{'}||s_i^{''}| + |t_i||s_i|} t_i \otimes s_i^{''} \otimes s_i^{'} = \sum_{i \in I, j \in
J} (-1)^{|t_{i}||s_{j}| + |s_j||s_i| + (|s_i| + |s_j|)( |t_i| + |t_j| )} t_i t_j \otimes s_j \otimes s_i = \]
	\[ = \sum_{i \in I, j \in J} (-1)^{|s_{i}||t_{i}| + |s_j||t_j| + |s_i|( |s_j| + |t_j| ) } t_i t_j \otimes s_j \otimes s_i.
\]
	
	Thus
	\[ ( id_{H} \otimes \Delta_{H^{cop}} ) ( \tau_{H,H} (R) ) = (\tau_{H,H} (R))_{13} (\tau_{H,H} (R))_{12}. \] 	
	
	2.2 We verify that the realtion \ref{eq:RuniversalBr1} holds.
	
	\[ (\tau_{H,H} \otimes id_{H}) \circ ( id_{H} \otimes \tau_{H,H} ) \circ (\tau_{H,H} \otimes id_{H}) \circ (id_{H} \otimes
\Delta_{H}) (R) = (\tau_{H,H} \otimes id_{H}) \circ ( id_{H} \otimes \tau_{H,H} ) \circ (\tau_{H,H} \otimes id_{H}) (R_{13}
R_{12}), \]
	\[ (\tau_{H,H} \otimes id_{H}) \circ ( id_{H} \otimes \tau_{H,H} ) ( \sum_{(t_i),i \in I} (-1)^{|s_i||t_i^{'}|} t_i^{'}
\otimes s_i \otimes t_i^{''} ) = \]
	\[ = (\tau_{H,H} \otimes id_{H}) \circ ( id_{H} \otimes \tau_{H,H} ) ( \sum_{i \in I, j \in J} (-1)^{ |t_j|( |s_i| + |s_j| )
+ |t_i|( |s_j| + |t_j| ) } t_j \otimes s_i s_j \otimes t_i), \]
	\[ \sum_{(t_i),i \in I} (-1)^{|s_i||t_i| + |t_i^{'}||t_i^{''}|} t_i^{''} \otimes t_i^{'} \otimes s_i = \sum_{i \in I, j \in
J} (-1)^{|s_i||t_i| + |t_j|(|s_j| + |s_i|)} t_{i} \otimes t_j \otimes s_{i} s_j. \]
	
	Thus
	\[ ( \Delta_{H^{cop}} \otimes id_{H} )( \tau_{H,H}(R) ) = (\tau_{H,H}(R))_{13} (\tau_{H,H}(R))_{23}. \]
	
\end{proof}

\propuniversalmatrix*
\begin{proof}
	Relation \ref{eq:RuniversalBr1} and the defintion of $R$ imply
	\[R_{12}R_{13}R_{23}=R_{12}(\Delta_{H}\otimes id_{H})(R)=\sum_{j\in I}(R\Delta_{H}(s_{j}))\otimes t_{j}= \sum_{j\in
I}((\tau_{H,H}\circ\Delta_{H}(s_{j}))R)\otimes t_{j}=\]
	\[=\sum_{i,j\in
I,(s_{j})}(-1)^{|(s_{j})^{'}||(s_{j})^{''}|+|s_{i}||(s_{j})^{'}|}(s_{j})^{''}s_{i}\otimes(s_{j})^{'}t_{i}\otimes
t_{j}=((\tau_{H,H}\circ\Delta_{H})\otimes id_{H})(R)R_{12}=\]
	\[=(\tau_{H,H}\otimes id_{H})(\Delta_{H}\otimes id_{H})(R)R_{12}=(\tau_{H,H}\otimes id_{H})(R_{13}R_{23})R_{12}=\]
	\[=(\sum_{i,j\in I}(-1)^{|s_{j}|(|s_{i}|+|t_{i}|)}s_{j}\otimes s_{i}\otimes t_{i}t_{j})R_{12}=R_{23}R_{13}R_{12}.\]
	
	It follows from the equality \ref{eq:RuniversalBr1} that
	\[R=(id_{H}\otimes id_{H})(R)=((\nu_{k,H}\circ(\epsilon_{H}\otimes id_{H}))\otimes id_{H})(\Delta_{H}\otimes
id_{H})(R)=((\nu_{k,H}\circ(\epsilon_{H}\otimes id_{H}))\otimes id_{H})R_{13}R_{23}=\]
	\[=((\nu_{k,H}\circ(\epsilon_{H}\otimes id_{H}))\otimes id_{H})(\sum_{i,j\in I}(-1)^{|t_{i}||s_{j}|}s_{i}\otimes
s_{j}\otimes t_{i}t_{j})=\sum_{i,j\in I}(-1)^{|t_{i}||s_{j}|}\epsilon_{H}(s_{i})s_{j}\otimes t_{i}t_{j}=\]
	\[=((\eta_{H}\circ\epsilon_{H})\otimes id_{H})(R)R \iff RR^{-1}=((\eta_{H}\circ\epsilon_{H})\otimes id_{H})(R)RR^{-1}\iff
1_{H}\otimes 1_{H}=((\eta_{H}\circ\epsilon_{H})\otimes id_{H})(R).\]
	
	It follows from the equality \ref{eq:RuniversalBr2} that
	\[R=(id_{H}\otimes id_{H})(R)=(id_{H}\otimes(\nu_{H,k}\circ(id_{H}\otimes\epsilon_{H})))(id_{H}\otimes\Delta_{H})(R)=\]
	
\[=(id_{H}\otimes(\nu_{H,k}\circ(id_{H}\otimes\epsilon_{H})))(R_{13}R_{12})=(id_{H}\otimes(\nu_{H,k}\circ(id_{H}\otimes\epsilon_{H})))(\sum_{i,j\in
I}s_{i}s_{j}\otimes t_{j}\otimes t_{i})=\]
	\[=\sum_{i,j\in I}s_{i}s_{j}\otimes\epsilon_{H}(t_{i})t_{j}=(id_{H}\otimes(\eta_{H}\circ\epsilon_{H}))(R)R\iff\]
	\[RR^{-1}=(id_{H}\otimes(\eta_{H}\circ\epsilon_{H}))(R)RR^{-1}\iff 1_{H}\otimes
1_{H}=(id_{H}\otimes(\eta_{H}\circ\epsilon_{H}))(R).\]
	
	We have
	\[((\mu_{H}\circ(S\otimes id)\circ\Delta_{H})\otimes id_{H})(R)=((\eta_{H}\circ\epsilon_{H})\otimes
id_{H})(R)=1_{H}\otimes1_{H}.\]
	
	Thus
	\[1_{H}\otimes1_{H}=(\mu_{H}\otimes id_{H}) \circ (S\otimes id_{H}\otimes id_{H}) \circ (\Delta_{H}\otimes
id_{H})(R)=(\mu_{H}\otimes id_{H}) \circ (S\otimes id_{H}\otimes id_{H})(R_{13}R_{23})=\]
	\[= \sum_{i,j\in I} (-1)^{|s_{j}||t_{i}|} S(s_{i}) s_{j} \otimes t_{i}t_{j} = (S\otimes id_{H})(R)R \iff R^{-1}=(S\otimes
id_{H})(R)RR^{-1}=(S\otimes id_{H})(R).\]
	
	Suppose that $H$ has an invertible antipode $S$. We can consider the braided Hopf superalgebra
	\[(H,\mu,\eta,\Delta^{op},\epsilon,S^{-1},S,\tau_{H,H}(R))\]
	of Lemma \ref{pr:newquasicom}. Then
	\[((\mu_{H}\circ(S^{-1}\otimes id_{H})\circ\ (\tau_{H,H} \circ \Delta_{H}) )\otimes id_{H})(\tau_{H,H}(R))=((\eta_{H}
\circ\epsilon_{H})\otimes id_{H})(\tau_{H,H}(R))=1_{H}\otimes1_{H}.\]	
	
	Consequently
	\[1_{H}\otimes1_{H}= ((\mu_{H}\circ(S^{-1}\otimes id_{H})\circ\ (\tau_{H,H} \circ \Delta_{H}) )\otimes
id_{H})(\tau_{H,H}(R)) = \]
	\[ = ( \mu_{H} \otimes id_{H} ) \circ ( S^{-1} \otimes id_{H} \otimes id_{H} ) \circ ( (\tau_{H,H} \circ \Delta_{H}) \otimes
id_{H} )(\tau_{H,H}(R)) = \]
	\[ = ( \mu_{H} \otimes id_{H} ) \circ ( S^{-1} \otimes id_{H} \otimes id_{H} ) ( (\tau_{H,H}(R))_{13} (\tau_{H,H}(R))_{23} )
= \]
	\[ = \sum_{i,j \in I} (-1)^{|s_{i}||t_{i}| + |s_{j}||t_{j}| + |t_j||s_i|} S^{-1}(t_{i}) t_j \otimes s_{i} s_{j} = ( S^{-1}
\otimes id_{H} ) ( \tau_{H,H}(R) ) \tau_{H,H}(R) \iff \]
	\[ \tau_{H,H}(R)^{-1} = ( S^{-1} \otimes id_{H} ) ( \tau_{H,H}(R) ) \tau_{H,H}(R) \tau_{H,H}(R)^{-1} = ( S^{-1} \otimes
id_{H} ) ( \tau_{H,H}(R) ) \iff \]
	\[ \tau_{H,H} \circ \tau_{H,H}(R^{-1}) =  \tau_{H,H} \circ ( S^{-1} \otimes id_{H} ) ( \tau_{H,H}(R) ) \iff \]
	\[ R^{-1} = \sum_{i \in I} (-1)^{ |s_i||t_i| + |s_i||t_i| } s_i \otimes S^{-1}(t_i) = (id_{H} \otimes S^{-1}) (R). \]
	
	Finally, we have
	\[ ( S \otimes S )(R) = ( id_{H} \otimes S ) ( S \otimes id_{H} ) (R) = ( id_{H} \otimes S ) (R^{-1}) = \]
	\[ = ( id_{H} \otimes S ) ( id_{H} \otimes S^{-1} ) (R) = ( id_{H} \otimes id_{H} ) (R) = R. \]
	
\end{proof}

\isomctouniveralmatrix*
\begin{proof}
	$c_{V,W}^{R}(?) = \tau_{V,W} \circ ((\alpha_{H,V} \otimes \alpha_{H,W}) \circ (id_H \otimes \tau_{H,V} \otimes id_{W})) ( R
\otimes ? )$, where $?$ is a numb variable, is a composition of superspace morphisms. Thus it is a superspace morphism.
	
	We prove the bijectivity of $c_{V,W}^{R}$ by constructing the inverse $(c_{V,W}^{R})^{-1}$. Note that $(c_{V,W}^{R})^{-1}(?)
= ((\alpha_{H,V} \otimes \alpha_{H,W}) \circ (id_H \otimes \tau_{H,V} \otimes id_{W})) ( R^{-1} \otimes \tau_{W,V}(?) ) $, where
$?$ is a numb variable, is a composition of superspace morphisms. Thus it is a superspace morphism.
	
	We have for all $v \in V, w \in W$
	\[ (c_{V,W}^{R})^{-1} \circ c_{V,W}^{R} (v \otimes w) = (c_{V,W}^{R})^{-1} ( \sum_{i \in I} (-1)^{|s_i| ( |t_i| + |w| ) +
|w||v|} \alpha_{H,W} (t_i \otimes w) \otimes \alpha_{H,V} (s_i \otimes v) ) = \]
	\[ = \sum_{i,j \in I} (-1)^{|s_i| ( |t_i| + |w| ) + |w||v| + (|s_i| + |v|)(|t_i| + |w|) + |t_j|(|s_i| + |v|)} \alpha_{H,V}
(S(s_j) s_i \otimes v) \otimes \alpha_{H,W} (t_j t_i \otimes w) = \]
	\[ = \sum_{i,j \in I} (-1)^{|v|(|t_i| + |t_j|) + |t_j||s_i| } \alpha_{H,V}  (S(s_j) s_i \otimes v) \otimes \alpha_{H,W} (t_j
t_i \otimes w) = \]
	\[ = ((\alpha_{H,V} \otimes \alpha_{H,W}) \circ (id_H \otimes \tau_{H,V} \otimes id_{W}))  ((\sum_{i,j \in I} (-1)^{
|t_j||s_i| }  S(s_j) s_i \otimes t_j t_i) \otimes ( v \otimes w )) = \]
	\[ = ((\alpha_{H,V} \otimes \alpha_{H,W}) \circ (id_H \otimes \tau_{H,V} \otimes id_{W})) ((R^{-1} R) \otimes ( v \otimes w
)) = \]
	\[ = ((\alpha_{H,V} \otimes \alpha_{H,W}) \circ (id_H \otimes \tau_{H,V} \otimes id_{W})) ((1_H \otimes 1_{H}) \otimes ( v
\otimes w )) = v \otimes w. \]
	
	We can rewrite this relation in the following form
	\[ (c_{V,W}^{R})^{-1} \circ c_{V,W}^{R} (v \otimes w) = (c_{V,W}^{R})^{-1} ( \sum_{i \in I} (-1)^{|s_i| ( |t_i| + |w| ) +
|w||v|} \alpha_{H,W} (t_i \otimes w) \otimes \alpha_{H,V} (s_i \otimes v) ) = \]
	\[ = \sum_{i,j \in I} (-1)^{|s_i| ( |t_i| + |w| ) + |w||v| + (|s_i| + |v|)(|t_i| + |w|) + |t_j|(|s_i| + |v|)} \alpha_{H,V}
(s_j s_i \otimes v) \otimes \alpha_{H,W} (S^{-1}(t_j) t_i \otimes w) = \]
	\[ = \sum_{i,j \in I} (-1)^{|v|(|t_i| + |t_j|) + |t_j||s_i| } \alpha_{H,V} (s_j s_i \otimes v) \otimes \alpha_{H,W}
(S^{-1}(t_j) t_i \otimes w) = \]
	\[ = ((\alpha_{H,V} \otimes \alpha_{H,W}) \circ (id_H \otimes \tau_{H,V} \otimes id_{W})) ((\sum_{i,j \in I} (-1)^{
|t_j||s_i| }  s_j s_i \otimes S^{-1}(t_j) t_i) \otimes ( v \otimes w )) = \]
	\[ = ((\alpha_{H,V} \otimes \alpha_{H,W}) \circ (id_H \otimes \tau_{H,V} \otimes id_{W})) ((R^{-1} R) \otimes ( v \otimes w
)) = \]
	\[ = ((\alpha_{H,V} \otimes \alpha_{H,W}) \circ (id_H \otimes \tau_{H,V} \otimes id_{W})) ((1_{H} \otimes 1_{H}) \otimes ( v
\otimes w )) = v \otimes w. \]
	
	Furthermore,
	\[ c_{V,W}^{R} \circ (c_{V,W}^{R})^{-1} (w \otimes v) = c_{V,W}^{R} ( \sum_{i \in I} (-1)^{|w||v| + |t_i||v|} \alpha_{H,V}
(S(s_i) \otimes v) \otimes \alpha_{H,W} (t_i \otimes w) ) =  \]
	\[ = \sum_{i,j \in I} (-1)^{|w||v| + |t_i||v| + |t_j|( |s_i| + |v| ) + ( |t_j| + |t_i| + 	|w| )( |s_j| + |s_i| + |v| ) }
\alpha_{H,W} (t_jt_i \otimes w) \otimes \alpha_{H,V} (s_jS(s_i) \otimes v)  = \]
	\[ = \sum_{i,j \in I} (-1)^{|t_j||s_j| + ( |t_i| + |w| )( |s_j| + |s_i| ) } \alpha_{H,W} (t_jt_i \otimes w) \otimes
\alpha_{H,V} (s_jS(s_i) \otimes v)  = \]
	\[ = ((\alpha_{H,W} \otimes \alpha_{H,V}) \circ (id_H \otimes \tau_{H,W} \otimes id_{V})) ((\sum_{i,j \in I}
(-1)^{|t_j||s_j| + |t_i||s_i| + |t_i||s_j| } t_jt_i \otimes s_jS(s_i)) \otimes ( w \otimes v )) = \]
	\[ = ((\alpha_{H,W} \otimes \alpha_{H,V}) \circ (id_H \otimes \tau_{H,W} \otimes id_{V})) ((\tau_{H,H}(R)
\tau_{H,H}(R^{-1})) \otimes ( w \otimes v )) = \]
	\[ = ((\alpha_{H,W} \otimes \alpha_{H,V}) \circ (id_H \otimes \tau_{H,W} \otimes id_{V})) ((\tau_{H,H}(R)
(\tau_{H,H}(R))^{-1}) \otimes ( w \otimes v )) = \]
	\[ = ((\alpha_{H,W} \otimes \alpha_{H,V}) \circ (id_H \otimes \tau_{H,W} \otimes id_{V})) ((1_{H} \otimes 1_{H} ) \otimes (
w \otimes v )) = w \otimes v. \]
	
	We can rewrite this relation in the following form
	\[ c_{V,W}^{R} \circ (c_{V,W}^{R})^{-1} (w \otimes v) = c_{V,W}^{R} ( \sum_{i \in I} (-1)^{|w||v| + |t_i||v|} \alpha_{H,V}
(s_i \otimes v) \otimes \alpha_{H,W} (S^{-1}(t_i) \otimes w) ) =  \]
	\[ = \sum_{i,j \in I} (-1)^{|w||v| + |t_i||v| + |t_j|( |s_i| + |v| ) + ( |t_j| + |t_i| + |w| )( |s_j| + |s_i| + |v| ) }
\alpha_{H,W} (t_jS^{-1}(t_i) \otimes w) \otimes \alpha_{H,V} (s_js_i \otimes v) = \]
	\[ = \sum_{i,j \in I} (-1)^{|t_j||s_j| + ( |t_i| + |w| )( |s_j| + |s_i| ) } \alpha_{H,W} (t_jS^{-1}(t_i) \otimes w) \otimes
\alpha_{H,V} (s_js_i \otimes v)  = \]
	\[ = ((\alpha_{H,W} \otimes \alpha_{H,V}) \circ (id_H \otimes \tau_{H,W} \otimes id_{V})) ((\sum_{i,j \in I}
(-1)^{|t_j||s_j| + |t_i||s_i| + |t_i||s_j| } t_jS^{-1}(t_i) \otimes s_js_i) \otimes ( w \otimes v )) = \]
	\[ = ((\alpha_{H,W} \otimes \alpha_{H,V}) \circ (id_H \otimes \tau_{H,W} \otimes id_{V})) ((\tau_{H,H}(R)
\tau_{H,H}(R^{-1})) \otimes ( w \otimes v )) = \]
	\[ = ((\alpha_{H,W} \otimes \alpha_{H,V}) \circ (id_H \otimes \tau_{H,W} \otimes id_{V})) ((\tau_{H,H}(R)
(\tau_{H,H}(R))^{-1}) \otimes ( w \otimes v )) = \]
	\[ = ((\alpha_{H,W} \otimes \alpha_{H,V}) \circ (id_H \otimes \tau_{H,W} \otimes id_{V})) ((1_{H} \otimes 1_{H}) \otimes ( w
\otimes v )) = w \otimes v. \]
	
	Therefore
	\[ (c_{V,W}^{R})^{-1} \circ c_{V,W}^{R} = c_{V,W}^{R} \circ (c_{V,W}^{R})^{-1} = id_{H} \otimes id_{H}. \]
\end{proof}

\isompropuniverslamatrix*
\begin{proof}
	1. It follows from Lemma \ref{lm:isomctouniveralmatrix} that $c_{V,W}^{R}$ is the superspace isomorphism. We prove that
$c_{V,W}^{R}$ is a mapping of left $H$-modules. We have for all $x \in H$, $v \in V$, $w \in W$
	\[ c_{V,W}^{R} ( \alpha_{H, V \otimes W} (x \otimes ( v \otimes w )) ) = c_{V,W}^{R} ( ((\alpha_{H,V} \otimes \alpha_{H,W})
\circ (id_{H} \otimes \tau_{H,V} \otimes id_{W})) (\Delta_{H}(x) \otimes ( v \otimes w )) ) = \]
	\[ = c_{V,W}^{R} ( \sum_{(x)} (-1)^{|x^{''}||v|} \alpha_{H,V} (x^{'} \otimes v) \otimes \alpha_{H,W} (x^{''} \otimes w) ) =
\]
	\[ = \tau_{V,W}( \sum_{(x), i \in I} (-1)^{|x^{''}||v| + |t_i|( |x^{'}| + |v| ) } \alpha_{H,V} (s_i x^{'} \otimes v) \otimes
\alpha_{H,W} (t_i x^{''} \otimes w) ) = \]	
	\[ = \tau_{V,W}( ((\alpha_{H,V} \otimes \alpha_{H,W}) \circ (id_H \otimes \tau_{H,V} \otimes id_{W})) (( \sum_{(x), i \in I}
(-1)^{|t_i||x^{'}| } s_i x^{'} \otimes t_i x^{''} ) \otimes ( v \otimes w )) ) = \]
	\[ = \tau_{V,W} ( ((\alpha_{H,V} \otimes \alpha_{H,W}) \circ (id_{H} \otimes \tau_{H,V} \otimes id_{W})) ((R \Delta_{H}(x))
\otimes ( v \otimes w )) ) = \]
	\[ = \tau_{V,W} ( ((\alpha_{H,V} \otimes \alpha_{H,W}) \circ (id_{H} \otimes \tau_{H,V} \otimes id_{W})) (( (\tau_{H,H}
\circ \Delta_{H}(x)) R ) \otimes ( v \otimes w )) ) = \]
	\[ = \sum_{(x), i \in I} (-1)^{ |x^{'}||x^{''}| + |x^{'}||s_i| + |v|( |x^{'}| + |t_i| ) + (|x^{''}| + |s_i| + |v|)(|x^{'}| +
|t_i| + |w| ) } \alpha_{H,W} (x^{'} t_i \otimes w) \otimes \alpha_{H,V} (x^{''}s_i \otimes v) = \]
	\[ = ((\alpha_{H,W} \otimes \alpha_{H,V}) \circ (id_{H} \otimes \tau_{H,W} \otimes id_{V})) (\Delta_{H}(x) \otimes ( \sum_{i
\in I} (-1)^{  |v||t_i| + (|s_i| + |v|)(|t_i| + |w| ) } \alpha_{H,W} (t_i \otimes w) \otimes \alpha_{H,V} (s_i \otimes v) )) =
\]
	\[ = ((\alpha_{H,W} \otimes \alpha_{H,V}) \circ (id_{H} \otimes \tau_{H,W} \otimes id_{V})) (\Delta_{H}(x) \otimes
(\tau_{V,W}( \sum_{i \in I} (-1)^{|v||t_i|} \alpha_{H,V} (s_i \otimes v) \otimes \alpha_{H,W} (t_i \otimes w) ))) = \]
	\[ = \alpha_{H , W \otimes V} (x \otimes (\tau_{V,W}( ((\alpha_{H,V} \otimes \alpha_{H,W}) \circ (id_H \otimes \tau_{H,V}
\otimes id_{W})) (R \otimes ( v \otimes w )) ))) =  \alpha_{H , W \otimes V} (x \otimes (c_{V,W}^{R}( v \otimes w ))). \]
	
	2. Using the relation \ref{eq:RuniversalBr1}, we have for all $u \in U$, $v \in V$ and $w \in W$:
	\[ (c_{U,W}^{R} \otimes id_V)(id_{U} \otimes c_{V,W}^{R}) ( u \otimes v \otimes w ) = \]
	\[ = (c_{U,W}^{R} \otimes id_V) ( \sum_{i \in I} (-1)^{|t_i||v| + (|s_i| + |v|) (|t_i| + |w|)} u \otimes \alpha_{H,W} (t_i
\otimes w) \otimes \alpha_{H,V} (s_i \otimes v) ) = \]
	\[ = \sum_{i,j \in I} (-1)^{|t_i||v| + (|s_i| + |v|) (|t_i| + |w|) + |t_j||u| + (|s_j| + |u|) (|t_j| + |t_i| + |w|)}
\alpha_{H,W} (t_jt_i \otimes w) \otimes \alpha_{H,U} (s_j \otimes u) \otimes \alpha_{H,V} (s_i \otimes v) = \]
	\[ = (\alpha_{H,W} \otimes \alpha_{H,U} \otimes \alpha_{H,V}) \circ (id_{H} \otimes \tau_{H,W} \otimes id_{U} \otimes id_{H}
\otimes id_{V}) \circ (id_{H} \otimes id_{H} \otimes \tau_{H,W \otimes U} \otimes id_{V}) \circ \]
	\[ \circ ( \sum_{i,j \in I} (-1)^{ |v||w|+ |u||w| + |s_i||t_j| + (|s_i|+ |s_j|) (|t_j| + |t_i|)} t_jt_i \otimes s_j \otimes
s_i \otimes (w \otimes u \otimes v) ) = \]
	\[ = (\alpha_{H,W} \otimes \alpha_{H,U} \otimes \alpha_{H,V}) \circ (id_{H} \otimes \tau_{H,W} \otimes id_{U} \otimes id_{H}
\otimes id_{V}) \circ (id_{H} \otimes id_{H} \otimes \tau_{H,W \otimes U} \otimes id_{V}) \circ \]
	\[ \circ ( \tau_{H \otimes H,H} (\sum_{i,j \in I} (-1)^{ |v||w|+ |u||w| + |s_i||t_j|} s_j \otimes s_i \otimes t_jt_i )
\otimes (w \otimes u \otimes v) ) = \]
	\[ = (\alpha_{H,W} \otimes \alpha_{H,U} \otimes \alpha_{H,V}) \circ (id_{H} \otimes \tau_{H,W} \otimes id_{U} \otimes id_{H}
\otimes id_{V}) \circ (id_{H} \otimes id_{H} \otimes \tau_{H,W \otimes U} \otimes id_{V}) \circ \]
	\[ \circ ( \tau_{H \otimes H,H} (\sum_{j \in I, (s_j)} (-1)^{ |v||w|+ |u||w|} (s_j)^{'} \otimes (s_j)^{''} \otimes t_j)
\otimes (w \otimes u \otimes v) ) = \]
	\[ = \sum_{j \in I, (s_j)} (-1)^{ |v||w|+ |u||w| + |t_j||s_j| + |w||s_j| + |u||(s_j)^{''}|} \alpha_{H,W} (t_j \otimes w)
\otimes \alpha_{H,U} ((s_j)^{'} \otimes u) \otimes \alpha_{H,V} ((s_j)^{''} \otimes v) = \]
	\[ = \tau_{U \otimes V, W} ( \sum_{j \in I, (s_j)} (-1)^{ |t_j|( |u| + |v| ) + |u||(s_j)^{''}|} \alpha_{H,U} ((s_j)^{'}
\otimes u) \otimes \alpha_{H,V} ((s_j)^{''} \otimes v) \otimes \alpha_{H,W} (t_j \otimes w) ) = \]
	\[ = \tau_{U \otimes V, W} ( \sum_{j \in I} (-1)^{ |t_j|( |u| + |v| )} \alpha_{H,U \otimes V} (s_j \otimes (u \otimes v))
\otimes \alpha_{H,W} (t_j \otimes w) ) = \]
	\[ = \tau_{U \otimes V, W} ((( \alpha_{H,U \otimes V} \otimes \alpha_{H,W} ) \circ (id_{H} \otimes \tau_{H,U \otimes V}
\otimes id_{W})) (R \otimes ( u \otimes v \otimes w ))) =  c_{U \otimes V,W}^{R} ( u \otimes v \otimes w ). \]
	
	Using the relation \ref{eq:RuniversalBr2}, we have for all $u \in U$, $v \in V$ and $w \in W$:
	\[ (id_{V} \otimes c_{U,W}^{R})(c_{U,V}^{R} \otimes id_{W}) ( u \otimes v \otimes w ) = \]
	\[ = (id_{V} \otimes c_{U,W}^{R}) ( \sum_{i \in I} (-1)^{|t_i||u| + (|t_i| + |v|)(|s_i| + |u|)} \alpha_{H,V} (t_i \otimes v)
\otimes \alpha_{H,U} (s_i \otimes u) \otimes w ) = \]
	\[ = \sum_{i,j \in I} (-1)^{|t_i||u| + (|t_i| + |v|)(|s_i| + |u|) + |t_j|( |s_i| + |u| ) + (|t_j| + |w|)(|s_j| + |s_i| +
|u|)} \alpha_{H,V} (t_i \otimes v) \otimes \alpha_{H,W} (t_j \otimes w) \otimes \alpha_{H,U} (s_j s_i \otimes u) = \]
	\[ = (\alpha_{H,V} \otimes \alpha_{H,W} \otimes \alpha_{H,U}) \circ (id_{H} \otimes \tau_{H,V} \otimes id_{W} \otimes id_{H}
\otimes id_{U}) \circ (id_{H} \otimes id_{H} \otimes \tau_{H,V \otimes W} \otimes id_{U}) \circ \]
	\[ \circ ((\sum_{i,j \in I} (-1)^{|t_i|(|s_j| + |s_i|) + |t_i||s_j| + |t_j||s_i| + |t_j|(|s_j| + |s_i|) + |v||u| + |w||u|}
t_i \otimes t_j \otimes s_j s_i) \otimes ( v \otimes w \otimes u )) = \]
	\[ = (\alpha_{H,V} \otimes \alpha_{H,W} \otimes \alpha_{H,U}) \circ (id_{H} \otimes \tau_{H,V} \otimes id_{W} \otimes id_{H}
\otimes id_{U}) \circ (id_{H} \otimes id_{H} \otimes \tau_{H,V \otimes W} \otimes id_{U}) \circ \]
	\[ \circ (\sum_{i,j \in I} (-1)^{|t_i|(|s_j| + |s_i|) + |s_j|(|t_i| + |s_i|) + |s_i|(|t_j| + |s_j|)+ |t_j|(|s_j| + |s_i|) +
|v||u| + |w||u|} t_i \otimes t_j \otimes s_j s_i) ( v \otimes w \otimes u ) = \]
	\[ = (\alpha_{H,V} \otimes \alpha_{H,W} \otimes \alpha_{H,U}) \circ (id_{H} \otimes \tau_{H,V} \otimes id_{W} \otimes id_{H}
\otimes id_{U}) \circ (id_{H} \otimes id_{H} \otimes \tau_{H,V \otimes W} \otimes id_{U}) \circ \]
	\[ \circ (\sum_{i,j \in I} (-1)^{|t_i|(|s_j| + |s_i|) + |t_j|(|s_j| + |s_i|) + |v||u| + |w||u|} t_i \otimes t_j \otimes s_j
s_i) ( v \otimes w \otimes u ) = \]
	\[ = (\alpha_{H,V} \otimes \alpha_{H,W} \otimes \alpha_{H,U}) \circ (id_{H} \otimes \tau_{H,V} \otimes id_{W} \otimes id_{H}
\otimes id_{U}) \circ (id_{H} \otimes id_{H} \otimes \tau_{H,V \otimes W} \otimes id_{U}) \circ \]
	\[ \circ (-1)^{|v||u| + |w||u|} ( (id_{H} \otimes \tau_{H,H}) \circ (\tau_{H,H} \otimes id_{H} ) ( \sum_{i,j \in I}  s_j s_i
\otimes t_i \otimes t_j )) ( v \otimes w \otimes u ) = \]
	\[ = (\alpha_{H,V} \otimes \alpha_{H,W} \otimes \alpha_{H,U}) \circ (id_{H} \otimes \tau_{H,V} \otimes id_{W} \otimes id_{H}
\otimes id_{U}) \circ (id_{H} \otimes id_{H} \otimes \tau_{H,V \otimes W} \otimes id_{U}) \circ \]
	\[ \circ (-1)^{|v||u| + |w||u|} ( (id_{H} \otimes \tau_{H,H}) \circ (\tau_{H,H} \otimes id_{H} ) ( \sum_{j \in I,(t_j)}  s_j
\otimes (t_j)^{'} \otimes (t_j)^{''} )) ( v \otimes w \otimes u ) = \]
	\[ = (\alpha_{H,V} \otimes \alpha_{H,W} \otimes \alpha_{H,U}) \circ (id_{H} \otimes \tau_{H,V} \otimes id_{W} \otimes id_{H}
\otimes id_{U}) \circ (id_{H} \otimes id_{H} \otimes \tau_{H,V \otimes W} \otimes id_{U}) \circ \]
	\[ \circ (\sum_{j \in I,(t_j)} (-1)^{|v||u| + |w||u| + |s_j||t_j|} (t_j)^{'} \otimes (t_j)^{''} \otimes s_j) ( v \otimes w
\otimes u ) = \]
	\[ = \sum_{j \in I,(t_j)} (-1)^{|v||u| + |w||u| + |s_j||t_j| + |s_j|(|v| + |w|) + |v||(t_j)^{''}|} \alpha_{H,V} ((t_j)^{'}
\otimes v) \otimes \alpha_{H,W} ((t_j)^{''} \otimes w) \otimes \alpha_{H,U} (s_j \otimes u) = \]
	\[ = \sum_{j \in I} (-1)^{|v||u| + |w||u| + |s_j||t_j| + |s_j|(|v| + |w|)} (( \alpha_{H,V} \otimes \alpha_{H,W} ) \circ
(id_{H} \otimes \tau_{H,V} \otimes id_{W}) (\Delta_{H}(t_j) \otimes (v \otimes w))) \otimes \alpha_{H,U} (s_j \otimes u) = \]
	\[ = \tau_{U,V \otimes W} (\sum_{j \in I} (-1)^{|u||t_j|} \alpha_{H,U} (s_j \otimes u) \otimes \alpha_{H, V \otimes W} (t_j
\otimes (v \otimes w))) = \]
	\[ = \tau_{U,V \otimes W} ( ((\alpha_{H,U} \otimes \alpha_{H, V \otimes W}) \circ (id_{H} \otimes \tau_{H,U} \otimes id_{V}
\otimes id_{W} )) (R \otimes ( u \otimes (v \otimes w) )) ) = c_{U,V \otimes W}^{R} ( u \otimes v \otimes w ). \]
	
	Finally, we have for all $u \in U$, $v \in V$ and $w \in W$:
	\[ (c_{V,W}^{R} \otimes id_{U})(id_{V} \otimes c_{U,W}^{R})(c_{U,V}^{R} \otimes id_{W} ) ( u \otimes v \otimes w ) = \]
	\[ = (c_{V,W}^{R} \otimes id_{U})(id_{V} \otimes c_{U,W}^{R}) ( \sum_{i \in I} (-1)^{|t_i||u| + (|s_i| + |u|) (|t_i| + |v|)}
\alpha_{H,V} (t_i \otimes v) \otimes \alpha_{H,U} (s_i \otimes u) \otimes w ) =  \]
	\[ = (c_{V,W}^{R} \otimes id_{U}) ( \sum_{i,j \in I} (-1)^{|t_i||u| + (|s_i| + |u|) (|t_i| + |v|) + |t_j|( |s_i| + |u| ) + (
|t_j| + |w| )( |s_j| + |s_i| + |u| ) } * \]
	\[ * \alpha_{H,V} (t_i \otimes v) \otimes \alpha_{H,W} (t_j \otimes w) \otimes \alpha_{H,U} (s_js_i \otimes u) ) =  \]
	\[ = \sum_{i,j,k \in I} (-1)^{|t_i||u| + (|s_i| + |u|) (|t_i| + |v|) + |t_j|( |s_i| + |u| ) + ( |t_j| + |w| )( |s_j| + |s_i|
+ |u| ) + |t_k|( |t_i| + |v| ) + ( |t_k| + |t_j| + |w| )( |s_k| + |t_i| + |v| ) } * \]
	\[ * \alpha_{H,W} (t_kt_j \otimes w) \otimes \alpha_{H,V} (s_kt_i \otimes v) \otimes \alpha_{H,U} (s_js_i \otimes u) =  \]
	\[ = (id_{W} \otimes \tau_{U,V}) \circ (\tau_{U,W} \otimes id_{V}) \circ (id_{V} \otimes \tau_{V,W}) \circ \]
	\[ \circ (\sum_{i,j,k \in I} (-1)^{|v||u| + |s_i| (|t_i| + |v|) + |t_j|( |s_i| + |u| ) + |t_k| ( |s_j| + |s_i| + |u| + |t_i|
+ |v| ) + (|s_j| + |s_i| + |u|)(|s_k| + |t_i| + |v|)} * \]
	\[ * \alpha_{H,U} (s_js_i \otimes u) \otimes \alpha_{H,V} (s_kt_i \otimes v) \otimes \alpha_{H,W} (t_kt_j \otimes w) )=  \]
	\[ = (id_{W} \otimes \tau_{U,V}) \circ (\tau_{U,W} \otimes id_{V}) \circ (id_{V} \otimes \tau_{V,W}) \circ (\alpha_{H,U}
\otimes \alpha_{H,V} \otimes \alpha_{H,W}) \circ \]
	\[ \circ (id_{H} \otimes id_{U} \otimes id_{H} \otimes \tau_{H,V} \otimes id_{W}) \circ (id_{H} \otimes \tau_{H \otimes H,U}
\otimes id_{V} \otimes id_{W}) \circ \]
	\[ \circ \sum_{i,j,k \in I} (-1)^{|s_i||t_i| + |t_j||s_i| + |t_k| ( |s_j| + |s_i| + |t_i| ) + (|s_j| + |s_i|)(|s_k| +
|t_i|)} *   ((s_js_i \otimes s_kt_i \otimes t_kt_j) \otimes ( u \otimes v \otimes w )) = \]
	\[ = (id_{W} \otimes \tau_{U,V}) \circ (\tau_{U,W} \otimes id_{V}) \circ (id_{V} \otimes \tau_{V,W}) \circ (\alpha_{H,U}
\otimes \alpha_{H,V} \otimes \alpha_{H,W}) \circ \]
	\[ \circ (id_{H} \otimes id_{U} \otimes id_{H} \otimes \tau_{H,V} \otimes id_{W}) \circ (id_{H} \otimes \tau_{H \otimes H,U}
\otimes id_{V} \otimes id_{W}) \circ \]
	\[ \circ  (( \sum_{i,j,k \in I} (-1)^{ |s_i||s_k| } s_js_i \otimes s_kt_i \otimes t_kt_j) \otimes ( u \otimes v \otimes w ))
= \]
	\[ = (id_{W} \otimes \tau_{U,V}) \circ (\tau_{U,W} \otimes id_{V}) \circ (id_{V} \otimes \tau_{V,W}) \circ (\alpha_{H,U}
\otimes \alpha_{H,V} \otimes \alpha_{H,W}) \circ \]
	\[ \circ (id_{H} \otimes id_{U} \otimes id_{H} \otimes \tau_{H,V} \otimes id_{W}) \circ (id_{H} \otimes \tau_{H \otimes H,U}
\otimes id_{V} \otimes id_{W}) \circ  (( R_{23} R_{13} R_{12} ) \otimes ( u \otimes v \otimes w )) = \]
	\[ = (id_{W} \otimes \tau_{U,V}) \circ (\tau_{U,W} \otimes id_{V}) \circ (id_{V} \otimes \tau_{V,W}) \circ (\alpha_{H,U}
\otimes \alpha_{H,V} \otimes \alpha_{H,W}) \circ \]
	\[ \circ (id_{H} \otimes id_{U} \otimes id_{H} \otimes \tau_{H,V} \otimes id_{W}) \circ (id_{H} \otimes \tau_{H \otimes H,U}
\otimes id_{V} \otimes id_{W}) \circ  (( R_{12} R_{13} R_{23} ) \otimes ( u \otimes v \otimes w )) = \]
	\[ = (id_{W} \otimes \tau_{U,V}) \circ (\tau_{U,W} \otimes id_{V}) \circ (id_{V} \otimes \tau_{V,W}) \circ (\alpha_{H,U}
\otimes \alpha_{H,V} \otimes \alpha_{H,W}) \circ \]
	\[ \circ (id_{H} \otimes id_{U} \otimes id_{H} \otimes \tau_{H,V} \otimes id_{W}) \circ (id_{H} \otimes \tau_{H \otimes H,U}
\otimes id_{V} \otimes id_{W}) \circ \]
	\[ \circ (( \sum_{i,j,k \in I} (-1)^{ |t_k||s_j| + |t_j||s_i| } s_k s_j \otimes t_k s_i \otimes t_j t_i ) \otimes ( u
\otimes v \otimes w )) = \]
	\[ = (id_{W} \otimes \tau_{U,V}) \circ (\tau_{U,W} \otimes id_{V}) \circ (id_{V} \otimes \tau_{V,W}) \circ \]	
	\[ \circ (\sum_{i,j,k \in I} (-1)^{ |t_k||s_j| + |t_j||s_i| + |u|(|t_k| + |t_j|) + |v|(|t_j| + |t_i|) } \alpha_{H,U} (s_k
s_j \otimes u) \otimes \alpha_{H,V} (t_k s_i \otimes v) \otimes \alpha_{H,W} (t_j t_i \otimes w)) = \]
	\[ = \sum_{i,j,k \in I} (-1)^{ |t_k||s_j| + |t_j||s_i| + |u|(|t_k| + |t_j|) + |v|(|t_j| + |t_i|) + ( |t_k| + |s_i| + |v| )
(|t_j| + |t_i| + |w|) } * \]
	\[ *(-1)^{(|s_k| + |s_j| + |u|)(|t_j| + |t_i| + |w|) + (|t_k| + |s_i| + |v|)(|s_k| + |s_j| + |u|)} \alpha_{H,W} (t_j t_i
\otimes w) \otimes \alpha_{H,V} (t_k s_i \otimes v) \otimes \alpha_{H,U} (s_k s_j \otimes u) = \]
	\[ = \sum_{i,j,k \in I} (-1)^{ |t_i||v| + ( |s_i| + |v| ) (|t_i| + |w|) + |t_j||u| + (|s_j| + |u|)(|t_j| + |t_i| + |w|) +
|t_k|(|s_j| + |u|) + (|s_k| + |s_j| + |u|)(|t_k| + |s_i| + |v|) } * \]
	\[ * \alpha_{H,W} (t_j t_i \otimes w) \otimes \alpha_{H,V} (t_k s_i \otimes v) \otimes \alpha_{H,U} (s_k s_j \otimes u) =
\]
	\[ = (id_{W} \otimes c_{U,V}^{R}) ( \sum_{i \in I} (-1)^{|t_i||v| + (|s_i| + |v|)( |t_i| + |w| ) + |t_j||u| + (|s_j| +
|u|)(|t_j| + |t_i| + |w|)} * \]
	\[ * \alpha_{H,W} (t_jt_i \otimes w) \otimes \alpha_{H,U} (s_j \otimes u) \otimes \alpha_{H,V} (s_i \otimes v) ) = \]
	\[ = (id_{W} \otimes c_{U,V}^{R})(c_{U,W}^{R} \otimes id_{V}) ( \sum_{i \in I} (-1)^{|t_i||v| + (|s_i| + |v|)( |t_i| + |w|
)} u \otimes \alpha_{H,W} (t_i \otimes w) \otimes \alpha_{H,V} (s_i \otimes v) ) = \]
	\[ = (id_{W} \otimes c_{U,V}^{R})(c_{U,W}^{R} \otimes id_{V})(id_{U} \otimes c_{V,W}^{R}) ( u \otimes v \otimes w ). \]	
\end{proof}

\twistedBialgebras*
\begin{proof}		
	1.Verify that the unity axiom holds
	\[ \mu_{X \bowtie A} \circ ( \eta_{X \bowtie A} \otimes id_{X \bowtie A} ) = \mu_{X \bowtie A} \circ ( id_{X \bowtie A}
\otimes \eta_{X \bowtie A} ), \]	
	\[ \mu_{X \bowtie A} \circ ( \eta_{X \bowtie A} \otimes id_{X \bowtie A} ) ( 1_{k} \otimes ( x \otimes a ) ) = \mu_{X
\bowtie A} ( (1_{X} \otimes 1_{A}) \otimes ( x \otimes a ) ) = \]
	\[ = ( \mu_{X} \otimes \mu_{A} ) \circ ( id_{X} \otimes \alpha \otimes \beta \otimes id_{A} ) ( 1_{X} \otimes 1_{A} \otimes
x^{'} \otimes 1_{A} \otimes x^{''} \otimes a ) = \]
	\[ = \sum_{(x)}(-1)^{|x^{'}||1_{A}|}1_{X}(1_{A}\cdot x^{'})\otimes(1_{A})^{x^{''}}a= \]
	\[ =\sum_{(x)} x^{'}\epsilon_{X}(x^{''})\otimes a =( (\nu_{X,k} \circ(id_{X} \otimes \epsilon_{X} )\circ\Delta_{X})\otimes
id_{A})(x\otimes a)=  (id_{X}\otimes id_{A})(x\otimes a)=x\otimes a \]
	for all $x \in X, \; a \in A$. Furthermore,
	
	\[ \mu_{X \bowtie A} \circ ( id_{X \bowtie A} \otimes \eta_{X \bowtie A} ) ( ( x \otimes a ) \otimes 1_{k} ) = \mu_{X
\bowtie A} ( ( x \otimes a ) \otimes (1_{X} \otimes 1_{A}) ) = \]
	\[ = ( \mu_{X} \otimes \mu_{A} ) \circ ( id_{X} \otimes \alpha \otimes \beta \otimes id_{A} ) ( x \otimes a^{'} \otimes
1_{X} \otimes a^{''} \otimes 1_{X} \otimes 1_{A} ) = \]
	\[ = \sum_{(a)} x(a^{'} \cdot 1_{X}) \otimes (a^{''})^{1_{X}} 1_{A} = \sum_{(a)} x\otimes\epsilon_{A}(a^{'})a^{''}= \]
	\[ = (id_{X} \otimes (\nu_{k,A} \circ (\epsilon_{A} \otimes id_{A}) \circ \Delta_{A})) (x\otimes a)=  (id_{X}\otimes
id_{A})(x\otimes a)= x\otimes a \]
	for all $x \in X, \; a \in A$.	
	
	2. Verify that the associativity axiom holds
	\[ (( x \otimes a ) ( y \otimes b )) ( z \otimes c ) = ( \sum_{(a),(y)} (-1)^{ |y^{'}||a^{''}| } x ( a^{'} \cdot y^{'} )
\otimes (a^{''})^{y^{''}} b ) ( z \otimes c ) = \]
	\[ ( \mu_{X} \otimes \mu_{A} ) \circ ( id_{X} \otimes \alpha \otimes \beta \otimes id_{A} ) \circ ( id_{X}  \otimes id_{A}
\otimes \tau_{A,X} \otimes id_{X} \otimes id_{A} ) \circ ( id_{X}  \otimes \Delta_{A} \otimes \Delta_{X} \otimes id_{A} ) \circ
\]
	\[ \circ ( \sum_{(a),(y)} (-1)^{ |y^{'}||a^{''}| } x ( a^{'} \cdot y^{'} ) \otimes (a^{''})^{y^{''}} b \otimes z \otimes c )
= \]
	\[ = ( \mu_{X} \otimes \mu_{A} ) \circ ( id_{X} \otimes \alpha \otimes \beta \otimes id_{A} ) \circ ( id_{X}  \otimes id_{A}
\otimes \tau_{A,X} \otimes id_{X} \otimes id_{A} ) \circ \]
	\[ \circ ( \sum_{(a),(y),(z),(a^{''}),(y^{''}),(b)} (-1)^{ |y^{'}||a^{''}| + |(y^{''})^{'}||(a^{''})^{''}|+
|(a^{''})^{''}||b^{'}| + |(y^{''})^{''}||b^{'}| } * \]
	\[ * x ( a^{'} \cdot y^{'} ) \otimes ((a^{''})^{'})^{(y^{''})^{'}} b^{'} \otimes ((a^{''})^{''})^{(y^{''})^{''}}) b^{''}
\otimes z^{'} \otimes z^{''} \otimes c ) = \]
	\[ = ( \mu_{X} \otimes \mu_{A} ) \circ ( id_{X} \otimes \alpha \otimes \beta \otimes id_{A} ) \circ \]
	\[ \circ ( \sum_{(a),(y),(z),(a^{''}),(y^{''}),(b)} (-1)^{ |y^{'}||a^{''}| + |(y^{''})^{'}||(a^{''})^{''}|+
|(a^{''})^{''}||b^{'}| + |(y^{''})^{''}||b^{'}| + |(a^{''})^{''}||z^{'}| + |(y^{''})^{''}||z^{'}| + |b^{''}||z^{'}| } * \]
	\[ * x ( a^{'} \cdot y^{'} ) \otimes ((a^{''})^{'})^{(y^{''})^{'}} b^{'} \otimes z^{'} \otimes
((a^{''})^{''})^{(y^{''})^{''}}) b^{''} \otimes z^{''} \otimes c ) = \]
	\[ = \sum_{(a),(y),(z),(a^{''}),(y^{''}),(b),(b^{''}),(z^{''})} (-1)^{ |y^{'}||a^{''}| + |(y^{''})^{'}||(a^{''})^{''}|+
|(a^{''})^{''}||b^{'}| + |(y^{''})^{''}||b^{'}| + |(a^{''})^{''}||z^{'}| + |(y^{''})^{''}||z^{'}| + |b^{''}||z^{'}| +
|(b^{''})^{''}||(z^{''})^{'}| } * \]
	\[ * x ( a^{'} \cdot y^{'} ) (((a^{''})^{'})^{(y^{''})^{'}} b^{'}) \cdot z^{'}) \otimes
(((a^{''})^{''})^{(y^{''})^{''}})^{(b^{''})^{'} \cdot (z^{''})^{'})} ( (b^{''})^{''} )^{ (z^{''})^{''} } c = \]
	\[ = ( (\mu_{X} \circ (id_{X} \otimes \alpha)) \otimes \mu_{A} ) \circ ( \mu_{X} \otimes \mu_{A} \otimes  id_{X} \otimes
\beta \otimes \mu_{A} ) \circ \]
	\[ \circ (id_{X} \otimes \alpha \otimes \beta \otimes id_{A} \otimes id_{X} \otimes \beta \otimes \alpha \otimes \beta
\otimes id_{A}) \circ \]
	\[ \circ ( id_{X} \otimes id_{A} \otimes id_{X} \otimes id_{A} \otimes id_{X} \otimes id_{A} \otimes id_{X} \otimes id_{A}
\otimes id_{X} \otimes id_{A} \otimes \tau_{A,X} \otimes id_{X} \otimes id_{A} ) \circ \]
	\[ \circ ( id_{X} \otimes id_{A} \otimes id_{X} \otimes id_{A} \otimes id_{X} \otimes id_{A} \otimes \tau_{A \otimes X
\otimes A \otimes A, X} \otimes id_{X} \otimes id_{X} \otimes id_{A} ) \circ \]
	\[ \circ ( id_{X} \otimes id_{A} \otimes id_{X} \otimes id_{A} \otimes id_{X} \otimes \tau_{A \otimes X, A} \otimes id_{A}
\otimes id_{A} \otimes id_{X} \otimes id_{X} \otimes id_{X} \otimes id_{A} ) \circ \]
	\[ \circ ( id_{X} \otimes id_{A} \otimes id_{X} \otimes id_{A} \otimes \tau_{A,X} \otimes id_{X} \otimes id_{A} \otimes
id_{A} \otimes id_{A} \otimes id_{X} \otimes id_{X} \otimes id_{X} \otimes id_{A} ) \circ \]
	\[ \circ ( id_{X} \otimes id_{A} \otimes \tau_{A \otimes A, X} \otimes id_{X} \otimes id_{X} \otimes id_{A} \otimes id_{A}
\otimes id_{A} \otimes id_{X} \otimes id_{X} \otimes id_{X} \otimes id_{A} ) \circ \]
	\[ \circ ( id_{X} \otimes ( ( id_{A} \otimes \Delta_{A} ) \circ \Delta_{A} ) \otimes ( ( id_{X} \otimes \Delta_{X} ) \circ
\Delta_{X} ) \otimes ( ( id_{A} \otimes \Delta_{A} ) \circ \Delta_{A} ) \otimes ( ( id_{X} \otimes \Delta_{X} ) \circ \Delta_{X}
) \otimes id_{A} ) \circ \]
	\[ \circ (x \otimes a \otimes y \otimes b \otimes z \otimes c) = \]
	\[ = ( (\mu_{X} \circ (id_{X} \otimes \alpha)) \otimes \mu_{A} ) \circ ( \mu_{X} \otimes \mu_{A} \otimes  id_{X} \otimes
\beta \otimes \mu_{A} ) \circ \]
	\[ \circ (id_{X} \otimes \alpha \otimes \beta \otimes id_{A} \otimes id_{X} \otimes \beta \otimes \alpha \otimes \beta
\otimes id_{A}) \circ \]
	\[ \circ ( id_{X} \otimes id_{A} \otimes id_{X} \otimes id_{A} \otimes id_{X} \otimes id_{A} \otimes id_{X} \otimes id_{A}
\otimes id_{X} \otimes id_{A} \otimes \tau_{A,X} \otimes id_{X} \otimes id_{A} ) \circ \]
	\[ \circ ( id_{X} \otimes id_{A} \otimes id_{X} \otimes id_{A} \otimes id_{X} \otimes id_{A} \otimes \tau_{A \otimes X
\otimes A \otimes A, X} \otimes id_{X} \otimes id_{X} \otimes id_{A} ) \circ \]
	\[ \circ ( id_{X} \otimes id_{A} \otimes id_{X} \otimes id_{A} \otimes id_{X} \otimes \tau_{A \otimes X, A} \otimes id_{A}
\otimes id_{A} \otimes id_{X} \otimes id_{X} \otimes id_{X} \otimes id_{A} ) \circ \]
	\[ \circ ( id_{X} \otimes id_{A} \otimes id_{X} \otimes id_{A} \otimes \tau_{A,X} \otimes id_{X} \otimes id_{A} \otimes
id_{A} \otimes id_{A} \otimes id_{X} \otimes id_{X} \otimes id_{X} \otimes id_{A} ) \circ \]
	\[ \circ ( id_{X} \otimes id_{A} \otimes \tau_{A \otimes A, X} \otimes id_{X} \otimes id_{X} \otimes id_{A} \otimes id_{A}
\otimes id_{A} \otimes id_{X} \otimes id_{X} \otimes id_{X} \otimes id_{A} ) \circ \]
	\[ \circ ( id_{X} \otimes ( ( \Delta_{A} \otimes id_{A} ) \circ \Delta_{A} ) \otimes ( ( \Delta_{X} \otimes id_{X} ) \circ
\Delta_{X} ) \otimes ( ( \Delta_{A} \otimes id_{A} ) \circ \Delta_{A} ) \otimes ( ( \Delta_{X} \otimes id_{X} ) \circ \Delta_{X}
) \otimes id_{A} ) \circ \]
	\[ \circ (x \otimes a \otimes y \otimes b \otimes z \otimes c) = \]
	\[ = \sum_{(b),(z),(a),(y),(b^{'}),(z^{'}),(a^{'}),(y^{'})} (-1)^{|(y^{'})^{'}| ( |(a^{'})^{''}| + |a^{''}| ) + |a^{''}|
|(y^{'})^{''}| + |(b^{'})^{'}| ( |a^{''}| + |y^{''}| ) + |(z^{'})^{'}| ( |a^{''}| + |y^{''}| + |(b^{'})^{''}| + |b^{''}| ) +
|(z^{'})^{''}| |b^{''}| } * \]
	\[ * x  ((a^{'})^{'} \cdot (y^{'})^{'}) ((((a^{'})^{''})^{(y^{'})^{''}} (b^{'})^{'}) \cdot (z^{'})^{'}) \otimes
((a^{''})^{y^{''}})^{(b^{'})^{''} \cdot (z^{'})^{''}} (b^{''})^{z^{''}}  c = \]
	\[ \sum_{(b),(z),(a),(y),(b^{'}),(z^{'}),(a^{'}),(y^{'})} (-1)^{|(y^{'})^{'}| |(a^{'})^{''}| + |a^{''}| |y^{'}| +
|(b^{'})^{'}| ( |a^{''}| + |y^{''}| ) + |(z^{'})^{'}| ( |a^{''}| + |y^{''}| + |(b^{'})^{''}| ) + |z^{'}| |b^{''}| } * \]
	\[ * x  ((a^{'})^{'} \cdot (y^{'})^{'}) ((((a^{'})^{''})^{(y^{'})^{''}} (b^{'})^{'}) \cdot (z^{'})^{'}) \otimes (a^{''})^{
y^{''} ((b^{'})^{''} \cdot (z^{'})^{''})} (b^{''})^{z^{''}} c \]
	for all $x,y,z \in X, \; a,b,c \in A$.
	
	Remark.
	
	2.1 We have by Definition \ref{df:twistedB}
	\[ \Delta_{A} ( (a^{''})^{y^{''}} b ) = \Delta_{A} ( (a^{''})^{y^{''}} ) \Delta_{A} ( b ) =  (\sum_{(a^{''}),(y^{''})}
(-1)^{|(y^{''})^{'}||(a^{''})^{''}|} ((a^{''})^{'})^{(y^{''})^{'}} \otimes ((a^{''})^{''})^{(y^{''})^{''}}) * \]
	\[ * ( \sum_{(b)} b^{'} \otimes b^{''} ) = \sum_{(a^{''}),(y^{''}),(b)} (-1)^{|(y^{''})^{'}||(a^{''})^{''}| +
|((a^{''})^{''})^{(y^{''})^{''}}||b^{'}|} ((a^{''})^{'})^{(y^{''})^{'}} b^{'} \otimes ((a^{''})^{''})^{(y^{''})^{''}}) b^{''} =
\]
	\[ = \sum_{(a^{''}),(y^{''}),(b)} (-1)^{|(y^{''})^{'}||(a^{''})^{''}|+ |(a^{''})^{''}||b^{'}| + |(y^{''})^{''}||b^{'}|}
((a^{''})^{'})^{(y^{''})^{'}} b^{'} \otimes ((a^{''})^{''})^{(y^{''})^{''}}) b^{''}. \]
	
	2.2 It also follows from Definition \ref{df:twistedB}
	\[ \beta( (a^{''})^{''})^{(y^{''})^{''}} b^{''} \otimes z^{''} ) = ((a^{''})^{''})^{(y^{''})^{''}} b^{''})^{z^{''}} = \]
	\[ = \sum_{(b^{''}),(z^{''})} (-1)^{ |(b^{''})^{''}||(z^{''})^{'}| }  ((a^{''})^{''})^{(y^{''})^{''}})^{(b^{''})^{'} \cdot
(z^{''})^{'}} ( (b^{''})^{''} )^{ (z^{''})^{''} }. \qed \]	
	
	Next we have
	\[ ( x \otimes a ) (( y \otimes b ) ( z \otimes c )) =  ( x \otimes a ) ( \sum_{(b),(z)} (-1)^{ |z^{'}||b^{''}| } y ( b^{'}
\cdot z^{'} ) \otimes (b^{''})^{z^{''}} c ) = \]
	\[ = ( \mu_{X} \otimes \mu_{A} ) \circ ( id_{X} \otimes \alpha \otimes \beta \otimes id_{A} ) \circ ( id_{X}  \otimes id_{A}
\otimes \tau_{A,X} \otimes id_{X} \otimes id_{A} ) \circ ( id_{X}  \otimes \Delta_{A} \otimes \Delta_{X} \otimes id_{A} ) \circ
\]
	\[ \circ ( \sum_{(b),(z)} (-1)^{ |z^{'}||b^{''}| } x \otimes a \otimes y ( b^{'} \cdot z^{'} ) \otimes (b^{''})^{z^{''}} c )
= \]
	\[ = ( \mu_{X} \otimes \mu_{A} ) \circ ( id_{X} \otimes \alpha \otimes \beta \otimes id_{A} ) \circ ( id_{X}  \otimes id_{A}
\otimes \tau_{A,X} \otimes id_{X} \otimes id_{A} ) \circ \]
	
	\[ \circ ( \sum_{(b),(z),(a),(y),(b^{'}),(z^{'})} (-1)^{ |z^{'}||b^{''}| + |(z^{'})^{'}||(b^{'})^{''}| + |y^{''}|
|(b^{'})^{'} | + |y^{''}| |(z^{'})^{'}| } * \]
	
	\[ * x \otimes a^{'} \otimes a^{''} \otimes y^{'} ((b^{'})^{'} \cdot (z^{'})^{'}) \otimes y^{''} ( (b^{'})^{''} \cdot
(z^{'})^{''} ) \otimes (b^{''})^{z^{''}} c ) = \]
	
	\[ = ( \mu_{X} \otimes \mu_{A} ) \circ ( id_{X} \otimes \alpha \otimes \beta \otimes id_{A} ) \circ \]
	
	\[ \circ ( \sum_{(b),(z),(a),(y),(b^{'}),(z^{'})} (-1)^{ |z^{'}||b^{''}| + |(z^{'})^{'}||(b^{'})^{''}| + |y^{''}|
|(b^{'})^{'} | + |y^{''}| |(z^{'})^{'}| + | a^{''} | | y^{'} | + | a^{''} | |(b^{'})^{'}| + | a^{''} | |(z^{'})^{'} | } * \]
	\[ * x \otimes a^{'} \otimes y^{'} ((b^{'})^{'} \cdot (z^{'})^{'}) \otimes a^{''} \otimes y^{''} ( (b^{'})^{''} \cdot
(z^{'})^{''} ) \otimes (b^{''})^{z^{''}} c ) = \]
	
	\[ = \sum_{(b),(z),(a),(y),(b^{'}),(z^{'}),(a^{'}),(y^{'})} (-1)^{ |z^{'}||b^{''}| + |(z^{'})^{'}||(b^{'})^{''}| + |y^{''}|
|(b^{'})^{'} | + |y^{''}| |(z^{'})^{'}| + | a^{''} | | y^{'} | + | a^{''} | |(b^{'})^{'}| + | a^{''} | |(z^{'})^{'} | +
|(y^{'})^{'}| |(a^{'})^{''}| } * \]
	
	\[ * x ( (a^{'})^{'} \cdot (y^{'})^{'} ) ( (((a^{'})^{''})^{ (y^{'})^{''} } (b^{'})^{'}) \cdot (z^{'})^{'} ) \otimes (
a^{''} ) ^ { y^{''} ( (b^{'}) ^ {''} \cdot (z^{'}) ^ {''} ) } (b^{''})^{z^{''}} c \]
	
	for all $x,y,z \in X, \; a,b,c \in A$.
	
	Remark.
	
	2.3 We have by the definition of a Hopf superalgebra morphism and by Definition \ref{df:twistedB}
	\[ \Delta_{X} ( y ( b^{'} \cdot z^{'} ) ) = \Delta_{X} (y) \Delta_{X} ( b^{'} \cdot z^{'} ) = \]
	\[ = ( \sum_{(y)} y^{'} \otimes y^{''} ) ( \sum_{(b^{'}),(z^{'})} (-1)^{|(z^{'})^{'}||(b^{'})^{''}|} (b^{'})^{'} \cdot
(z^{'})^{'} \otimes (b^{'})^{''} \cdot (z^{'})^{''} = \]
	\[ = \sum_{(y),(b^{'}),(z^{'})} (-1)^{|(z^{'})^{'}||(b^{'})^{''}| + |y^{''}| |(b^{'})^{'} \cdot (z^{'})^{'}| } y^{'}
((b^{'})^{'} \cdot (z^{'})^{'}) \otimes y^{''} ( (b^{'})^{''} \cdot (z^{'})^{''} ) = \]
	\[ = \sum_{(y),(b^{'}),(z^{'})} (-1)^{|(z^{'})^{'}||(b^{'})^{''}| + |y^{''}| |(b^{'})^{'} | + |y^{''}| |(z^{'})^{'}| } y^{'}
((b^{'})^{'} \cdot (z^{'})^{'}) \otimes y^{''} ( (b^{'})^{''} \cdot (z^{'})^{''} ). \]
	
	2.4 Notice that
	\[ ( ( a^{''} ) ^ { y^{''} } ) ^ { (b^{'}) ^ {''} \cdot (z^{'}) ^ {''} } = ( a^{''} ) ^ { y^{''} ( (b^{'}) ^ {''} \cdot
(z^{'}) ^ {''} ) } , \]
	since $\beta$ endows $A$ with a structure of a right $X$-module.
	
	2.5 We also have
	\[ \alpha ( a^{'} \otimes y^{'} ( (b^{'})^{'} \cdot (z^{'})^{'} ) ) = \]
	\[ = \sum_{(a^{'}),(y^{'})} (-1)^{ |(y^{'})^{'}| |(a^{'})^{''}| } ( (a^{'})^{'} \cdot (y^{'})^{'} ) ( ((a^{'})^{''})^{
(y^{'})^{''} } \cdot ( (b^{'})^{'} \cdot (z^{'})^{'} ) ) = \]
	\[ = \sum_{(a^{'}),(y^{'})} (-1)^{ |(y^{'})^{'}| |(a^{'})^{''}| } ( (a^{'})^{'} \cdot (y^{'})^{'} ) ( (( (a^{'})^{''} )^{
(y^{'})^{''} } (b^{'})^{'}) \cdot (z^{'})^{'} ), \]
	\[ ((a^{'})^{''})^{ (y^{'})^{''} } \cdot ( (b^{'})^{'} \cdot (z^{'})^{'} ) = (((a^{'})^{''})^{ (y^{'})^{''} } (b^{'})^{'})
\cdot (z^{'})^{'}, \]
	since $\alpha$ endows $X$ with a structure of a left $A$-module. $\qed$

	3. We deduce from the relations mentioned above that the bicrossed product is a supercoalgebra as it is the tensor product
of supercoalgebras $X$ and $A$.
	
	4. Verify that the counit is a superalgebra morphism. We have for all $x,y \in X, \; a,b \in A$
	\[ \epsilon_{X \bowtie A} ( x \otimes a ) \epsilon_{X \bowtie A} ( y \otimes b ) = \epsilon_{X}(x) \epsilon_{A}(a)
\epsilon_{X}(y) \epsilon_{A}(b), \]
	\[ \epsilon_{X \bowtie A} ( \mu_{X \bowtie A} ( (x \otimes a) \otimes (y \otimes b) ) ) = \]
	\[ = \epsilon_{X \bowtie A} ( \sum_{(a),(y)} (-1)^{ |y^{'}||a^{''}| } x ( a^{'} \cdot y^{'} ) \otimes (a^{''})^{y^{''}} b )
= \sum_{(a),(y)} (-1)^{ |y^{'}||a^{''}| } \epsilon_{X}(x) \epsilon_{X}(a^{'} \cdot y^{'}) \epsilon_{A}((a^{''})^{y^{''}})
\epsilon_{A}(b) = \]
	\[ = \sum_{(a),(y)} (-1)^{ |y^{'}||a^{''}| } \epsilon_{X}(x) \epsilon_{A}(a^{'}) \epsilon_{X}(y^{'}) \epsilon_{A}(a^{''})
\epsilon_{X}(y^{''}) \epsilon_{A}(b) = \]
	\[ = \epsilon_{X}(x) \epsilon_{A}(b) \sum_{(a),(y)} (-1)^{ |y^{'}||a^{''}| } \epsilon_{A}(a^{'}) \epsilon_{X}(y^{'})
\epsilon_{A}(a^{''}) \epsilon_{X}(y^{''}) = \]
	\[ = \epsilon_{X}(x) \epsilon_{A}(b) \sum_{(a),(y)} \epsilon_{A}(a^{'} \epsilon_{A}(a^{''}))
\epsilon_{X}(y^{'}\epsilon_{X}(y^{''})) =  \epsilon_{X}(x) \epsilon_{A}(b) \epsilon_{A}(a) \epsilon_{X}(y). \]
	
	\[ \epsilon_{X \bowtie A} \circ \eta_{X \bowtie A} = id_{k}, \]
	\[ \epsilon_{X \bowtie A} \circ \eta_{X \bowtie A} (1_{k}) = \epsilon_{X \bowtie A} ( 1_{X} \otimes 1_{A} ) =
\epsilon_{X}(1_{X}) \epsilon_{A}(1_{A}) = 1_{k} = id_{k}(1_k). \]
	
	5. Verify that the comultiplication is a superalgebra morphism. We have for all $x,y \in X, \; a,b \in A$:
	\[ \mu_{X \bowtie A \otimes X \bowtie A} ( \Delta_{X \bowtie A}(x \otimes a) \otimes \Delta_{X \bowtie A} (y \otimes b) ) =
\]
	\[ = ( \sum_{(a),(x)} (-1)^{ |a^{'}| |x^{''}| } (x^{'} \otimes a^{'}) \otimes ( x^{''} \otimes a^{''} ) ) ( \sum_{(y),(b)}
(-1)^{ |b^{'}| |y^{''}| } (y^{'} \otimes b^{'}) \otimes ( y^{''} \otimes b^{''} ) ) = \]
	\[ = \sum_{(a),(x),(y),(b),(y^{'}),(y^{''}),(a^{'}),(a^{''})} (-1)^{ |a^{'}| |x^{''}| + |b^{'}| |y^{''}| +
|(y^{'})^{'}||(a^{'})^{''}| + |(y^{''})^{'}||(a^{''})^{''}| + (|x^{''}|+|a^{''}|)(|y^{'}| + |b^{'}| ) } * \]
	\[ * x^{'} ( (a^{'})^{'} \cdot (y^{'})^{'} ) \otimes ((a^{'})^{''})^{(y^{'})^{''}} b^{'} \otimes x^{''} ( (a^{''})^{'} \cdot
(y^{''})^{'} ) \otimes ((a^{''})^{''})^{(y^{''})^{''}} b^{''},\]
	
	\[ \Delta_{X \bowtie A} ( \mu_{X \bowtie A} ((x \otimes a) \otimes (y \otimes b) )) = \Delta_{X \bowtie A} ( \sum_{(a),(y)}
(-1)^{ |y^{'}||a^{''}| } x ( a^{'} \cdot y^{'} ) \otimes (a^{''})^{y^{''}} b ) = \]
	\[ = ( id_{X} \otimes \tau_{X,A} \otimes id_{A} ) \circ  ( \sum_{(a),(y)} (-1)^{ |y^{'}||a^{''}| } \Delta_{X} ( x ( a^{'}
\cdot y^{'} ) ) \otimes \Delta_{A} ( (a^{''})^{y^{''}} b ) ) ) = \]
	\[ = ( id_{X} \otimes \tau_{X,A} \otimes id_{A} ) \circ \]
	\[ \circ ( \sum_{(a),(y),(x),(a^{'}),(y^{'}),(a^{''}),(y^{''}),(b)} (-1)^{ |y^{'}||a^{''}| + |(y^{'})^{'}||(a^{'})^{''}| +
|x^{''}| ( |(a^{'})^{'}| + |(y^{'})^{'}| ) + |(y^{''})^{'}||(a^{''})^{''}| + |b^{'}|(|(a^{''})^{''}| + |(y^{''})^{''}|) } * \]
	\[ * x^{'} ((a^{'})^{'} \cdot (y^{'})^{'}) \otimes x^{''} ( (a^{'})^{''} \cdot (y^{'})^{''} ) ) \otimes
((a^{''})^{'})^{(y^{''})^{'}} b^{'} \otimes ((a^{''})^{''})^{(y^{''})^{''}} b^{''} ) = \]
	\[ = \sum_{(a),(y),(x),(a^{'}),(y^{'}),(a^{''}),(y^{''}),(b)} (-1)^{ |y^{'}||a^{''}| + |(y^{'})^{'}||(a^{'})^{''}| +
|x^{''}| ( |(a^{'})^{'}| + |(y^{'})^{'}| ) + |(y^{''})^{'}||(a^{''})^{''}| + |b^{'}|(|(a^{''})^{''}| + |(y^{''})^{''}|) } * \]
	\[ * (-1)^{ ( |x^{''}| + |(a^{'})^{''}| + |(y^{'})^{''}| )( |(a^{''})^{'}| + |(y^{''})^{'}| + |b^{'}| ) } x^{'} ((a^{'})^{'}
\cdot (y^{'})^{'}) \otimes ((a^{''})^{'})^{(y^{''})^{'}} b^{'} \otimes x^{''} ( (a^{'})^{''} \cdot (y^{'})^{''} ) \otimes
((a^{''})^{''})^{(y^{''})^{''}} b^{''} = \]
	\[ = \sum_{(x),(a),(a^{'}),(a^{''}),(y),(y^{'}),(y^{''}),(b)} (-1)^{ |x^{''}||(a^{'})^{'}| + |(y^{'})^{'}|( |x^{''}| +
|(a^{'})^{''}| + |a^{''}| ) + |(a^{''})^{'}| ( |x^{''}| + |(a^{'})^{''}| ) } * \]
	\[ * (-1)^{ |(y^{''})^{'}| ( |x^{''}| + |(a^{'})^{''}| + |(a^{''})^{''}| + |(y^{'})^{''}| ) + |b^{'}| ( |x^{''}| +
|(a^{'})^{''}| + |(a^{''})^{''}| + |(y^{'})^{''}| + |(y^{''})^{''}| ) + |(y^{'})^{''}||(a^{''})^{''}| } * \]
	\[ * x^{'} ((a^{'})^{'} \cdot (y^{'})^{'}) \otimes ((a^{''})^{'})^{(y^{''})^{'}} b^{'} \otimes x^{''} ( (a^{'})^{''} \cdot
(y^{'})^{''} ) \otimes ((a^{''})^{''})^{(y^{''})^{''}} b^{''} = \]
	\[ = ((\mu_{X} \circ ( id_{X} \otimes \alpha )) \otimes (\mu_{A} \circ ( \beta \otimes id_{A} )) \otimes (\mu_{X} \circ (
id_{X} \otimes \alpha )) \otimes (\mu_{A} \circ ( \beta \otimes id_{A} ))) \circ \]
	\[ \circ (id_{X} \otimes id_{A} \otimes id_{X} \otimes id_{A} \otimes id_{X} \otimes id_{A} \otimes id_{X} \otimes id_{A}
\otimes \tau_{A,X} \otimes id_{X} \otimes id_{A}) \circ \]
	\[ \circ (id_{X} \otimes id_{A} \otimes id_{X} \otimes id_{A} \otimes id_{X} \otimes \tau_{X \otimes A \otimes A \otimes X
\otimes X, A} \otimes id_{A}) \circ \]
	\[ \circ (id_{X} \otimes id_{A} \otimes id_{X} \otimes id_{A} \otimes \tau_{X \otimes A \otimes A \otimes X, X} \otimes
id_{X} \otimes id_{A} \otimes id_{A}) \circ \]
	\[ \circ (id_{X} \otimes id_{A} \otimes id_{X} \otimes \tau_{X \otimes A, A} \otimes id_{A} \otimes id_{X} \otimes id_{X}
\otimes id_{X} \otimes id_{A} \otimes id_{A}) \circ \]
	\[ \circ (id_{X} \otimes id_{A} \otimes \tau_{X \otimes A \otimes A \otimes A, X} \otimes id_{X} \otimes id_{X} \otimes
id_{X} \otimes id_{A} \otimes id_{A} ) \circ \]
	\[ \circ ( id_{X} \otimes \tau_{X,A} \otimes id_{A} \otimes id_{A} \otimes id_{X} \otimes id_{X} \otimes id_{X} \otimes
id_{X} \otimes id_{A} \otimes id_{A} ) \circ \]
	\[ \circ (\Delta_{X} \otimes (( \Delta_{A} \otimes \Delta_{A} ) \circ \Delta_{A}) \otimes (( \Delta_{X} \otimes \Delta_{X} )
\circ \Delta_{X}) \otimes \Delta_{A}) (x \otimes a \otimes y \otimes b) = \]
	\[ = ((\mu_{X} \circ ( id_{X} \otimes \alpha )) \otimes (\mu_{A} \circ ( \beta \otimes id_{A} )) \otimes (\mu_{X} \circ (
id_{X} \otimes \alpha )) \otimes (\mu_{A} \circ ( \beta \otimes id_{A} ))) \circ \]
	\[ \circ (id_{X} \otimes id_{A} \otimes id_{X} \otimes id_{A} \otimes id_{X} \otimes id_{A} \otimes id_{X} \otimes id_{A}
\otimes \tau_{A,X} \otimes id_{X} \otimes id_{A}) \circ \]
	\[ \circ (id_{X} \otimes id_{A} \otimes id_{X} \otimes id_{A} \otimes id_{X} \otimes \tau_{X \otimes A \otimes A \otimes X
\otimes X, A} \otimes id_{A}) \circ \]
	\[ \circ (id_{X} \otimes id_{A} \otimes id_{X} \otimes id_{A} \otimes \tau_{X \otimes A \otimes A \otimes X, X} \otimes
id_{X} \otimes id_{A} \otimes id_{A}) \circ \]
	\[ \circ (id_{X} \otimes id_{A} \otimes id_{X} \otimes \tau_{X \otimes A, A} \otimes id_{A} \otimes id_{X} \otimes id_{X}
\otimes id_{X} \otimes id_{A} \otimes id_{A}) \circ \]
	\[ \circ (id_{X} \otimes id_{A} \otimes \tau_{X \otimes A \otimes A \otimes A, X} \otimes id_{X} \otimes id_{X} \otimes
id_{X} \otimes id_{A} \otimes id_{A} ) \circ \]
	\[ \circ ( id_{X} \otimes \tau_{X,A} \otimes id_{A} \otimes id_{A} \otimes id_{X} \otimes id_{X} \otimes id_{X} \otimes
id_{X} \otimes id_{A} \otimes id_{A} ) \circ \]
	\[ \circ (\Delta_{X} \otimes ( ( id_{A} \otimes \Delta_{A} \otimes id_{A} ) \circ ( id_{A} \otimes \Delta_{A} ) \circ
\Delta_{A}) \otimes ( ( id_{X} \otimes \Delta_{X} \otimes id_{X} ) \circ ( id_{X} \otimes \Delta_{X} ) \circ \Delta_{X}) \otimes
\Delta_{A}) (x \otimes a \otimes y \otimes b) = \]
	\[ = \sum_{(a),(x),(y),(b),(y^{''}),((y^{''})^{'}),(a^{''}),((a^{''})^{'})} (-1)^{ |a^{'}| |x^{''}| + |y^{'}| ( |x^{''}| +
|((a^{''})^{'})^{'}| + |((a^{''})^{'})^{''}| + |(a^{''})^{''}| ) + |((a^{''})^{'})^{''}| ( |x^{''}| + |((a^{''})^{'})^{'}| ) } *
\]
	\[ * (-1)^{ |((y^{''})^{'})^{''}| ( |x^{''}| + |((a^{''})^{'})^{'}| + |(a^{''})^{''}| + |((y^{''})^{'})^{'}| ) + |b^{'}| (
|x^{''}| + |((a^{''})^{'})^{'}| + |(a^{''})^{''}| + |((y^{''})^{'})^{'}| + |(y^{''})^{''}| ) + |((y^{''})^{'})^{'}|
|(a^{''})^{''}| } * \]
	\[ * x^{'} (a^{'} \cdot y^{'}) \otimes (((a^{''})^{'})^{''})^{((y^{''})^{'})^{''}} b^{'} \otimes x^{''} (((a^{''})^{'})^{'}
\cdot ((y^{''})^{'})^{'}) \otimes ((a^{''})^{''})^{(y^{''})^{''}} b^{''} = \]
	\[ = ( id_{X} \otimes \mu_{A} \otimes \mu_{X} \otimes id_{A} ) \circ ( id_{X} \otimes id_{A} \otimes \tau_{X \otimes X, A}
\otimes id_{A} ) \circ ( id_{X} \otimes \tau_{X,A} \otimes id_{X} \otimes id_{A} \otimes id_{A} ) \circ \]
	\[ \circ ( \sum_{(a),(x),(y),(b),(y^{''}),(a^{''})} (-1)^{ |a^{'}| |x^{''}| + |y^{'}| ( |x^{''}| + |a^{''}| ) +
|(y^{''})^{'}| |(a^{''})^{''}| + |b^{'}| ( |(a^{''})^{''}| + |(y^{''})^{''}| ) } * \]
	\[ * \sum_{((y^{''})^{'}),((a^{''})^{'})} (-1)^{ |((a^{''})^{'})^{'}| |((a^{''})^{'})^{''}| + |((y^{''})^{'})^{'}|
|((y^{''})^{'})^{''}| + |((y^{''})^{'})^{''}| |((a^{''})^{'})^{'}| } * \]
	\[ * x^{'} (a^{'} \cdot y^{'}) \otimes x^{''} \otimes (((a^{''})^{'})^{''})^{((y^{''})^{'})^{''}} \otimes ((a^{''})^{'})^{'}
\cdot ((y^{''})^{'})^{'} \otimes b^{'} \otimes ((a^{''})^{''})^{(y^{''})^{''}} b^{''} ) = \]
	\[ = ( id_{X} \otimes \mu_{A} \otimes \mu_{X} \otimes id_{A} ) \circ ( id_{X} \otimes id_{A} \otimes \tau_{X \otimes X, A}
\otimes id_{A} ) \circ ( id_{X} \otimes \tau_{X,A} \otimes id_{X} \otimes id_{A} \otimes id_{A} ) \circ \]
	\[ ( \sum_{(a),(x),(y),(b),(y^{''}),(a^{''})} (-1)^{ |a^{'}| |x^{''}| + |y^{'}| ( |x^{''}| + |a^{''}| ) + |(y^{''})^{'}|
|(a^{''})^{''}| + |b^{'}| ( |(a^{''})^{''}| + |(y^{''})^{''}| ) } * \]
	\[ * \sum_{((y^{''})^{'}),((a^{''})^{'})} (-1)^{ |((a^{''})^{'})^{''}| |((y^{''})^{'})^{'}| } * \]
	\[ * x^{'} (a^{'} \cdot y^{'}) \otimes x^{''} \otimes (((a^{''})^{'})^{'})^{((y^{''})^{'})^{'}} \otimes ((a^{''})^{'})^{''}
\cdot ((y^{''})^{'})^{''} \otimes b^{'} \otimes ((a^{''})^{''})^{(y^{''})^{''}} b^{''} ) = \]
	\[ = \sum_{(a),(x),(y),(b),(y^{''}),(a^{''})} (-1)^{ |a^{'}| |x^{''}| + |y^{'}| ( |x^{''}| + |a^{''}| ) + |(y^{''})^{'}|
|(a^{''})^{''}| + |b^{'}| ( |(a^{''})^{''}| + |(y^{''})^{''}| ) } * \]
	\[ * \sum_{((y^{''})^{'}),((a^{''})^{'})} (-1)^{ |((a^{''})^{'})^{''}| |((y^{''})^{'})^{'}| + |x^{''}| (
|(((a^{''})^{'})^{'})| + |((y^{''})^{'})^{'}| ) + |b^{'}| ( |x^{''}| + |((a^{''})^{'})^{''}| + |((y^{''})^{'})^{''}| ) } * \]
	\[ * x^{'} (a^{'} \cdot y^{'}) \otimes (((a^{''})^{'})^{'})^{((y^{''})^{'})^{'}} b^{'} \otimes x^{''} (((a^{''})^{'})^{''}
\cdot ((y^{''})^{'})^{''}) \otimes ((a^{''})^{''})^{(y^{''})^{''}} b^{''} = \]
	\[ = \sum_{(x),(a),(a^{''}),((a^{''})^{'}),(y),(y^{''}),((y^{''})^{'}),(b)} (-1)^{|a^{'}||x^{''}| + |y^{'}|( |x^{''}| +
|a^{''}| ) + |((a^{''})^{'})^{'}||x^{''}| + |((y^{''})^{'})^{'}|( |x^{''}| + |((a^{''})^{'})^{''}| + |(a^{''})^{''}| ) } * \]
	\[ * (-1)^{|b^{'}|( |x^{''}| + |((a^{''})^{'})^{''}| + |(a^{''})^{''}| + |((y^{''})^{'})^{''}| +  |(y^{''})^{''}| ) +
|((y^{''})^{'})^{''}||(a^{''})^{''}| } * \]
	\[ * x^{'} (a^{'} \cdot y^{'}) \otimes (((a^{''})^{'})^{'})^{((y^{''})^{'})^{'}} b^{'} \otimes x^{''} (((a^{''})^{'})^{''}
\cdot ((y^{''})^{'})^{''}) \otimes ((a^{''})^{''})^{(y^{''})^{''}} b^{''} = \]
	\[ = ((\mu_{X} \circ ( id_{X} \otimes \alpha )) \otimes (\mu_{A} \circ ( \beta \otimes id_{A} )) \otimes (\mu_{X} \circ (
id_{X} \otimes \alpha )) \otimes (\mu_{A} \circ ( \beta \otimes id_{A} ))) \circ \]
	\[ \circ (id_{X} \otimes id_{A} \otimes id_{X} \otimes id_{A} \otimes id_{X} \otimes id_{A} \otimes id_{X} \otimes id_{A}
\otimes \tau_{A,X} \otimes id_{X} \otimes id_{A}) \circ \]
	\[ \circ (id_{X} \otimes id_{A} \otimes id_{X} \otimes id_{A} \otimes id_{X} \otimes \tau_{X \otimes A \otimes A \otimes X
\otimes X,A} \otimes id_{A}) \circ \]
	\[ \circ (id_{X} \otimes id_{A} \otimes id_{X} \otimes id_{A} \otimes \tau_{X \otimes A \otimes A, X} \otimes id_{X} \otimes
id_{X} \otimes id_{A} \otimes id_{A}) \circ \]
	\[ \circ (id_{X} \otimes id_{A} \otimes id_{X} \otimes \tau_{X,A} \otimes id_{A} \otimes id_{A} \otimes id_{X} \otimes
id_{X} \otimes id_{X} \otimes id_{A} \otimes id_{A}) \circ \]
	\[ \circ (id_{X} \otimes id_{A} \otimes \tau_{X \otimes A \otimes A \otimes A, X} \otimes id_{X} \otimes id_{X} \otimes
id_{X} \otimes id_{A} \otimes id_{A} ) \circ \]
	\[ \circ ( id_{X} \otimes \tau_{X,A} \otimes id_{A} \otimes id_{A} \otimes id_{X} \otimes id_{X} \otimes id_{X} \otimes
id_{X} \otimes id_{A} \otimes id_{A} ) \circ \]
	\[ \circ (\Delta_{X} \otimes ( ( id_{A} \otimes \Delta_{A} \otimes id_{A} ) \circ ( id_{A} \otimes \Delta_{A} ) \circ
\Delta_{A}) \otimes ( ( id_{X} \otimes \Delta_{X} \otimes id_{X} ) \circ ( id_{X} \otimes \Delta_{X} ) \circ \Delta_{X}) \otimes
\Delta_{A}) (x \otimes a \otimes y \otimes b) = \]
	\[ = ((\mu_{X} \circ ( id_{X} \otimes \alpha )) \otimes (\mu_{A} \circ ( \beta \otimes id_{A} )) \otimes (\mu_{X} \circ (
id_{X} \otimes \alpha )) \otimes (\mu_{A} \circ ( \beta \otimes id_{A} ))) \circ \]
	\[ \circ (id_{X} \otimes id_{A} \otimes id_{X} \otimes id_{A} \otimes id_{X} \otimes id_{A} \otimes id_{X} \otimes id_{A}
\otimes \tau_{A,X} \otimes id_{X} \otimes id_{A}) \circ \]
	\[ \circ (id_{X} \otimes id_{A} \otimes id_{X} \otimes id_{A} \otimes id_{X} \otimes \tau_{X \otimes A \otimes A \otimes X
\otimes X,A} \otimes id_{A}) \circ \]
	\[ \circ (id_{X} \otimes id_{A} \otimes id_{X} \otimes id_{A} \otimes \tau_{X \otimes A \otimes A, X} \otimes id_{X} \otimes
id_{X} \otimes id_{A} \otimes id_{A}) \circ \]
	\[ \circ (id_{X} \otimes id_{A} \otimes id_{X} \otimes \tau_{X,A} \otimes id_{A} \otimes id_{A} \otimes id_{X} \otimes
id_{X} \otimes id_{X} \otimes id_{A} \otimes id_{A}) \circ \]
	\[ \circ (id_{X} \otimes id_{A} \otimes \tau_{X \otimes A \otimes A \otimes A, X} \otimes id_{X} \otimes id_{X} \otimes
id_{X} \otimes id_{A} \otimes id_{A} ) \circ \]
	\[ \circ ( id_{X} \otimes \tau_{X,A} \otimes id_{A} \otimes id_{A} \otimes id_{X} \otimes id_{X} \otimes id_{X} \otimes
id_{X} \otimes id_{A} \otimes id_{A} ) \circ \]
	\[ \circ (\Delta_{X} \otimes (( \Delta_{A} \otimes \Delta_{A} ) \circ \Delta_{A}) \otimes (( \Delta_{X} \otimes \Delta_{X} )
\circ \Delta_{X}) \otimes \Delta_{A}) (x \otimes a \otimes y \otimes b) = \]
	\[ = \sum_{(a),(x),(y),(b),(y^{'}),(y^{''}),(a^{'}),(a^{''})} (-1)^{ |(a^{'})^{'}| |x^{''}| + |(y^{'})^{'}| ( |x^{''}| +
|(a^{'})^{''}| + |(a^{''})^{'}| + |(a^{''})^{''}| ) + |(a^{'})^{''}| |x^{''}| } * \]
	\[ * (-1)^{ |(y^{'})^{''}| ( |x^{''}| + |(a^{''})^{'}| + |(a^{''})^{''}| ) + |b^{'}| ( |x^{''}| + |(a^{''})^{'}| +
|(a^{''})^{''}| + |(y^{''})^{'}| + |(y^{''})^{''}| ) + |(y^{''})^{'}| |(a^{''})^{''}| } * \]
	\[ * x^{'} ((a^{'})^{'} \cdot (y^{'})^{'}) \otimes ((a^{'})^{''})^{(y^{'})^{''}} b^{'} \otimes x^{''} ((a^{''})^{'} \cdot
(y^{''})^{'}) \otimes ((a^{''})^{''})^{(y^{''})^{''}} b^{''} = \]
	\[ = \sum_{(a),(x),(y),(b),(y^{'}),(y^{''}),(a^{'}),(a^{''})} (-1)^{ |a^{'}| |x^{''}| + |b^{'}| |y^{''}| +
|(y^{'})^{'}||(a^{'})^{''}| + |(y^{''})^{'}||(a^{''})^{''}| + (|x^{''}|+|a^{''}|)(|y^{'}| + |b^{'}| ) } * \]
	\[ * x^{'} ((a^{'})^{'} \cdot (y^{'})^{'}) \otimes ((a^{'})^{''})^{(y^{'})^{''}} b^{'} \otimes x^{''} ((a^{''})^{'} \cdot
(y^{''})^{'}) \otimes ((a^{''})^{''})^{(y^{''})^{''}} b^{''}. \]
	
	Remark.
	5.1 We have by Definition \ref{df:twistedB}
	\[ \Delta_{X} ( x ( a^{'} \cdot y^{'} ) ) = ( \sum_{(x)} x^{'} \otimes x^{''} ) ( \sum_{(a),(y)}
(-1)^{|(y^{'})^{'}||(a^{'})^{''}|} (a^{'})^{'} \cdot (y^{'})^{'}  \otimes (a^{'})^{''} \cdot (y^{'})^{''} ) = \]
	\[ = \sum_{(x),(a^{'}),(y^{'})} (-1)^{|(y^{'})^{'}||(a^{'})^{''}| + |x^{''}| ( |(a^{'})^{'}| + |(y^{'})^{'}| )} x^{'}
((a^{'})^{'} \cdot (y^{'})^{'}) \otimes x^{''} ( (a^{'})^{''} \cdot (y^{'})^{''} ) ). \]
	
	5.2 Notice that by Definition \ref{df:twistedB}
	\[ \Delta_{A} ( (a^{''})^{y^{''}} b ) ) ) = ( \sum_{(a^{''}),(y^{''})} (-1)^{ |(y^{''})^{'}||(a^{''})^{''}| }
((a^{''})^{'})^{(y^{''})^{'}} \otimes ((a^{''})^{''})^{(y^{''})^{''}} ) ( \sum_{(b)} b^{'} \otimes b^{''} ) = \]
	\[ = \sum_{(a^{''}),(y^{''}),(b)} (-1)^{ |(y^{''})^{'}||(a^{''})^{''}| + (|(a^{''})^{''}| + |(y^{''})^{''}|)|b^{'}| }
((a^{''})^{'})^{(y^{''})^{'}} b^{'} \otimes ((a^{''})^{''})^{(y^{''})^{''}} b^{''}. \qed \]
	
	Verify that
	\[ \Delta_{X \bowtie A} \circ \eta_{X \bowtie A} = (\eta_{X \bowtie A} \otimes \eta_{X \bowtie A}) \circ \Delta_{k}. \]
	Indeed,
	\[ \Delta_{X \bowtie A} \circ \eta_{X \bowtie A} (1_{k}) = \Delta_{X \bowtie A} ( 1_{X} \otimes 1_{A} ) = 1_{X} \otimes
1_{A} \otimes 1_{X} \otimes 1_{A} = \]
	\[ = (\eta_{X \bowtie A} \otimes \eta_{X \bowtie A}) (1_{k} \otimes 1_{k}) = (\eta_{X \bowtie A} \otimes \eta_{X \bowtie A})
\circ \Delta_{k} (1_k). \]
	
	6. Verify that $S_{X \bowtie A}$ is an antipode on $X \otimes A$.	
	\[ \mu_{X \bowtie A} \circ ( S_{X \bowtie A} \otimes id_{X \bowtie A}) \circ \Delta_{X \bowtie A} = \mu_{X \bowtie A} \circ
( id_{X \bowtie A} \otimes S_{X \bowtie A} ) \circ \Delta_{X \bowtie A} = \eta_{X \bowtie A} \otimes \epsilon_{X \bowtie A}, \]
	\[ \mu_{X \bowtie A} \circ ( S_{X \bowtie A} \otimes id_{X \bowtie A}) \circ \Delta_{X \bowtie A} ( x \otimes a ) = \mu_{X
\bowtie A} ( \sum_{(x),(a)} (-1)^{|x^{''}||a^{'}|} S_{X \bowtie A} (x^{'} \otimes a^{'}) \otimes x^{''} \otimes a^{''} ) = \]
	\[ = \mu_{X \bowtie A} ( \sum_{(x),(a),(x^{'}),(a^{'})} (-1)^{|x^{''}||a^{'}| + |(x^{'})^{'}||(x^{'})^{''}| +
|(a^{'})^{'}||(a^{'})^{''}| + |(x^{'})^{'}||(a^{'})^{''}| + |(x^{'})^{''}||(a^{'})^{''}| + |(x^{'})^{'}||(a^{'})^{'}|} * \]
	\[ * S_{A}((a^{'})^{''}) \cdot S_{X}((x^{'})^{''}) \otimes ( S_{A}((a^{'})^{'}) )^{S_{X}((x^{'})^{'})} \otimes x^{''}
\otimes a^{''} ) = \]
	\[ = \sum_{(x),(a),(x^{'}),(a^{'}),((a^{'})^{'}),((x^{'})^{'}),(x^{''})} (-1)^{|x^{''}||a^{'}| + |(x^{'})^{'}||(x^{'})^{''}|
+ |(a^{'})^{'}||(a^{'})^{''}| + |(x^{'})^{'}||(a^{'})^{''}| + |(x^{'})^{''}||(a^{'})^{''}| + |(x^{'})^{'}||(a^{'})^{'}| } * \]
	\[ * (-1)^{|((a^{'})^{'})^{'}||((a^{'})^{'})^{''}| + |((x^{'})^{'})^{'}||((x^{'})^{'})^{''}| + |((a^{'})^{'})^{'}|
|((x^{'})^{'})^{''}| + |(x^{''})^{'}| | ((a^{'})^{'})^{'} | + |(x^{''})^{'}| | ((x^{'})^{'})^{'} | } * \]
	\[ * ( S_{A}((a^{'})^{''}) \cdot S_{X}((x^{'})^{''}) ) ( (S_{A}(((a^{'})^{'})^{''}))^{ S_{X}(((x^{'})^{'})^{''}) } \cdot
(x^{''})^{'} ) \otimes (S_{A}( ((a^{'})^{'})^{'} ))^{ S_{X}( ((x^{'})^{'})^{'} ) (x^{''})^{''} } a^{''} = \]	
	\[ = \sum_{(x),(a),(x^{'}),(a^{'}),((a^{'})^{'}),((x^{'})^{'}),(x^{''})} (-1)^{ |(a^{'})^{''}| ( |(a^{'})^{'}| + |x| ) +
|(x^{'})^{''}| |(x^{'})^{'}| + |((a^{'})^{'})^{''}| ( |x^{''}| + |(x^{'})^{'}| + |((a^{'})^{'})^{'}| ) +
|((x^{'})^{'})^{''}||((x^{'})^{'})^{'}|} * \]
	\[ * (-1)^{ |(x^{''})^{'}| | ((x^{'})^{'})^{'} | + | ((a^{'})^{'})^{'} | ( |(x^{''})^{''}| +  |((x^{'})^{'})^{'}| ) } * \]
	\[ * ( S_{A}((a^{'})^{''}) \cdot S_{X}((x^{'})^{''}) ) ( (S_{A}(((a^{'})^{'})^{''}))^{ S_{X}(((x^{'})^{'})^{''}) } \cdot
(x^{''})^{'} ) \otimes (S_{A}( ((a^{'})^{'})^{'} ))^{ S_{X}( ((x^{'})^{'})^{'} ) (x^{''})^{''} } a^{''} = \]
	\[ = (\mu_{X} \circ (\alpha \otimes (\alpha \circ (\beta \otimes id_{X}))) \otimes (\mu_{A} \circ (\beta \otimes id_{A})
\circ (id_{A} \otimes \mu_{X} \otimes id_{A}))) \circ \]
	\[ \circ (S_{A} \otimes S_{X} \otimes S_{A} \otimes S_{X} \otimes id_{X} \otimes S_{A} \otimes S_{X} \otimes id_{X} \otimes
id_{A}) \circ \]
	\[ \circ (id_{A} \otimes id_{X} \otimes id_{A} \otimes id_{X} \otimes id_{X} \otimes \tau_{X \otimes X,A} \otimes id_{A})
\circ (id_{A} \otimes id_{X} \otimes id_{A} \otimes id_{X} \otimes \tau_{X,X} \otimes id_{X} \otimes id_{A} \otimes id_{A})
\circ \]
	\[ \circ (id_{A} \otimes id_{X} \otimes id_{A} \otimes \tau_{X,X} \otimes id_{X} \otimes id_{X} \otimes id_{A} \otimes
id_{A}) \circ (id_{A} \otimes id_{X} \otimes \tau_{X \otimes X \otimes X \otimes X \otimes A, A} \otimes id_{A}) \circ \]
	\[ \circ (id_{A} \otimes \tau_{X \otimes X,X} \otimes id_{X} \otimes id_{X} \otimes id_{A} \otimes id_{A} \otimes id_{A})
\circ (\tau_{X \otimes X \otimes X \otimes X \otimes X \otimes A \otimes A,A} \otimes id_{A}) \circ \]
	\[ \circ (((((\Delta_{X} \otimes id_{X}) \circ \Delta_{X}) \otimes id_{X} \otimes id_{X}) \circ (( id_{X} \otimes \Delta_{X}
) \circ \Delta_{X})) \otimes(((( \Delta_{A} \otimes id_{A}) \circ \Delta_{A}) \otimes id_{A}) \circ \Delta_{A}))(x \otimes a)=
	\]
	\[ = (\mu_{X} \circ (\alpha \otimes (\alpha \circ (\beta \otimes id_{X}))) \otimes (\mu_{A} \circ (\beta \otimes id_{A})
\circ (id_{A} \otimes \mu_{X} \otimes id_{A}))) \circ \]
	\[ \circ (S_{A} \otimes S_{X} \otimes S_{A} \otimes S_{X} \otimes id_{X} \otimes S_{A} \otimes S_{X} \otimes id_{X} \otimes
id_{A}) \circ \]
	\[ \circ (id_{A} \otimes id_{X} \otimes id_{A} \otimes id_{X} \otimes id_{X} \otimes \tau_{X \otimes X,A} \otimes id_{A})
\circ (id_{A} \otimes id_{X} \otimes id_{A} \otimes id_{X} \otimes \tau_{X,X} \otimes id_{X} \otimes id_{A} \otimes id_{A})
\circ \]
	\[ \circ (id_{A} \otimes id_{X} \otimes id_{A} \otimes \tau_{X,X} \otimes id_{X} \otimes id_{X} \otimes id_{A} \otimes
id_{A}) \circ (id_{A} \otimes id_{X} \otimes \tau_{X \otimes X \otimes X \otimes X \otimes A, A} \otimes id_{A}) \circ \]
	\[ \circ (id_{A} \otimes \tau_{X \otimes X,X} \otimes id_{X} \otimes id_{X} \otimes id_{A} \otimes id_{A} \otimes id_{A})
\circ (\tau_{X \otimes X \otimes X \otimes X \otimes X \otimes A \otimes A,A} \otimes id_{A}) \circ \]	
	\[ (((((id_{X} \otimes \Delta_{X}) \circ \Delta_{X}) \otimes id_{X} \otimes id_{X}) \circ (( id_{X} \otimes \Delta_{X} )
\circ \Delta_{X})) \otimes((((id_{A}\otimes\Delta_{A})\circ\Delta_{A})\otimes id_{A})\circ\Delta_{A}))(x\otimes a) = \]
	\[ = \sum_{(x),(a),(x^{'}),((x^{'})^{''}),(x^{''}),(a^{'}),((a^{'})^{''})} (-1)^{ |((a^{'})^{''})^{''}| ( |x| +
|(a^{'})^{'}| + |((a^{'})^{''})^{'}| ) + |((x^{'})^{''})^{''}| ( |(x^{'})^{'}| + |((x^{'})^{''})^{'}| ) } * \]
	\[ * (-1)^{ |((a^{'})^{''})^{'}| ( |(x^{'})^{'}| + |((x^{'})^{''})^{'}| + |x^{''}| + |(a^{'})^{'}| ) +
|((x^{'})^{''})^{'}||(x^{'})^{'}| + |(x^{''})^{'}| |(x^{'})^{'}| + |(a^{'})^{'}| ( |(x^{'})^{'}| + |(x^{''})^{''}| ) } * \]
	\[ * (S_{A} ( ((a^{'})^{''})^{''} ) \cdot S_{X} ( ((x^{'})^{''})^{''} ) ) ((S_{A} ( ((a^{'})^{''})^{'} )) ^{S_{X} (
((x^{'})^{''})^{'} )} \cdot (x^{''})^{'}) \otimes (S_{A} ((a^{'})^{'}))^{S_{X} ((x^{'})^{'}) (x^{''})^{''}} a^{''} = \]
	\[ = \sum_{(x),(a),(x^{'}),(x^{''}),(a^{'})} (-1)^{ |(a^{'})^{''}| ( |x| + |(a^{'})^{'}| ) + |(x^{'})^{''}| |(x^{'})^{'}| +
|(x^{''})^{'}| |(x^{'})^{'}| + |(a^{'})^{'}| ( |(x^{'})^{'}| + |(x^{''})^{''}| ) } * \]
	\[ * S_{A} ( (a^{'})^{''} ) \cdot ( S_{X} ((x^{'})^{''}) (x^{''})^{'}) \otimes (S_{A} ((a^{'})^{'}))^{S_{X} ((x^{'})^{'})
(x^{''})^{''}} a^{''} = \]
	\[ = ( (\alpha \circ (id_{A} \otimes \mu_{X})) \otimes (\mu_{A} \circ (\beta \otimes id_{A}) \circ (id_{A} \otimes \mu_{X}
\otimes id_{A}) ) ) \circ ( S_{A} \otimes S_{X} \otimes id_{X} \otimes S_{A} \otimes S_{X} \otimes id_{X} \otimes id_{A} ) \circ
\]
	\[ \circ ( id_{A} \otimes id_{X} \otimes id_{X} \otimes \tau_{X \otimes X, A} \otimes id_{A} ) \circ (id_{A} \otimes
\tau_{X,X \otimes X} \otimes id_{X} \otimes id_{A} \otimes id_{A} ) \circ \]
	\[ \circ ( \tau_{X \otimes X \otimes X \otimes X \otimes A, A} \otimes id_{A} ) \circ (( (\Delta_{X} \otimes \Delta_{X})
\circ \Delta_{X} ) \otimes ((\Delta_{A} \otimes id_{A}) \circ \Delta_{A})) (x \otimes a) = \]
	\[ = ( (\alpha \circ (id_{A} \otimes \mu_{X})) \otimes (\mu_{A} \circ ((\beta \otimes id_{A}) \circ ( id_{A} \otimes \mu_{X}
\otimes id_{A})) ) \circ ( S_{A} \otimes S_{X} \otimes id_{X} \otimes S_{A} \otimes S_{X} \otimes id_{X} \otimes id_{A} ) \circ
\]
	\[ \circ ( id_{A} \otimes id_{X} \otimes id_{X} \otimes \tau_{X \otimes X, A} \otimes id_{A} ) \circ (id_{A} \otimes
\tau_{X,X \otimes X} \otimes id_{X} \otimes id_{A} \otimes id_{A} ) \circ \]
	\[ \circ ( \tau_{X \otimes X \otimes X \otimes X \otimes A, A} \otimes id_{A} ) \circ ( ( id_{X} \otimes \Delta_{X} \otimes
id_{X} ) \circ (id_{X} \otimes \Delta_{X}) \circ \Delta_{X} ) \otimes ((\Delta_{A} \otimes id_{A}) \circ \Delta_{A})) (x \otimes
a) = \]
	\[ = \sum_{(x),(a),(x^{''}),((x^{''})^{'}),(a^{'})} (-1)^{ |(a^{'})^{''}| ( |x| + |(a^{'})^{'}| ) + |(x^{''})^{'}||x^{'}| +
|(a^{'})^{'}| ( |x^{'}| + |(x^{''})^{''}| ) } * \]
	\[ * (S_{A}((a^{'})^{''}) \cdot (S_{X}(((x^{''})^{'})^{'}) ((x^{''})^{'})^{''})) \otimes (S_{A}((a^{'})^{'})) ^{S_{X}(x^{'})
(x^{''})^{''}} a^{''} = \]
	\[ = \sum_{(x),(a),(x^{''}),(a^{'})} (-1)^{ |(a^{'})^{''}| ( |x| + |(a^{'})^{'}| ) + |(a^{'})^{'}| ( |x^{'}| +
|(x^{''})^{''}| ) } * \]
	\[ * \epsilon_{X} ((x^{''})^{'}) ((S_{A}((a^{'})^{''}) \cdot 1_{X}) \otimes (S_{A}((a^{'})^{'})) ^{S_{X}(x^{'})
(x^{''})^{''}} a^{''} = \]
	\[ = \sum_{(x),(a),(x^{''}),(a^{'})} (-1)^{ |(a^{'})^{''}| ( |x| + |(a^{'})^{'}| ) + |(a^{'})^{'}| ( |x^{'}| +
|(x^{''})^{''}| ) } * \]
	\[ * \epsilon_{X}((x^{''})^{'}) \epsilon_{A}(S_{A}((a^{'})^{''}) 1_{X} \otimes (S_{A}((a^{'})^{'})) ^{S_{X}(x^{'})
(x^{''})^{''}} a^{''} = \]
	\[ = \sum_{(x),(a),(x^{''}),(a^{'})} (-1)^{ |(a^{'})^{'}| ( |x^{'}| + |(x^{''})^{''}| ) }  \epsilon_{X}((x^{''})^{'})
\epsilon_{A}((a^{'})^{''}) 1_{X} \otimes (S_{A}((a^{'})^{'})) ^{S_{X}(x^{'}) (x^{''})^{''}} a^{''} = \]
	\[ = (id_{X} \otimes \mu_{A}) \circ (id_{X} \otimes \beta \otimes id_{A}) \circ (id_{X} \otimes id_{A} \otimes \mu_{X}
\otimes id_{A}) \circ (id_{X} \otimes \tau_{X \otimes X, A} \otimes id_{A}) \circ ( id_{X} \otimes S_{X} \otimes id_{X} \otimes
S_{A} \otimes id_{A} ) \circ \]
	\[ \circ ( id_{X} \otimes id_{X} \otimes ( \nu_{k,X} \circ ( \epsilon_{X} \otimes id_{X} ) \circ \Delta_{X}) \otimes (
\nu_{A,k} \circ ( id_{A} \otimes \epsilon_{A} ) \circ \Delta_{A}) \otimes id_{A} ) (\sum_{(x),(a)} 1_{X} \otimes x^{'} \otimes
x^{''} \otimes a^{'} \otimes a^{''}) = \]
	\[ = (id_{X} \otimes \mu_{A}) \circ (id_{X} \otimes \beta \otimes id_{A}) \circ (id_{X} \otimes id_{A} \otimes \mu_{X}
\otimes id_{A}) \circ (id_{X} \otimes \tau_{X \otimes X, A} \otimes id_{A}) \circ ( id_{X} \otimes S_{X} \otimes id_{X} \otimes
S_{A} \otimes id_{A} ) \circ \]
	\[ \circ ( id_{X} \otimes id_{X} \otimes id_{X} \otimes id_{A} \otimes id_{A} )  (\sum_{(x),(a)} 1_{X} \otimes x^{'} \otimes
x^{''} \otimes a^{'} \otimes a^{''}) = \]
	\[ = \sum_{(x),(a)} (-1)^{|a^{'}||x|} 1_{X} \otimes (S_{A}(a^{'}))^{S_{X}(x^{'}) x^{''}} a^{''} = \]
	\[ = (id_{X} \otimes (\mu_{A} \circ (\beta \otimes id_{A}))) \circ (id_{X} \otimes \tau_{X,A} \otimes id_{A} ) \circ \]
	\[ \circ (id_{X} \otimes ( \mu_{X} \circ (S_{X} \otimes id_{X}) \circ \Delta_{X} ) \otimes id_{A} \otimes id_{A}) (
\sum_{(a)} 1_{X} \otimes x \otimes S_{A}(a^{'}) \otimes a^{''} ) = \]
	\[ = (id_{X} \otimes (\mu_{A} \circ (\beta \otimes id_{A}))) \circ (id_{X} \otimes \tau_{X,A} \otimes id_{A} ) \circ (id_{X}
\otimes (\eta_{X} \circ \epsilon_{X}) \otimes id_{A} \otimes id_{A}) ( \sum_{(a)} 1_{X} \otimes x \otimes S_{A}(a^{'}) \otimes
a^{''} ) = \]
	\[ = \epsilon_{X}(x) \sum_{(a)} 1_{X} \otimes (S_{A}(a^{'}))^{1_{X}} a^{''} = \epsilon_{X}(x) \sum_{(a)} 1_{X} \otimes
S_{A}(a^{'}) a^{''} = \]
	\[ = \epsilon_{X}(x) \epsilon_{A}(a) 1_{X} \otimes 1_{A} = (\eta_{X \bowtie A} \otimes \epsilon_{X \bowtie A}) (x \otimes a)
. \]
	
	Remark. We use the following relations that we deduce from Definition \ref{df:twistedB}
	
	6.1.1
	\[ \Delta_{A} ( S_{A}((a^{'})^{'}) = \sum_{((a^{'})^{'})} (-1)^{|((a^{'})^{'})^{'}||((a^{'})^{'})^{''}|}
S_{A}(((a^{'})^{'})^{''}) \otimes S_{A}( ((a^{'})^{'})^{'} ), \]
	
	6.1.2
	\[ \Delta_{X} ( S_{X}((x^{'})^{'}) = \sum_{((x^{'})^{'})} (-1)^{|((x^{'})^{'})^{'}||((x^{'})^{'})^{''}|}
S_{X}(((x^{'})^{'})^{''}) \otimes S_{X}( ((x^{'})^{'})^{'} ). \]
	
	6.1.3
	\[ \Delta_{A} ( ( S_{A}((a^{'})^{'}) )^{S_{X}((x^{'})^{'})} ) = \sum_{((a^{'})^{'}),((x^{'})^{'})}
(-1)^{|((a^{'})^{'})^{'}||((a^{'})^{'})^{''}| + |((x^{'})^{'})^{'}||((x^{'})^{'})^{''}| + |((a^{'})^{'})^{'}|
|((x^{'})^{'})^{''}|} * \]
	\[ * (S_{A}(((a^{'})^{'})^{''}))^{ S_{X}(((x^{'})^{'})^{''}) } \otimes (S_{A}( ((a^{'})^{'})^{'} ))^{ S_{X}(
((x^{'})^{'})^{'} ) }.\]
	
	6.2
	\[ S_{A} ( (a^{'})^{''} ) \cdot ( S_{X} ((x^{'})^{''}) (x^{''})^{'}) = \]
	\[ = \sum_{((a^{'})^{''}),((x^{'})^{''})} (-1)^{ |((x^{'})^{''})^{'}| |((x^{'})^{''})^{''}| + |((a^{'})^{''})^{'}|
|((a^{'})^{''})^{''}| + |((x^{'})^{''})^{''}| |((a^{'})^{''})^{'}| }  * \]
	\[ * (S_{A} ( ((a^{'})^{''})^{''} ) \cdot S_{X} ( ((x^{'})^{''})^{''} )) ((S_{A} ( ((a^{'})^{''})^{'} )) ^{S_{X} (
((x^{'})^{''})^{'} )} \cdot (x^{''})^{'}). \]
	
	6.3 We have from the definition of a counit
	\[ \epsilon_{X}((x^{''})^{'}) \Rightarrow |(x^{''})^{'}|=0, \; \epsilon_{A}((a^{'})^{''}) \Rightarrow |(a^{'})^{''}|=0, \;
\epsilon_{X}(x) \Rightarrow |x|=0. \qed \]
	
	Next we have
	\[ \mu_{X \bowtie A} \circ ( id_{X \bowtie A} \otimes S_{X \bowtie A} ) \circ \Delta_{X \bowtie A} (x \otimes a) = \mu_{X
\bowtie A} ( \sum_{(x),(a)} (-1)^{|x^{''}||a^{'}|} x^{'} \otimes a^{'} \otimes S_{X \bowtie A} ( x^{''} \otimes a^{''} ) ) = \]
	
	\[ = \mu_{X \bowtie A} ( \sum_{(x),(a),(x^{''}),(a^{''})} (-1)^{|x^{''}||a^{'}| + |(x^{''})^{'}||(x^{''})^{''}| +
|(a^{''})^{'}||(a^{''})^{''}| + |(x^{''})^{'}||(a^{''})^{''}| + |(x^{''})^{''}||(a^{''})^{''}| + |(x^{''})^{'}||(a^{''})^{'}|} *
\]
	\[ * x^{'} \otimes a^{'} \otimes S_{A}((a^{''})^{''}) \cdot S_{X}((x^{''})^{''}) \otimes ( S_{A}((a^{''})^{'})
)^{S_{X}((x^{''})^{'})} )= \]
	\[ = \sum_{(x),(a),(x^{''}),(a^{''}),((a^{''})^{''}),((x^{''})^{''}),(a^{'})} (-1)^{|x^{''}||a^{'}| +
|(x^{''})^{'}||(x^{''})^{''}| + |(a^{''})^{'}||(a^{''})^{''}| + |(x^{''})^{'}||(a^{''})^{''}| + |(x^{''})^{''}||(a^{''})^{''}| +
|(x^{''})^{'}||(a^{''})^{'}|} * \]
	\[ * (-1)^{|((a^{''})^{''})^{'}||((a^{''})^{''})^{''}| + |((x^{''})^{''})^{'}||((x^{''})^{''})^{''}| + |((a^{''})^{''})^{'}|
|((x^{''})^{''})^{''}| + |(a^{'})^{''}| |((a^{''})^{''})^{''}| + |(a^{'})^{''}| |((x^{''})^{''})^{''}| } * \]
	\[ * x^{'} ( (a^{'})^{'} \cdot (S_{A}(((a^{''})^{''})^{''})) \cdot S_{X}(((x^{''})^{''})^{''}) ) \otimes ((a^{'})^{''}) ^ {
(S_{A}( ((a^{''})^{''})^{'} )) \cdot S_{X}( ((x^{''})^{''})^{'} ) } ( S_{A}((a^{''})^{'}) )^{S_{X}((x^{''})^{'})} = \]	
	\[ = \sum_{(x),(a),(x^{''}),(a^{''}),((a^{''})^{''}),((x^{''})^{''}),(a^{'})} (-1)^{|x^{''}||a^{'}| +
|(x^{''})^{'}||(x^{''})^{''}| + |(a^{''})^{'}||(a^{''})^{''}| + |(x^{''})^{'}||(a^{''})^{''}| + |(x^{''})^{''}||(a^{''})^{''}| +
|(x^{''})^{'}||(a^{''})^{'}|} * \]
	\[ * (-1)^{|((a^{''})^{''})^{'}||((a^{''})^{''})^{''}| + |((x^{''})^{''})^{'}||((x^{''})^{''})^{''}| + |((a^{''})^{''})^{'}|
|((x^{''})^{''})^{''}| + |(a^{'})^{''}| |((a^{''})^{''})^{''}| + |(a^{'})^{''}| |((x^{''})^{''})^{''}| } * \]
	\[ * x^{'} ( ( (a^{'})^{'} (S_{A}(((a^{''})^{''})^{''})) ) \cdot S_{X}(((x^{''})^{''})^{''}) ) \otimes ((a^{'})^{''}) ^ {
(S_{A}( ((a^{''})^{''})^{'} )) \cdot S_{X}( ((x^{''})^{''})^{'} ) } ( S_{A}((a^{''})^{'}) )^{S_{X}((x^{''})^{'})} = \]
	\[ = \sum_{(x),(a),(x^{''}),(a^{''}),((a^{''})^{''}),((x^{''})^{''}),(a^{'})} (-1)^{ |(a^{'})^{'}| |x^{''}| +
|((a^{''})^{''})^{''}| ( |((a^{''})^{''})^{'}| + |(a^{''})^{'}| + |(a^{'})^{''}| + |x^{''}| ) + |((x^{''})^{''})^{''}| (
|(x^{''})^{'}| + |((x^{''})^{''})^{'}| ) } * \]
	\[ * (-1)^{ |(a^{'})^{''}| ( |(x^{''})^{'}| + |((x^{''})^{''})^{'}| ) + |((a^{''})^{''})^{'}| ( |(x^{''})^{'}| +
|((x^{''})^{''})^{'}| + |(a^{''})^{'}| ) + |((x^{''})^{''})^{'}| |(x^{''})^{'}| + |(a^{''})^{'}| |(x^{''})^{'}| } * \]
	\[ * x^{'} ( ( (a^{'})^{'} (S_{A}(((a^{''})^{''})^{''})) ) \cdot S_{X}(((x^{''})^{''})^{''}) ) \otimes ((a^{'})^{''}) ^ {
(S_{A}( ((a^{''})^{''})^{'} )) \cdot S_{X}( ((x^{''})^{''})^{'} ) } ( S_{A}((a^{''})^{'}) )^{S_{X}((x^{''})^{'})} = \]	
	\[ = (\mu_{X} \circ (id_{X} \otimes (\alpha \circ (\mu_{A} \otimes id_{X}))) \otimes (\mu_{A} \circ (\beta \otimes id_{A})
\circ (id_{A} \otimes \alpha \otimes \beta))) \circ \]
	\[ \circ (id_{X} \otimes id_{A} \otimes S_{A} \otimes S_{X} \otimes id_{A} \otimes S_{A} \otimes S_{X} \otimes S_{A} \otimes
S_{X}) \circ ( id_{X} \otimes id_{A} \otimes id_{A} \otimes id_{X} \otimes id_{A} \otimes id_{A} \otimes id_{X} \otimes
\tau_{X,A} ) \circ \]
	\[ \circ (id_{X} \otimes id_{A} \otimes id_{A} \otimes id_{X} \otimes id_{A} \otimes id_{A} \otimes \tau_{X ,X} \otimes
id_{A}) \circ (id_{X} \otimes id_{A} \otimes id_{A} \otimes id_{X} \otimes id_{A} \otimes \tau_{X \otimes X \otimes A,A} ) \circ
\]	
	\[ \circ (id_{X} \otimes id_{A} \otimes id_{A} \otimes id_{X} \otimes \tau_{X \otimes X,A} \otimes id_{A} \otimes id_{A})
\circ (id_{X} \otimes id_{A} \otimes id_{A} \otimes \tau_{X \otimes X , X} \otimes id_{A} \otimes id_{A} \otimes id_{A}) \circ
\]
	\[ \circ (id_{X} \otimes id_{A} \otimes \tau_{X \otimes X \otimes X \otimes A \otimes A \otimes A,A}) \circ (id_{X} \otimes
\tau_{X \otimes X \otimes X ,A} \otimes id_{A} \otimes id_{A} \otimes id_{A} \otimes id_{A}) \circ \]
	\[ \circ ( (((id_{X} \otimes ((id_{X} \otimes \Delta_{X}) \circ \Delta_{X}))) \circ \Delta_{X}) \otimes ((\Delta_{A} \otimes
((id_{A} \otimes \Delta_{A}) \circ \Delta_{A})) \circ \Delta_{A}) ) (x\otimes a) =\]
	\[ = (\mu_{X} \circ (id_{X} \otimes (\alpha \circ (\mu_{A} \otimes id_{X}))) \otimes (\mu_{A} \circ (\beta \otimes id_{A})
\circ (id_{A} \otimes \alpha \otimes \beta))) \circ \]
	\[ \circ (id_{X} \otimes id_{A} \otimes S_{A} \otimes S_{X} \otimes id_{A} \otimes S_{A} \otimes S_{X} \otimes S_{A} \otimes
S_{X}) \circ ( id_{X} \otimes id_{A} \otimes id_{A} \otimes id_{X} \otimes id_{A} \otimes id_{A} \otimes id_{X} \otimes
\tau_{X,A} ) \circ \]
	\[ \circ (id_{X} \otimes id_{A} \otimes id_{A} \otimes id_{X} \otimes id_{A} \otimes id_{A} \otimes \tau_{X ,X} \otimes
id_{A}) \circ (id_{X} \otimes id_{A} \otimes id_{A} \otimes id_{X} \otimes id_{A} \otimes \tau_{X \otimes X \otimes A,A} ) \circ
\]	
	\[ \circ (id_{X} \otimes id_{A} \otimes id_{A} \otimes id_{X} \otimes \tau_{X \otimes X,A} \otimes id_{A} \otimes id_{A})
\circ (id_{X} \otimes id_{A} \otimes id_{A} \otimes \tau_{X \otimes X , X} \otimes id_{A} \otimes id_{A} \otimes id_{A}) \circ
\]
	\[ \circ (id_{X} \otimes id_{A} \otimes \tau_{X \otimes X \otimes X \otimes A \otimes A \otimes A,A}) \circ (id_{X} \otimes
\tau_{X \otimes X \otimes X ,A} \otimes id_{A} \otimes id_{A} \otimes id_{A} \otimes id_{A}) \circ \]
	\[ \circ ( (((id_{X} \otimes (( \Delta_{X} \otimes id_{X}) \circ \Delta_{X}))) \circ \Delta_{X}) \otimes ((\Delta_{A}
\otimes (( \Delta_{A} \otimes id_{A}) \circ \Delta_{A})) \circ \Delta_{A}) ) (x\otimes a) =\]	
	\[ = \sum_{(x),(a),(a^{'}),(a^{''}),((a^{''})^{'}),(x^{''}),((x^{''})^{'})} (-1)^{ |(a^{'})^{'}| |x^{''}| + |(a^{''})^{''}|
( |x^{''}| + |(a^{'})^{''}| + |(a^{''})^{'}| ) } * \]
	\[ * (-1)^{|(x^{''})^{''}| |(x^{''})^{'}| + |(a^{'})^{''}| |(x^{''})^{'}| + |((a^{''})^{'})^{''}| ( |((a^{''})^{'})^{'}| +
|(x^{''})^{'}| ) + |((x^{''})^{'})^{''}| |((x^{''})^{'})^{'}| + |((a^{''})^{'})^{'}| |((x^{''})^{'})^{'}|} * \]
	\[ * x^{'} (((a^{'})^{'} S_{A}((a^{''})^{''})) \cdot S_{X}((x^{''})^{''})) \otimes
(((a^{'})^{''})^{S_{A}(((a^{''})^{'})^{''}) \cdot S_{X}(((x^{''})^{'})^{''})} )
(S_{A}(((a^{''})^{'})^{'}))^{S_{X}(((x^{''})^{'})^{'})} = \]
	\[ = \sum_{(x),(a),(a^{'}),(a^{''}),(x^{''})} (-1)^{ |(a^{'})^{'}| |x^{''}| + |(a^{''})^{''}| ( |x^{''}| + |(a^{'})^{''}| +
|(a^{''})^{'}| ) + |(x^{''})^{''}| |(x^{''})^{'}| + |(a^{'})^{''}| |(x^{''})^{'}| + |(a^{''})^{'}| |(x^{''})^{'}| } * \]
	\[ * x^{'} (((a^{'})^{'} S_{A}((a^{''})^{''})) \cdot S_{X}((x^{''})^{''})) \otimes ((a^{'})^{''}
S_{A}((a^{''})^{'}))^{S_{X}((x^{''})^{'})} = \]
	\[ = ((\mu_{X} \circ (id_{X} \otimes \alpha) \circ (id_{X} \otimes \mu_{A} \otimes id_{X})) \otimes (\beta \circ (\mu_{A}
\otimes id_{X}))) \circ \]
	\[ \circ (id_{X} \otimes id_{A} \otimes S_{A} \otimes S_{X} \otimes id_{A} \otimes S_{A} \otimes S_{X}) \circ \]
	\[ \circ (id_{X} \otimes id_{A} \otimes id_{A} \otimes \tau_{X,X \otimes A \otimes A}) \circ (id_{X} \otimes id_{A} \otimes
\tau_{X \otimes X \otimes A \otimes A,A}) \circ \]
	\[ \circ (id_{X} \otimes \tau_{X \otimes X, A} \otimes id_{A} \otimes id_{A} \otimes id_{A}) \circ (((id_{X} \otimes
\Delta_{X}) \circ \Delta_{X}) \otimes ( (\Delta_{A} \otimes \Delta_{A}) \circ \Delta_{A} )) (x \otimes a) = \]	
	\[ = ((\mu_{X} \circ (id_{X} \otimes \alpha) \circ (id_{X} \otimes \mu_{A} \otimes id_{X})) \otimes (\beta \circ (\mu_{A}
\otimes id_{X}))) \circ \]
	\[ \circ (id_{X} \otimes id_{A} \otimes S_{A} \otimes S_{X} \otimes id_{A} \otimes S_{A} \otimes S_{X}) \circ \]
	\[ \circ (id_{X} \otimes id_{A} \otimes id_{A} \otimes \tau_{X,X \otimes A \otimes A}) \circ (id_{X} \otimes id_{A} \otimes
\tau_{X \otimes X \otimes A \otimes A,A}) \circ \]
	\[ \circ (id_{X} \otimes \tau_{X \otimes X, A} \otimes id_{A} \otimes id_{A} \otimes id_{A}) \circ ( ((id_{X} \otimes
\Delta_{X}) \circ \Delta_{X}) \otimes ( ( id_{A} \otimes \Delta_{A} \otimes id_{A} ) \circ ( (id_{A} \otimes \Delta_{A}) \circ
\Delta_{A} )) ) (x \otimes a) = \]	
	\[ = \sum_{(x),(a),(x^{''}),(a^{''}),((a^{''})^{'})} (-1)^{ |a^{'}| ||x^{''}| + |(a^{''})^{''}| ( |x^{''}| + |(a^{''})^{'}|
) + |(x^{''})^{'}| ( |(x^{''})^{''}| + |(a^{''})^{'}| ) } * \]
	\[ * x^{'} ((a^{'} S_{A}((a^{''})^{''})) \cdot S_{X}((x^{''})^{''})) \otimes (((a^{''})^{'})^{'} S_{A}
(((a^{''})^{'})^{''}))^{S_{X} ((x^{''})^{'})} = \]
	\[ = (id_{X} \otimes (\beta \circ ((\mu_{A} \circ (id_{A} \otimes S_{A}) \circ \Delta_{A}) \otimes id_{X})))  \circ\]
	\[ \circ (\sum_{(x),(a),(x^{''}),(a^{''})} (-1)^{ |a^{'}| ||x^{''}| + |(a^{''})^{''}| ( |x^{''}| + |(a^{''})^{'}| ) +
|(x^{''})^{'}| ( |(x^{''})^{''}| + |(a^{''})^{'}| ) } * \]
	\[ * x^{'} ((a^{'} S_{A}((a^{''})^{''})) \cdot S_{X}((x^{''})^{''})) \otimes (a^{''})^{'} \otimes S_{X} ((x^{''})^{'}) ) =
\]
	\[ = \epsilon_{A}((a^{''})^{'}) \sum_{(x),(a),(x^{''}),(a^{''})} (-1)^{ |a^{'}| ||x^{''}| + |(a^{''})^{''}| |x^{''}| +
|(x^{''})^{'}| |(x^{''})^{''}| } x^{'} ((a^{'} S_{A}( (a^{''})^{''})) \cdot S_{X}((x^{''})^{''})) \otimes \epsilon_{X}(S_{X}
((x^{''})^{'})) 1_{A} = \]
	\[ = \epsilon_{A}((a^{''})^{'}) \epsilon_{X}((x^{''})^{'})  \sum_{(x),(a),(x^{''}),(a^{''})} (-1)^{ |a^{'}| ||x^{''}| +
|(a^{''})^{''}| |x^{''}| } x^{'} ((a^{'} S_{A}( (a^{''})^{''})) \cdot S_{X}((x^{''})^{''})) \otimes 1_{A} = \]
	\[ = ((\mu_{X} \circ (id_{X} \otimes \alpha) \circ (id_{X} \otimes \mu_{A} \otimes id_{X})) \otimes id_{A}) \circ ( id_{X}
\otimes id_{A} \otimes S_{A} \otimes S_{X} \otimes id_{A}) \circ ( id_{X} \otimes \tau_{X, A \otimes A} \otimes id_{A} ) \circ
\]
	\[ \circ (id_{X} \otimes ( \nu_{k,X} \circ ( \epsilon_{X} \otimes id_{X} ) \circ \Delta_{X} ) \otimes id_{A} \otimes (
\nu_{k,A} \circ  ( \epsilon_{A} \otimes id_{A} ) \circ \Delta_{A} ) \otimes id_{A}) ( \sum_{(x),(a)} x^{'} \otimes x^{''}
\otimes a^{'} \otimes a^{''} \otimes 1_{A}) = \]
	\[ = ((\mu_{X} \circ (id_{X} \otimes \alpha) \circ (id_{X} \otimes \mu_{A} \otimes id_{X})) \otimes id_{A}) \circ ( id_{X}
\otimes id_{A} \otimes S_{A} \otimes S_{X} \otimes id_{A}) \circ ( id_{X} \otimes \tau_{X, A \otimes A} \otimes id_{A} ) \circ
\]
	\[ \circ (id_{X} \otimes id_{X} \otimes id_{A} \otimes id_{A} \otimes id_{A}) ( \sum_{(x),(a)} x^{'} \otimes x^{''} \otimes
a^{'} \otimes a^{''} \otimes 1_{A}) = \]
	\[ = \sum_{(x),(a)} (-1)^{ |x^{''}||a| } x^{'} ( (a^{'} S_{A} (a^{''})) \cdot S_{X}(x^{''} )) \otimes 1_{A} = \]	
	\[ = ((\mu_{X} \circ (id_{X} \otimes \alpha) ) \otimes id_{A}) \circ (id_{X} \otimes \tau_{X,A} \otimes id_{A}) \circ \]
	\[ \circ ( id_{X} \otimes id_{X} \otimes ( \mu_{A} \otimes ( id_{A} \otimes S_{A} ) \otimes \Delta_{A} ) \otimes id_{A} ) (
\sum_{(x),(a)} x^{'} \otimes S_{X}(x^{''}) \otimes a \otimes 1_{A}) = \]
	\[ = ((\mu_{X} \circ (id_{X} \otimes \alpha) ) \otimes id_{A}) (\sum_{(x)} x^{'} \otimes \epsilon_{A}(a) 1_{A} \otimes
S_{X}(x^{''}) \otimes 1_{A}) = \]
	\[ = \epsilon_{A}(a) \sum_{(x)} x^{'} (1_{A} \cdot S_{X}(x^{''})) \otimes 1_{A} = \epsilon_{A}(a) \sum_{(x)} x^{'}
S_{X}(x^{''}) \otimes 1_{A} = \]
	\[ = \epsilon_{X}(x) \epsilon_{A}(a) 1_{X} \otimes 1_{A} = (\eta_{X \bowtie A} \otimes \epsilon_{X \bowtie A}) (x \otimes a)
. \]
	Remark. We use the following relations that we deduce from Definition \ref{df:twistedB}
	
	6.4.1
	\[ \Delta_{A} ( S_{A}((a^{''})^{''}) = \sum_{((a^{''})^{''})} (-1)^{|((a^{''})^{''})^{'}||((a^{''})^{''})^{''}|}
S_{A}(((a^{''})^{''})^{''}) \otimes S_{A}( ((a^{''})^{''})^{'} ), \]
	
	6.4.2
	\[ \Delta_{X} ( S_{X}((x^{''})^{''}) = \sum_{((x^{''})^{''})} (-1)^{|((x^{''})^{''})^{'}||((x^{''})^{''})^{''}|}
S_{X}(((x^{''})^{''})^{''}) \otimes S_{X}( ((x^{''})^{''})^{'} ). \]
	
	6.4.3
	\[ \Delta_{X}( S_{A}((a^{''})^{''}) \cdot S_{X}((x^{''})^{''}) ) = \sum_{((a^{''})^{''}),((x^{''})^{''})}
(-1)^{|((a^{''})^{''})^{'}||((a^{''})^{''})^{''}| + |((x^{''})^{''})^{'}||((x^{''})^{''})^{''}| + |((a^{''})^{''})^{'}|
|((x^{''})^{''})^{''}|} *  \]
	\[ * (S_{A}(((a^{''})^{''})^{''})) \cdot S_{X}(((x^{''})^{''})^{''}) \otimes (S_{A}( ((a^{''})^{''})^{'} )) \cdot S_{X}(
((x^{''})^{''})^{'} ).\]	
	
	6.5
	\[ ((a^{'})^{''} S_{A}((a^{''})^{'}))^{S_{X}((x^{''})^{'})} = \sum_{((a^{''})^{'}),((x^{''})^{'})} (-1)^{
|((a^{''})^{'})^{'}| |((a^{''})^{'})^{''}| + |((x^{''})^{'})^{'}| |((x^{''})^{'})^{''}| + |((a^{''})^{'})^{'}|
|((x^{''})^{'})^{''}| } * \]
	\[ * (((a^{'})^{''})^{S_{A}(((a^{''})^{'})^{''}) \cdot S_{X}(((x^{''})^{'})^{''})} )
(S_{A}(((a^{''})^{'})^{'}))^{S_{X}(((x^{''})^{'})^{'})}.  \]
	
	6.6 We have from the definition of a counit
	\[ \epsilon_{A} ((a^{''})^{'}) \Rightarrow |(a^{''})^{'}| = 0, \; \epsilon_{X} ((x^{''})^{'}) \Rightarrow |(x^{''})^{'}| =
0, \]
	\[ \epsilon_{A}(a) \Rightarrow |a|=0, \; \epsilon_{X}(x) \Rightarrow |x|=0. \qed \]
	
	7. Verify that $i_{X}$ and $i_{A}$ are supercoalgebras morphisms
	\[ (i_{X} \otimes i_{X}) \circ \Delta_{X} = \Delta_{X \bowtie A} \circ i_{X}, \]
	\[ (i_{X} \otimes i_{X}) \circ \Delta_{X} (x) = \sum_{(x)} ( x^{'} \otimes 1_{A} ) \otimes ( x^{''} \otimes 1_{A} ), \]	
	\[ \Delta_{X \bowtie A} \circ i_{X} (x) = \sum_{(a)} (-1)^{|1_{A}||x^{''}|} ( x^{'} \otimes 1_{A} ) \otimes ( x^{''} \otimes
1_{A} )  = \sum_{(x)} ( x^{'} \otimes 1_{A} ) \otimes ( x^{''} \otimes 1_{A} ) \]
	for all $x \in X$.
	\[ \epsilon_{X \bowtie A} \circ i_{X} = \epsilon_{X}, \]
	\[ \epsilon_{X \bowtie A} \circ i_{X} (x) = \epsilon_{X}(x) \epsilon_{A}(1_{A}) = 1_{k} \epsilon_{X}(x) = \epsilon_{X}(x)\]
	for all $x \in X$.
	
	\[ (i_{A} \otimes i_{A}) \circ \Delta_{A} = \Delta_{X \bowtie A} \circ i_{A}, \]
	\[ (i_{A} \otimes i_{A}) \circ \Delta_{A} (a) = \sum_{(a)} ( 1_{X} \otimes a^{'} ) \otimes ( 1_{X} \otimes a^{''} ), \]
	\[ \Delta_{X \bowtie A} \circ i_{A} (a) = \sum_{(a)} (-1)^{|1_{X}||a^{'}|} ( 1_{X} \otimes a^{'} ) \otimes ( 1_{X} \otimes
a^{''} ) = \sum_{(a)} ( 1_{X} \otimes a^{'} ) \otimes ( 1_{X} \otimes a^{''} ) \]
	for all $a \in A$.
	\[ \epsilon_{X \bowtie A} \circ i_{A} = \epsilon_{A}, \]
	\[ \epsilon_{X \bowtie A} \circ i_{A} (a) = \epsilon_{X}(1_{X}) \epsilon_{A}(a) = 1_{k} \epsilon_{A}(a) = \epsilon_{A}(a)
\]
	for all $a \in A$.
	
	Verify that $i_X$ and $i_A$ are superalgebras morphisms:
	\[ i_{X}(xy) = xy \otimes 1_{A}, \]
	\[ i_{X}(x) i_{X}(y) = ( x \otimes 1_{A} ) ( y \otimes 1_{A} ) = \]
	\[ = \sum_{(y)} (-1)^{|y^{'}||1_{A}|} x ( 1_{A} \cdot y^{'} ) \otimes (1_{A})^{y{''}} 1_{A} = \sum_{(y)} x y^{'} \otimes
\epsilon_{X} (y^{''}) 1_{A} = x ( \sum_{(y)} y^{'} \epsilon_{X} (y^{''}) ) \otimes 1_{A} = xy \otimes 1_{A} \]
	for all $x,y \in X$.
	
	\[ i_{X} \circ \eta_{X} = \eta_{X \bowtie A} \iff  i_{X} \circ \eta_{X} ( 1_{k} ) = i_{X} ( 1_{X} ) = 1_{X} \otimes 1_{A}
\iff 		\eta_{X \bowtie A} ( 1_{k} ) = 1_{X} \otimes 1_{A}. \]
	
	\[ i_{A}(ab) = 1_{X} \otimes ab, \]
	\[ i_{A}(a) i_{A}(b) = ( 1_{X} \otimes a ) ( 1_{X} \otimes b ) = \]
	\[ \sum_{(a)} (-1)^{|1_{X}||a^{''}|} 1_{X} ( a^{'} \cdot 1_{X} ) \otimes (a^{''})^{1_{X}} b = \sum_{(a)} 1_{X}
\epsilon_{A}(a^{'}) \otimes a^{''} b = 1_{X} \otimes ( \sum_{(a)} \epsilon_{A}(a^{'}) a^{''}) b = 1_{X} \otimes ab \]
	for all $a,b \in A$.
	
	\[ i_{A} \circ \eta_{A} = \eta_{X \bowtie A}  \iff i_{A} \circ \eta_{A} ( 1_{k} ) = i_{A} ( 1_{A} ) = 1_{X} \otimes 1_{A}
\iff \eta_{X \bowtie A} ( 1_{k} ) = 1_{X} \otimes 1_{A}. \]
	
	8. We prove that
	\[ x \otimes a = \mu_{X \bowtie A} ( (x \otimes 1_{A}) \otimes (1_{X} \otimes a) ) \]
	for all $x \in X, \; a \in A$. Indeed,
	\[ \mu_{X \bowtie A} ( (x \otimes 1_{A}) \otimes (1_{X} \otimes a) ) = x ( 1_{A} \cdot 1_{X} ) \otimes (1_{A})^{1_{X}} a = x
\otimes a. \]
\end{proof}

\modularalgebra*
\begin{proof}
	It is sufficient to verify that if the relation \ref{eq:tensormodule} (\ref{eq:tensormodule_1} in case of a right
$H$-module) holds for $x$ and $y$ from $X$,then it will be true for a product $xy$. Indeed, we have for all $x,y \in X$, $a,b
\in A$ in a case of a left $H$-module
	\[ \alpha_{H,A} ((xy) \otimes (ab)) = \alpha_{H,A} (x \otimes \alpha_{H,A} (y \otimes (ab))) = \alpha_{H,A} (x \otimes (
\sum_{(y)} (-1)^{|y^{''}||a|} \alpha_{H,A} (y^{'} \otimes a) \alpha_{H,A}(y^{''} \otimes b) )) = \]
	\[ = \sum_{(x),(y)} (-1)^{|y^{''}||a| + |x^{''}||y^{'}| + |x^{''}||a|} \alpha_{H,A} ( x^{'} \otimes \alpha_{H,A} (y^{'}
\otimes a) ) \alpha_{H,A} ( x^{''} \otimes \alpha_{H,A} (y^{''} \otimes b) ) = \]
	\[ = \sum_{(x),(y)} (-1)^{|y^{''}||a| + |x^{''}||y^{'}| + |x^{''}||a|} \alpha_{H,A} ((x^{'}y^{'}) \otimes a)
\alpha_{H,A}((x^{''}y^{''}) \otimes b) ) = \]	
	\[ = \mu_{A}\circ\mu_{H\otimes H,A\otimes A}(\Delta(xy)\otimes(a\otimes b))=\mu_{A}\circ\mu_{H\otimes H,A\otimes
A}(\sum_{(xy)}(xy)^{'}\otimes(xy)^{''}\otimes(a\otimes b))= \]
	\[ = \mu_{A} (\sum_{(xy)}(-1)^{|(xy)^{''}||a|} \alpha_{H,A} ((xy)^{'} \otimes a) \otimes \alpha_{H,A} ((xy)^{''} \otimes b)
) = \sum_{(xy)} (-1)^{|(xy)^{''}||a|} \alpha_{H,A}( (xy)^{'} \otimes a) \alpha_{H,A}((xy)^{''} \otimes b ). \]
	
	We have for all $x,y \in H$, $a,b \in A$ in a case of a right $H$-module:
	\[ \beta_{A,H} ((ab) \otimes (xy))=\beta_{A,H}(\beta_{A,H}((ab) \otimes x) \otimes
y)=\sum_{(x)}(-1)^{|b||x^{'}|}\beta_{A,H}((\beta_{A,H}(a\otimes x^{'}) \beta_{A,H}(b\otimes x^{''}))\otimes y)= \]
	\[ = \sum_{(x),(y)} (-1)^{|b||x^{'}|+|y^{'}|(|b|+|x^{''}|)} \beta_{A,H}(\beta_{A,H}(a \otimes x^{'}) \otimes y^{'})
\beta_{A,H}(\beta_{A,H}(b \otimes x^{''}) \otimes y^{''}) = \]
	\[ = \sum_{(x),(y)} (-1)^{|b||x^{'}|+|y^{'}|(|b|+|x^{''}|)} \beta_{A,H}(a \otimes (x^{'}y^{'})) \beta_{A,H}(b \otimes
(x^{''}y^{''})) = \]
	\[ = \mu_{A} \circ \mu_{A\otimes A,H\otimes H} ((a\otimes b)\otimes\Delta(xy)) = \mu_{A} \circ \mu_{A\otimes A,H\otimes H}
(\sum_{(xy)} (a\otimes b) \otimes (xy)^{'} \otimes (xy)^{''})=  \]
	\[ = \mu_{A} (\sum_{(xy)} (-1)^{|b||(xy)^{'}|} \beta_{A,H} (a \otimes (x y)^{'}) \otimes \beta_{A,H} (b \otimes (xy)^{''}) )
= \sum_{(xy)} (-1)^{|b||(xy)^{'}|} \beta_{A,H}(a \otimes (xy)^{'}) \beta_{A,H}(b \otimes (xy)^{''}). \]
	
\end{proof}

\moduleAlgebra*
\begin{proof}
	First we verify that a mapping $\gamma: H \otimes H \to H$ induces on $H$ a structure of a left $H$-module. Indeed, we have
	\[ \gamma( 1_H \otimes x ) = 1_H x S(1_H) = x, \]
	\[ \gamma ( b \otimes \gamma( a \otimes x )) = \gamma ( b \otimes ( \sum_{(a)} (-1)^{|x||a^{''}|} a^{'} x S(a^{''}) )) =
\sum_{(a),(b)} (-1)^{|x||a^{''}| + |b^{''}|(|a^{'}|+|x|+|a^{''}|)} b^{'} a^{'} x S(a^{''}) S(b^{''}) = \]
	\[ = \sum_{(a),(b)} (-1)^{|x||a^{''}| + |b^{''}|(|a^{'}|+|x|)} b^{'} a^{'} x S(b^{''}a^{''}) = \mu \circ ( \mu \otimes id )
\circ ( id \otimes \tau_{H,H} ) \circ ( id \otimes S \otimes id ) ( \Delta(ba) \otimes x ) = \]
	\[ = \sum_{(ba)} (-1)^{(ba)^{''}||x|} (ba)^{'} x S((ba)^{''}) = \gamma ((ba) \otimes x) \]
	for all $a,b,x \in H$.
	
	We show that $H$ is a left module-superalgebra on over $H$. We have
	\[ \gamma(a \otimes 1_H) = \sum_{(a)} a^{'} S(a^{''}) = \epsilon(a) 1_H, \]			
	\[ \sum_{(a)} (-1)^{|x||a^{''}|} \gamma(a^{'} \otimes x) \gamma(a^{''} \otimes y) = \sum_{(a),(a^{'}),(a^{''})}
(-1)^{|x||a^{''}| + |x||(a^{'})^{''}| + |y||(a^{''})^{''}|} (a^{'})^{'} x S((a^{'})^{''}) (a^{''})^{'} y S((a^{''})^{''}) = \]
	\[ = \mu \circ ( \mu \otimes id_{H} ) \circ (\mu \otimes \mu \otimes \mu) \circ (id_{H} \otimes id_{H} \otimes S \otimes
id_{H} \otimes id_{H} \otimes S) \circ (id_{H} \otimes id_{H} \otimes id_{H} \otimes id_{H} \otimes \tau_{H,H}) \circ \]
	\[ \circ (id_{H} \otimes \tau_{H \otimes H \otimes H,H} \otimes id_{H}) \circ (((\Delta \otimes \Delta) \circ \Delta)
\otimes id_{H} \otimes id_{H}) ( a \otimes x \otimes y ) = \]
	\[ = \mu \circ ( \mu \otimes id_{H} ) \circ (\mu \otimes \mu \otimes \mu) \circ (id_{H} \otimes id_{H} \otimes S \otimes
id_{H} \otimes id_{H} \otimes S) \circ (id_{H} \otimes id_{H} \otimes id_{H} \otimes id_{H} \otimes \tau_{H,H}) \circ \]
	\[ \circ (id_{H} \otimes \tau_{H \otimes H \otimes H,H} \otimes id_{H}) \circ (((id_{H} \otimes ((\Delta \otimes id_{H})
\circ \Delta)) \circ \Delta) \otimes id_{H} \otimes id_{H}) ( a \otimes x \otimes y ) = \]
	\[ = \sum_{(a),(a^{''}),((a^{''})^{'})} (-1)^{||x| |a^{''}| + |y||(a^{''})^{''}|} a^{'} x S(((a^{''})^{'})^{'})
((a^{''})^{'})^{''} y S((a^{''})^{''}) = \]
	\[ = \sum_{(a),(a^{''})} (-1)^{|x||a^{''}| + |y||(a^{''})^{''}|} a^{'} x \epsilon((a^{''})^{'}) y S((a^{''})^{''}) = \]
	\[ = \mu \circ (\mu \otimes \mu) \circ ( id_{H} \otimes id_{H} \otimes id_{H} \otimes S ) \circ (id_{H} \otimes \tau_{H,H
\otimes H}) \circ (((id_{H} \otimes (\nu_{k,H} \circ ( \epsilon \otimes id_{H} ) \circ \Delta)) \circ \Delta) \otimes id_{H}
\otimes id_{H}) (a \otimes x \otimes y) = \]
	\[ = \mu \circ (\mu \otimes \mu) \circ ( id_{H} \otimes id_{H} \otimes id_{H} \otimes S ) \circ (id_{H} \otimes \tau_{H,H
\otimes H}) \circ (((id_{H} \otimes id_{H}) \circ \Delta) \otimes id_{H} \otimes id_{H}) (a \otimes x \otimes y) = \]
	\[ = \sum_{(a)} (-1)^{|x||a^{''}| + |y||a^{''}|} a^{'} x y S(a^{''}) = \sum_{(a)} (-1)^{|a^{''}|( |x| + |y| )} a^{'} xy
S(a^{''}) = \gamma(a \otimes (xy)) \]
	for all $a,x,y \in H$.
	
	We verify that a mapping $\delta: H \otimes H \to H$ induces on $H$ a structure of a $H$-module. Indeed, we have
	\[ \delta(x \otimes 1_H) = S(1_H) x 1_H = x, \]
	\[ \delta ( \delta(x \otimes {a}) \otimes b) = \beta ( ( \sum_{(a)} (-1)^{|x||a^{'}|} S(a^{'}) x a^{''} ) \otimes {b} ) =
\]
	\[ = \sum_{(a),(b)} (-1)^{|x||a^{'}| + |b^{'}|(|a^{'}|+|x|+|a^{''}|)} S(b^{'}) S(a^{'}) x a^{''} b^{''} = \sum_{(a),(b)}
(-1)^{|x||a^{'}| + |b^{'}|(|x|+|a^{''}|)} S(a^{'}b^{'}) x a^{''} b^{''} = \]
	\[ = \mu \circ ( \mu \otimes id ) \circ ( \tau_{H,H} \otimes id ) \circ ( id \otimes S \otimes id ) \circ ( x \otimes
\Delta(ab) ) = \sum_{(ab)} (-1)^{|x||(ab)^{'}|} S((ab)^{'}) x (ab)^{''} = \delta ( x \otimes (ab) ) \]
	for all $a,b,x \in H$.
	
	We show that $H$ is a right module-superalgebra on $H$. We have
	\[ \delta(1_H \otimes a) = \sum_{(a)} S(a^{'}) a^{''} = \epsilon(a) 1_H, \]
	
	\[ \sum_{(a)} (-1)^{|y||a^{'}|} \delta(x \otimes a^{'}) \delta(y \otimes a^{''}) = \sum_{(a),(a^{'}),(a^{''})}
(-1)^{|y||a^{'}| + |x||(a^{'})^{'}| + |y||(a^{''})^{'}|} S((a^{'})^{'}) x (a^{'})^{''} S((a^{''})^{'}) y (a^{''})^{''} = \]	
	\[ = \mu \circ ( \mu \otimes id_{H} ) \circ (\mu \otimes \mu \otimes \mu) \circ (S \otimes id_{H} \otimes id_{H} \otimes S
\otimes id_{H} \otimes id_{H}) \circ (\tau_{H,H} \otimes id_{H} \otimes id_{H} \otimes id_{H} \otimes id_{H}) \circ \]
	\[ \circ (id_{H} \otimes \tau_{H,H \otimes H \otimes H} \otimes id_{H}) \circ (id_{H} \otimes id_{H} \otimes ((\Delta
\otimes \Delta) \circ \Delta)) ( x \otimes y \otimes a ) = \]
	\[ = \mu \circ ( \mu \otimes id_{H} ) \circ (\mu \otimes \mu \otimes \mu) \circ (S \otimes id_{H} \otimes id_{H} \otimes S
\otimes id_{H} \otimes id_{H}) \circ (\tau_{H,H} \otimes id_{H} \otimes id_{H} \otimes id_{H} \otimes id_{H}) \circ \]
	\[ \circ (id_{H} \otimes \tau_{H,H \otimes H \otimes H} \otimes id_{H}) \circ (id_{H} \otimes id_{H} \otimes ((id_{H}
\otimes ((\Delta \otimes id_{H}) \circ \Delta)) \circ \Delta)) ( x \otimes y \otimes a ) = \]
	\[ = \sum_{(a),(a^{''}),((a^{''})^{'})} (-1)^{|y|( |a^{'}| + |(a^{''})^{'}| ) + |x||a^{'}| } S(a^{'}) x ((a^{''})^{'})^{'}
S(((a^{''})^{'})^{''}) y (a^{''})^{''} = \]
	\[ = \sum_{(a),(a^{''}),((a^{''})^{'})} (-1)^{|y|( |a^{'}| + |(a^{''})^{'}| ) + |x||a^{'}| } S(a^{'}) x
\epsilon((a^{''})^{'}) y (a^{''})^{''} =  \sum_{(a),(a^{''}),((a^{''})^{'})} (-1)^{|y||a^{'}| + |x||a^{'}| } S(a^{'}) x
\epsilon((a^{''})^{'}) y (a^{''})^{''} = \]
	\[ = \mu \circ (\mu \otimes \mu) \circ ( S \otimes id_{H} \otimes id_{H} \otimes id_{H} ) \circ (\tau_{H \otimes H,H}
\otimes id_{H} ) \circ (id_{H} \otimes id_{H} \otimes ((id_{H} \otimes (\nu_{k,H} \circ ( \epsilon \otimes id_{H} ) \circ
\Delta)) \circ \Delta)) (x \otimes y \otimes a) = \]
	\[ = \mu \circ (\mu \otimes \mu) \circ ( S \otimes id_{H} \otimes id_{H} \otimes id_{H} ) \circ (\tau_{H \otimes H,H}
\otimes id_{H} ) \circ (id_{H} \otimes id_{H} \otimes \Delta) (x \otimes y \otimes a) = \]
	\[ = \sum_{(a)} (-1)^{(|y| + |x|)|a^{'}|} S(a^{'}) xy a^{''} = \delta((xy) \otimes a) \]
	for all $a,x,y \in H$.
\end{proof}

\RightLeftaction*
\begin{proof}
	We verify that a left action $\alpha$ induces a structure of a left $H$-module on $A^{*}$.
	\[ \alpha ( 1_{H} \otimes f )(a) = (-1)^{|1_{H}||f|}f(\alpha_{H,A}(S^{-1}(1_{H}) \otimes a))=f(a),\]
	\[ \alpha ( y \otimes \alpha(x \otimes f) )(a) = (-1)^{|x||f|} \alpha (y \otimes f (\alpha_{H,A}(S^{-1}(x) \otimes ?))) (a)
= \]
	\[ = (-1)^{|y|(|x| + |f|) + |x||f|} f(\alpha_{H,A}(S^{-1}(x)S^{-1}(y) \otimes a)) = (-1)^{(|x| + |y|)|f|}
f(\alpha_{H,A}(S^{-1}(yx) \otimes a)) = \alpha ( (yx) \otimes f ) (a) \]
	for  all $x,y \in H, \; f \in A^{*}, a \in A$.
	
	We verify that a mapping $\alpha: H \otimes (A^{op})^{*} \to (A^{op})^{*}$, which is a left action of $H$ on $(A^{op})^{*}$,
is also a supercoalgebra morphism.
	
	1.
	\[ \epsilon_{A^{*}} \circ \alpha = \epsilon_{H \otimes A^{*}}, \]
	\[ \epsilon_{A^{*}}(\alpha(x \otimes f)) = \alpha(x \otimes f)(1_{A}) = (-1)^{|x||f|} f( \alpha_{H,A} (S^{-1}(x) \otimes
1_{A})) = (-1)^{|x||f|} \epsilon_{H}(S^{-1}(x)) f(1_{A}) = \]
	\[ = \epsilon_{H}(x) \epsilon_{A^{*}}(f) = \epsilon_{H \otimes A^{*}} (x \otimes f) \]
	for all $x \in H, \; f \in (A^{op})^{*}$.
	
	2.
	\[ \Delta_{(A^{op})^{*}} \circ \alpha = ( \alpha \otimes \alpha ) \circ \Delta_{H \otimes (A^{op})^{*}}, \]
	\[ \lambda_{A,A} (\Delta_{(A^{op})^{*}} \circ \alpha(x \otimes f)) (a \otimes b) = \sum_{(\alpha(x \otimes f))} (-1)^{|a||b|
+ |(\alpha(x \otimes f))^{''}||b|} (\alpha(x \otimes f))^{'} (b) (\alpha(x \otimes f))^{''} (a) = \]
	\[ = (-1)^{|a||b|} \sum_{(\alpha(x \otimes f))} (-1)^{|(\alpha(x \otimes f))^{''}||b|} (\alpha(x \otimes f))^{'} (b)
(\alpha(x \otimes f))^{''} (a) = (-1)^{|a||b|}  \lambda_{A,A}  (\Delta_{A^{*}} (\alpha(x \otimes f)))  ( b \otimes a ) = \]
	\[ = (-1)^{|a||b|} \alpha(x \otimes f)  (ba) = (-1)^{|a||b| + |x| |f|} f( \alpha_{H,A} (S^{-1}(x) \otimes (ba))) = \]
	\[ = \sum_{(x)} (-1)^{|a||b| + |x| |f| + |x^{'}||x^{''}| + |b||x^{'}|} f( \alpha_{H,A} (S^{-1}(x^{''}) \otimes b)
\alpha_{H,A} (S^{-1}(x^{'}) \otimes a) ) = \]
	\[ = \sum_{(x)} (-1)^{|a||b| + |x| |f| + |x^{'}||x^{''}| + |b||x^{'}|} \lambda_{A,A}  (\Delta_{A^{*}} (f)) (
\alpha_{H,A}(S^{-1}(x^{''}) \otimes b) \otimes \alpha_{H,A}(S^{-1}(x^{'}) \otimes a)) = \]
	\[ = \sum_{(x),(f)} (-1)^{|a||b| + |x| |f| + |x^{'}||x^{''}| + |b||x^{'}| + |f^{''}| ( |x^{''}| + |b| ) } f^{'} (
\alpha_{H,A} (S^{-1}(x^{''})  \otimes b)) f^{''} ( \alpha_{H,A} (S^{-1}(x^{'}) \otimes a)) = \]
	\[ = \sum_{(x),(f)} (-1)^{|a||b| + |x^{'}| |f^{'}| + |x^{''}| |f^{''}| + |x^{'}||x^{''}| + |b||x^{'}| + |f^{''}| ( |x^{''}|
+ |b| ) } \alpha (x^{''} \otimes f^{'}) (b) \alpha ( x^{'} \otimes f^{''}) (a) = \]	
	\[ = \sum_{(x),(f)} (-1)^{|a|( |x^{''}| + |f^{'}| ) + |f^{''}||x^{''}| + |f^{''}||f^{'}| } \alpha ( x^{'} \otimes f^{''})
(a) \alpha (x^{''} \otimes f^{'}) (b) = \]
	\[ = \lambda_{A,A} (( \alpha \otimes \alpha ) \circ \Delta_{H \otimes (A^{op})^{*}} ( x \otimes f )) ( a \otimes b ) \]
	for all $x \in H, \; f \in (A^{op})^{*}, \; a,b \in A$. Since $\lambda_{A,A}$ is the superspace isomorphism, the result
follows.
	
	We verify that a right action $\beta$ induces a structure of a right $H$-module on $A^{*}$.
	\[ \beta ( f \otimes 1_{H} )(a) = (-1)^{|1_{H}||a|}f(\beta_{A,H}(a \otimes S^{-1}(1_{H})))=f(a),\]
	\[ \beta ( \beta(f \otimes x) \otimes y )(a) = (-1)^{|x||?|} \beta ( f (\beta_{A,H}(? \otimes S^{-1}(x))) \otimes y ) (a) =
\]
	\[ = (-1)^{|y||a| + |x|(|a|+|y|)} f(\beta_{A,H}(a \otimes S^{-1}(y)S^{-1}(x))) = (-1)^{(|x| + |y|)|a|} f(\beta_{A,H}(a
\otimes S^{-1}(xy))) = \beta ( f \otimes xy ) (a) \]
	for  all $x,y \in H, \; f \in A^{*}, a \in A$.
	
	We verify that a mapping $\beta:(A^{op})^{*} \otimes H \to (A^{op})^{*}$, which is a right action of $H$ on $A^{*}$, is also
a supercoalgebra morphism.
	
	1.
	\[ \epsilon_{A^{*}} \circ \beta= \epsilon_{H \otimes A^{*}}, \]
	\[ \epsilon_{A^{*}}(\beta(f \otimes x)) = \beta(f \otimes x)(1_{A}) = (-1)^{|x||1_{A}|} f( \beta_{A,H} (1_{A} \otimes
S^{-1}(x))) = \epsilon_{H}(S^{-1}(x)) f(1_{A}) = \]
	\[ = \epsilon_{H}(x) \epsilon_{A^{*}}(f) = \epsilon_{H \otimes A^{*}} (x \otimes f) \]
	for all $x \in H, \; f \in (A^{op})^{*}$.
	
	2.
	\[ \Delta_{(A^{op})^{*}} \circ \beta = ( \beta \otimes \beta ) \circ \Delta_{(A^{op})^{*} \otimes H}, \]
	\[\lambda_{A,A}(\Delta_{(A^{op})^{*}}\circ\beta(f\otimes x))(a\otimes b)=\sum_{(\beta(f\otimes
x))}(-1)^{|a||b|+|(\beta(f\otimes x))^{''}||b|}(\beta(f\otimes x))^{'}(b)(\beta(f\otimes x))^{''}(a)=\]
	\[=(-1)^{|a||b|}\sum_{(\beta(f\otimes x))}(-1)^{|(\beta(f\otimes x))^{''}||b|}(\beta(f\otimes x))^{'}(b)(\beta(f\otimes
x))^{''}(a)=(-1)^{|a||b|}\lambda_{A,A}(\Delta_{A^{*}}(\beta(f\otimes x)))(b\otimes a)=\]
	\[=(-1)^{|a||b|}\beta(f\otimes x)(ba)=(-1)^{|a||b|+|x|(|b|+|a|)}f(\beta_{A,H} (ba \otimes S^{-1}(x)))=\]
	\[=\sum_{(x)}(-1)^{|a||b|+|x|(|b|+|a|)+|x^{'}||x^{''}|+|x^{''}||a|}f(\beta_{A,H}(b \otimes S^{-1}(x^{''})) \beta_{A,H}(a
\otimes S^{-1}(x^{'})))=\]
	\[=\sum_{(x)}(-1)^{|a||b|+|x|(|b|+|a|)+|x^{'}||x^{''}|+|x^{''}||a|}\lambda_{A,A}(\Delta_{A^{*}}(f))( \beta_{A,H} (b \otimes
S^{-1}(x^{''}))\otimes \beta_{A,H} (a \otimes S^{-1}(x^{'})))=\]
	\[=\sum_{(x)}(-1)^{|a||b|+|x|(|b|+|a|)+|x^{'}||x^{''}|+|x^{''}||a|+|f^{''}|(|b|+|x^{''}|)}f^{'}(\beta_{A,H}(b \otimes
S^{-1}(x^{''})))f^{''}( \beta_{A,H} (a \otimes S^{-1}(x^{'})))=\]
	\[=\sum_{(x)}(-1)^{|a||b|+|x^{'}||b|+|x^{''}||a|+|x^{'}||x^{''}|+|x^{''}||a|+|f^{''}|(|b|+|x^{''}|)}\beta(f^{'}\otimes
x^{''})(b)\beta(f^{''}\otimes x^{'})(a)=\]
	
\[=\sum_{(x)}(-1)^{|a|(|f^{'}|+|x^{''}|)+|f^{''}||x^{''}|+|x^{'}||x^{''}|+(|x^{'}|+|f^{''}|)(|f^{'}|+|x^{''}|)}\beta(f^{'}\otimes
x^{''})(b)\beta(f^{''}\otimes x^{'})(a)=\]
	\[=\sum_{(x)}(-1)^{|f^{'}||f^{''}|+|f^{'}||x^{'}|+|a|(|f^{'}|+|x^{''}|)}\beta(f^{''}\otimes x^{'})(a)\beta(f^{'}\otimes
x^{''})(b)=\]
	\[=\lambda_{A,A}((\beta\otimes\beta)\circ\Delta_{(A^{op})^{*}\otimes H}(f\otimes x))(a\otimes b)\]
	for all $x \in H, \; f \in (A^{op})^{*}, \; a,b \in A$. Since $\lambda_{A,A}$ is the superspace isomorphism, the result
follows.
\end{proof}

\hondualop*
\begin{proof}
	It follows from Proposition \ref{pr:moduleAlgebra} that a Hopf superalgebra $H$ induces a structure of a left (respectively
of a right) module-superalgebra over itself by the mapping \ref{eq:leftmult} (\ref{eq:rightmult}). It follows from Lemma
\ref{cl:action} that we can endow $(H^{op})^{*}$ with a structure of a left (respectively of a right) module-supercoalgebra over
$H$ by the formula \ref{eq:modulesupercoalgebraleft} (\ref{eq:modulesupercoalgebraright}).
	\[\alpha(x\otimes f)(a)=(-1)^{|x||f|}f(\gamma(S^{-1}(x)\otimes
a))=\sum_{(x)}(-1)^{|x||f|+|x^{'}||x^{''}|+|x^{'}||a|}f(S^{-1}(x^{''})aS(S^{-1}(x^{'})))=\]
	\[=\sum_{(x)}(-1)^{|x||f|+|x^{'}||x^{''}|+|x^{'}||a|}f(S^{-1}(x^{''})ax^{'}),\]
	\[\beta(f\otimes x)(a)=(-1)^{|x||a|}f(\delta(a\otimes
S^{-1}(x)))=\sum_{(x)}(-1)^{|x||a|+|x^{'}||x^{''}|+|x^{''}||a|}f(S(S^{-1}(x^{''}))aS^{-1}(x^{'}))=\]
	\[=\sum_{(x)}(-1)^{|x||a|+|x^{'}||x^{''}|+|x^{''}||a|}f(x^{''}aS^{-1}(x^{'}))\]
	for all $x,a\in H, \; f\in(H^{op})^{*}$.
\end{proof}

\rightAction*
\begin{proof}
	Consider a Hopf superalgebra $(H^{cop})^{*}$. It follows from Corollary \ref{cl:caction} that we can endow a Hopf
superalgebra $(((H^{cop})^{*})^{op})^{*}$ with a structure of a right module-supercoalgebra over $(H^{cop})^{*}$ by the formula
	\[\beta^{'}(a\otimes f)(g)=\sum_{(f)}(-1)^{|f||g|+|f^{'}||f^{''}|+|f^{''}||g|}a(f^{''}gS^{*}(f^{'}))\]
	for all $a\in(((H^{cop})^{*})^{op})^{*}, \; f,g \in (H^{cop})^{*}$.
	It follows from Proposition \ref{pr:Dualcopop} that there exists a Hopf superalgebra morphism $ev_{H}:H \to
(((H^{cop})^{*})^{op})^{*}$, such that $ev_{H}(x)=E_{x}$, where $E_{x}(f)=(-1)^{|x||f|}f(x)$ for all $x\in H,f\in H^{*}$. Then
	\[\beta^{'}(ev_{H}(a)\otimes f)(g)=\sum_{(f)}(-1)^{|f||g|+|f^{'}||f^{''}|+|f^{''}||g|}ev_{H}(a)(f^{''}gS^{*}(f^{'}))=\]
	\[=\sum_{(f)}(-1)^{|f||g|+|f^{'}||f^{''}|+|f^{''}||g|+|a|(|f|+|g|)}(f^{''}gS^{*}(f^{'}))(a)=\]
	\[ = \sum_{(f)}(-1)^{|f||g|+|f^{'}||f^{''}|+|f^{''}||g|+|a|(|f|+|g|)} \mu_{(H^{cop})^{*}} ( f^{''}g \otimes S^{*}(f^{'}) )
(a) = \]
	\[ = \sum_{(f)}(-1)^{|f||g|+|f^{'}||f^{''}|+|f^{''}||g|+|a|(|f|+|g|)} (\tau_{H,H} \circ \Delta_{H})^{*} (\lambda_{H,H} (
f^{''}g \otimes S^{*}(f^{'}) )) (a) = \]
	\[ = \sum_{(f),(a)}(-1)^{|f||g|+|f^{'}||f^{''}|+|f^{''}||g|+|a|(|f|+|g|) + |a^{'}||a^{''}| + |f^{'}||a^{''}|}
\mu_{(H^{cop})^{*}} ( f^{''} \otimes g ) (a^{''}) S^{*}(f^{'})(a^{'}) = \]
	\[ = \sum_{(f),(a),(a^{''})}(-1)^{|f||g|+|f^{'}||f^{''}|+|f^{''}||g|+|a|(|f|+|g|) + |a^{'}||a^{''}| + |f^{'}||a^{''}| +
|g||(a^{''})^{''}| + |(a^{''})^{'}||(a^{''})^{''}|} f^{''} ((a^{''})^{''}) g ((a^{''})^{'}) f^{'} (S(a^{'})) = \]
	\[ = \sum_{(a),(a^{''})}(-1)^{|f||g|+|a|(|f|+|g|) + |(a^{''})^{'}||(a^{''})^{''}|} \lambda_{H,H}  (\Delta_{(H^{cop})^{*}}
(f)) (S(a^{'}) \otimes (a^{''})^{''}) g((a^{''})^{'}) = \]
	\[ = \sum_{(a),(a^{''})}(-1)^{|f||g|+|a|(|f|+|g|) + |(a^{''})^{'}||(a^{''})^{''}|} g( f(S(a^{'}) (a^{''})^{''})
(a^{''})^{'}) = \]
	\[ = \sum_{(a),(a^{''})}(-1)^{|a||f| + |(a^{''})^{'}||(a^{''})^{''}|} ev_{H}(f(S(a^{'}) (a^{''})^{''}) (a^{''})^{'}) (g) \]
	for all $a \in H, \; f,g\in(H^{cop})^{*}$.
	
	Consequently,
	\[ \beta^{'}(ev_{H}(a)\otimes f) = \sum_{(a),(a^{''})} (-1)^{|a||f| + |(a^{''})^{'}||(a^{''})^{''}|}  ev_{H}
(f(S(a^{'})(a^{''})^{''}) (a^{''})^{'}). \]
	
	It follows from Remark \ref{rm:antipodeisom} that a superspace morphism $(S^{-1})^{*}:(H^{op})^{*} \to (H^{cop})^{*}$ is a
Hopf superalgebra isomorphism. Thus we have a right action of the Hopf superalgebra $(H^{op})^{*}$ on $H$, defined by the
formula
	\[ \beta(a \otimes f) = ev_{H}^{-1} \circ \beta^{'}(ev_{H}(a)\otimes (S^{-1})^{*}(f)) = ev_{H}^{-1} \circ
\beta^{'}(ev_{H}(a)\otimes (f \circ S^{-1})) = \]
	\[ = \sum_{(a),(a^{''})} (-1)^{|a||f| + |(a^{''})^{'}||(a^{''})^{''}|} ev_{H}^{-1} \circ ev_{H} (f( S^{-1}
(S(a^{'})(a^{''})^{''}) ) (a^{''})^{'}) = \]
	\[ = \sum_{(a),(a^{''})} (-1)^{|a||f| + |(a^{''})^{''}|(|(a^{''})^{'}| + |a^{'}|)} f ( S^{-1} ((a^{''})^{''}) a^{'} )
(a^{''})^{'} \]	
	for all $a \in H, \; f \in (H^{op})^{*}$.
\end{proof}

\QBT*
\begin{proof}
	It follows from Corollary \ref{cl:caction} that there exists a superspace morphism $\alpha$ which induces on $X$ a structure
of a left module-supercoalgebra over $H$. It follows from Corollary \ref{cl:rightAction} that there exists a superspace morphism
$\beta$ which induces on $H$ a structure of a right module-supercoalgebra over $X$.
	
	We verify that relations \ref{eq:twistBi1} - \ref{eq:twistBi5} from Definition \ref{df:twistedB} hold.
	
	Relation \ref{eq:twistBi1}. We show that for all $a \in H$, $f,g \in X$
	\[ a \cdot (fg) = \sum_{(a),(f)} (-1)^{ |f^{'}||f^{''}| + |f^{''}||a^{''}| } (a^{'} \cdot f^{''}) ((a^{''})^{f^{'}} \cdot
g). \]
	
	Indeed, we have for all $x \in H$
	\[ \sum_{(a),(f)} (-1)^{|f^{'}||f^{''}| + |f^{''}||a^{''}|} (a^{'} \cdot f^{''}) ((a^{''})^{f^{'}} \cdot g) (x) = \]
	\[ = \sum_{(a),(f),(x)} (-1)^{|f^{'}||f^{''}| + |f^{''}||a^{''}| + |x^{'}||((a^{''})^{f^{'}} \cdot g| } (a^{'} \cdot f^{''})
(x^{'}) ((a^{''})^{f^{'}} \cdot g) (x^{''}) = \]		
	\[ = \sum_{(a),(f),(x),(a^{'}),(a^{''}),((a^{''})^{''})} (-1)^{|f^{'}||f^{''}| + |f^{''}||a^{''}| +
|x^{'}||((a^{''})^{f^{'}} \cdot g| + |a^{'}||f^{''}| + |(a^{'})^{'}| |(a^{'})^{''}| + |x^{'}||(a^{'})^{'}| + |a^{''}||f^{'}|} *
\]
	\[ * (-1)^{|((a^{''})^{''})^{''}| ( |((a^{''})^{''})^{'}| + |(a^{''})^{'}| )} * \]
	\[ * f^{''} ( S^{-1}((a^{'})^{''}) x^{'} (a^{'})^{'} )  (f^{'}( S^{-1}(((a^{''})^{''})^{''}) (a^{''})^{'} )
((a^{''})^{''})^{'} \cdot g) (x^{''}) = \]				
	\[ = \sum_{(a),(f),(x),(a^{'}),(a^{''}),((a^{''})^{''}),(((a^{''})^{''})^{'})} (-1)^{|f^{'}||f^{''}| + |f^{''}||a^{''}| +
|x^{'}||((a^{''})^{f^{'}} \cdot g| + |a^{'}||f^{''}| + |(a^{'})^{'}| |(a^{'})^{''}| + |x^{'}||(a^{'})^{'}| + |a^{''}||f^{'}|} *
\]
	\[ * (-1)^{|((a^{''})^{''})^{''}| ( |((a^{''})^{''})^{'}| + |(a^{''})^{'}| ) + |(a^{''})^{''})^{'}||g| +
|(((a^{''})^{''})^{'})^{'}||(((a^{''})^{''})^{'})^{''}| + |x^{''}||(((a^{''})^{''})^{'})^{'}| } * \]
	\[ * f^{''} ( S^{-1}((a^{'})^{''}) x^{'} (a^{'})^{'} )  f^{'}( S^{-1}(((a^{''})^{''})^{''}) (a^{''})^{'} ) g( S^{-1}(
(((a^{''})^{''})^{'})^{''} ) x^{''} (((a^{''})^{''})^{'})^{'} ) = \]
	\[ = \sum_{(a),(x),(a^{'}),(a^{''}),((a^{''})^{''}),(((a^{''})^{''})^{'})} (-1)^{|f||a| + |x^{'}|(|a^{''}| + |g|) +
|(a^{'})^{'}| |(a^{'})^{''}| + |x^{'}||(a^{'})^{'}| + |((a^{''})^{''})^{''}| ( |((a^{''})^{''})^{'}| + |(a^{''})^{'}| )} * \]
	\[ * (-1)^{ |(a^{''})^{''})^{'}||g| + |(((a^{''})^{''})^{'})^{'}||(((a^{''})^{''})^{'})^{''}| +
|x^{''}||(((a^{''})^{''})^{'})^{'}| } * \]
	\[ * \lambda_{H,H} (\Delta_{(H^{op})^{*}} (f)) ( S^{-1}((a^{'})^{''}) x^{'} (a^{'})^{'} \otimes S^{-1}(((a^{''})^{''})^{''})
(a^{''})^{'} ) g( S^{-1}( (((a^{''})^{''})^{'})^{''} ) x^{''} (((a^{''})^{''})^{'})^{'} ) = \]
	\[ = \sum_{(a),(x),(a^{'}),(a^{''}),((a^{''})^{''}),(((a^{''})^{''})^{'})} (-1)^{ (|((a^{''})^{''})^{''}| +
(a^{''})^{'}|)(|a^{'}| + |x^{'}|) + |f||a| + |x^{'}|(|a^{''}| + |g|) + |(a^{'})^{'}| |(a^{'})^{''}| + |x^{'}||(a^{'})^{'}|} *
\]
	\[ * (-1)^{ |((a^{''})^{''})^{''}| ( |((a^{''})^{''})^{'}| + |(a^{''})^{'}| ) + |(a^{''})^{''})^{'}||g| +
|(((a^{''})^{''})^{'})^{'}||(((a^{''})^{''})^{'})^{''}| + |x^{''}||(((a^{''})^{''})^{'})^{'}| } * \]
	\[ * f ( S^{-1}(((a^{''})^{''})^{''}) (a^{''})^{'} S^{-1}((a^{'})^{''}) x^{'} (a^{'})^{'}) g( S^{-1}(
(((a^{''})^{''})^{'})^{''} ) x^{''} (((a^{''})^{''})^{'})^{'} ) = \]
	\[ = \sum_{(a),(f),(x),(a^{'}),(a^{''}),((a^{''})^{''}),(((a^{''})^{''})^{'})} (-1)^{ |f||a| + |((a^{''})^{''})^{''}| (
|a^{'}| + |((a^{''})^{''})^{'}| + |(a^{''})^{'}| ) + (a^{''})^{'}| |a^{'}| + |(a^{'})^{'}| |(a^{'})^{''}| } * \]
	\[ * (-1)^{ |x^{'}| ( |g| + |((a^{''})^{''})^{'}| + |(a^{'})^{'}| ) + |g||(a^{''})^{''})^{'}| +
|(((a^{''})^{''})^{'})^{''}||(((a^{''})^{''})^{'})^{'}| + |x^{''}||(((a^{''})^{''})^{'})^{'}| } * \]
	\[ * f ( S^{-1}(((a^{''})^{''})^{''}) (a^{''})^{'} S^{-1}((a^{'})^{''}) x^{'} (a^{'})^{'}) g( S^{-1}(
(((a^{''})^{''})^{'})^{''} ) x^{''} (((a^{''})^{''})^{'})^{'} ) = \]		
	\[ = \nu_{k,k} \circ (\rho_{H^{*},H} \otimes \rho_{H^{*},H}) \circ (id_{X} \otimes (\mu_{H} \circ (\mu_{H} \otimes \mu_{H}))
\otimes id_{X} \otimes (\mu_{H} \circ (id_{H} \otimes \mu_{H}))) \circ \]
	\[ \circ (id_{X} \otimes S^{-1} \otimes id_{H} \otimes S^{-1} \otimes id_{H} \otimes id_{H} \otimes id_{X} \otimes S^{-1}
\otimes id_{H} \otimes id_{H}) \circ \]
	\[ \circ (id_{X} \otimes id_{H} \otimes id_{H} \otimes id_{H} \otimes id_{H} \otimes id_{H} \otimes id_{X} \otimes id_{H}
\otimes \tau_{H,H}) \circ (id_{X} \otimes id_{H} \otimes id_{H} \otimes id_{H} \otimes id_{H} \otimes id_{H} \otimes id_{X}
\otimes \tau_{H,H} \otimes id_{H}) \circ \]
	\[ \circ (id_{X} \otimes id_{H} \otimes id_{H} \otimes id_{H} \otimes id_{H} \otimes id_{H} \otimes \tau_{H \otimes H,X}
\otimes id_{H}) \circ (id_{X} \otimes id_{H} \otimes id_{H} \otimes id_{H} \otimes \tau_{H \otimes H \otimes H \otimes X, H}
\otimes id_{H}) \circ \]
	\[ \circ (id_{X} \otimes id_{H} \otimes id_{H} \otimes \tau_{H,H} \otimes id_{H} \otimes id_{H} \otimes id_{X} \otimes
id_{H} \otimes id_{H}) \circ (id_{X} \otimes id_{H} \otimes \tau_{H \otimes H,H} \otimes id_{H} \otimes id_{H} \otimes id_{X}
\otimes id_{H} \otimes id_{H}) \circ \]
	\[ \circ (id_{X} \otimes \tau_{H \otimes H \otimes H \otimes H \otimes H,H} \otimes id_{X} \otimes id_{H} \otimes id_{H})
\circ (\tau_{H \otimes H \otimes H \otimes H \otimes H \otimes H,X} \otimes id_{X} \otimes id_{H} \otimes id_{H}) \circ \]
	\[ \circ (((id_{H} \otimes id_{H} \otimes id_{H} \otimes ((\Delta_{H} \otimes id_{H}) \circ \Delta_{H})) \circ (\Delta_{H}
\otimes \Delta_{H}) \circ \Delta_{H}) \otimes id_{X} \otimes id_{X} \otimes \Delta_{H}) (a \otimes f \otimes g \otimes x) = \]
	\[ = \nu_{k,k} \circ (\rho_{H^{*},H} \otimes \rho_{H^{*},H}) \circ (id_{X} \otimes (\mu_{H} \circ (\mu_{H} \otimes \mu_{H}))
\otimes id_{X} \otimes (\mu_{H} \circ (id_{H} \otimes \mu_{H}))) \circ \]
	\[ \circ (id_{X} \otimes S^{-1} \otimes id_{H} \otimes S^{-1} \otimes id_{H} \otimes id_{H} \otimes id_{X} \otimes S^{-1}
\otimes id_{H} \otimes id_{H}) \circ \]
	\[ \circ (id_{X} \otimes id_{H} \otimes id_{H} \otimes id_{H} \otimes id_{H} \otimes id_{H} \otimes id_{X} \otimes id_{H}
\otimes \tau_{H,H}) \circ (id_{X} \otimes id_{H} \otimes id_{H} \otimes id_{H} \otimes id_{H} \otimes id_{H} \otimes id_{X}
\otimes \tau_{H,H} \otimes id_{H}) \circ \]
	\[ \circ (id_{X} \otimes id_{H} \otimes id_{H} \otimes id_{H} \otimes id_{H} \otimes id_{H} \otimes \tau_{H \otimes H,X}
\otimes id_{H}) \circ (id_{X} \otimes id_{H} \otimes id_{H} \otimes id_{H} \otimes \tau_{H \otimes H \otimes H \otimes X, H}
\otimes id_{H}) \circ \]
	\[ \circ (id_{X} \otimes id_{H} \otimes id_{H} \otimes \tau_{H,H} \otimes id_{H} \otimes id_{H} \otimes id_{X} \otimes
id_{H} \otimes id_{H}) \circ (id_{X} \otimes id_{H} \otimes \tau_{H \otimes H,H} \otimes id_{H} \otimes id_{H} \otimes id_{X}
\otimes id_{H} \otimes id_{H}) \circ \]
	\[ \circ (id_{X} \otimes \tau_{H \otimes H \otimes H \otimes H \otimes H,H} \otimes id_{X} \otimes id_{H} \otimes id_{H})
\circ (\tau_{H \otimes H \otimes H \otimes H \otimes H \otimes H,X} \otimes id_{X} \otimes id_{H} \otimes id_{H}) \circ \]
	\[ \circ (((id_{H} \otimes id_{H} \otimes id_{H} \otimes ((\Delta_{H} \otimes id_{H}) \circ \Delta_{H})) \circ ( id_{H}
\otimes \Delta_{H} \otimes id_{H} ) \circ (\Delta_{H} \otimes id_{H}) \circ \Delta_{H}) \otimes id_{X} \otimes id_{X} \otimes
\Delta_{H}) (a \otimes f \otimes g \otimes x) = \]
	\[ = \sum_{(a),(a^{'}),((a^{'})^{''}),(a^{''}),((a^{''})^{'}),(x)} (-1)^{ |f||a| + |(a^{''})^{''}| ( |a^{'}| +
|(a^{''})^{'}| ) + |((a^{'})^{''})^{''}| ( |(a^{'})^{'}| + |((a^{'})^{''})^{'}| ) + |((a^{'})^{''})^{'}||(a^{'})^{'}| + |x^{'}|(
|(a^{'})^{'}| + |(a^{''})^{'}| + |g| ) } * \]
	\[ * (-1)^{ |g||(a^{''})^{'}| + |((a^{''})^{'})^{''}||((a^{''})^{'})^{'}| + |x^{''}||((a^{''})^{'})^{'}| } * \]
	\[ * f(S^{-1}((a^{''})^{''}) ((a^{'})^{''})^{''} S^{-1}(((a^{'})^{''})^{'}) x^{'} (a^{'})^{'})  g
(S^{-1}(((a^{''})^{'})^{''}) x^{''} ((a^{''})^{'})^{'}) = \]
	\[ = \sum_{(a),(a^{'}),((a^{'})^{''}),(a^{''}),((a^{''})^{'}),(x)} (-1)^{ |f||a| + |(a^{''})^{''}| ( |a^{'}| +
|(a^{''})^{'}| ) + |(a^{'})^{''}||(a^{'})^{'}| + |x^{'}|( |(a^{'})^{'}| + |(a^{''})^{'}| + |g| ) } * \]
	\[ * (-1)^{ |g||(a^{''})^{'}| + |((a^{''})^{'})^{''}||((a^{''})^{'})^{'}| + |x^{''}||((a^{''})^{'})^{'}| } * \]
	\[ * f(S^{-1}((a^{''})^{''}) \epsilon((a^{'})^{''}) x^{'} (a^{'})^{'})  g (S^{-1}(((a^{''})^{'})^{''}) x^{''}
((a^{''})^{'})^{'}) = \]
	\[ = \sum_{(a),(a^{'}),((a^{'})^{''}),(a^{''}),((a^{''})^{'}),(x)} (-1)^{ |f||a| + |(a^{''})^{''}| ( |a^{'}| +
|(a^{''})^{'}| ) + |x^{'}|( |a^{'}| + |(a^{''})^{'}| + |g| ) } * \]
	\[ * (-1)^{ |g||(a^{''})^{'}| + |((a^{''})^{'})^{''}||((a^{''})^{'})^{'}| + |x^{''}||((a^{''})^{'})^{'}| } * \]
	\[ * f(S^{-1}((a^{''})^{''}) x^{'} a^{'})  g (S^{-1}(((a^{''})^{'})^{''}) x^{''} ((a^{''})^{'})^{'}) = \]
	\[ = \nu_{k,k} \circ (\rho_{H^{*},H} \otimes \rho_{H^{*},H}) \circ (id_{X} \otimes (\mu_{H} \circ (id_{H} \otimes \mu_{H}))
\otimes id_{X} \otimes (\mu_{H} \circ (id_{H} \otimes \mu_{H}))) \circ \]
	\[ \circ ( id_{X} \otimes S^{-1} \otimes id_{H} \otimes id_{H} \otimes id_{X} \otimes S^{-1} \otimes id_{H} \otimes id_{H} )
\circ \]
	\[ \circ (id_{X} \otimes id_{H} \otimes id_{H} \otimes id_{H} \otimes id_{X} \otimes id_{H} \otimes \tau_{H,H}) \circ
(id_{X} \otimes id_{H} \otimes id_{H} \otimes id_{H} \otimes id_{X} \otimes \tau_{H,H} \otimes id_{H}) \circ \]
	\[ \circ (id_{X} \otimes id_{H} \otimes id_{H} \otimes id_{H} \otimes \tau_{H \otimes H,X} \otimes id_{H}) \circ (id_{X}
\otimes id_{H} \otimes \tau_{H \otimes H \otimes H \otimes X,H} \otimes id_{H}) \circ \]
	\[ \circ (id_{X} \otimes \tau_{H \otimes H \otimes H, H} \otimes id_{X} \otimes id_{H} \otimes id_{H}) \circ (\tau_{H
\otimes H \otimes H \otimes H,X} \otimes id_{X} \otimes id_{H} \otimes id_{H}) \circ \]
	\[ \circ ((( id_{H} \otimes \Delta_{H} \otimes id_{H} ) \circ (id_{H} \otimes \Delta_{H}) \circ \Delta_{H}) \otimes id_{X}
\otimes id_{X} \otimes \Delta_{H}) (a \otimes f \otimes g \otimes x) = \]
	\[ = \nu_{k,k} \circ (\rho_{H^{*},H} \otimes \rho_{H^{*},H}) \circ (id_{X} \otimes (\mu_{H} \circ (id_{H} \otimes \mu_{H}))
\otimes id_{X} \otimes (\mu_{H} \circ (id_{H} \otimes \mu_{H}))) \circ \]
	\[ \circ ( id_{X} \otimes S^{-1} \otimes id_{H} \otimes id_{H} \otimes id_{X} \otimes S^{-1} \otimes id_{H} \otimes id_{H} )
\circ \]
	\[ \circ (id_{X} \otimes id_{H} \otimes id_{H} \otimes id_{H} \otimes id_{X} \otimes id_{H} \otimes \tau_{H,H}) \circ
(id_{X} \otimes id_{H} \otimes id_{H} \otimes id_{H} \otimes id_{X} \otimes \tau_{H,H} \otimes id_{H}) \circ \]
	\[ \circ (id_{X} \otimes id_{H} \otimes id_{H} \otimes id_{H} \otimes \tau_{H \otimes H,X} \otimes id_{H}) \circ (id_{X}
\otimes id_{H} \otimes \tau_{H \otimes H \otimes H \otimes X,H} \otimes id_{H}) \circ \]
	\[ \circ (id_{X} \otimes \tau_{H \otimes H \otimes H, H} \otimes id_{X} \otimes id_{H} \otimes id_{H}) \circ (\tau_{H
\otimes H \otimes H \otimes H,X} \otimes id_{X} \otimes id_{H} \otimes id_{H}) \circ \]
	\[ \circ (((\Delta_{H} \otimes \Delta_{H}) \circ \Delta_{H}) \otimes id_{X} \otimes id_{X} \otimes \Delta_{H}) (a \otimes f
\otimes g \otimes x) = \]
	\[ = \sum_{(a),(a^{'}),(a^{''}),(x)} (-1)^{ |f||a| + |(a^{''})^{''}|(|a^{'}| + |(a^{''})^{'}|) + |x^{'}|(|a^{'}| +
|(a^{''})^{'}| + |g|) + |g|(|(a^{'})^{''}| + |(a^{''})^{'}|) + |(a^{''})^{'}||(a^{'})^{''}| + |x^{''}||(a^{'})^{''}| } * \]
	\[ * f(S^{-1}((a^{''})^{''}) x^{'} (a^{'})^{'})  g(S^{-1}((a^{''})^{'}) x^{''} (a^{'})^{''}) = \]
	\[ = \sum_{(a)} (-1)^{ |a| (|f| + |g|) + |a^{'}|(|a^{''}| + |x|) } \mu_{X} ( f \otimes g ) ( S^{-1}(a^{''}) x a^{'} ) = a
\cdot (fg). \]

	Relation \ref{eq:twistBi2}. We show that for all $a \in H$
	\[ a \cdot 1_{X} = \epsilon_{H}(a) 1_{X}. \]
	We have
	\[ a \cdot 1_{X} = a \cdot \epsilon_{H} = \sum_{(a)} (-1)^{|a^{'}||a^{''}| + |?||a^{'}|} \epsilon_{H} (S^{-1}(a^{''})?a^{'})
= \sum_{(a)} \epsilon_{H}(a^{''}) \epsilon_{H}( a^{'} ) \epsilon_{H}(?) = \epsilon_{H}( \sum_{(a)} \epsilon_{H}(a^{'}) a^{''} )
\epsilon_H(?) = \]
	\[ = \epsilon_{H}(a) \epsilon_{H} = \epsilon_{H}(a) 1_{X}. \]

	Relation \ref{eq:twistBi3}. We show that for all $a,b \in H$, $f \in X$
	\[ (ab) ^ {f} = \sum_{(b),(f)} (-1)^{|f^{'}||f^{''}| + |b^{''}||f^{''}|} a^{b^{'} \cdot f^{''}} (b^{''})^{ f^{'}}. \]
	Indeed,
	\[ \sum_{(b),(f)} (-1)^{|f^{'}||f^{''}| + |b^{''}||f^{''}|} a^{b^{'} \cdot f^{''}} (b^{''})^{ f^{'}} = \]
	\[ = \sum_{(b),(f),(b^{''}),((b^{''})^{''})} (-1)^{|f^{'}||f^{''}| + |b^{''}||f^{''}| + |b^{''}||f^{'}| +
|((b^{''})^{''})^{''}|(|((b^{''})^{''})^{'}| + |(b^{''})^{'}|)} * \]
	\[ * a^{\sum_{(b^{'})} (-1)^{ |b^{'}||f^{''}| + |(b^{'})^{'}||(b^{'})^{''}| + |?||(b^{'})^{'}| }
f^{''}(S^{-1}((b^{'})^{''})?(b^{'})^{'}) } f^{'}(S^{-1}(((b^{''})^{''})^{''}) (b^{''})^{'}) ((b^{''})^{''})^{'} = \]
	\[ = \sum_{(b),(f),(b^{''}),((b^{''})^{''}),(b^{'})} (-1)^{|f^{'}||f^{''}| + |b^{''}||f^{''}| + |b^{''}||f^{'}| +
|((b^{''})^{''})^{''}|(|((b^{''})^{''})^{'}| + |(b^{''})^{'}|)} * \]
	\[ * (-1)^{|b^{'}||f^{''}| + |(b^{'})^{'}||(b^{'})^{''}| + |(b^{'})^{'}|(|a^{'}| + |(a^{''})^{''}|) + |a|(|f^{''}| +
|b^{'}|) + |(a^{''})^{''}|(|(a^{''})^{'}|+|a^{'}|)} * \]
	\[ * f^{''}(S^{-1}((b^{'})^{''}) S^{-1}((a^{''})^{''})a^{'} (b^{'})^{'}) (a^{''})^{'} f^{'}(S^{-1}(((b^{''})^{''})^{''})
(b^{''})^{'}) ((b^{''})^{''})^{'} = \]		
	\[ = \sum_{(b),(b^{''}),((b^{''})^{''}),(b^{'})} (-1)^{|b||f| + |a||f| +
|(a^{''})^{'}|(|(b^{''})^{'}|+|((b^{''})^{''})^{''}|) + |((b^{''})^{''})^{''}|(|((b^{''})^{''})^{'}| + |(b^{''})^{'}|)} * \]
	\[ * (-1)^{|(b^{'})^{'}||(b^{'})^{''}| + |(b^{'})^{'}|(|a^{'}| + |(a^{''})^{''}|) + |a||b^{'}| +
|(a^{''})^{''}|(|(a^{''})^{'}|+|a^{'}|)} * \]
	\[ * \lambda_{H,H} (\Delta_{(H^{op})^{*}} (f)) (S^{-1}((b^{'})^{''}) S^{-1}((a^{''})^{''})a^{'} (b^{'})^{'} \otimes
S^{-1}(((b^{''})^{''})^{''}) (b^{''})^{'}) (a^{''})^{'} ((b^{''})^{''})^{'} = \]
	\[ = \sum_{(b),(b^{''}),((b^{''})^{''}),(b^{'})} (-1)^{ (|(b^{'})^{''}| + |(a^{''})^{''}| + |a^{'}| + |(b^{'})^{'}|) (
|((b^{''})^{''})^{''}| + |(b^{''})^{'}| ) + |b||f| + |a||f| + |(a^{''})^{'}|(|(b^{''})^{'}|+|((b^{''})^{''})^{''}|) } * \]
	\[ * (-1)^{|((b^{''})^{''})^{''}|(|((b^{''})^{''})^{'}| + |(b^{''})^{'}|) + |(b^{'})^{'}||(b^{'})^{''}| +
|(b^{'})^{'}|(|a^{'}| + |(a^{''})^{''}|) + |a||b^{'}| + |(a^{''})^{''}|(|(a^{''})^{'}|+|a^{'}|)} * \]
	\[ * f(S^{-1}(((b^{''})^{''})^{''}) (b^{''})^{'} S^{-1}((b^{'})^{''}) S^{-1}((a^{''})^{''})a^{'} (b^{'})^{'}) (a^{''})^{'}
((b^{''})^{''})^{'} = \]
	\[ \sum_{(b),(b^{''}),((b^{''})^{''}),(b^{'})} (-1)^{ |f|(|a| + |b|) + |((b^{''})^{''})^{''}| (|((b^{''})^{''})^{'}| +
|(b^{''})^{'}| + |b^{'}| + |a| ) + |(b^{''})^{'}|( |b^{'}| + |a| ) } * \]
	\[ * (-1)^{|(b^{'})^{''}|(|(b^{'})^{'}| + |a|) + |(a^{''})^{''}|(|(a^{''})^{'}|+|a^{'}|) + |(b^{'})^{'}||(a^{''})^{'}|} *
\]
	\[ * f(S^{-1}(((b^{''})^{''})^{''}) (b^{''})^{'} S^{-1}((b^{'})^{''}) S^{-1}((a^{''})^{''})a^{'} (b^{'})^{'}) (a^{''})^{'}
((b^{''})^{''})^{'} = \]
	\[ = \nu_{k,H} \circ ((\rho_{H^{*},H} \circ (id_{X} \otimes  \mu_{H})) \otimes id_{H}) \circ (id_{X} \otimes \mu_{H} \otimes
id_{H} \otimes id_{H}) \circ (id_{X} \otimes \mu_{H} \otimes \mu_{H} \otimes \mu_{H} \otimes \mu_{H}) \circ \]
	\[ \circ (id_{X} \otimes S^{-1} \otimes id_{H} \otimes S^{-1} \otimes S^{-1} \otimes id_{H} \otimes id_{H} \otimes id_{H}
\otimes id_{H}) \circ \]
	\[ \circ (id_{X} \otimes id_{H} \otimes id_{H} \otimes id_{H} \otimes id_{H} \otimes id_{H} \otimes \tau_{H,H} \otimes
id_{H}) \circ (id_{X} \otimes id_{H} \otimes id_{H} \otimes id_{H} \otimes \tau_{H \otimes H,H} \otimes id_{H} \otimes id_{H})
\circ \]
	\[ \circ (id_{X} \otimes id_{H} \otimes id_{H} \otimes \tau_{H \otimes H \otimes H \otimes H,H} \otimes id_{H}) \circ
(id_{X} \otimes id_{H} \otimes \tau_{H \otimes H \otimes H \otimes H \otimes H,H} \otimes id_{H}) \circ \]
	\[ \circ (id_{X} \otimes \tau_{H \otimes H \otimes H \otimes H \otimes H \otimes H \otimes H,H}) \circ \tau_{H \otimes H
\otimes H \otimes H \otimes H \otimes H \otimes H \otimes H,X} \circ \]
	\[ \circ (((id_{H} \otimes \Delta_{H}) \circ \Delta_{H}) \otimes ((id_{H} \otimes id_{H} \otimes id_{H} \otimes \Delta_{H})
\circ (\Delta_{H} \otimes \Delta_{H}) \circ \Delta_{H}) \otimes id_{X}) (a \otimes b \otimes f) = \]
	\[ = \nu_{k,H} \circ ((\rho_{H^{*},H} \circ (id_{X} \otimes  \mu_{H})) \otimes id_{H}) \circ (id_{X} \otimes \mu_{H} \otimes
id_{H} \otimes id_{H}) \circ (\mu_{H} \otimes \mu_{H} \otimes \mu_{H} \otimes \mu_{H}) \circ \]
	\[ \circ (id_{X} \otimes S^{-1} \otimes id_{H} \otimes S^{-1} \otimes S^{-1} \otimes id_{H} \otimes id_{H} \otimes id_{H}
\otimes id_{H}) \circ \]
	\[ \circ (id_{X} \otimes id_{H} \otimes id_{H} \otimes id_{H} \otimes id_{H} \otimes id_{H} \otimes \tau_{H,H} \otimes
id_{H}) \circ (id_{X} \otimes id_{H} \otimes id_{H} \otimes id_{H} \otimes \tau_{H \otimes H,H} \otimes id_{H} \otimes id_{H})
\circ \]
	\[ \circ (id_{X} \otimes id_{H} \otimes id_{H} \otimes \tau_{H \otimes H \otimes H \otimes H,H} \otimes id_{H}) \circ
(id_{X} \otimes id_{H} \otimes \tau_{H \otimes H \otimes H \otimes H \otimes H,H} \otimes id_{H}) \circ \]
	\[ \circ (id_{X} \otimes \tau_{H \otimes H \otimes H \otimes H \otimes H \otimes H \otimes H,H}) \circ \tau_{H \otimes H
\otimes H \otimes H \otimes H \otimes H \otimes H \otimes H,X} \circ \]
	\[ \circ (((id_{H} \otimes \Delta_{H}) \circ \Delta_{H}) \otimes ((id_{H} \otimes id_{H} \otimes id_{H} \otimes \Delta_{H})
\circ (id_{H} \otimes \Delta_{H} \otimes id_{H}) \circ (id_{H} \otimes \Delta_{H}) \circ \Delta_{H}) \otimes id_{X}) (a \otimes
b \otimes f) = \]
	\[ = \sum_{(a),(a^{''}),(b),(b^{''}),((b^{''})^{'}),((b^{''})^{''})} (-1)^{|f|(|a| + |b|) + |((b^{''})^{''})^{''}|( |a| +
|b^{'}| + |(b^{''})^{'}| + |((b^{''})^{''})^{'}| ) + |((b^{''})^{'})^{''}|( |a| + |b^{'}| + |((b^{''})^{'})^{'}| ) } * \]
	\[ * (-1)^{|((b^{''})^{'})^{'}|( |a| + |b^{'}| ) + |(a^{''})^{''}|(|a^{'}| + |(a^{''})^{'}|) + |b^{'}||(a^{''})^{'}|} * \]
	\[ * f(S^{-1}(((b^{''})^{''})^{''})  ((b^{''})^{'})^{''}  S^{-1}(((b^{''})^{'})^{'})  S^{-1}((a^{''})^{''}) a^{'} b^{'})
(a^{''})^{'} ((b^{''})^{''})^{'} = \]
	\[ = \sum_{(a),(a^{''}),(b),(b^{''}),((b^{''})^{'}),((b^{''})^{''})} (-1)^{|f|(|a| + |b|) + |((b^{''})^{''})^{''}|( |a| +
|b^{'}| + |(b^{''})^{'}| + |((b^{''})^{''})^{'}| ) + |(b^{''})^{'}|( |a| + |b^{'}| ) } * \]
	\[ * (-1)^{|(a^{''})^{''}|(|a^{'}| + |(a^{''})^{'}|) + |b^{'}||(a^{''})^{'}|} * \]
	\[ * f(S^{-1}(((b^{''})^{''})^{''})  \epsilon_{H}((b^{''})^{'})  S^{-1}((a^{''})^{''}) a^{'} b^{'}) (a^{''})^{'}
((b^{''})^{''})^{'} = \]	
	\[ = \nu_{k,H} \circ ((\rho_{H^{*},H} \circ (id_{X} \otimes (\mu_{H} \circ (\mu_{H} \otimes \mu_{H})))) \otimes \mu_{H})
\circ (id_{X} \otimes S^{-1} \otimes S^{-1} \otimes id_{H} \otimes id_{H} \otimes id_{H} \otimes id_{H}) \circ \]
	\[ \circ (id_{X} \otimes id_{H} \otimes id_{H} \otimes id_{H} \otimes \tau_{H,H} \otimes id_{H}) \circ (id_{X} \otimes
id_{H} \otimes \tau_{H \otimes H, H} \otimes id_{H} \otimes id_{H}) \circ \]
	\[ \circ (id_{X} \otimes \tau_{H \otimes H \otimes H \otimes H \otimes H, H}) \circ \tau_{H \otimes H \otimes H \otimes H
\otimes H \otimes H,X} \circ \]
	\[ \circ (((id_{H} \otimes \Delta_{H}) \circ \Delta_{H}) \otimes ((id_{H} \otimes (\Delta_{H} \circ \nu_{k,H} \circ
(\epsilon_{H} \otimes id_{H}) \circ \Delta_{H})) \circ \Delta_{H}) \otimes id_{X}) (a \otimes b \otimes f) = \]		
	\[ = \nu_{k,H} \circ ((\rho_{H^{*},H} \circ (id_{X} \otimes (\mu_{H} \circ (\mu_{H} \otimes \mu_{H})))) \otimes \mu_{H})
\circ (id_{X} \otimes S^{-1} \otimes S^{-1} \otimes id_{H} \otimes id_{H} \otimes id_{H} \otimes id_{H}) \circ \]
	\[ \circ (id_{X} \otimes id_{H} \otimes id_{H} \otimes id_{H} \otimes \tau_{H,H} \otimes id_{H}) \circ (id_{X} \otimes
id_{H} \otimes \tau_{H \otimes H, H} \otimes id_{H} \otimes id_{H}) \circ \]
	\[ \circ (id_{X} \otimes \tau_{H \otimes H \otimes H \otimes H \otimes H, H}) \circ \tau_{H \otimes H \otimes H \otimes H
\otimes H \otimes H,X} \circ \]
	\[ \circ (((id_{H} \otimes \Delta_{H}) \circ \Delta_{H}) \otimes ((id_{H} \otimes (\Delta_{H} \circ id_{H})) \circ
\Delta_{H}) \otimes id_{X}) (a \otimes b \otimes f) = \]
	\[ = \sum_{(a),(b),(a^{''}),(b^{''})} (-1)^{ |f|(|a| + |b|) + |(b^{''})^{''}|( |a| + |b^{'}| + |(b^{''})^{'}| ) +
|(a^{''})^{''}|(|a^{'}| + |(a^{''})^{'}|) + |b^{'}||(a^{''})^{'}| } * \]
	\[ * f(S^{-1}((b^{''})^{''}) S^{-1}((a^{''})^{''}) a^{'} b^{'}) (a^{''})^{'} (b^{''})^{'} = \]		
	\[ = \sum_{(a),(b),(a^{''}),(b^{''})} (-1)^{ |f|(|a| + |b|) + |(b^{''})^{''}|( |a^{'}| + |(a^{''})^{'}| + |b^{'}| +
|(b^{''})^{'}| ) + |(a^{''})^{''}|(|a^{'}| + |(a^{''})^{'}| + |b^{'}| + |(b^{''})^{'}|) + |b^{'}||a^{''}| +
|(b^{''})^{'}||(a^{''})^{''}|} * \]
	\[ * f((S^{-1}((a^{''})^{''}(b^{''})^{''})) a^{'} b^{'}) (a^{''})^{'} (b^{''})^{'} = \]	
	\[ = \nu_{k,H} \circ (\rho_{H^{*},H} \otimes id_{H}) \circ (id_{X} \otimes \mu_{H} \otimes id_{H}) \circ (id_{X} \otimes
\tau_{H \otimes H,H}) \circ \tau_{H \otimes H \otimes H,X} \circ ( id_{H} \otimes id_{H} \otimes S^{-1} \otimes id_{X} ) \circ
\]
	\[ \circ (\sum_{(a),(b),(a^{''}),(b^{''})} (-1)^{|b^{'}||a^{''}| + |(b^{''})^{'}||(a^{''})^{''}|} a^{'} b^{'} \otimes
(a^{''})^{'} (b^{''})^{'} \otimes (a^{''})^{''}(b^{''})^{''} \otimes f) = \]	
	\[ = \nu_{k,H} \circ (\rho_{H^{*},H} \otimes id_{H}) \circ (id_{X} \otimes \mu_{H} \otimes id_{H}) \circ (id_{X} \otimes
\tau_{H \otimes H,H}) \circ \tau_{H \otimes H \otimes H,X} \circ ( id_{H} \otimes id_{H} \otimes S^{-1} \otimes id_{X} ) \circ
\]
	\[ \circ (\sum_{(ab),((ab)^{''})} (ab)^{'} \otimes ((ab)^{''})^{'} \otimes ((ab)^{''})^{''} \otimes f) = \]	
	\[ = \sum_{(ab),((ab)^{''})} (-1)^{|ab||f| + |((ab)^{''})^{''}|( |((ab)^{''})^{'}| + |(ab)^{'}| )}
f(S^{-1}(((ab)^{''})^{''}) (ab)^{'} ) ((ab)^{''})^{'} = (ab)^{f}. \]
	
	Relation \ref{eq:twistBi4}. We show that for all $f \in X$
	\[ (1_{H})^{f} = \epsilon_{X}(f) 1_{H}. \]
	We have
	\[ (1_{H})^{f} = f(S^{-1}(1_{H})1_{H})1_{H} = f(1_H) 1_{H} = \epsilon_{X}(f) 1_{H}. \]
	
	Relation \ref{eq:twistBi5}. We verify that for all $a \in H$, $f \in X$
	\[ \sum_{(a),(f)} (-1)^{|f^{'}||f^{''}| + |a^{''}||f^{''}|} (a^{'})^{f^{''}} \otimes a^{''} \cdot f^{'} = \sum_{(a),(f)}
(-1)^{|a^{'}||a^{''}| + |f^{'}||a^{'}|} (a^{''})^{f^{'}} \otimes a^{'} \cdot f^{''}. \]
	
	We have
	\[ \sum_{(a),(f)} (-1)^{|f^{'}||f^{''}| + |a^{''}||f^{''}|} (a^{'})^{f^{''}} \otimes a^{''} \cdot f^{'} = \]
	\[ = \sum_{(a),(f),(a^{'}),((a^{'})^{''}),(a^{''})} (-1)^{|f^{'}||f^{''}| + |a^{''}||f^{''}| + |a^{'}||f^{''}| +
|((a^{'})^{''})^{''}|(|((a^{'})^{''})^{'}| + |(a^{'})^{'}|) + |a^{''}||f^{'}| + |(a^{''})^{'}||(a^{''})^{''}| +
|?||(a^{''})^{'}|} * \]
	\[ * f^{''}(S^{-1}(((a^{'})^{''})^{''}) (a^{'})^{'} ) ((a^{'})^{''})^{'} \otimes f^{'}(S^{-1}((a^{''})^{''}) ? (a^{''})^{'}
) = \]
	\[ = \sum_{(a),(a^{'}),((a^{'})^{''}),(a^{''})} (-1)^{|a||f| + |((a^{'})^{''})^{''}|(|((a^{'})^{''})^{'}| + |(a^{'})^{'}|) +
|((a^{'})^{''})^{'}|( |a^{''}| + |?| ) + |(a^{''})^{'}||(a^{''})^{''}| + |?||(a^{''})^{'}|} * \]
	\[ * ((a^{'})^{''})^{'} \otimes \lambda_{H,H} ( \Delta_{(H^{op})^{*}} (f)) ( S^{-1}(((a^{'})^{''})^{''}) (a^{'})^{'} \otimes
S^{-1}((a^{''})^{''}) ? (a^{''})^{'} ) = \]
	\[ = \sum_{(a),(a^{'}),((a^{'})^{''}),(a^{''})} (-1)^{|a||f| + |((a^{'})^{''})^{''}|(|((a^{'})^{''})^{'}| + |(a^{'})^{'}|) +
|((a^{'})^{''})^{'}|( |a^{''}| + |?| ) + |(a^{''})^{'}||(a^{''})^{''}| + |?||(a^{''})^{'}|} * \]
	\[ * (-1)^{(|((a^{'})^{''})^{''}| + |(a^{'})^{'}|)(|(a^{''})^{''}| + |?| + |(a^{''})^{'}|)} * \]		
	\[ * ((a^{'})^{''})^{'} \otimes f(S^{-1}((a^{''})^{''}) ? (a^{''})^{'} S^{-1}(((a^{'})^{''})^{''}) (a^{'})^{'}) = \]			

	\[ = \sum_{(a),(a^{'}),((a^{'})^{''}),(a^{''})} (-1)^{ |((a^{'})^{''})^{'}||(a^{'})^{'}| + |f|(|a| + |((a^{'})^{''})^{'}|) +
|(a^{''})^{''}|(|(a^{'})^{'}| + |((a^{'})^{''})^{''}| + |(a^{''})^{'}|) + |?|(|(a^{'})^{'}| + |((a^{'})^{''})^{''}| +
|(a^{''})^{'}|)} * \]
	\[ * (-1)^{|(a^{''})^{'}|(|((a^{'})^{''})^{''}| + |(a^{'})^{'}|) + |((a^{'})^{''})^{''}||(a^{'})^{'}|} * \]
	\[ * ((a^{'})^{''})^{'} \otimes f(S^{-1}((a^{''})^{''}) ? (a^{''})^{'} S^{-1}(((a^{'})^{''})^{''}) (a^{'})^{'}) = \]
	\[ = (id_{H} \otimes (\rho_{H^{*},H} \circ (id_{X} \otimes (\mu_{H} \circ (id_{H} \otimes \mu_{H}) \circ (id_{H} \otimes
\mu_{H} \otimes \mu_{H}))))) \circ (id_{H} \otimes id_{X} \otimes S^{-1} \otimes id_{H} \otimes id_{H} \otimes S^{-1} \otimes
id_{H}) \circ \]
	\[ \circ (id_{H} \otimes id_{X} \otimes id_{H} \otimes id_{H} \otimes id_{H} \otimes \tau_{H,H}) \circ (id_{H} \otimes
id_{X} \otimes id_{H} \otimes id_{H} \otimes \tau_{H \otimes H,H}) \circ \]
	\[ \circ (id_{H} \otimes id_{X} \otimes id_{H} \otimes \tau_{H \otimes H \otimes H,H}) \circ (id_{H} \otimes id_{X} \otimes
\tau_{H \otimes H \otimes H,H} \otimes id_{H}) \circ \]
	\[ \circ (id_{H} \otimes \tau_{H \otimes H \otimes H \otimes H,X} \otimes id_{H}) \circ (\tau_{H,H} \otimes id_{H} \otimes
id_{H} \otimes id_{H} \otimes id_{X} \otimes id_{H}) \circ \]
	\[ \circ (((id_{H} \otimes \Delta_{H} \otimes id_{H} \otimes id_{H}) \circ (\Delta_{H} \otimes \Delta_{H}) \circ \Delta_{H})
\otimes id_{X} \otimes id_{H}) (a \otimes f \otimes ?) = \]
	\[ = (id_{H} \otimes (\rho_{H^{*},H} \circ (id_{X} \otimes (\mu_{H} \circ (id_{H} \otimes \mu_{H}) \circ (id_{H} \otimes
\mu_{H} \otimes \mu_{H}))))) \circ (id_{H} \otimes id_{X} \otimes S^{-1} \otimes id_{H} \otimes id_{H} \otimes S^{-1} \otimes
id_{H}) \circ \]
	\[ \circ (id_{H} \otimes id_{X} \otimes id_{H} \otimes id_{H} \otimes id_{H} \otimes \tau_{H,H}) \circ (id_{H} \otimes
id_{X} \otimes id_{H} \otimes id_{H} \otimes \tau_{H \otimes H,H}) \circ \]
	\[ \circ (id_{H} \otimes id_{X} \otimes id_{H} \otimes \tau_{H \otimes H \otimes H,H}) \circ (id_{H} \otimes id_{X} \otimes
\tau_{H \otimes H \otimes H,H} \otimes id_{H}) \circ \]
	\[ \circ (id_{H} \otimes \tau_{H \otimes H \otimes H \otimes H,X} \otimes id_{H}) \circ (\tau_{H,H} \otimes id_{H} \otimes
id_{H} \otimes id_{H} \otimes id_{X} \otimes id_{H}) \circ \]
	\[ \circ (((id_{H} \otimes id_{H} \otimes \Delta_{H} \otimes id_{H}) \circ (\Delta_{H} \otimes \Delta_{H}) \circ \Delta_{H})
\otimes id_{X} \otimes id_{H}) (a \otimes f \otimes ?) = \]
	\[ = \sum_{(a),(a^{'}),(a^{''}),((a^{''})^{'})} (-1)^{ |(a^{'})^{'}||(a^{'})^{''}| + |f|(|a| + |(a^{'})^{''}|) +
|(a^{''})^{''}|(|(a^{'})^{'}| + |(a^{''})^{'}|) + |?|(|(a^{'})^{'}| + |(a^{''})^{'}|) } * \]
	\[ * (-1)^{|((a^{''})^{'})^{''}|(|(a^{'})^{'}| + |((a^{''})^{'})^{'}|) + |((a^{''})^{'})^{'}||(a^{'})^{'}|} * \]
	\[ * (a^{'})^{''} \otimes f(S^{-1}((a^{''})^{''}) ? ((a^{''})^{'})^{''} S^{-1}(((a^{''})^{'})^{'}) (a^{'})^{'}) = \]
	\[ = \sum_{(a),(a^{'}),(a^{''})} (-1)^{ |(a^{'})^{'}||(a^{'})^{''}| + |f|(|a| + |(a^{'})^{''}|) +
|(a^{''})^{''}|(|(a^{'})^{'}| + |(a^{''})^{'}|) + |?|(|(a^{'})^{'}| + |(a^{''})^{'}|) + |(a^{'})^{'}||(a^{''})^{'}| } * \]
	\[ * (a^{'})^{''} \otimes f(S^{-1}((a^{''})^{''}) ? \epsilon_{H}((a^{''})^{'}) (a^{'})^{'}) = \]
	\[ = \sum_{(a),(a^{'}),(a^{''})} (-1)^{ |(a^{'})^{'}||(a^{'})^{''}| + |f|(|a| + |(a^{'})^{''}|) + |a^{''}||(a^{'})^{'}| +
|?||(a^{'})^{'}| } * \]
	\[ * (a^{'})^{''} \otimes f(S^{-1}(\epsilon_{H}((a^{''})^{'})(a^{''})^{''}) ? (a^{'})^{'}) = \]	
	\[ = \sum_{(a),(a^{'})} (-1)^{ |(a^{'})^{'}||(a^{'})^{''}| + |f|(|a| + |(a^{'})^{''}|) + |a^{''}||(a^{'})^{'}| +
|?||(a^{'})^{'}| } (a^{'})^{''} \otimes f(S^{-1}(a^{''}) ? (a^{'})^{'}). \]	
	
	Remark 1. We have by the coassociativity
	\[ (id_{H} \otimes \Delta_{H} \otimes id_{H} \otimes id_{H}) \circ (\Delta_{H} \otimes \Delta_{H}) \circ \Delta_{H} =
(id_{H} \otimes ((\Delta_{H} \otimes id_{H}) \circ \Delta_{H}) \otimes id_{H} ) \circ (id_{H} \otimes \Delta_{H}) \circ
\Delta_{H} = \]
	\[ = (id_{H} \otimes ((id_{H} \otimes \Delta_{H}) \circ \Delta_{H}) \otimes id_{H} ) \circ ( id_{H} \otimes \Delta_{H})
\circ \Delta_{H} =  (id_{H} \otimes id_{H} \otimes \Delta_{H} \otimes id_{H}) \circ (\Delta_{H} \otimes \Delta_{H}) \circ
\Delta_{H}. \qed \]
	
	We have for all $a \in H$, $f \in X$
	\[ \sum_{(a),(f)} (-1)^{|a^{'}||a^{''}| + |f^{'}||a^{'}|} (a^{''})^{f^{'}} \otimes a^{'} \cdot f^{''} = \]
	\[ = \sum_{(a),(f),(a^{'}),(a^{''}),((a^{''})^{''})} (-1)^{|a^{'}||a^{''}| + |f^{'}||a^{'}| + |a^{''}||f^{'}| +
|((a^{''})^{''})^{''}|(|((a^{''})^{''})^{'}| + |(a^{''})^{'}|) + |a^{'}||f^{''}| + |(a^{'})^{'}||(a^{'})^{''}| +
|?||(a^{'})^{'}| } * \]
	\[ * f^{'}(S^{-1}(((a^{''})^{''})^{''}) (a^{''})^{'} ) ((a^{''})^{''})^{'} \otimes f^{''}(S^{-1}((a^{'})^{''}) ? (a^{'})^{'}
) = \]
	\[ = \sum_{(a),(a^{'}),(a^{''}),((a^{''})^{''})} (-1)^{|a^{'}||a^{''}| + |f||a| +
|((a^{''})^{''})^{''}|(|((a^{''})^{''})^{'}| + |(a^{''})^{'}|) + |((a^{''})^{''})^{'}|(|a^{'}|+|?|) +
|(a^{'})^{'}||(a^{'})^{''}| + |?||(a^{'})^{'}| } * \]
	\[ * ((a^{''})^{''})^{'} \otimes \lambda_{H,H} ( \Delta_{H^{*}} (f)) (S^{-1}(((a^{''})^{''})^{''}) (a^{''})^{'} \otimes
S^{-1}((a^{'})^{''}) ? (a^{'})^{'}) = \]
	\[ = \sum_{(a),(a^{'}),(a^{''}),((a^{''})^{''})} (-1)^{|a^{'}||a^{''}| + |f||a| +
|((a^{''})^{''})^{''}|(|((a^{''})^{''})^{'}| + |(a^{''})^{'}|) + |((a^{''})^{''})^{'}|(|a^{'}|+|?|) +
|(a^{'})^{'}||(a^{'})^{''}| + |?||(a^{'})^{'}| } * \]
	\[ * ((a^{''})^{''})^{'} \otimes f(S^{-1}(((a^{''})^{''})^{''}) (a^{''})^{'} S^{-1}((a^{'})^{''}) ? (a^{'})^{'}) = \]
	\[ = \sum_{(a),(a^{'}),(a^{''}),((a^{''})^{''})} (-1)^{|((a^{''})^{''})^{'}|( |a^{'}| + |(a^{''})^{'}|) + |f|(|a| +
|((a^{''})^{''})^{'}|) + |((a^{''})^{''})^{''}|( |a^{'}| + |(a^{''})^{'}| ) + |(a^{''})^{'}||a^{'}| +
|(a^{'})^{''}||(a^{'})^{'}| +  |(a^{'})^{'}||?| } * \]
	\[ * ((a^{''})^{''})^{'} \otimes f(S^{-1}(((a^{''})^{''})^{''}) (a^{''})^{'} S^{-1}((a^{'})^{''}) ? (a^{'})^{'}) = \]		

	\[ = (id_{H} \otimes (\rho_{H^{*},H} \circ (id_{X} \otimes (\mu_{H} \circ (id_{H} \otimes \mu_{H}) \circ (id_{H} \otimes
\mu_{H} \otimes \mu_{H}))))) \circ (id_{H} \otimes id_{X} \otimes S^{-1} \otimes id_{H} \otimes S^{-1} \otimes id_{H}  \otimes
id_{H}) \circ \]
	\[ \circ (id_{H} \otimes id_{X} \otimes id_{H} \otimes id_{H} \otimes id_{H} \otimes \tau_{H,H}) \circ (id_{H} \otimes
id_{X} \otimes id_{H} \otimes id_{H} \otimes \tau_{H,H} \otimes id_{H}) \circ \]
	\[ \circ (id_{H} \otimes id_{X} \otimes id_{H} \otimes \tau_{H \otimes H,H} \otimes id_{H}) \circ (id_{H} \otimes id_{X}
\otimes \tau_{H \otimes H \otimes H,H} \otimes id_{H}) \circ \]
	\[ \circ (id_{H} \otimes \tau_{H \otimes H \otimes H \otimes H,X} \otimes id_{H}) \circ (\tau_{H \otimes H \otimes H,H}
\otimes id_{H} \otimes id_{X} \otimes id_{H}) \circ \]
	\[ \circ (((id_{H} \otimes id_{H} \otimes id_{H} \otimes \Delta_{H}) \circ (\Delta_{H} \otimes \Delta_{H}) \circ \Delta_{H})
\otimes id_{X} \otimes id_{H}) (a \otimes f \otimes ?) = \]
	\[ = (id_{H} \otimes (\rho_{H^{*},H} \circ (id_{X} \otimes (\mu_{H} \circ (id_{H} \otimes \mu_{H}) \circ (id_{H} \otimes
\mu_{H} \otimes \mu_{H}))))) \circ (id_{H} \otimes id_{X} \otimes S^{-1} \otimes id_{H} \otimes S^{-1} \otimes id_{H}  \otimes
id_{H}) \circ \]
	\[ \circ (id_{H} \otimes id_{X} \otimes id_{H} \otimes id_{H} \otimes id_{H} \otimes \tau_{H,H}) \circ (id_{H} \otimes
id_{X} \otimes id_{H} \otimes id_{H} \otimes \tau_{H,H} \otimes id_{H}) \circ \]
	\[ \circ (id_{H} \otimes id_{X} \otimes id_{H} \otimes \tau_{H \otimes H,H} \otimes id_{H}) \circ (id_{H} \otimes id_{X}
\otimes \tau_{H \otimes H \otimes H,H} \otimes id_{H}) \circ \]
	\[ \circ (id_{H} \otimes \tau_{H \otimes H \otimes H \otimes H,X} \otimes id_{H}) \circ (\tau_{H \otimes H \otimes H,H}
\otimes id_{H} \otimes id_{X} \otimes id_{H}) \circ \]
	\[ \circ (((id_{H} \otimes \Delta_{H} \otimes id_{H} \otimes id_{H}) \circ ( (( \Delta_{H} \otimes id_{H} ) \circ
\Delta_{H}) \otimes id_{H}) \circ \Delta_{H}) \otimes id_{X} \otimes id_{H}) (a \otimes f \otimes ?) = \]
	\[ = \sum_{(a),(a^{'}),((a^{'})^{'}),(((a^{'})^{'})^{''})} (-1)^{ |(a^{'})^{''}||(a^{'})^{'}| + |f|( |(a^{'})^{'}| +
|a^{''}| ) + |a^{''}||(a^{'})^{'}| + |(((a^{'})^{'})^{''})^{''}|(|((a^{'})^{'})^{'}| + |(((a^{'})^{'})^{''})^{'}|) } * \]
	\[ * (-1)^{|(((a^{'})^{'})^{''})^{'}||((a^{'})^{'})^{'}| + |?||((a^{'})^{'})^{'}|} *  \]
	\[ * (a^{'})^{''} \otimes f(S^{-1}(a^{''}) (((a^{'})^{'})^{''})^{''} S^{-1}((((a^{'})^{'})^{''})^{'}) ? ((a^{'})^{'})^{'}) =
\]
	\[ = \sum_{(a),(a^{'}),((a^{'})^{'})} (-1)^{ |(a^{'})^{''}||(a^{'})^{'}| + |f|( |(a^{'})^{'}| + |a^{''}| ) +
|a^{''}||(a^{'})^{'}| + |((a^{'})^{'})^{''}||((a^{'})^{'})^{'}| + |?||(a^{'})^{'}| } * \]
	\[ * (a^{'})^{''} \otimes f(S^{-1}(a^{''}) \epsilon_{H}(((a^{'})^{'})^{''}) ? ((a^{'})^{'})^{'}) = \]
	\[ = \sum_{(a),(a^{'}),((a^{'})^{'})} (-1)^{ |(a^{'})^{''}||(a^{'})^{'}| + |f|( |(a^{'})^{'}| + |a^{''}| ) +
|a^{''}||(a^{'})^{'}| + |?||(a^{'})^{'}| } * \]
	\[ * (a^{'})^{''} \otimes f(S^{-1}(a^{''}) ? ((a^{'})^{'})^{'} \epsilon_{H}(((a^{'})^{'})^{''})) = \]
	\[ = \sum_{(a),(a^{'})} (-1)^{ |(a^{'})^{''}||(a^{'})^{'}| + |f|( |(a^{'})^{'}| + |a^{''}| ) + |a^{''}||(a^{'})^{'}| +
|?||(a^{'})^{'}| } (a^{'})^{''} \otimes f(S^{-1}(a^{''}) ? (a^{'})^{'}) = \]
	\[ = \sum_{(a),(a^{'})} (-1)^{ |(a^{'})^{'}||(a^{'})^{''}| + |f|(|a| + |(a^{'})^{''}|) + |a^{''}||(a^{'})^{'}| +
|?||(a^{'})^{'}| } (a^{'})^{''} \otimes f(S^{-1}(a^{''}) ? (a^{'})^{'}). \]
	
	Remark 2. We have by the coassociativity
	\[ (id_{H} \otimes id_{H} \otimes id_{H} \otimes \Delta_{H}) \circ (\Delta_{H} \otimes \Delta_{H}) \circ \Delta_{H} = \]
	\[ = (id_{H} \otimes id_{H} \otimes ((id_{H} \otimes \Delta_{H}) \circ \Delta_{H})) \circ (\Delta_{H} \otimes id_{H}) \circ
\Delta_{H} = \]
	\[ = (id_{H} \otimes id_{H} \otimes ((\Delta_{H} \otimes id_{H}) \circ \Delta_{H})) \circ (\Delta_{H} \otimes id_{H}) \circ
\Delta_{H} = \]
	\[ = (id_{H} \otimes id_{H} \otimes \Delta_{H} \otimes id_{H}) \circ (\Delta_{H} \otimes \Delta_{H}) \circ \Delta_{H}= \]
	\[ = (id_{H} \otimes id_{H} \otimes \Delta_{H} \otimes id_{H}) \circ (id_{H} \otimes \Delta_{H} \otimes id_{H}) \circ
(id_{H} \otimes \Delta_{H}) \circ \Delta_{H} = \]
	\[ = (id_{H} \otimes ((id_{H} \otimes \Delta_{H}) \circ \Delta_{H}) \otimes id_{H}) \circ (id_{H} \otimes \Delta_{H}) \circ
\Delta_{H} = \]
	\[ = (id_{H} \otimes ((\Delta_{H} \otimes id_{H}) \circ \Delta_{H}) \otimes id_{H}) \circ (\Delta_{H} \otimes id_{H}) \circ
\Delta_{H} = \]
	\[ = (id_{H} \otimes \Delta_{H} \otimes id_{H} \otimes id_{H}) \circ ( (( id_{H} \otimes \Delta_{H} ) \circ \Delta_{H})
\otimes id_{H}) \circ \Delta_{H} = \]
	\[ = (id_{H} \otimes \Delta_{H} \otimes id_{H} \otimes id_{H}) \circ ( (( \Delta_{H} \otimes id_{H} ) \circ \Delta_{H})
\otimes id_{H}) \circ \Delta_{H}. \qed \]
	
	We see that the right and left sides of the considered relation are equal. The result follows.
\end{proof}

\MultQD*
\begin{proof}
	It follows from Theorem \ref{theorem:twistedBialgebras} that we have for all $a,b,c \in H, \; f,g,h \in H^{*}$
	
	\[ ((id_{X} \otimes ev_{H}) ((f \otimes a)( g \otimes b ))) (c \otimes h) = ((id_{X} \otimes ev_{H}) (\sum_{(a),(g)}
(-1)^{|g^{'}||g^{''}| + |g^{''}||a^{''}|}  f(a^{'} \cdot g^{''}) \otimes (a^{''})^{g^{'}} b)) (c \otimes h) = \]
	\[ = \sum_{(a),(g),(a^{'}),(a^{''}),((a^{''})^{''}),(c)} (-1)^{|g^{'}||g^{''}| + |g^{''}||a^{''}| + |a^{'}||g^{''}| +
|(a^{'})^{'}||(a^{'})^{''}| + |(a^{'})^{'}||?| + |a^{''}||g^{'}| + |((a^{''})^{''})^{''}|(|((a^{''})^{''})^{'}| +
|(a^{''})^{'}|) } * \]
	\[ * (-1)^{ |c|( |((a^{''})^{''} )^{'}| + |b| ) } * \]
	\[ * \lambda_{H,H} ((f \otimes g^{''}( S^{-1}( (a^{'})^{''} ) ? (a^{'})^{'} ))) (c^{'} \otimes c^{''}) \otimes g^{'}( S^{-1}
( ((a^{''})^{''})^{''} ) (a^{''})^{'} ) ev_{H} (((a^{''})^{''} )^{'} b) (h) = \]
	\[ = \sum_{(a),(g),(a^{'}),(a^{''}),((a^{''})^{''}),(c)} (-1)^{|g^{'}||g^{''}| + |g^{''}||a^{''}| + |a^{'}||g^{''}| +
|(a^{'})^{'}||(a^{'})^{''}| + |(a^{'})^{'}||c^{''}| + |a^{''}||g^{'}| + |((a^{''})^{''})^{''}|(|((a^{''})^{''})^{'}| +
|(a^{''})^{'}|) } * \]
	\[ * (-1)^{ |c|( |((a^{''})^{''} )^{'}| + |b| ) + |c^{'}|( |g^{''}| + |a^{'}| ) } * \]
	\[ * f(c^{'}) g^{''}( S^{-1}( (a^{'})^{''} ) c^{''} (a^{'})^{'} ) \otimes g^{'}( S^{-1} ( ((a^{''})^{''})^{''} )
(a^{''})^{'} ) ev_{H} (((a^{''})^{''} )^{'} b) (h) = \]
	\[ = \sum_{(a),(a^{'}),(a^{''}),((a^{''})^{''}),(c)} (-1)^{|g||a| + |(a^{'})^{'}||(a^{'})^{''}| + |(a^{'})^{'}||c^{''}| +
|a^{'}|( |((a^{''})^{''})^{''}| + |(a^{''})^{'}| ) + |((a^{''})^{''})^{''}|(|((a^{''})^{''})^{'}| + |(a^{''})^{'}|) } * \]
	\[ * (-1)^{ |c|( |((a^{''})^{''} )^{'}| + |b| ) + |c^{'}|( |g| + |a^{'}| + |((a^{''})^{''})^{''}| + |(a^{''})^{'}| ) } * \]
	\[ * f(c^{'}) \lambda_{H,H} (\Delta_{H^{*}} (g)) ( S^{-1} ( ((a^{''})^{''})^{''} ) (a^{''})^{'} \otimes S^{-1}( (a^{'})^{''}
) c^{''}  (a^{'})^{'} ) \otimes ev_{H} (((a^{''})^{''} )^{'} b) (h) = \]
	\[ = \sum_{(a),(a^{'}),(a^{''}),((a^{''})^{''}),(c)} (-1)^{|c^{'}|( |g| + |a| + |b| ) + |g||a| + |((a^{''})^{''} )^{''}|(
|a^{'}| + |(a^{''})^{'}| + |((a^{''})^{''} )^{'}| ) + |(a^{''})^{'}||a^{'}| + |(a^{'})^{'}||(a^{'})^{''}|} * \]
	\[ * (-1)^{ |c^{''}|( |(a^{'})^{'}| + |((a^{''})^{''} )^{'}| + |b| ) } * \]
	\[ * f(c^{'}) g( S^{-1} ( ((a^{''})^{''})^{''} ) (a^{''})^{'} S^{-1}( (a^{'})^{''} ) c^{''}  (a^{'})^{'} ) \otimes ev_{H}
(((a^{''})^{''} )^{'} b) (h) = \]
	\[ = ((\nu_{k,k} \circ (\rho_{H^{*},H} \otimes \rho_{H^{*},H})) \otimes id_{k}) \circ \]
	\[ \circ (id_{X} \otimes id_{H} \otimes id_{X} \otimes (\mu_{H} \circ (id_{H} \otimes \mu_{H}) \circ (id_{H} \otimes \mu_{H}
\otimes \mu_{H}))) \otimes (\rho_{H^{**},H^{*}} \circ (ev_{H} \otimes id_{H^{*}})) \circ \]
	\[ \circ (id_{X} \otimes id_{H} \otimes id_{X} \otimes S^{-1} \otimes id_{H} \otimes S^{-1} \otimes id_{H} \otimes id_{H}
\otimes id_{H} \otimes id_{H} \otimes id_X) \circ \]
	\[ \circ ( id_{X} \otimes id_{H} \otimes id_{X} \otimes id_{H} \otimes id_{H} \otimes id_{H} \otimes \tau_{H \otimes H
\otimes H, H} \otimes id_{X}) \circ \]
	\[ \circ (id_{X} \otimes id_{H} \otimes id_{X} \otimes id_{H} \otimes id_{H} \otimes \tau_{H,H} \otimes id_{H} \otimes
id_{H} \otimes id_{H} \otimes id_{X}) \circ \]
	\[ (id_{X} \otimes id_{H} \otimes id_{X} \otimes id_{H} \otimes \tau_{H \otimes H,H} \otimes id_{H} \otimes id_{H} \otimes
id_{H} \otimes id_{X}) \circ (id_{X} \otimes id_{H} \otimes id_{X} \otimes \tau_{H \otimes H \otimes H \otimes H,H} \otimes
id_{H} \otimes id_{H} \otimes id_{X}) \circ \]
	\[ \circ (id_{X} \otimes id_{H} \otimes \tau_{H \otimes H \otimes H \otimes H \otimes H,X} \otimes id_{H} \otimes id_{H}
\otimes id_{X}) \circ (id_{X} \otimes \tau_{H \otimes H \otimes H \otimes H \otimes H \otimes X \otimes H, H} \otimes id_{H}
\otimes id_{X}) \circ \]
	\[ \circ (id_{X} \otimes ((id_{H} \otimes id_{H} \otimes id_{H} \otimes \Delta_{H}) \circ (\Delta_{H} \otimes \Delta_{H})
\circ \Delta_{H}) \otimes id_{X} \otimes id_{H} \otimes \Delta_{H} \otimes id_{X}) (f \otimes a \otimes g \otimes b \otimes c
\otimes h) = \]
	\[ = ((\nu_{k,k} \circ (\rho_{H^{*},H} \otimes \rho_{H^{*},H})) \otimes id_{k}) \circ \]
	\[ \circ (id_{X} \otimes id_{H} \otimes id_{X} \otimes (\mu_{H} \circ (id_{H} \otimes \mu_{H}) \circ (id_{H} \otimes \mu_{H}
\otimes \mu_{H}))) \otimes (\rho_{H^{**},H^{*}} \circ (ev_{H} \otimes id_{H^{*}})) \circ \]
	\[ \circ (id_{X} \otimes id_{H} \otimes id_{X} \otimes S^{-1} \otimes id_{H} \otimes S^{-1} \otimes id_{H} \otimes id_{H}
\otimes id_{H} \otimes id_{H} \otimes id_X) \circ \]
	\[ \circ ( id_{X} \otimes id_{H} \otimes id_{X} \otimes id_{H} \otimes id_{H} \otimes id_{H} \otimes \tau_{H \otimes H
\otimes H, H} \otimes id_{X}) \circ \]
	\[ \circ (id_{X} \otimes id_{H} \otimes id_{X} \otimes id_{H} \otimes id_{H} \otimes \tau_{H,H} \otimes id_{H} \otimes
id_{H} \otimes id_{H} \otimes id_{X}) \circ \]
	\[ (id_{X} \otimes id_{H} \otimes id_{X} \otimes id_{H} \otimes \tau_{H \otimes H,H} \otimes id_{H} \otimes id_{H} \otimes
id_{H} \otimes id_{X}) \circ (id_{X} \otimes id_{H} \otimes id_{X} \otimes \tau_{H \otimes H \otimes H \otimes H,H} \otimes
id_{H} \otimes id_{H} \otimes id_{X}) \circ \]
	\[ \circ (id_{X} \otimes id_{H} \otimes \tau_{H \otimes H \otimes H \otimes H \otimes H,X} \otimes id_{H} \otimes id_{H}
\otimes id_{X}) \circ (id_{X} \otimes \tau_{H \otimes H \otimes H \otimes H \otimes H \otimes X \otimes H, H} \otimes id_{H}
\otimes id_{X}) \circ \]
	\[ \circ (id_{X} \otimes ((id_{H} \otimes \Delta_{H} \otimes id_{H} \otimes id_{H}) \circ ( (( \Delta_{H} \otimes id_{H} )
\circ \Delta_{H}) \otimes id_{H}) \circ \Delta_{H}) \otimes id_{X} \otimes id_{H} \otimes \Delta_{H} \otimes id_{X}) (f \otimes
a \otimes g \otimes b \otimes c \otimes h) = \]
	\[ = \sum_{(a),(a^{'}),((a^{'})^{'}),(((a^{'})^{'})^{''}),(c)} (-1)^{ |c^{'}|( |a| + |g| + |b| ) + |g||a| + |a^{''}||a^{'}|
+ |(((a^{'})^{'})^{''})^{''}|( |((a^{'})^{'})^{'}| + |(((a^{'})^{'})^{''})^{'}| ) +
|(((a^{'})^{'})^{''})^{'}||((a^{'})^{'})^{'}| } * \]
	\[ * (-1)^{|c^{''}|( (|(a^{'})^{'})^{'}| + |(a^{'})^{''}| + |b| )} * \]
	\[ * f(c^{'}) g(S^{-1}(a^{''}) (((a^{'})^{'})^{''})^{''} S^{-1}((((a^{'})^{'})^{''})^{'}) c^{''} ((a^{'})^{'})^{'}) \otimes
ev_{H}((a^{'})^{''} b)(h) = \]
	\[ = \sum_{(a),(a^{'}),((a^{'})^{'}),(c)} (-1)^{ |c^{'}|( |a| + |g| + |b| ) + |g||a| + |a^{''}||a^{'}| +
|((a^{'})^{'})^{''}||((a^{'})^{'})^{'}| + |c^{''}|( (|(a^{'})^{'})^{'}| + |(a^{'})^{''}| + |b| ) } * \]
	\[ * f(c^{'}) g(S^{-1}(a^{''}) c^{''} ((a^{'})^{'})^{'} \epsilon_{H}(((a^{'})^{'})^{''}) ) \otimes ev_{H}((a^{'})^{''} b)(h)
= \]
	\[ = \sum_{(a),(a^{'})} (-1)^{ |c^{'}|( |a| + |g| + |b| ) + |g||a| + |a^{''}||a^{'}| + |c^{''}| (|a^{'}| + |b| ) } f(c^{'})
g(S^{-1}(a^{''}) c^{''} (a^{'})^{'} ) \otimes ev_{H}((a^{'})^{''} b)(h) = \]
	\[ = ((id_{X} \otimes ev_{H}) ( \sum_{(a),(a^{'})} (-1)^{ |g||a| + |a^{''}||a^{'}| + |?||(a^{'})^{'}| } f g(S^{-1}(a^{''}) ?
(a^{'})^{'} ) \otimes (a^{'})^{''} b )) (c \otimes h). \]
	
	Thus
	\[ (f \otimes a)( g \otimes b ) =  \sum_{(a),(a^{'})} (-1)^{ |g||a| + |a^{''}||a^{'}| + |?||(a^{'})^{'}| } f
g(S^{-1}(a^{''}) ? (a^{'})^{'} ) \otimes (a^{'})^{''} b. \]
	
	Remark. We have by the coassociativity
	\[ (id_{H} \otimes id_{H} \otimes id_{H} \otimes \Delta_{H}) \circ (\Delta_{H} \otimes \Delta_{H}) \circ \Delta_{H} = \]
	\[ = (id_{H} \otimes id_{H} \otimes ((id_{H} \otimes \Delta_{H}) \circ \Delta_{H})) \circ (\Delta_{H} \otimes id_{H}) \circ
\Delta_{H} = \]
	\[ = (id_{H} \otimes id_{H} \otimes ((\Delta_{H} \otimes id_{H}) \circ \Delta_{H})) \circ (\Delta_{H} \otimes id_{H}) \circ
\Delta_{H} = \]
	\[ = (id_{H} \otimes id_{H} \otimes \Delta_{H} \otimes id_{H}) \circ (\Delta_{H} \otimes \Delta_{H}) \circ \Delta_{H}= \]
	\[ = (id_{H} \otimes id_{H} \otimes \Delta_{H} \otimes id_{H}) \circ (id_{H} \otimes \Delta_{H} \otimes id_{H}) \circ
(id_{H} \otimes \Delta_{H}) \circ \Delta_{H} = \]
	\[ = (id_{H} \otimes ((id_{H} \otimes \Delta_{H}) \circ \Delta_{H}) \otimes id_{H}) \circ (id_{H} \otimes \Delta_{H}) \circ
\Delta_{H} = \]
	\[ = (id_{H} \otimes ((\Delta_{H} \otimes id_{H}) \circ \Delta_{H}) \otimes id_{H}) \circ (\Delta_{H} \otimes id_{H}) \circ
\Delta_{H} = \]
	\[ = (id_{H} \otimes \Delta_{H} \otimes id_{H} \otimes id_{H}) \circ ( (( id_{H} \otimes \Delta_{H} ) \circ \Delta_{H})
\otimes id_{H}) \circ \Delta_{H} = \]
	\[ = (id_{H} \otimes \Delta_{H} \otimes id_{H} \otimes id_{H}) \circ ( (( \Delta_{H} \otimes id_{H} ) \circ \Delta_{H})
\otimes id_{H}) \circ \Delta_{H}. \]	
\end{proof}

\DHprop*
\begin{proof}
	
	It follows from Lemma \ref{lm:MultQDr} that we have for all $a \in H, \; f \in X$
	
	\[ ( 1_{X} \otimes 1_{H} )( f \otimes a ) = \epsilon_{H} f( S^{-1}(1_{H}) ? 1_{H} ) \otimes 1_{H} a = \epsilon_{H} f \otimes
a = f \otimes a, \]
	\[ ( f \otimes a )( 1_{X} \otimes 1_{H} ) = \sum_{(a),(a^{'})} (-1)^{|a^{''}||a^{'}| + |?||(a^{'})^{'}|} f \epsilon_{H}(
S^{-1}( a^{''} ) ? (a^{'})^{'} ) \otimes (a^{'})^{''} 1_{H} = \]
	\[ = \sum_{(a),(a^{'})} f \epsilon_{H} \otimes \epsilon_{H}((a^{'})^{'}) (a^{'})^{''} \epsilon_{H}(a^{''}) =  \sum_{(a)} f
\otimes a^{'} \epsilon_{H}(a^{''}) = f \otimes a. \]
	
	Statements about a counit and a comultiplication in $D(H)$ follow from Theorem \ref{theorem:twistedBialgebras}.
	
	We also have from Theorem \ref{theorem:twistedBialgebras} that an antipode $S_{X \bowtie A}$ is given by formula
	\[ ((id_{X} \otimes ev_{H}) (S_{X \bowtie A}(f \otimes a)))  ( b \otimes g ) = \]
	\[ = ((id_{X} \otimes ev_{H}) (\sum_{(f),(a)} (-1)^{|a^{'}||a^{''}| + |f^{''}||a^{''}| + |f^{'}||a^{''}| + |f^{''}||a^{'}|}
S(a^{''}) \cdot (S^{-1})^{*}(f^{'}) \otimes S(a^{'})^{(S^{-1})^{*}(f^{''})})) ( b \otimes g ) = \]
	\[ = \sum_{(f),(a),(a^{'}),((a^{'})^{'}),(a^{''})} (-1)^{|a^{'}||a^{''}| + |f^{''}||a^{''}| + |f^{'}||a^{''}| +
|f^{''}||a^{'}| + |a^{''}||f^{'}| + |b||(a^{''})^{''}| + |a^{'}||f^{''}| + |(a^{'})^{'}||(a^{'})^{''}| +
|((a^{'})^{'})^{'}||(a^{'})^{''}| + |b||((a^{'})^{'})^{''}| } * \]
	\[ * (S^{-1})^{*}(f^{'})( S^{-1}(S((a^{''})^{'})) b S( (a^{''})^{''} ) ) \otimes (S^{-1})^{*}(f^{''})(S^{-1}(
S(((a^{'})^{'})^{'}) ) S((a^{'})^{''}) ) ev_{H}(S(((a^{'})^{'})^{''}))(g) = \]
	\[ = \sum_{(f),(a),(a^{'}),((a^{'})^{'}),(a^{''})} (-1)^{|a^{'}||a^{''}| + |f^{''}||a^{''}| + |b||(a^{''})^{'}| +
|(a^{''})^{'}||(a^{''})^{''}| + |(a^{'})^{'}||(a^{'})^{''}| + |b||((a^{'})^{'})^{''}| } * \]	
	\[ * f^{'}( (a^{''})^{''} S^{-1}(b) S^{-1}((a^{''})^{'}) ) \otimes f^{''}( (a^{'})^{''}  S^{-1}(((a^{'})^{'})^{'}) )
ev_{H}(S(((a^{'})^{'})^{''}))(g) = \]
	\[ = \sum_{(a),(a^{'}),((a^{'})^{'}),(a^{''})} (-1)^{|a^{'}||a^{''}| + |b||(a^{''})^{'}| + |(a^{''})^{'}||(a^{''})^{''}| +
|(a^{'})^{'}||(a^{'})^{''}| + |b||a^{'}| } * \]	
	\[ * \lambda_{H,H} (\Delta_{H^{*}} (f)) ( (a^{''})^{''} S^{-1}(b) S^{-1}((a^{''})^{'}) \otimes (a^{'})^{''}
S^{-1}(((a^{'})^{'})^{'}) ) \otimes ev_{H}(S(((a^{'})^{'})^{''}))(g) = \]
	\[ = \sum_{(a),(a^{'}),((a^{'})^{'}),(a^{''})} (-1)^{|a^{'}||a^{''}| + |b||(a^{''})^{'}| + |(a^{''})^{'}||(a^{''})^{''}| +
|(a^{'})^{'}||(a^{'})^{''}| + |b||a^{'}| } * \]	
	\[ * f( (a^{''})^{''} S^{-1}(b) S^{-1}((a^{''})^{'}) (a^{'})^{''} S^{-1}(((a^{'})^{'})^{'}) ) \otimes
ev_{H}(S(((a^{'})^{'})^{''}))(g) = \]
	\[ = \sum_{(a),(a^{'}),((a^{'})^{'}),(a^{''})} (-1)^{|(a^{''})^{''}|(|(a^{''})^{'}| + |a^{'}|) + |b|(|a^{'}| +
|(a^{''})^{'}|) + |(a^{''})^{'}||a^{'}| + |(a^{'})^{''}||(a^{'})^{'}| } * \]
	\[ * f( (a^{''})^{''} S^{-1}(b) S^{-1}((a^{''})^{'}) (a^{'})^{''} S^{-1}(((a^{'})^{'})^{'}) ) \otimes
ev_{H}(S(((a^{'})^{'})^{''}))(g) = \]	
	\[ = (\rho_{H^{*},H} \otimes id_{k}) \circ (id_{X} \otimes (\mu_{H} \circ (id_{H} \otimes \mu_{H}) \circ (id_{H} \otimes
\mu_{H} \otimes \mu_{H})) \otimes (\rho_{H^{**},H^{*}} \circ (ev_{H} \otimes id_{X}) )) \circ \]
	\[\circ (id_{X} \otimes id_{H} \otimes S^{-1} \otimes S^{-1} \otimes id_{H} \otimes S^{-1} \otimes S \otimes id_{X}) \circ
\]
	\[ \circ (id_{X} \otimes id_{H} \otimes id_{H} \otimes id_{H} \otimes \tau_{H \otimes H,H} \otimes id_{X}) \circ (id_{X}
\otimes id_{H} \otimes id_{H} \otimes \tau_{H \otimes H \otimes H,H} \otimes id_{X}) \circ \]
	\[ \circ (id_{X} \otimes id_{H} \otimes \tau_{H \otimes H \otimes H \otimes H,H} \otimes id_{X}) \circ (id_{X} \otimes
\tau_{H \otimes H \otimes H \otimes H,H} \otimes id_{H} \otimes id_{X}) \circ \]
	\[ \circ (id_{X} \otimes ((\Delta_{H} \otimes id_{H} \otimes id_{H} \otimes id_{H}) \circ (\Delta_{H} \otimes \Delta_{H})
\circ \Delta_{H}) \otimes id_{H} \otimes id_{X}) (f \otimes a \otimes b \otimes g) = \]
	\[ = (\rho_{H^{*},H} \otimes id_{k}) \circ (id_{X} \otimes (\mu_{H} \circ (id_{H} \otimes \mu_{H}) \circ (id_{H} \otimes
\mu_{H} \otimes \mu_{H})) \otimes (\rho_{H^{**},H^{*}} \circ (ev_{H} \otimes id_{X}) )) \circ \]
	\[\circ (id_{X} \otimes id_{H} \otimes S^{-1} \otimes S^{-1} \otimes id_{H} \otimes S^{-1} \otimes S \otimes id_{X}) \circ
\]
	\[ \circ (id_{X} \otimes id_{H} \otimes id_{H} \otimes id_{H} \otimes \tau_{H \otimes H,H} \otimes id_{X}) \circ (id_{X}
\otimes id_{H} \otimes id_{H} \otimes \tau_{H \otimes H \otimes H,H} \otimes id_{X}) \circ \]
	\[ \circ (id_{X} \otimes id_{H} \otimes \tau_{H \otimes H \otimes H \otimes H,H} \otimes id_{X}) \circ (id_{X} \otimes
\tau_{H \otimes H \otimes H \otimes H,H} \otimes id_{H} \otimes id_{X}) \circ \]
	\[ \circ (id_{X} \otimes ((id_{H} \otimes id_{H} \circ \Delta_{H} \otimes id_{H}) \circ (id_{H} \otimes \Delta_{H} \otimes
id_{H}) \circ (\Delta_{H} \otimes id_{H}) \circ \Delta_{H}) \otimes id_{H} \otimes id_{X}) (f \otimes a \otimes b \otimes g) =
\]
	\[ = \sum_{(a),(a^{'}),((a^{'})^{''}),(((a^{'})^{''})^{''})} (-1)^{ |a^{''}||a^{'}| + |b||a^{'}| +
|(((a^{'})^{''})^{''})^{''}|( |(a^{'})^{'}| + |((a^{'})^{''})^{'}| + |(((a^{'})^{''})^{''})^{'}| ) + |(((a^{'})^{''})^{''})^{'}
|( |(a^{'})^{'}| + |((a^{'})^{''})^{'}| ) } * \]
	\[ * f(a^{''} S^{-1}(b) S^{-1}((((a^{'})^{''})^{''})^{''}) (((a^{'})^{''})^{''})^{'} S^{-1}((a^{'})^{'})) \otimes
ev_{H}(S(((a^{'})^{''})^{'}))(g) = \]
	\[ = \sum_{(a),(a^{'}),((a^{'})^{''})} (-1)^{ |a^{''}||a^{'}| + |b||a^{'}| + |((a^{'})^{''})^{''}|( |(a^{'})^{'}| +
|((a^{'})^{''})^{'}| ) } * \]
	\[ * f(a^{''} S^{-1}(b) \epsilon_{H}(((a^{'})^{''})^{''}) S^{-1}((a^{'})^{'})) \otimes ev_{H}(S(((a^{'})^{''})^{'}))(g) =
\]
	\[ = \sum_{(a),(a^{'}),((a^{'})^{''})} (-1)^{ |a^{''}||a^{'}| + |b||a^{'}| + |((a^{'})^{''})^{''}|( |(a^{'})^{'}| +
|((a^{'})^{''})^{'}| ) } * \]
	\[ * f(a^{''} S^{-1}(b) S^{-1}((a^{'})^{'})) \otimes ev_{H}(S( ((a^{'})^{''})^{'})\epsilon_{H}(((a^{'})^{''})^{''}) )(g) =
\]
	\[ = \sum_{(a),(a^{'})} (-1)^{ |a^{''}||a^{'}| + |b||a^{'}| } f(a^{''} S^{-1}(b) S^{-1}((a^{'})^{'})) \otimes ev_{H}(S(
(a^{'})^{''} )(g) = \]
	\[ = ((id_{X} \otimes ev_{H}) (\sum_{(a),(a^{'})} (-1)^{ |a^{''}||a^{'}| + |(a^{'})^{'}||?| } f(a^{''} S^{-1}(?)
S^{-1}((a^{'})^{'})) \otimes S( (a^{'})^{''} ))) ( b \otimes g )  \]
	for all $f,g \in X, \; a,b \in H$.
	
	Thus
	\[ S_{X \bowtie A}(f \otimes a) = \sum_{(a),(a^{'})} (-1)^{ |a^{''}||a^{'}| + |(a^{'})^{'}||?| } f(a^{''} S^{-1}(?)
S^{-1}((a^{'})^{'})) \otimes S( (a^{'})^{''} ) \]
	for all $f \in X, \; a \in H$.
	
	Remark. We have by the coassociativity
	\[(\Delta_{H} \otimes id_{H} \otimes id_{H} \otimes id_{H}) \circ (\Delta_{H} \otimes \Delta_{H}) \circ \Delta_{H} = \]
	\[ = (((\Delta_{H} \otimes id_{H}) \circ \Delta_{H}) \otimes id_{H} \otimes id_{H}) \circ (id_{H} \otimes \Delta_{H}) \circ
\Delta_{H} = \]
	\[ = (((id_{H} \otimes \Delta_{H}) \circ \Delta_{H}) \otimes id_{H} \otimes id_{H}) \circ (id_{H} \otimes \Delta_{H}) \circ
\Delta_{H} = \]
	\[ = (id_{H} \otimes \Delta_{H} \otimes id_{H} \otimes id_{H}) \circ (\Delta_{H} \otimes \Delta_{H}) \circ \Delta_{H} = \]
	\[ = (id_{H} \otimes \Delta_{H} \otimes id_{H} \otimes id_{H}) \circ (id_{H} \otimes \Delta_{H} \otimes id_{H}) \circ
(id_{H} \otimes \Delta_{H}) \circ \Delta_{H} = \]
	\[ = (id_{H} \otimes \Delta_{H} \otimes id_{H} \otimes id_{H}) \circ (id_{H} \otimes \Delta_{H} \otimes id_{H}) \circ
(\Delta_{H} \otimes id_{H}) \circ \Delta_{H}  = \]
	\[ = (id_{H} \otimes ((\Delta_{H} \otimes id_{H}) \circ \Delta_{H}) \otimes id_{H}) \circ (\Delta_{H} \otimes id_{H}) \circ
\Delta_{H} = \]
	\[ = (id_{H} \otimes ((id_{H} \otimes \Delta_{H}) \circ \Delta_{H}) \otimes id_{H}) \circ (\Delta_{H} \otimes id_{H}) \circ
\Delta_{H} = \]
	\[ = (id_{H} \otimes id_{H} \circ \Delta_{H} \otimes id_{H}) \circ (id_{H} \otimes \Delta_{H} \otimes id_{H}) \circ
(\Delta_{H} \otimes id_{H}) \circ \Delta_{H}. \qed \]

	We show that $i_{X}(X)$ and $i_{H}(H)$ are Hopf subsuperalgebras. Mappings $i_{X}$ and $i_{H}$ are bijective. Furthermore,
$i_{X}(X)$ is a vector superspace: $(i_{X}(X))_{0} = X_{0} \otimes 1_{H}$,  $(i_{X}(X))_{1} = X_{1} \otimes 1_{H}$. We have for
all $\lambda_{1}, \lambda_{2} \in k$, $f,g \in X$
	\[ \lambda_{1}i_{X}(f) + \lambda_{2}i_{X}(g) = \lambda_{1} f \otimes 1_{H} + \lambda_{2} g \otimes 1_{H} = (\lambda_{1} +
\lambda_{2}) ( f + g ) \otimes 1_{H} = (\lambda_{1} + \lambda_{2}) i_{X}(f + g). \]
	
	In the same way  $i_{H}(H)$ is a vector superspace: $(i_{H}(H))_{0} = 1_{X} \otimes H_{0}$,  $(i_{H}(H))_{1} = 1_{X} \otimes
H_{1}$. We have for all $\lambda_{1}, \lambda_{2} \in k$, $a,b \in H$
	\[ \lambda_{1}i_{H}(a) + \lambda_{2}i_{H}(b) = \lambda_{1} 1_{X} \otimes a + \lambda_{2} 1_{X} \otimes b = (\lambda_{1} +
\lambda_{2}) 1_{X} \otimes (a + b) = (\lambda_{1} + \lambda_{2}) i_{H}(a + b). \]
	
	We prove that superspace morphisms $i_{X}$ and $i_{H}$ are Hopf superalgebras morphisms.
	First we prove that they are supercoalgebras morphisms.
	
	We show for $i_{X}$ that
	\[ (i_{X} \otimes i_{X}) \otimes \Delta_{X} = \Delta_{X \bowtie H} \circ i_{X}. \]
	Indeed, we have for all $f \in X$
	\[ (i_{X} \otimes i_{X}) \otimes \Delta_{X}(f) = \sum_{(f)} (-1)^{|f^{'}||f^{''}|} ( f^{''} \otimes 1_{H} ) \otimes ( f^{'}
\otimes 1_{H} ) = \]
	\[ = \sum_{(f)} (-1)^{|f^{'}||f^{''}| + |1_{H}||f^{'}|} ( f^{''} \otimes 1_{H} ) \otimes ( f^{'} \otimes 1_{H} ) = \Delta_{X
\bowtie H} \circ i_{X}(f). \]
	
	Next we show that
	\[ \epsilon_{X \bowtie H} \circ i_{X} = \epsilon_{X}. \]
	It follows from the fact that we have for all $f \in X$
	\[ \epsilon_{X \bowtie H} \circ i_{X}(f) = \epsilon_{X \bowtie H}( f \otimes 1_{H}) = \epsilon_{X}(f) \epsilon_{H}(1_{H}) =
1_{k} \epsilon_{X}(f) = \epsilon_{X}(f).\]
	
	We show for $i_{H}$ that
	\[ (i_{H} \otimes i_{H}) \otimes \Delta_{H} = \Delta_{X \bowtie H} \circ i_{H}. \]
	Indeed, we have for all $a \in H$
	\[ (i_{H} \otimes i_{H}) \otimes \Delta_{H}(a) = \sum_{(a)} ( 1_{X} \otimes a^{'} ) \otimes ( 1_{X} \otimes a^{''} ) =
\sum_{(a)} (-1)^{|\epsilon_{H}||a^{'}|} ( 1_{X} \otimes a^{'} ) \otimes ( 1_{X} \otimes a^{''} ) = \Delta_{X \bowtie H} \circ
i_{H}(a). \]
	
	Next we show that
	\[ \epsilon_{X \bowtie H} \circ i_{H} = \epsilon_{H}. \]
	It follows from the fact that we have for all $a \in H$
	\[ \epsilon_{X \bowtie H} \circ i_{H}(a) = \epsilon_{X}(1_{X}) \epsilon_{H}(a) = 1_{X}(1_{H}) \epsilon_{H}(a) =
\epsilon_{H}(1_{H}) \epsilon_{H}(a) = 1_{k} \epsilon_{H}(a) = \epsilon_{H}(a).\]
	
	We show that $i_X$ and $i_A$ are superalgebras morphisms.
	
	We show for $i_{X}$ that
	\[ i_{X}(f) i_{X}(g) = i_{X}(fg). \]
	Indeed, we have for all $f ,g\in X$
	\[ i_{X}(f) i_{X}(g) = ( f \otimes 1_{H} ) ( g \otimes 1_{H} ) = fg( S^{-1}(1_{H}) ? 1_{H} ) \otimes 1_{H} 1_{H} = fg
\otimes 1_{H} = i_{X}(fg). \]
	
	Next we show that
	\[ i_{X} \circ \eta_{X} = \eta_{X \bowtie H}. \]
	It follows from the fact that
	\[ i_{X} \circ \eta_{X} ( 1_{k} ) = i_{X} ( 1_{X} ) = 1_{X} \otimes 1_{H} = \eta_{X \bowtie H} ( 1_{k} ). \]
	
	We show for $i_{H}$ that
	\[ i_{H}(a) i_{H}(b) = i_{H}(ab). \]
	Indeed, we have for all $a,b \in H$
	\[ i_{H}(a) i_{H}(b) = ( 1_{X} \otimes a ) ( 1_{X} \otimes b ) = \]
	\[ = \sum_{(a)} (-1)^{|a^{'}||a^{''}|+|?||(a^{'})^{'}|} 1_{X} \epsilon_{H} ( S^{-1}( a^{''} ) ? (a^{'})^{'} ) \otimes
(a^{'})^{''} b = \]
	\[ = \sum_{(a)} 1_{X} \epsilon_{H} \otimes \epsilon_{H}((a^{'})^{'}) (a^{'})^{''} \epsilon_{H}(a^{''}) b = 1_{X} \otimes ab
= i_{H}(ab). \]
	
	Next we show that
	\[ i_{H} \circ \eta_{H} = \eta_{X \bowtie H}.  \]
	It follows from the fact that
	\[ i_{H} \circ \eta_{H} ( 1_{k} ) = i_{H} ( 1_{H} ) = 1_{X} \otimes 1_{H} = \eta_{X \bowtie H} ( 1_{k} ). \]
	
	We show that $i_{X}$ commutes with antipodes
	\[ S_{X \bowtie H} \circ i_{X} = i_{X} \circ S_{X}. \]
	Indeed, we have for all $f \in X$
	\[ S_{X \bowtie H} \circ i_{X}(f) = S_{X \bowtie H}(f \otimes 1_{H}) = \]
	\[ = f( 1_{H} S^{-1}(?) S^{-1}(1_{H} )) \otimes S(1_{H}) = (f \circ S^{-1}) \otimes 1_{H} = (S^{-1})^{*} (f) \otimes 1_{H} =
i_{X} \circ S_{X}(f). \]
	
	We show that $i_{H}$ commutes with antipodes
	\[ S_{X \bowtie H} \circ i_{H} = i_{H} \circ S_{H}. \]
	It follows from the fact that we have for all $a \in H$
	\[ S_{X \bowtie H} \circ i_{H}(a) = S_{X \bowtie H}(1_{X} \otimes a) = \sum_{(a),(a^{'})} (-1)^{ |a^{''}||a^{'}| +
|(a^{'})^{'}||?| } \epsilon_{H}(a^{''} S^{-1}(?) S^{-1}((a^{'})^{'})) \otimes S( (a^{'})^{''} ) = \]
	\[ = \sum_{(a),(a^{'})} (-1)^{ |a^{''}||a^{'}| } \epsilon_{H} \otimes S( \epsilon_{H}((a^{'})^{'}) (a^{'})^{''}
\epsilon_{H}(a^{''}) ) = \sum_{(a)} \epsilon_{H} \otimes S( a^{'} \epsilon_{H}(a^{''}) ) = \]
	\[ = \epsilon_{H} \otimes S(a) = 1_{X} \otimes S(a) = i_{H} \circ S_{H}(a). \]

	Also note that we have for all $f \in X, \; a \in H$
	\[ i_{X}(f) i_{H}(a) = ( f \otimes 1_{H} ) ( 1_{X} \otimes a ) = f \epsilon_{H} ( S^{-1}( 1_{H} ) ? 1_{H} ) \otimes 1_{H} a
= f \epsilon_{H} \otimes a = f \otimes a,\]
	
	\[ i_{H}(a)i_{X}(f) = ( 1_{X} \otimes a ) ( f \otimes 1_{H} ) = \sum_{(a),(a^{'})} (-1)^{|f||a| + |a^{''}||a^{'}| +
|?||(a^{'})^{'}|} \epsilon_{H} f(S^{-1}( a^{''} ) ? (a^{'})^{'}) \otimes (a^{'})^{''} 1_{H} = \]
	\[ = \sum_{(a),(a^{'})} (-1)^{|f||a| + |a^{''}||a^{'}| + |?||(a^{'})^{'}|} f(S^{-1}( a^{''} ) ? (a^{'})^{'}) \otimes
(a^{'})^{''}. \]
\end{proof}

\RQD*
\begin{proof}
	We show that $R$ respects conditions of Defenitions \ref{quasicoco} and \ref{twistedquasicoco}.

	1.
	We argue that $R$ is an invertible even element and its inverse is an even element
	\[ \overline{R} = \sum_{i \in I} ( 1_{X} \otimes e_{i} ) \otimes ( ( e^{i} \circ S ) \otimes 1_{H} ). \]
	
	Consider the element $R\overline{R}$.
	\[ R\overline{R} = ( \sum_{i \in I} ( 1_{X} \otimes e_{i} ) \otimes ( e^{i} \otimes 1_{H} )  ) ( \sum_{j \in I} ( 1_{X}
\otimes e_{j} ) \otimes ( (e^{j} \circ S) \otimes 1_{H} ) ) = \]
	\[ = \sum_{i,j \in I} (-1)^{|e^{i}||e_{j}|} ( 1_{X} \otimes e_{i} )( 1_{X} \otimes e_{j} ) \otimes ( e^{i} \otimes 1_{H} )(
(e^{j} \circ S) \otimes 1_{H} ) = \]
	\[ = \sum_{i,j \in I} (-1)^{|e^{i}||e_{j}|} ( 1_{X} \otimes e_{i}e_{j} ) \otimes ( e^{i}(e^{j} \circ S) \otimes 1_{H} ). \]
	
	We evaluate $\xi=b \otimes u \otimes c \otimes v \in H \otimes X \otimes H \otimes X$ on $R\overline{R}$.
	\[ \lambda_{H,H^{*},H,H^{*}} (((id_{X} \otimes ev_{H} \otimes id_{X} \otimes ev_{H}) (R\overline{R}))) (\xi) = \]
	\[ = \lambda_{H,H^{*},H,H^{*}} ( (\sum_{i,j \in I} (-1)^{|e^{i}||e_{j}|} ( 1_{X} \otimes ev_{H}(e_{i}e_{j}) ) \otimes (
(e^{i}(e^{j} \circ S)) \otimes ev_{H}(1_{H}) ))) (b \otimes u \otimes c \otimes v) = \]
	\[ = v(1_{H}) \epsilon_{H} (b) \sum_{i,j \in I} (-1)^{|e^{i}||e_{j}| + ( |b| + |u| ) |e^{i}(e^{j} \circ S)| +
|e_{i}e_{j}||b| + |u||e_{i} e_{j}| } u( e_{i} e_{j} ) (e^{i}(e^{j} \circ S)) (c) = \]
	\[ = v(1_{H}) \epsilon_{H} (b) \sum_{i,j \in I} (-1)^{|e_{i}||e_{j}| + (|e_{i}| + |e_{j}|)(|b| + |u|) + (|e_{i}| +
|e_{j}|)|b| + |u|(|e_{i}| + |e_{j}|) } u( e_{i} e_{j} ) (e^{i}(e^{j} \circ S)) (c) = \]
	\[ = v(1_{H}) \epsilon_{H} (b) \sum_{i,j \in I} (-1)^{|e_{i}||e_{j}|} u( e_{i} e_{j} )  (\Delta_{H})^{*}  (\lambda_{H,H}
(e^{i} \otimes (e^{j} \circ S))) (c) = \]	
	\[ = v(1_{H}) \epsilon_{H} (b) \sum_{(c),i,j \in I} (-1)^{|e_{i}||e_{j}| + |c^{'}||e_{j}|} u( e_{i} e_{j} ) e^{i}(c^{'})
e^{j}(S(c^{''})) = \]	
	\[ = v(1_{H}) \epsilon_{H} (b) \sum_{(c),i,j \in I} (-1)^{|c^{'}||c^{''}| + |c^{'}||c^{''}|} u( e_{i} e_{j} ) e^{i}(c^{'})
e^{j}(S(c^{''})) = \]	
	\[ = v(1_{H}) \epsilon_{H} (b) \sum_{(c)} u( (\sum_{i \in I} e^{i}(c^{'}) e_{i}) (\sum_{j \in I} e^{j}( S(c^{''}) )e_{j}) )
= \]
	\[ = v(1_{H}) \epsilon_{H} (b) u( \sum_{(c)} c^{'} S(c^{''}) ) = \epsilon_{H}(b) u(1_{H}) \epsilon_{H}(c) v(1_{H}) =  \]
	\[ = \lambda_{H,H^{*},H,H^{*}} ( (id_{X} \otimes ev_{H} \otimes id_{X} \otimes ev_{H}) (1_{X} \otimes 1_{H} \otimes 1_{X}
\otimes 1_{H})) (\xi). \]
	
	Consequently, $R\overline{R}=1_{X} \otimes 1_{H} \otimes 1_{X} \otimes 1_{H}$.
	
	Consider the element $\overline{R}R$.		
	\[ \overline{R}R = ( \sum_{i \in I} ( 1_{X} \otimes e_{i} ) \otimes ( (e^{i} \circ S) \otimes 1_{H} ) ) ( \sum_{j \in I} (
1_{X} \otimes e_{j} ) \otimes ( e^{j} \otimes 1_{H} )  ) = \]
	\[ = \sum_{i,j \in I} (-1)^{|e^{i}||e_{j}|} ( 1_{X} \otimes e_{i} )( 1_{X} \otimes e_{j} ) \otimes ( (e^{i} \circ S) \otimes
1_{H} )( e^{j} \otimes 1_{H} ) = \]
	\[ = \sum_{i,j \in I} (-1)^{|e^{i}||e_{j}|} ( 1_{X} \otimes e_{i}e_{j} ) \otimes ( (e^{i} \circ S) e^{j} \otimes 1_{H} ).
\]
	
	We evaluate $\xi=b \otimes u \otimes c \otimes v \in H \otimes X \otimes H \otimes X$ on $\overline{R}R$.
	\[ \lambda_{H,H^{*},H,H^{*}} ( ((id_{X} \otimes ev_{H} \otimes id_{X} \otimes ev_{H}) (\overline{R}R))) (\xi) = \]
	\[ = \lambda_{H,H^{*},H,H^{*}} ( (\sum_{i,j \in I} (-1)^{|e^{i}||e_{j}|} ( 1_{X} \otimes e_{i}e_{j} ) \otimes ( (e^{i} \circ
S) e^{j} \otimes 1_{H} )))  (b \otimes u \otimes c \otimes v) = \]
	\[ = v(1_{H}) \epsilon_{H}(b) \sum_{i,j \in I} (-1)^{|e^{i}||e_{j}| + ( |b| + |u| ) |(e^{i} \circ S)e^{j}| + |e_{i}e_{j}||b|
+ |u||e_{i} e_{j}| } u( e_{i} e_{j} ) ((e^{i} \circ S) e^{j})(c) = \]
	\[ = v(1_{H}) \epsilon_{H}(b) \sum_{i,j \in I} (-1)^{|e_{i}||e_{j}| + (|b| + |u|)(|e_{i}| + |e_{j}|) + (|e_{i}| +
|e_{j}|)|b| + |u|(|e_{i}| + |e_{j}|) } u(e_{i} e_{j}) ((e^{i} \circ S) e^{j})(c) = \]
	\[ = v(1_{H}) \epsilon_{H}(b) \sum_{i,j \in I} (-1)^{|e_{i}||e_{j}|} u(e_{i} e_{j}) (\Delta_{H})^{*} ( \lambda_{H,H} ((e^{i}
\circ S) \otimes e^{j})) (c) = \]
	\[ = v(1_{H}) \epsilon_{H}(b) \sum_{(c), i,j \in I} (-1)^{|e_{i}||e_{j}| + |e_{j}||c^{'}|} u(e_{i} e_{j}) e^{i}(S(c^{'}))
e^{j}(c^{''}) = \]
	\[ = v(1_{H}) \epsilon_{H}(b) \sum_{(c), i,j \in I} (-1)^{|c^{'}||c^{''}| + |c^{'}||c^{''}|} u(e_{i} e_{j}) e^{i}(S(c^{'}))
e^{j}(c^{''}) = \]
	\[ = v(1_{H}) \epsilon_{H}(b) \sum_{(c)}  u( (\sum_{i \in I} e^{i}( S(c^{'}) )e_{i}) (\sum_{j \in I} e^{j}(c^{''}) e_{j}) )
= \]
	\[ = v(1_{H}) \epsilon_{H} (b) u(\sum_{(c)} S(c^{'}) c^{''}) = \epsilon_{H}(b) u(1_{H}) \epsilon_{H}(c) v(1_{H}) = \]
	\[ = \lambda_{H,H^{*},H,H^{*}} ( (id_{X} \otimes ev_{H} \otimes id_{X} \otimes ev_{H}) (1_{X} \otimes 1_{H} \otimes 1_{X}
\otimes 1_{H})) (\xi). \]
	
	Consequently, $\overline{R}R=1_{X} \otimes 1_{H} \otimes 1_{X} \otimes 1_{H}$.
	
	Thus $ \overline{R}R = R\overline{R} = 1_{X} \otimes 1_{H} \otimes 1_{X} \otimes 1_{H}$.
	
	2. We fix an element $f \in X, \; a \in H$. Then
	\[ ((\tau_{X \bowtie A, X \bowtie A} \circ \Delta_{X \bowtie A}) (f \otimes a)) R = \]
	\[ = \sum_{(f),(a),i \in I} (-1)^{|f^{'}||f^{''}| + |f^{'}||a^{'}| + (|f^{''}| + |a^{'}|)(|f^{'}| + |a^{''}|) +
|e_{i}|(|f^{''}| + |a^{'}|)} (f^{'} \otimes a^{''})( 1_{X} \otimes e_{i} ) \otimes (f^{''} \otimes a^{'})( e^{i} \otimes 1_{H} )
= \]
	\[ = \sum_{(f),(a),(a^{'}),((a^{'})^{'}),i \in I} (-1)^{|f^{'}||f^{''}| + |f^{'}||a^{'}| + (|f^{''}| + |a^{'}|)(|f^{'}| +
|a^{''}|) + |e_{i}|(|f^{''}| + |a^{'}|) + |e^{i}||a^{'}| + |(a^{'})^{'}||(a^{'})^{''}| + |?||((a^{'})^{'})^{'}| } * \]
	\[ * ( f^{'} \otimes a^{''}e_{i} ) \otimes (f^{''} e^{i}( S^{-1}((a^{'})^{''}) ? ((a^{'})^{'})^{'} ) \otimes
((a^{'})^{'})^{''}) = \]
	\[ = \sum_{(f),(a),(a^{'}),((a^{'})^{'}),i \in I} (-1)^{|f^{''}|(|a^{''}| + |e_{i}|) + |a^{'}||a^{''}|  +
|(a^{'})^{'}||(a^{'})^{''}| + |?||((a^{'})^{'})^{'}|} * \]
	\[ * ( f^{'} \otimes a^{''}e_{i} ) \otimes (f^{''} e^{i}( S^{-1}((a^{'})^{''}) ? ((a^{'})^{'})^{'} ) \otimes
((a^{'})^{'})^{''}).\]
	
	We fix an element $\xi= b \otimes u \otimes c \otimes v \in H \otimes X \otimes H \otimes X$. Then
	
	\[ \lambda_{H,H^{*},H,H^{*}} ( ((id_{X} \otimes ev_{H} \otimes id_{X} \otimes ev_{H}) (\Delta_{X \bowtie A}^{op} (f \otimes
a) R))) (\xi) = \]
	\[ = \sum_{(f),(a),(a^{'}),((a^{'})^{'}),(c),i \in I} (-1)^{|f^{''}|(|a^{''}| + |e_{i}|) + |a^{'}||a^{''}|  +
|(a^{'})^{'}||(a^{'})^{''}| + |c^{''}||((a^{'})^{'})^{'}| + |u|( |a^{''}| + |e_{i}| ) + |((a^{'})^{'})^{''}||v|} * \]
	\[ * (-1)^{ |b| ( |a| + |f^{''}| ) + |u| ( |a^{'}| + |f^{''}| + |e^{i}| ) + |c||((a^{'})^{'})^{''}| + |c^{'}|( |e^{i}| +
|(a^{'})^{''}| + |((a^{'})^{'})^{'}| )} * \]
	\[ * f^{'}(b)  u(a^{''}e_{i}) f^{''}(c^{'}) e^{i}( S^{-1}((a^{'})^{''}) c^{''} ((a^{'})^{'})^{'} )  v(((a^{'})^{'})^{''}) =
\]
	\[ = \sum_{(a),(a^{'}),((a^{'})^{'}),(c),i \in I} (-1)^{|c^{'}|(|a^{''}| + |u|) + |a^{'}||a^{''}|  +
|(a^{'})^{'}||(a^{'})^{''}| + |c^{''}||((a^{'})^{'})^{'}| + |u||a| + |((a^{'})^{'})^{''}||v|} * \]
	\[ * (-1)^{ |b| |a| + |c||((a^{'})^{'})^{''}| + |c^{'}|( |(a^{'})^{''}| + |((a^{'})^{'})^{'}| )}    \lambda_{H,H} (
\Delta_{H^{*}} (f)) (b \otimes c^{'})  u(a^{''}e_{i}) e^{i}( S^{-1}((a^{'})^{''}) c^{''} ((a^{'})^{'})^{'} )
v(((a^{'})^{'})^{''}) = \]	
	\[ = \sum_{(a),(a^{'}),((a^{'})^{'}),(c)} (-1)^{|c^{'}|(|a^{''}| + |u|) + |a^{'}||a^{''}|  + |(a^{'})^{'}||(a^{'})^{''}| +
|u||a| + |((a^{'})^{'})^{''}||v|} * \]
	\[ * (-1)^{ |b| |a| + |c||(a^{'})^{'}| + |c^{'}||(a^{'})^{''}|}  f(b c^{'})  u(a^{''} \sum_{i \in I}  e^{i}(
S^{-1}((a^{'})^{''}) c^{''} ((a^{'})^{'})^{'} ) e_{i}) v(((a^{'})^{'})^{''}) = \]	
	\[ = \sum_{(a),(a^{'}),((a^{'})^{'}),(c)} (-1)^{|c^{'}|(|a^{''}| + |u|) + |a^{'}||a^{''}|  + |(a^{'})^{'}||(a^{'})^{''}| +
|u||a| + |((a^{'})^{'})^{''}||v| + |b| |a| + |c||(a^{'})^{'}| + |c^{'}||(a^{'})^{''}|} * \]
	\[ * f(b c^{'})  u(a^{''} S^{-1}((a^{'})^{''}) c^{''} ((a^{'})^{'})^{'}) v(((a^{'})^{'})^{''}) = \]	
	\[ = \sum_{(a),(a^{'}),((a^{'})^{'}),(c)} (-1)^{ |b||a| + |c^{'}|(|a| + |u|) + |u||a| + |a^{''}||a^{'}|  +
|(a^{'})^{''}||(a^{'})^{'}| + |c^{''}||(a^{'})^{'}| + |v||((a^{'})^{'})^{''}|} * \]
	\[ * f(b c^{'})  u(a^{''} S^{-1}((a^{'})^{''}) c^{''} ((a^{'})^{'})^{'}) v(((a^{'})^{'})^{''}) = \]	
	\[ = \nu_{k,k} \circ (id_{k} \otimes \nu_{k,k}) \circ (\rho_{H^{*},H} \otimes \rho_{H^{*},H} \otimes id_{k}) \circ \]
	\[ \circ (id_{X} \otimes \mu_{H} \otimes id_{X} \otimes (\mu_{H} \circ ((\mu_{H} \circ ( id_{H} \otimes S^{-1} )) \otimes
\mu_{H})) \otimes \rho_{H^{*},H}) \circ \]
	\[ \circ (id_{X} \otimes id_{H} \otimes id_{H} \otimes  id_{X} \otimes id_{H} \otimes id_{H} \otimes id_{H} \otimes id_{H}
\otimes \tau_{H,X}) \circ \]
	\[ \circ (id_{X} \otimes id_{H} \otimes id_{H} \otimes  id_{X} \otimes id_{H} \otimes id_{H} \otimes \tau_{H \otimes H,H}
\otimes id_{X}) \circ (id_{X} \otimes id_{H} \otimes id_{H} \otimes  id_{X} \otimes id_{H} \otimes \tau_{H \otimes H,H} \otimes
id_{H} \otimes id_{X}) \circ \]
	\[ \circ (id_{X} \otimes id_{H} \otimes id_{H} \otimes  id_{X} \otimes \tau_{H \otimes H \otimes H,H} \otimes id_{H} \otimes
id_{X}) \circ (id_{X} \otimes id_{H} \otimes id_{H} \otimes \tau_{H \otimes H \otimes H \otimes H,X} \otimes id_{H} \otimes
id_{X}) \circ \]
	\[ \circ (id_{X} \otimes id_{H} \otimes \tau_{H \otimes H \otimes H \otimes H \otimes X,H} \otimes id_{H} \otimes id_{X})
\circ (id_{X} \otimes \tau_{H \otimes H \otimes H \otimes H,H} \otimes id_{X} \otimes id_{H} \otimes id_{H} \otimes id_{X})
\circ \]
	\[ \circ (id_{X} \otimes ((\Delta_{H} \otimes id_{H} \otimes id_{H}) \circ (\Delta_{H} \otimes id_{H}) \circ \Delta_{H})
\otimes id_{H} \otimes id_{X} \otimes \Delta_{H} \otimes id_{X}) (f \otimes a \otimes b \otimes u \otimes c \otimes v) = \]
	\[ = \nu_{k,k} \circ (id_{k} \otimes \nu_{k,k}) \circ (\rho_{H^{*},H} \otimes \rho_{H^{*},H} \otimes id_{k}) \circ \]
	\[ \circ (id_{X} \otimes \mu_{H} \otimes id_{X} \otimes (\mu_{H} \circ ((\mu_{H} \circ ( id_{H} \otimes S^{-1} )) \otimes
\mu_{H})) \otimes \rho_{H^{*},H}) \circ \]
	\[ \circ (id_{X} \otimes id_{H} \otimes id_{H} \otimes  id_{X} \otimes id_{H} \otimes id_{H} \otimes id_{H} \otimes id_{H}
\otimes \tau_{H,X}) \circ \]
	\[ \circ (id_{X} \otimes id_{H} \otimes id_{H} \otimes  id_{X} \otimes id_{H} \otimes id_{H} \otimes \tau_{H \otimes H,H}
\otimes id_{X}) \circ (id_{X} \otimes id_{H} \otimes id_{H} \otimes  id_{X} \otimes id_{H} \otimes \tau_{H \otimes H,H} \otimes
id_{H} \otimes id_{X}) \circ \]
	\[ \circ (id_{X} \otimes id_{H} \otimes id_{H} \otimes  id_{X} \otimes \tau_{H \otimes H \otimes H,H} \otimes id_{H} \otimes
id_{X}) \circ (id_{X} \otimes id_{H} \otimes id_{H} \otimes \tau_{H \otimes H \otimes H \otimes H,X} \otimes id_{H} \otimes
id_{X}) \circ \]
	\[ \circ (id_{X} \otimes id_{H} \otimes \tau_{H \otimes H \otimes H \otimes H \otimes X,H} \otimes id_{H} \otimes id_{X})
\circ (id_{X} \otimes \tau_{H \otimes H \otimes H \otimes H,H} \otimes id_{X} \otimes id_{H} \otimes id_{H} \otimes id_{X})
\circ \]	
	\[ \circ (id_{X} \otimes ((id_{H} \otimes id_{H} \otimes \Delta_{H}) \circ (id_{H} \otimes \Delta_{H}) \circ \Delta_{H})
\otimes id_{H} \otimes id_{X} \otimes \Delta_{H} \otimes id_{X}) (f \otimes a \otimes b \otimes u \otimes c \otimes v) = \]
	\[ = \sum_{(a),(a^{''}),((a^{''})^{''}),(c)} (-1)^{ |b||a| + |c^{'}|(|a| + |u|) + |u||a| + |((a^{''})^{''})^{''}|( |a^{'}| +
|(a^{''})^{'}| + |((a^{''})^{''})^{'}| ) + |((a^{''})^{''})^{'}|( |a^{'}| + |(a^{''})^{'}| ) } * \]
	\[ * (-1)^{|c^{''}|( |a^{'}| + |(a^{''})^{'}| ) + |v||(a^{''})^{'}| }  f(b c^{'}) u(((a^{''})^{''})^{''}
S^{-1}(((a^{''})^{''})^{'}) c^{''} a^{'}) v((a^{''})^{'}) = \]
	\[ = \sum_{(a),(a^{''}),(c)} (-1)^{ |b||a| + |c^{'}|(|a| + |u|) + |u||a| + |(a^{''})^{''}|( |a^{'}| + |(a^{''})^{'}| ) +
|c^{''}|( |a^{'}| + |(a^{''})^{'}| ) + |v||(a^{''})^{'}| } * \]
	\[ * f(b c^{'}) u( \epsilon_{H}((a^{''})^{''}) c^{''} a^{'}) v((a^{''})^{'}) = \]
	\[ = \sum_{(a),(a^{''}),(c)} (-1)^{ |b||a| + |c^{'}|(|a| + |u|) + |u||a| + |c^{''}||a| + |v||a^{''}| } f(b c^{'}) u(c^{''}
a^{'}) v((a^{''})^{'} \epsilon_{H}((a^{''})^{''}) ) = \]
	\[ = \sum_{(a),(c)} (-1)^{ |b||a| + |c^{'}|(|a| + |u|) + |u||a| + |c^{''}||a| + |v||a^{''}| } f(b c^{'}) u(c^{''} a^{'})
v(a^{''}). \]	
	
	Remark. 2.1 We have by the coassociativity
	\[ (\Delta_{H} \otimes id_{H} \otimes id_{H}) \circ (\Delta_{H} \otimes id_{H}) \circ \Delta_{H} = (((\Delta_{H} \otimes
id_{H}) \circ \Delta_{H}) \otimes id_{H}) \circ \Delta_{H} = \]
	\[ = (((id_{H} \otimes \Delta_{H}) \circ \Delta_{H}) \otimes id_{H}) \circ \Delta_{H} = (id_{H} \otimes \Delta_{H} \otimes
id_{H}) \circ (\Delta_{H} \otimes id_{H}) \circ \Delta_{H} = \]
	\[ = (id_{H} \otimes \Delta_{H} \otimes id_{H}) \circ (id_{H} \otimes \Delta_{H}) \circ \Delta_{H} = (id_{H} \otimes
((\Delta_{H} \otimes id_{H}) \circ \Delta_{H})) \circ \Delta_{H} = \]
	\[ = (id_{H} \otimes ((id_{H} \otimes \Delta_{H}) \circ \Delta_{H})) \circ \Delta_{H} = (id_{H} \otimes id_{H} \otimes
\Delta_{H}) \circ (id_{H} \otimes \Delta_{H}) \circ \Delta_{H}. \qed \]
	
	For early fixed $f \in X, \; a \in H$ it follows that
	\[ R \Delta_{X \bowtie A} (f \otimes a) = ( \sum_{i \in I} ( 1_{X} \otimes e_{i} ) \otimes ( e^{i} \otimes 1_{H} ) ) (
\sum_{(f),(a)} (-1)^{|f^{'}||f^{''}| + |f^{'}||a^{'}|} (f^{''} \otimes a^{'}) \otimes (f^{'} \otimes a^{''}) ) = \]
	\[ = \sum_{(f),(a),i \in I} (-1)^{|f^{'}||f^{''}| + |f^{'}||a^{'}| + |e^{i}|( |f^{''}| + |a^{'}| ) } ( 1_{X} \otimes e_{i}
)(f^{''} \otimes a^{'}) \otimes ( e^{i} \otimes 1_{H} )(f^{'} \otimes a^{''}) = \]
	\[ = \sum_{(f),(a),i \in I, (e_{i}),((e_{i})^{'})} (-1)^{|f^{'}||f^{''}| + |f^{'}||a^{'}| + |e^{i}|( |f^{''}| + |a^{'}| ) +
|f^{''}||e_{i}| + |(e_{i})^{'}||(e_{i})^{''}| + |?||((e_{i})^{'})^{'}|} * \]
	\[ * f^{''}( S^{-1}( (e_{i})^{''} ) ? ((e_{i})^{'})^{'} ) \otimes ((e_{i})^{'})^{''} a^{'} \otimes e^{i} f^{'} \otimes
a^{''}. \]
	
	Consider the early fixed element $\xi= b \otimes u \otimes c \otimes v \in H \otimes X \otimes H \otimes X$. Then
	
	\[ \lambda_{H,H^{*},H,H^{*}} ( ((id_{X} \otimes ev_{H} \otimes id_{X} \otimes ev_{H}) (R \Delta_{X \bowtie A} (f \otimes
a)))) (\xi) = \]	
	\[ = \sum_{(f),(a),i \in I, (e_{i}),((e_{i})^{'})} (-1)^{|f^{'}||f^{''}| + |f^{'}||a^{'}| + |e^{i}||a^{'}| +
|(e_{i})^{'}||(e_{i})^{''}| + |b|( |a| + |e^{i}| + |f^{'}| + |(e_{i})^{'}| )} * \]
	\[ * (-1)^{ |u|( |a| + |f^{'}| + |e^{i}| + |((e_{i})^{'})^{''}| ) + |c||a^{''}| + |c^{'}||f^{'}| + |v||a^{''}| }  f^{''}(
S^{-1}( (e_{i})^{''} ) b ((e_{i})^{'})^{'} ) u(((e_{i})^{'})^{''} a^{'}) e^{i}(c^{'}) f^{'}(c^{''}) v(a^{''}) = \]
	\[ = \sum_{(f),(a),i \in I, (e_{i}),((e_{i})^{'})} (-1)^{|c^{''}||a^{'}| + |e^{i}||a^{'}| + |(e_{i})^{'}||(e_{i})^{''}| +
|b|( |a| + |e^{i}| + |(e_{i})^{'}| ) + |c^{''}|( |(e_{i})^{''}| + |((e_{i})^{'})^{'}| )} * \]
	\[ * (-1)^{ |u|( |a| + |c^{''}| + |e^{i}| + |((e_{i})^{'})^{''}| ) + |c||a^{''}| + |e^{i}||c^{''}| + |v||a^{''}| } * \]
	\[ * \lambda_{H,H} ( \Delta_{(H^{op})^{*}} (f)) ( S^{-1}( (e_{i})^{''} ) b ((e_{i})^{'})^{'} \otimes c^{''}) v(a^{''} )
u(((e_{i})^{'})^{''} a^{'}) e^{i}(c^{'}) = \]
	\[ = \sum_{(f),(a),i \in I, (e_{i}),((e_{i})^{'})} (-1)^{|c^{''}||a^{'}| + |e^{i}||a^{'}| + |(e_{i})^{'}||(e_{i})^{''}| +
|b|( |a| + |c^{''}| + |e^{i}| + |(e_{i})^{'}| )} * \]
	\[ * (-1)^{ |u|( |a| + |c^{''}| + |e^{i}| + |((e_{i})^{'})^{''}| ) + |c||a^{''}| + |e_{i}||c^{''}| + |v||a^{''}| } f(c^{''}
S^{-1}( (e_{i})^{''} ) b ((e_{i})^{'})^{'}) e^{i}(c^{'}) u(((e_{i})^{'})^{''} a^{'}) v(a^{''} ) = \]		
	\[ = \sum_{(f),(a),i \in I, (e_{i}),((e_{i})^{'})} (-1)^{ |c^{''}|( |e_{i}| + |u| + |b| + |a| ) + |(e_{i})^{''}|(
|(e_{i})^{'}| + |b| + |a| + |u| ) + |b||a| + |((e_{i})^{'})^{'}|( |a| + |u| ) + |u||a| + |((e_{i})^{'})^{''}||a| + |v||a^{''}| }
* \]
	\[ * f(c^{''} S^{-1}( (e_{i})^{''} ) b ((e_{i})^{'})^{'}) e^{i}(c^{'}) u(((e_{i})^{'})^{''} a^{'}) v(a^{''} ) = \]
	\[ = \nu_{k,k} \circ (\rho_{H^{*},H} \otimes \nu_{k,k}) \circ (id_{X} \otimes (\mu_{H} \circ ((\mu_{H} \circ (id_{H} \otimes
S^{-1})) \otimes \mu_{H})) \otimes (\rho_{H^{*},H} \circ ( id_{X} \otimes \mu_{H} )) \otimes \rho_{H^{*},H}) \circ \]
	\[ \circ (id_{X} \otimes id_{H} \otimes id_{H} \otimes id_{H} \otimes id_{H} \otimes id_{X} \otimes id_{H} \otimes id_{H}
\otimes \tau_{H,X}) \circ (id_{X} \otimes id_{H} \otimes id_{H} \otimes id_{H} \otimes id_{H} \otimes id_{X} \otimes id_{H}
\otimes \Delta_{H} \otimes id_{X}) \circ \]
	\[ \circ (id_{X} \otimes id_{H} \otimes id_{H} \otimes id_{H} \otimes id_{H} \otimes id_{X} \otimes \tau_{H,H} \otimes
id_{X}) \circ (id_{X} \otimes id_{H} \otimes id_{H} \otimes id_{H} \otimes id_{H} \otimes \tau_{H,X} \otimes id_{H} \otimes
id_{X}) \circ \]
	\[ \circ (id_{X} \otimes id_{H} \otimes id_{H} \otimes id_{H} \otimes \tau_{H \otimes X,H} \otimes id_{H} \otimes id_{X})
\circ (id_{X} \otimes id_{H} \otimes id_{H} \otimes \tau_{H,H} \otimes id_{X} \otimes id_{H} \otimes id_{H} \otimes id_{X})
\circ \]
	\[ \circ (id_{X} \otimes id_{H} \otimes \tau_{H \otimes H \otimes X \otimes H \otimes H,H} \otimes id_{X}) \circ (id_{X}
\otimes \tau_{H \otimes H \otimes X \otimes H \otimes H \otimes H,H} \otimes id_{X}) \circ \]
	\[ \circ (id_{X} \otimes id_{H} \otimes id_{H} \otimes id_{X} \otimes (\nu_{k,H \otimes H \otimes H} \circ ( \rho_{H^{*},H}
\otimes ((\Delta_{H} \otimes id_{H}) \circ \Delta_{H}))) \otimes id_{H} \otimes id_{X}) \circ \]
	\[ \circ (id_{X} \otimes id_{H} \otimes id_{H} \otimes id_{X} \otimes \tau_{H, X \otimes H} \otimes id_{H} \otimes id_{X})
\circ \]
	\[ \circ (\tau_{H \otimes X,X \otimes H \otimes H \otimes X} \otimes \Delta_{H} \otimes id_{X}) (e_{i} \otimes e^{i} \otimes
f \otimes a \otimes b \otimes u \otimes c \otimes v) = \]
	\[ = \nu_{k,k} \circ (\rho_{H^{*},H} \otimes \nu_{k,k}) \circ (id_{X} \otimes (\mu_{H} \circ ((\mu_{H} \circ (id_{H} \otimes
S^{-1})) \otimes \mu_{H})) \otimes (\rho_{H^{*},H} \circ ( id_{X} \otimes \mu_{H} )) \otimes \rho_{H^{*},H}) \circ \]
	\[ \circ (id_{X} \otimes id_{H} \otimes id_{H} \otimes id_{H} \otimes id_{H} \otimes id_{X} \otimes id_{H} \otimes id_{H}
\otimes \tau_{H,X}) \circ (id_{X} \otimes id_{H} \otimes id_{H} \otimes id_{H} \otimes id_{H} \otimes id_{X} \otimes id_{H}
\otimes \Delta_{H} \otimes id_{X}) \circ \]
	\[ \circ (id_{X} \otimes id_{H} \otimes id_{H} \otimes id_{H} \otimes id_{H} \otimes id_{X} \otimes \tau_{H,H} \otimes
id_{X}) \circ (id_{X} \otimes id_{H} \otimes id_{H} \otimes id_{H} \otimes id_{H} \otimes \tau_{H,X} \otimes id_{H} \otimes
id_{X}) \circ \]
	\[ \circ (id_{X} \otimes id_{H} \otimes id_{H} \otimes id_{H} \otimes \tau_{H \otimes X,H} \otimes id_{H} \otimes id_{X})
\circ (id_{X} \otimes id_{H} \otimes id_{H} \otimes \tau_{H,H} \otimes id_{X} \otimes id_{H} \otimes id_{H} \otimes id_{X})
\circ \]
	\[ \circ (id_{X} \otimes id_{H} \otimes \tau_{H \otimes H \otimes X \otimes H \otimes H,H} \otimes id_{X}) \circ (id_{X}
\otimes \tau_{H \otimes H \otimes X \otimes H \otimes H \otimes H,H} \otimes id_{X}) \circ \]
	\[ \circ (\sum_{(c),(c^{''}),((c^{''})^{''})} f \otimes a \otimes b \otimes u \otimes c^{'} \otimes (c^{''})^{'} \otimes
((c^{''})^{''})^{'} \otimes ((c^{''})^{''})^{''} \otimes v) = \]
	\[ = \sum_{(c),(c^{''}),((c^{''})^{''})} (-1)^{ |((c^{''})^{''})^{''}|( |a| + |b| + |u| + |c^{'}| + |(c^{''})^{'}| +
|((c^{''})^{''})^{'}| ) + |((c^{''})^{''})^{'}|( |a| + |b| + |u| + |c^{'}| + |(c^{''})^{'}| ) + |b||a| + |c^{'}|( |a| + |u| ) }
* \]
	\[ * (-1)^{ |u||a| + |(c^{''})^{'}||a| + |v||a^{''}| } f(((c^{''})^{''})^{''} S^{-1}(((c^{''})^{''})^{'}) b c^{'})
u((c^{''})^{'} a^{'}) v(a^{''}) = \]	
	\[ = \sum_{(c),(c^{''})} (-1)^{ |b||a| + |c^{'}|( |a| + |u| ) + |u||a| + |c^{''}||a| + |v||a^{''}| } f( b c^{'}) u(
(c^{''})^{'} \epsilon_{H}((c^{''})^{''}) a^{'}) v(a^{''}) = \]	
	\[ = \sum_{(c)} (-1)^{ |b||a| + |c^{'}|( |a| + |u| ) + |u||a| + |c^{''}||a| + |v||a^{''}| } f( b c^{'}) u( c^{''} a^{'})
v(a^{''}). \]		
	
	Remark.	
	2.2 Set
	\[ c = \sum_{i \in I} e^{i} (c) e_{i}. \]
	Then it follows
	\[ (\Delta_{H} \otimes id_{H}) \circ \Delta_{H} (c) = \sum_{(c),(c^{'})} (c^{'})^{'} \otimes (c^{'})^{''} \otimes c^{''} =
\sum_{i \in I, (e_{i}),(e_{i}^{'})} e^{i} (c) (e_{i}^{'})^{'} \otimes (e_{i}^{'})^{''} \otimes e_{i}^{''}. \]
	
	2.3 It is easy to see that
	\[ (id_{X} \otimes id_{H} \otimes id_{H} \otimes id_{X} \otimes (\nu_{k,H \otimes H \otimes H} \circ ( \rho_{H^{*},H}
\otimes ((\Delta_{H} \otimes id_{H}) \circ \Delta_{H}))) \otimes id_{H} \otimes id_{X}) \circ \]
	\[ \circ (id_{X} \otimes id_{H} \otimes id_{H} \otimes id_{X} \otimes \tau_{H, X \otimes H} \otimes id_{H} \otimes id_{X})
\circ \]
	\[ \circ (\tau_{H \otimes X,X \otimes H \otimes H \otimes X} \otimes \Delta_{H} \otimes id_{X}) (e_{i} \otimes e^{i} \otimes
f \otimes a \otimes b \otimes u \otimes c \otimes v) = \]
	\[ = \sum_{i \in I, (e_{i}),(e_{i}^{'}), (c)} f \otimes a \otimes b \otimes u \otimes e^{i}(c^{'}) (e_{i}^{'})^{'} \otimes
(e_{i}^{'})^{''} \otimes e_{i}^{''} \otimes  c^{''} \otimes v = \]
	\[ = \sum_{(c),(c^{'}),((c^{'})^{'})} f \otimes a \otimes b \otimes u \otimes ((c^{'})^{'})^{'} \otimes ((c^{'})^{'})^{''}
\otimes (c^{'})^{''} \otimes  c^{''} \otimes v = \]
	\[ = \sum_{(c),(c^{''}),((c^{''})^{''})} f \otimes a \otimes b \otimes u \otimes c^{'} \otimes (c^{''})^{'} \otimes
((c^{''})^{''})^{'} \otimes ((c^{''})^{''})^{''} \otimes v. \qed \]
	
	Thus we have for all $f \in X, \; a \in H$
	\[ ((\tau_{X \bowtie A, X \bowtie A} \circ \Delta_{X \bowtie A}) (f \otimes a)) R = R \Delta_{X \bowtie A} (f \otimes a).
\]
	
	3. We prove that
	\[ ( \Delta_{D(H)} \otimes id_{D(H)} ) (R) = R_{13} R_{23}. \]
	
	Note that
	\[ ( \Delta_{D(H)} \otimes id_{D(H)} ) (R) = ( \Delta_{D(H)} \otimes id_{D(H)} ) ( \sum_{i \in I} ( 1_{X} \otimes e_{i} )
\otimes ( e^{i} \otimes 1_{H} ) ) = \]
	\[ = \sum_{i \in I} 1_{X} \otimes (e_{i})^{'} \otimes 1_{X} \otimes (e_{i})^{''} \otimes e^{i} \otimes 1_{H}. \]
	
	We fix an element $a \otimes t \otimes b \otimes u \otimes c \otimes v \in H \otimes X \otimes H \otimes X \otimes H \otimes
X$.
	\[\lambda_{H,H^{*},H,H^{*},H,H^{*}}((id_{X}\otimes ev_{H}\otimes id_{X}\otimes ev_{H}\otimes id_{X}\otimes
ev_{H})((\Delta_{D(H)}\otimes id_{D(H)})(R)))(a\otimes t\otimes b\otimes u\otimes c\otimes v)=\]
	\[=\epsilon_{H}(a)\epsilon_{H}(b)v(1_{H})\sum_{i\in
I,(e_{i})}(-1)^{|e^{i}|(|a|+|t|+|b|+|u|)+|(e_{i})^{''}|(|a|+|t|+|b|+|u|)+|(e_{i})^{'}|(|a|+|t|)}t((e_{i})^{'})u((e_{i})^{''})e^{i}(c)=\]
	\[=\epsilon_{H}(a)\epsilon_{H}(b)v(1_{H})\sum_{i\in I,(e_{i})}(-1)^{|(e_{i})^{'}||u|}t((e_{i})^{'})u((e_{i})^{''})e^{i}(c)=
\epsilon_{H}(a)\epsilon_{H}(b)v(1_{H})\sum_{(c)}(-1)^{|u||c^{'}|}t(c^{'})u(c^{''}).\]
	
	Remark. 3.1 Set
	\[ c = \sum_{i \in I} e^{i}(c) e_{i}. \]
	Then
	\[ \Delta(c) = \sum_{i \in I,(e_{i})} e^{i}(c) (e_{i})^{'} \otimes (e_{i})^{''}. \]
	It is easy to see that
	\[\lambda_{H,H}(t\otimes u)(\sum_{(c)}c^{'}\otimes c^{''})=\lambda_{H,H}(t\otimes u)(\sum_{i\in
I,(e_{i})}e^{i}(c)(e_{i})^{'}\otimes(e_{i})^{''}) \iff \]
	\[ \iff \sum_{(c)}(-1)^{|u||c^{'}|}t(c^{'})u(c^{''})=\sum_{i\in
I,(e_{i})}(-1)^{|u||(e_{i})^{'}|}e^{i}(c)t((e_{i})^{'})u((e_{i})^{''}). \qed \]
	
	Evaluate
	\[R_{13}R_{23}=(\sum_{i\in I}(1_{X}\otimes e_{i})\otimes(1_{X}\otimes1_{H})\otimes(e^{i}\otimes1_{H}))(\sum_{j\in
I}(1_{X}\otimes1_{H})\otimes(1_{X}\otimes e_{j})\otimes(e^{j}\otimes1_{H}))=\]
	\[=\sum_{i,j\in I}(-1)^{|e^{i}||e_{j}|}(1_{X}\otimes e_{i})\otimes(1_{X}\otimes e_{j})\otimes(e^{i}e^{j}\otimes1_{H}).\]
	
	Consider the early fixed element $a \otimes t \otimes b \otimes u \otimes c \otimes v \in H \otimes X \otimes H \otimes X
\otimes H \otimes X$. Then
	\[\lambda_{H,H^{*},H,H^{*},H,H^{*}}((id_{X}\otimes ev_{H}\otimes id_{X}\otimes ev_{H}\otimes id_{X}\otimes
ev_{H})(R_{13}R_{23}))(a\otimes t\otimes b\otimes u\otimes c\otimes v)=\]
	\[=\epsilon_{H}(a)\epsilon_{H}(b)v(1_{H})\sum_{i,j\in
I,(c)}(-1)^{|e^{i}||e_{j}|+(|e_{i}|+|e_{j}|)(|a|+|t|+|b|+|u|)+|e^{j}||c^{'}|+|e_{j}|(|a|+|t|+|b|+|u|)+|e_{i}|(|a|+|t|)} *\]
	\[ * t(e_{i})u(e_{j})e^{i}(c^{'})e^{j}(c^{''})=\]
	\[=\epsilon_{H}(a)\epsilon_{H}(b)v(1_{H})\sum_{i,j\in
I,(c)}(-1)^{|e^{i}||e_{j}|+|e_{i}|(|b|+|u|)+|e^{j}||c^{'}|}t(e_{i})u(e_{j})e^{i}(c^{'})e^{j}(c^{''})=\]
	\[=\epsilon_{H}(a)\epsilon_{H}(b)v(1_{H})\sum_{i,j\in
I,(c)}(-1)^{|e^{i}||e_{j}|+|e_{i}||u|+|e^{j}||e^{i}|}t(e_{i})u(e_{j})e^{i}(c^{'})e^{j}(c^{''})=\]
	\[=\epsilon_{H}(a)\epsilon_{H}(b)v(1_{H})\sum_{(c)}(-1)^{|c^{'}||u|}t(\sum_{i\in I}e^{i}(c^{'})e_{i})u(\sum_{j\in
I}e^{j}(c^{''})e_{j})=\epsilon_{H}(a)\epsilon_{H}(b)v(1_{H})\sum_{(c)}(-1)^{|u||c^{'}|}t(c^{'})u(c^{''}).\]
	
	Therefore
	\[ ( \Delta_{D(H)} \otimes id_{D(H)} ) (R) = R_{13} R_{23}. \]
	
	4. We prove that
	\[ ( id_{D(H)} \otimes \Delta_{D(H)} ) (R) = R_{13} R_{12}. \]
	
	Note that
	\[(id_{D(H)}\otimes\Delta_{D(H)})(R)=(id_{D(H)}\otimes\Delta_{D(H)})(\sum_{i\in I}(1_{X}\otimes
e_{i})\otimes(e^{i}\otimes1_{H}))=\]
	\[=\sum_{i\in I,(e^{i})}(-1)^{|(e^{i})^{'}||(e^{i})^{''}|}1_{X}\otimes
e_{i}\otimes(e^{i})^{''}\otimes1_{H}\otimes(e^{i})^{'}\otimes1_{H}.\]
	
	We fix an element $a \otimes t \otimes b \otimes u \otimes c \otimes v \in H \otimes X \otimes H \otimes X \otimes H \otimes
X$.
	\[\lambda_{H,H^{*},H,H^{*},H,H^{*}}((id_{X}\otimes ev_{H}\otimes id_{X}\otimes ev_{H}\otimes id_{X}\otimes
ev_{H})((id_{D(H)}\otimes\Delta_{D(H)})(R)))(a\otimes t\otimes b\otimes u\otimes c\otimes v)=\]
	\[=\epsilon_{H}(a)u(1_{H})v(1_{H})\sum_{i\in
I,(e^{i})}(-1)^{|(e^{i})^{'}||(e^{i})^{''}|+|(e^{i})^{'}|(|a|+|t|+|b|+|u|)+|(e^{i})^{''}|(|a|+|t|)+|e_{i}|(|a|+|t|)}t(e_{i})(e^{i})^{''}(b)(e^{i})^{'}(c)=\]
	\[=\epsilon_{H}(a)u(1_{H})v(1_{H})\sum_{i\in I}t(e_{i})\lambda_{H,H}(\Delta_{(H^{op})^{*}}(e^{i}))(b\otimes c)=\]
	\[=\epsilon_{H}(a)u(1_{H})v(1_{H})\sum_{i\in
I}(-1)^{|b||c|}t(e_{i})e^{i}(cb)=(-1)^{|b||c|}\epsilon_{H}(a)u(1_{H})v(1_{H})t(\sum_{i\in I}e^{i}(cb)e_{i})=\]
	\[=(-1)^{|b||c|}\epsilon_{H}(a)u(1_{H})v(1_{H})t(cb).\]
	
	Next we have
	\[ R_{13} R_{12} = ( \sum_{i \in I} ( 1_{X} \otimes e_{i} ) \otimes ( 1_{X} \otimes 1_{H} ) \otimes ( e^{i} \otimes 1_{H} )
) ( \sum_{j \in I} ( 1_{X} \otimes e_{j} ) \otimes ( e^{j} \otimes 1_{H} ) \otimes ( 1_{X} \otimes 1_{H} ) ) = \]
	\[ = \sum_{i,j \in I} (-1)^{|e^{i}| ( |e_{j}| + |e^{j}| ) } ( 1_{X} \otimes e_{i}e_{j} ) \otimes ( e^{j} \otimes 1_{H} )
\otimes ( e^{i} \otimes 1_{H} ) = \sum_{i,j \in I} ( 1_{X} \otimes e_{i}e_{j} ) \otimes ( e^{j} \otimes 1_{H} ) \otimes ( e^{i}
\otimes 1_{H} ). \]
	
	Consider the early fixed element $a \otimes t \otimes b \otimes u \otimes c \otimes v \in H \otimes X \otimes H \otimes X
\otimes H \otimes X$. Then		
	\[\lambda_{H,H^{*},H,H^{*},H,H^{*}}((id_{X}\otimes ev_{H}\otimes id_{X}\otimes ev_{H}\otimes id_{X}\otimes
ev_{H})(R_{13}R_{12}))(a\otimes t\otimes b\otimes u\otimes c\otimes v)=\]
	\[=\epsilon_{H}(a)u(1_{H})v(1_{H})\sum_{i,j\in
I}(-1)^{|e^{i}|(|a|+|t|+|b|+|u|)+|e^{j}|(|a|+|t|)+(|e_{i}|+|e_{j}|)(|a|+|t|)}t(e_{i}e_{j})e^{j}(b)e^{i}(c)=\]
	\[=\epsilon_{H}(a)u(1_{H})v(1_{H})(-1)^{|b||c|}t((\sum_{i\in I}e^{i}(c)e_{i})(\sum_{j\in I}e^{j}(b)e_{j}))=
(-1)^{|b||c|}\epsilon_{H}(a)u(1_{H})v(1_{H})t(cb).\]
	
	Therefore
	\[ ( id_{D(H)} \otimes \Delta_{D(H)} ) (R) = R_{13} R_{12}. \]
\end{proof}

\eq*
\begin{proof}
	We use mathematical induction on two variables to proof all identities. All relations are easy to proof. Therefore we
consider only the identity
	\[ e_1^r f_1^w = f_1^w e_1^r + [w,r > 0] \sum_{u=1}^{min(r,w)} \frac{[r]! [w]!}{[u]! [r-u]! [w-u]!}  f_1^{w-u} \times \]
	\[ \times [k_1; \; 2u-r-w] [k_1; \; 2u-r-w-1] ... [k_1; \; u-r-w+1] e_1^{r-u}, \]
	where $w,r \in \mathbb{N} $.
	
	If $w=0, \; r=0$ the statement is trivial. If $w=1, \; r=1$ we get			
	\[ e_1 f_1 = f_1 e_1 + [k_1; \; 0] = f_1 e_1 + \frac{k_1 - k_1^{-1}}{q - q^{-1}}. \]
	
	Suppose that the claim is true for all $1 \le p < r$ for $e_1$ and for all $1 \le q \le v$ for $f_1$. Using the induction
hypothesis,
	\[ e_1 f_1^v = f_1^v e_1 + \frac{[1]! [v]!}{[1]! [0]! [v-1]!} f_1^{v-1} [k_1; \; 1-v] = f_1^v e_1 + [v] f_1^{v-1} [k_1; \;
1-v]. \]
	We now prove that the equality holds for $r$:
	\[ e_1^r f_1^v =  e_1 ( f_1^v e_1^{r-1} + \sum_{u=1}^{min(r-1,v)} \frac{[r-1]! [v]!}{[u]! [r-u-1]! [v-u]!}  f_1^{v-u} \times	
\]
	\[ \times [k_1; \; 2u-r-v+1] [k_1; \; 2u-r-v] ... [k_1; \; u-r-v+2] e_1^{r-u-1} ) = \]
	\[ = f_1^v e_1^r + [r][v] f_1^{v-1} [k_1; 2-r-v] e_1^{r-1} +  \sum_{u=2}^{min(r-2,v-1)} \frac{[r]! [v]!}{[u]! [r-u]! [v-u]!}
f_1^{v-u} \times \]
	\[ \times [k_1; \; 2u-r-v] [k_1; \; 2u-r-v-1] [k_1; \; 2u-r-v-2] ... [k_1; \; u-r-v+1] e_1^{r-u} + \]
	\[ + [ r > v ] \frac{[r]!}{[r-v]!} [k_1; v-r] [k_1; v-r-1] ... [k_1; -r+1] e_1^{r-v} + \]
	\[ + [ v \ge r ] ( \frac{[r]! [v]!}{[r-1]! [v-r+1]!} f_1^{v-r+1} [k_1; \; r-v-2] [k_1; \; r-v-3] ... [k_1; \; -v] e_1 + \]
	\[ + \frac{ [v]!}{ [v-r]!} f_1^{v-r} [k_1; \; r-v] [k_1; \; r-v-1] [k_1; \; r-v-2] ... [k_1; \; -v+1]) = \]
	\[ = f_1^v e_1^r + \sum_{u=1}^{min(r,v)} \frac{[r]! [v]!}{[u]! [r-u]! [v-u]!} f_1^{v-u} \times \]
	\[ \times [k_1; \; 2u-r-v] [k_1; \; 2u-r-v-1] [k_1; \; 2u-r-v-2] ... [k_1; \; u-r-v+1] e_1^{r-u}. \]
	
	Suppose that the claim is true for all $1 \le p \le z$ for $e_1$ and for all $1 \le q < w$ for $f_1$. Using the induction
hypothesis,
	\[ e_1^z f_1 = f_1 e_1^z + \frac{[z]! [1]!}{[1]! [z-1]! [0]!} [k_1; 1-z] e_1^{z-1} = f_1 e_1^z + [z] [k_1; 1-z] e_1^{z-1}.
\]			
	We now prove that the equality holds for $w$:
	\[ e_1^z f_1^w = f_1^{w-1} e_1^z f_1 + \sum_{u=1}^{min(z,w-1)} \frac{[z]! [w-1]!}{[u]! [z-u]! [w-u-1]!}  f_1^{w-u-1} \times
\]
	\[ \times [k_1; \; 2u-z-w+1] [k_1; \; 2u-z-w] ... [k_1; \; u-z-w+2] e_1^{z-u} f_1 = \]
	\[ = f_1^w e_1^z + [z] [w] f_1^{w-1} [k_1; 2-z-w] e_1^{z-1} + \]
	\[ + \sum_{u=2}^{min(z-1,w-2)} \frac{[z]! [w]!}{[u]! [z-u]! [w-u]!}  f_1^{w-u} \times \]
	\[ \times [k_1; \; 2u-z-w] [k_1; \; 2u-z-w-1] [k_1; \; 2u-z-w-2] ... [k_1; \; u-z-w+1] e_1^{z-u} + \]
	\[ + [z \ge w] ( \frac{[z]![w]}{[z-w+1]!} f_1 [k_1; \; w-z-2] [k_1; \; w-z-3] [k_1; \; w-z-4] ... [k_1; \; -z] e_1^{z-w+1}
+ \]
	\[ + \frac{[z]!}{[z-w]!} [k_1; w-z] [k_1; \; w-z-1] [k_1; \; w-z-2] ... [k_1; \; -z+1] e_1^{z-w} ) + \]						

	\[ + [ w > z ] \frac{[w]!}{[w-z]!} f_1^{w-z} [k_1;z-w] [k_1; \; z-w-1] [k_1; \; z-w-2] ... [k_1; \; -w+1] = \]
	\[ = f_1^w e_1^z + \sum_{u=1}^{min(z,w)} \frac{[z]! [w]!}{[u]! [z-u]! [w-u]!}  f_1^{w-u} \times \]
	\[ \times [k_1; \; 2u-z-w] [k_1; \; 2u-z-w-1] ... [k_1; \; u-z-w+1] e_1^{z-u}. \]	
	
\end{proof}

\cent*
\begin{proof}
	We verify that the element $k_1^{d}$ belongs to the center of $U_q$. It follows from Lemma \ref{lm:relonbasis} that  		

	\[k_1^{d}f_1=q^{-2d}f_1k_1^d=f_1k_1^d, \; k_1^df_2=q^{d}f_2k_1^d=f_2k_1^d,\]
	\[k_1^df_3=q^{-d}f_3k_1^d=f_3k_1^d, \; k_2k_1^d=k_1^dk_2,\]
	\[e_1k_1^d=q^{-2d}k_1^{d}e_1=k_1^{d}e_1, \; e_2k_1^d=q^{d}k_1^de_2=k_1^de_2,\]
	\[e_3k_1^d=q^{-d}k_1^de_3=k_1^de_3.\]
	
	We verify that the element $k_2^{d}$ belongs to the center of $U_q$.
	\[k_2^df_1=q^{d}f_1k_2^d=f_1k_2^d, \; k_2^df_2=f_2k_2^d,\]
	\[k_2^df_3=q^{d}f_3k_2^d=f_3k_2^d, \; k_2^dk_1=k_1k_2^d,\]
	\[e_1k_2^d=q^{d}k_2^de_1=k_2^de_1, \; e_2k_2^d=k_2^de_2,\]
	\[e_3k_2^d=q^{d}k_2^de_3=k_2^de_3.\]
	
	We verify that the element $f_1^d$ belongs to the center of $U_q$. It follows from Lemma \ref{lm:relonbasis} that
	\[ f_2f_1^d = q^d f_1^d f_2 + [d] f_1^{d-1} f_3 = f_1^d f_2, \; f_3f_1^d=q^{-d}f_1^df_3=f_1^df_3, \]
	\[ k_1f_1^{d}=q^{-2d}f_1^dk_1=f_1^dk_1, \; k_2f_1^d=q^{d}f_1^dk_2=f_1^dk_2,\]
	\[ e_1f_1^d = f_1^d e_1 + \frac{[1]![d]!}{[1]![0]![d-1]!} f_1^{d-1} [k_1;1-d] = f_1^d e_1, \]
	\[e_2f_1^d=f_1^de_2, \; e_3 f_1^d = f_1^d e_3 - q^{d-2} [d] f_1^{d-1} k_1^{-1} e_2 = f_1^d e_3. \]		
	
	We verify that the element $e_1^d$ belongs to the center of $U_q$. It follows from Lemma \ref{lm:relonbasis} that
	\[ e_1^d f_1 = f_1 e_1^d + \frac{[d]![1]!}{[1]![d-1]![0]!} [k_1;1-d] e_1^{d-1} = f_1 e_1^d, \]
	\[e_1^df_2=f_2e_1^d, \; e_1^d f_3 = f_3 e_1^d + q^{2-d} [d] f_2 k_1 e_1^{d-1} = f_3 e_1^d, \]
	\[e_1^dk_1=q^{-2d}k_1^{d}e_1=k_1^{d}e_1, e_1^dk_2=q^{d}k_2e_1^d=k_2e_1^d,\]
	\[ e_2e_1^d = q^d e_1^d e_2 - q[d] e_1^{d-1} e_3 = e_1^d e_2, \; e_3e_1^d=q^{-d}e_1^de_3=e_1^de_3.\]    	
\end{proof}

\strHUU*
\begin{proof}
	We know from Proposition \ref{prHopfQuotient} that $\bar{U}_q$ is the Hopf superalgebra if a two-sided $\mathbb{Z}_2$-graded
ideal
	\[I=(f_1^d, k_1^d - 1, k_2^d - 1, e_1^d)\]
	is a $\mathbb{Z}_{2}$-graded Hopf ideal in $U_q$.	
	
	Note that
	\[ (k_1 \otimes e_1)(e_1 \otimes 1) = k_1 e_1 \otimes e_1 = q^2 e_1 k_1 \otimes e_1 = q^2 (e_1 \otimes 1) (k_1 \otimes e_1),
\]
	\[ ( f_1 \otimes k_1^{-1} ) ( 1 \otimes f_1 ) = f_1 \otimes k_1^{-1} f_1 = q^2 ( f_1 \otimes f_1 k_1^{-1} )= q^2 (1 \otimes
f_1)( f_1 \otimes k_1^{-1} ). \]
	We have for all $v \in \mathbb{N}$
	\[ (f_1k_1)^v = q^{-v(v-1)} f_1^v k_1^v, \]
	\[ ( k_1^{-1}e_1 )^v = q^{v(v-1)} k_1^{-v}e_1^{v}. \]
	
	We now prove that $\Delta(I) \subset I \otimes H + H \otimes I$. It is sufficient to consider only generation relations of
$I$. For $i \in \{1,2\}$
	\[ \Delta(k_i^d - 1) = \Delta^d(k_i) - \Delta(1) = k_i^d \otimes k_i^d - 1 \otimes 1 = \]
	\[ = k_i^d \otimes k_i^d - 1 \otimes k_1^d + 1 \otimes k_1^d - 1 \otimes 1 = (k_i^d - 1) \otimes k_i^d + 1 \otimes (k_i^d -
1) \subset I \otimes H + H \otimes I. \]
	
	\[\Delta(e_1^d) = \Delta^d(e_1)= (e_1 \otimes 1 + k_1 \otimes e_1)^d =\sum_{k=0}^{d} q^{k(d-k)} {d \brack k} (e_1 \otimes
1)^k (k_1 \otimes e_1)^{d-k} = \]
	\[ = \sum_{k=0}^{d} q^{k(k-d)} {d \brack k} k_1^{d-k} e_1^k \otimes e_1^{d-k} = k_1^d \otimes e_1^d + \sum_{k=1}^{d-1}
q^{k(k-d)} {d \brack k} k_1^{d-k} e_1^k \otimes e_1^{d-k} + e_1^d \otimes 1= \]
	\[ =  e_1^d \otimes 1 + k_1^d \otimes e_1^d \subset I \otimes H + H \otimes I. \]
	
	\[ \Delta^d(f_1)= (1 \otimes f_1 + f_1 \otimes k_1^{-1} )^d = \sum_{k=0}^{d} q^{k(d-k)} {d \brack k} ( 1 \otimes f_1 )^k (
f_1 \otimes k_1^{-1} )^{d-k} = \]
	\[ = \sum_{k=0}^{d} q^{k(d-k)} {d \brack k} f_1^{d-k} \otimes f_1^k k_1^{k-d} = f_1^d \otimes k_1^{-d} + \sum_{k=1}^{d-1}
q^{k(d-k)} {d \brack k} f_1^{d-k} \otimes f_1^k k_1^{k-d} + 1 \otimes f_1^d = \]
	\[ = f_1^d \otimes k_1^{-d} + 1 \otimes f_1^d \subset I \otimes H + H \otimes I. \]
	
	We prove that $\epsilon(I)=0$. It is sufficient to consider only generation relations of $I$. For $i \in \{1,2\}$
	\[ \epsilon(k_i^d-1)=\epsilon^d(k_i)-1=1-1=0, \; \epsilon(f_1^d)=\epsilon^d(f_1)=0, \; \epsilon(e_1^d)=\epsilon^d(e_1)=0.
\]
	
	We prove that $S(I) \subset I$. It is sufficient to consider only generation relations of $I$. For $i \in \{1,2\}$
	\[ S(k_i^d-1) = S^d(k_i) - S(1) = k_i^{-d} -1 \in I, \]
	\[ S(f_1^d) = S^d(f_1) = (-1)^d (f_1k_1)^d = (-1)^d q^{-d(d-1)} f_1^dk_1^d \in I, \]
	\[ S(e_1^d) = S^d(e_1) = (-1)^d ( k_1^{-1}e_1 )^d = (-1)^d q^{d(d-1)} k_1^{-d}e_1^{d} \in I.\]
\end{proof}

\subH*
\begin{proof}
	Since $B_q^+$ and $B_q^{-}$ are multiplicatively generated by homogeneous elements, they are $\mathbb{Z}_2$-graded subspaces
of vector superspace $\bar{U}_q$.
	
	We prove that $B_q^-$ is a Hopf subsuperalgebra. We have to prove that $\mu_{\bar{U}_q}(B_q^- \otimes B_q^-) \subset B_q^-$,
$\Delta_{\bar{U}_q}(B_q^-) \subset B_q^-$ and $S_{\bar{U}_q}(B_q^-) \subset B_q^-$.
	We verify that $B_q^-$ is closed under multiplication. It is sufficient to notice that
	\[ ( f_1^{w_1} f_3^{s_1} f_2^{l_1} k_1^{i_1} k_2^{j_1}) ( f_1^{w_2} f_3^{s_2} f_2^{l_2} k_1^{i_2} k_2^{j_2}) =  \]
	\[ = (-1)^{l_1s_2} q^{(j_1-2i_1-s_1)w_2+(j_1-i_1)s_2+i_1l_2+(w_2+s_2)l_1} f_1^{w_1+w_2} f_3^{s_1+s_2} f_2^{l_1+l_2}
k_1^{i_1+i_2} k_2^{j_1+j_2} + \]
	\[ + [s_1=s_2=0,l_1=1] q^{(j_1-2i_1)w_2+i_1l_2} [w_2] f_1^{w_1+w_2-1} f_3 f_2^{l_2} k_1^{i_1+i_2} k_2^{j_1+j_2} \in B_q^-.
\]
	
	We verify that $B_q^-$ is closed under comultiplication. It is sufficient to notice that
	\[ \Delta_{\bar{U}_q} (f_1^{w} f_3^{s} f_2^{l} k_1^{i} k_2^{j}) = \]
	\[ = \sum_{0 \le v \le w} q^{v(w-v)} {w \brack v} ( q^{(w-v)(s-l)} f_1^{w-v} k_1^i k_2^j \otimes f_1^v f_3^s f_2^l
k_1^{v-w+i} k_2^j + \]
	\[ + (-1)^s [l=1] q^{(w-v)s} f_1^{w-v} f_2 k_1^i k_2^j \otimes  f_1^v f_3^s k_1^{v-w+i} k_2^{j-1} + \]
	\[ + [s=1] q^{(2 - l)(w-v)} ( q^{-1} - q ) f_1^{w-v} f_2 k_1^i k_2^j \otimes f_1^{v+1} f_2^l k_1^{v-w+i} k_2^{j-1} + \]
	\[ + [s=1] q^{(v-w-1)l} f_1^{w-v} f_3 k_1^i k_2^j \otimes f_1^v f_2^l k_1^{v-w-1+i} k_2^{j-1} + \]
	\[ + [s=l=1] f_1^{w-v} f_3 f_2 k_1^i k_2^j \otimes f_1^v k_1^{v-w-1+i} k_2^{j-2} ) \in B_q^- \otimes B_q^-. \]
	
	Note that
	\[ \epsilon_{\bar{U}_q}(\sum_{ \substack{0 \leq w, i, j \leq d-1, \\  0 \leq s, l \leq 1}} \alpha_{w s l i j} f_1^{w}
f_3^{s} f_2^{l} k_1^{i} k_2^{j}) = \sum_{ 0 \leq i, j \leq d-1} \alpha_{0, 0, 0, i, j}. \]
	
	We verify that $S_{\bar{U}_q}(B_q^-) \subset B_q^-$. It is sufficient to notice that
	\[ S_{\bar{U}_q}( f_1^{w} f_3^{s} f_2^{l} k_1^{i} k_2^{j} ) = (-1)^{sl} S_{\bar{U}_q}(k_2)^j S_{\bar{U}_q}(k_1)^i
S_{\bar{U}_q}(f_2)^l S_{\bar{U}_q}(f_3)^s S_{\bar{U}_q}(f_1)^w = \]
	\[ = [s=0] (-1)^{l+w} q^{w(1-w+2l+2i-j)-il}  f_1^w f_2^l k_1^{w-i} k_2^{l-j} + \]
	\[ + [s=0,l=1] (-1)^{w+1} q^{w(2-w+2i-j)-i} [w] f_1^{w-1} f_3 k_1^{w-i} k_2^{1-j} + \]
	\[ + [s=1,l=0] (-1)^{w} q^{w(1-w+2i-j)+i-j} (q - q^3) f_1^{w+1} f_2 k_1^{w-i+1} k_2^{1-j} + \]
	\[ + [s=l=1] (-1)^{w} q^{w(1-w+2i-j) - j+2} f_1^w f_3 f_2 k_1^{1+w-i} k_2^{2-j} + \]
	\[ + [s=1,l=0] (-1)^{w+1} q^{w(-w+2i-j+1)+i-j+2} f_1^w f_3 k_1^{w-i+1} k_2^{1-j} \in B_q^-. \]
	
	Note that
	\[ S^{-1}_{\bar{U}_q}( f_1^{w} f_3^{s} f_2^{l} k_1^{i} k_2^{j} ) = (-1)^{sl} S^{-1}_{\bar{U}_q}(k_2)^j
S^{-1}_{\bar{U}_q}(k_1)^i S^{-1}_{\bar{U}_q}(f_2)^l S^{-1}_{\bar{U}_q}(f_3)^s S^{-1}_{\bar{U}_q}(f_1)^w = \]
	\[ = (1 - [s=l=1]) (-1)^{l+w+s} q^{w(-w-2s+2l+2i-j-1)+l(2s-i)+s(i-j)} f_1^w f_3^s f_2^l k_1^{w+s-i} k_2^{s+l-j} + \]
	\[ + [s=0,l=1] (-1)^{w+1} q^{w(-w+2i-j)-i} [w] f_1^{w-1} f_3 k_1^{w-i} k_2^{1-j} + \]		
	\[ + [s=1,l=0] (-1)^{w} q^{w(-w-1-j+2i)+i-j} (q^{-1} - q) f_1^{w+1} f_2 k_1^{w-i+1} k_2^{1-j} + \]
	\[ + [s=l=1] (-1)^{w} q^{w(-w-j+2i-1)-j} f_1^{w} f_3 f_2 k_1^{w-i+1} k_2^{2-j} \in B_q^-. \]
	
	We prove that $B_q^+$ is a Hopf subsuperalgebra. We have to prove that $\mu_{\bar{U}_q}(B_q^+ \otimes B_q^+) \subset B_q^+$,
$\Delta_{\bar{U}_q}(B_q^+) \subset B_q^+$ and $S_{\bar{U}_q}(B_q^+) \subset B_q^+$.
	We verify that $B_q^+$ is closed under multiplication
	\[ (k_1^{i_1}k_2^{j_1}e_1^{r_1}e_3^{h_1}e_2^{t_1}) (k_1^{i_2}k_2^{j_2}e_1^{r_2}e_3^{h_2}e_2^{t_2}) = \]
	\[ = (-1)^{t_1h_2} q^{(t_1-h_1-2r_1)i_2+(h_1+r_1)j_2-h_1r_2+(r_2+h_2)t_1}  k_1^{i_1+i_2}k_2^{j_1+j_2}e_1^{r_1+r_2}
e_3^{h_1+h_2} e_2^{t_1+t_2} - \]
	\[ - [h_1=h_2=0, t_1=1] q^{(1-2r_1)i_2+r_1j_2+1} [r_2] k_1^{i_1+i_2}k_2^{j_1+j_2}e_1^{r_1 +r_2-1} e_3 e_2^{t_2} \in B_q^+.
\]
	
	We verify that $B_q^+$ is closed under comultiplication
	\[\Delta_{\bar{U}_q}(k_1^i k_2^j e_1^r e_3^h e_2^t) = \Delta_{\bar{U}_q}(k_1)^i \Delta_{\bar{U}_q}(k_2)^j
\Delta_{\bar{U}_q}(e_1)^r \Delta_{\bar{U}_q}(e_3)^h \Delta_{\bar{U}_q}(e_2)^t = \]
	\[ = \sum_{v=0}^{r} q^{v(v-r)} {r \brack v} (k_1^{r-v+i} k_2^j e_1^v e_3^h e_2^t \otimes k_1^i k_2^j e_1^{r-v} + \]
	\[ + [t=1] q^{h+v} k_1^{r-v+i} k_2^{j+1} e_1^v e_3^h \otimes k_1^i k_2^j e_1^{r-v} e_2 + \]
	\[ + [h=1] (-1)^{t} q^{-v} k_1^{r-v+i+1} k_2^{j+1} e_1^v e_2^t \otimes k_1^i k_2^j e_1^{r-v} e_3 + \]
	\[ + [h=1] (-1)^{t} q^{v} (q-q^{-1}) k_1^{r-v+i} k_2^{j+1} e_1^{v+1} e_2^t \otimes k_1^i k_2^j e_1^{r-v} e_2 + \]
	\[ + [h=t=1] k_1^{r-v+i+1} k_2^{j+2} e_1^v \otimes k_1^i k_2^j e_1^{r-v} e_3 e_2 ) \in B_q^+ \otimes B_q^+. \]
	
	Note that
	\[ \epsilon_{\bar{U}_q}(\sum_{ \substack{0 \leq i, j, r \leq d-1, \\  0 \leq h, t \leq 1}} \alpha_{i j r h t}
k_1^{i}k_2^{j}e_1^{r}e_3^{h}e_2^{t}) = \sum_{ 0 \leq i, j \leq d-1} \alpha_{i, j, 0, 0, 0}. \]
	
	We verify that $S_{\bar{U}_q}(B_q^+) \subset B_q^+$. It is sufficient to notice that
	\[S_{\bar{U}_q}(k_1^i k_2^j e_1^r e_3^h e_2^t) = (-1)^{ht} S_{\bar{U}_q}(e_2)^t S_{\bar{U}_q}(e_3)^h S_{\bar{U}_q}(e_1)^r
S_{\bar{U}_q}(k_2)^j S_{\bar{U}_q}(k_1)^i =\]
	\[ = [h=0] (-1)^{t+r} q^{r(r-1-j+2i)-it} k_1^{-r-i} k_2^{-t-j} e_1^r e_2^t + \]
	\[+ [h=0,t=1] (-1)^{r} q^{r(r-2-j+2i)-i+1} [r] k_1^{-r-i} k_2^{-j-1} e_1^{r-1} e_3 +\]
	\[+ [h=1,t=0] (-1)^{r} q^{r(r-j+2i+1)-j+i} (1-q^{-2}) k_1^{-r-i-1} k_2^{-j-1} e_1^{r+1} e_2 +\]
	\[ + [h=t=1] (-1)^{r} q^{r(r-j+2i-1)-j-2} k_1^{-1-r-i} k_2^{-2-j} e_1^r e_3 e_2 + \]
	\[ + [h=1,t=0] (-1)^{r+1} q^{r(r-j+2i+1)-j+i} k_1^{-r-i-1} k_2^{-j-1} e_1^{r} e_3 \in B_q^+. \]
	
	Notice that
	\[S^{-1}_{\bar{U}_q}(k_1^i k_2^j e_1^r e_3^h e_2^t) = (-1)^{ht} S^{-1}_{\bar{U}_q}(e_2)^t S^{-1}_{\bar{U}_q}(e_3)^h
S^{-1}_{\bar{U}_q}(e_1)^r S^{-1}_{\bar{U}_q}(k_2)^j S^{-1}_{\bar{U}_q}(k_1)^i =\]
	\[ = [h=0] (-1)^{t+r} q^{r(r+1+2i-j)-ti} k_1^{-i-r} k_2^{-j-t} e_1^{r} e_2^t + \]
	\[ + [h=0,t=1] (-1)^{r} q^{r(r+2i-j)-i+1} [r] k_1^{-i-r} k_2^{-j-1} e_1^{r-1} e_3 + \]
	\[ + [h=1,t=0] (-1)^{r} q^{r(r+2i+3-j)+i-j} (q^2-1) k_1^{-1-r-i} k_2^{-1-j} e_1^{r+1} e_2 + \]
	\[ + [h=1,t=1] (-1)^{r} q^{r(r+2i+1-j)-j} k_1^{-1-r-i} k_2^{-2-j} e_1^{r} e_3 e_2 + \]
	\[ + [h=1,t=0] (-1)^{r+1} q^{r(r+3+2i-j)+i-j+2} k_1^{-i-r-1} k_2^{-1-j} e_1^{r} e_3 \in B_q^+. \]
	
\end{proof}

\Xmultip*
\begin{proof}
	Let $\beta$ and $\gamma$ be two linear $\mathbb{Z}_2$-graded maps on $B_q^+$. Then the product $\beta \gamma$ in $X$ is
given by
	\[ \beta \gamma (k_1^i k_2^j e_1^r e_3^h e_2^t) =  \]		
	\[ = \sum_{v=0}^{r} q^{v(v-r)} {r \brack v} ( (-1)^{|\gamma|(h+t)} \beta(k_1^{r-v+i} k_2^j e_1^v e_3^h e_2^t) \gamma(k_1^i
k_2^j e_1^{r-v}) + \]
	\[ + [t=1] (-1)^{|\gamma|h} q^{h+v} \beta(k_1^{r-v+i} k_2^{j+1} e_1^v e_3^h) \gamma(k_1^i k_2^j e_1^{r-v} e_2) + \]
	\[ + [h=1] (-1)^{t(1+|\gamma|)} q^{-v} \beta(k_1^{r-v+i+1} k_2^{j+1} e_1^v e_2^t) \gamma(k_1^i k_2^j e_1^{r-v} e_3) + \]
	\[ + [h=1] (-1)^{t(1+|\gamma|)} q^{v} (q-q^{-1}) \beta(k_1^{r-v+i} k_2^{j+1} e_1^{v+1} e_2^t) \gamma(k_1^i k_2^j e_1^{r-v}
e_2) + \]
	\[ + [h=t=1] \beta(k_1^{r-v+i+1} k_2^{j+2} e_1^v) \gamma (k_1^i k_2^j e_1^{r-v} e_3 e_2) ). \]
	
	It is easy to prove using mathematical induction on $n \in \mathbb{N}$ that
	\[ \alpha_{k_1}^n (k_1^i k_2^j e_1^r e_3^h e_2^t) = [r=h=t=0] q^{n(-2i+j)}, \]
	\[ \alpha_{k_2}^n (k_1^i k_2^j e_1^r e_3^h e_2^t) = [r=h=t=0] q^{ni}, \]
	\[ \alpha_{e_1}^n (k_1^i k_2^j e_1^r e_3^h e_2^t) = [r=n,h=t=0] (-1)^{n} (q-q^{-1})^{-n} q^{n(2i-j)+\frac{n(n-1)}{2}} [n]!.
\]

	We evaluate $\alpha_{k_1}^{i_1} \alpha_{k_2}^{j_1}$. We have for all $i_1,j_1,r_1 \in \mathbb{N}$
	\[ \alpha_{k_1}^{i_1} \alpha_{k_2}^{j_1} (k_1^i k_2^j e_1^r e_3^h e_2^t) = [r=h=t=0] q^{i_1(-2i+j) + j_1i}. \]

	We evaluate $\alpha_{e_1}^{r_1} \alpha_{k_1}^{i_1} \alpha_{k_2}^{j_1}$.
	\[ \alpha_{e_1}^{r_1} \alpha_{k_1}^{i_1} \alpha_{k_2}^{j_1}(k_1^{i} k_2^{j} e_1^{r} e_3^{h} e_2^{t}) = [r=r_1,h=t=0]
\alpha_{e_1} ^{r_1}( k_1^i k_2^j e_1^{r_1} ) \alpha_{k_1}^{i_1} \alpha_{k_2}^{j_1}(k_1^i k_2^j) = \]
	\[ = [r=r_1,h=t=0] (-1)^{r_1} (q-q^{-1})^{-r_1} q^{r_1(2i-j) + \frac{r_1(r_1-1)}{2} + i_1(-2i+j)+ j_1i} [r_1]! = \]
	\[ = [r=r_1,h=t=0] (-1)^{r_1} (q-q^{-1})^{-r_1} q^{(r_1-i_1)(2i-j) + \frac{r_1(r_1-1)}{2} + j_1i} [r_1]! \Rightarrow \]
	\[ \alpha_{e_1}^{r_1} \alpha_{k_1}^{i_1} \alpha_{k_2}^{j_1} = \sum_{0 \le v,p \le d-1} (-1)^{r_1} (q-q^{-1})^{-r_1}
q^{(r_1-i_1)(2v-p) + \frac{r_1(r_1-1)}{2} + j_1v} [r_1]! (k_1^v k_2^p e_1^{r_1})^{*}, \]

	We evaluate $\alpha_{e_1}^{r_1} \alpha_{e_3} \alpha_{k_1}^{i_1} \alpha_{k_2}^{j_1}$.
	\[ \alpha_{e_1}^{r_1} \alpha_{e_3} (k_1^{i} k_2^{j} e_1^{r} e_3^{h} e_2^{t}) = \]
	\[ = [r=r_1+1,h=0,t=1] (-1)^{r_1} (q-q^{-1})^{-r_1-1} q^{r_1(2i-j+1) + \frac{r_1(r_1-1)}{2} + i - j} [r_1+1]! + \]
	\[ + [r=r_1,h=1,t=0] (-1)^{r_1} (q-q^{-1})^{-r_1-1} q^{-r_1 + r_1(2i-j+1) + \frac{r_1(r_1-1)}{2} + i - j} [r_1]! + \]
	\[ + [r=r_1,h=1,t=0] [r_1] (-1)^{r_1} (q-q^{-1})^{-r_1} q^{r_1(2i-j+1) + \frac{r_1(r_1-1)}{2} + i - j} [r_1]! = \]
	\[ = [r=r_1+1,h=0,t=1] (-1)^{r_1} (q-q^{-1})^{-r_1-1} q^{r_1(2i-j+1) + \frac{r_1(r_1-1)}{2} + i - j} [r_1+1]! + \]
	\[ + [r=r_1,h=1,t=0] (-1)^{r_1} (q-q^{-1})^{-r_1-1} q^{r_1 + r_1(2i-j+1) + \frac{r_1(r_1-1)}{2} + i - j} [r_1]!, \]
	\[ \alpha_{e_1}^{r_1} \alpha_{e_3} \alpha_{k_1}^{i_1} \alpha_{k_2}^{j_1} (k_1^{i} k_2^{j} e_1^{r} e_3^{h} e_2^{t}) = \]
	\[ = [r=r_1+1,h=0,t=1] (-1)^{r_1} (q-q^{-1})^{-r_1-1} q^{r_1(2i-j+1) + \frac{r_1(r_1-1)}{2} + i - j + i_1(-2i+j) + j_1i}
[r_1+1]! + \]
	\[ + [r=r_1,h=1,t=0] (-1)^{r_1} (q-q^{-1})^{-r_1-1} q^{r_1 + r_1(2i-j+1) + \frac{r_1(r_1-1)}{2} + i - j + i_1(-2i+j) + j_1i}
[r_1]! \Rightarrow \]
	\[ \alpha_{e_1}^{r_1} \alpha_{e_3} \alpha_{k_1}^{i_1} \alpha_{k_2}^{j_1} = \]
	\[ = \sum_{0 \le v,p \le d-1} ( (-1)^{r_1} (q-q^{-1})^{-r_1-1} q^{r_1(2v-p+1) + \frac{r_1(r_1-1)}{2} + v - p + i_1(-2v+p) +
j_1v} [r_1+1]! (k_1^v k_2^p e_1^{r_1+1} e_2 )^{*} + \]
	\[ + (-1)^{r_1} (q-q^{-1})^{-r_1-1} q^{r_1 + r_1(2v-p+1) + \frac{r_1(r_1-1)}{2} + v - p + i_1(-2v+p) + j_1v} [r_1]! (k_1^v
k_2^p e_1^{r_1} e_3 )^{*} ) = \]
	\[ = (-1)^{r_1} (q-q^{-1})^{-r_1-1} q^{r_1(2v-p+1) + \frac{r_1(r_1-1)}{2} + v - p + i_1(-2v+p) + j_1v} [r_1]! \times \]
	\[ \times \sum_{0 \le v,p \le d-1} ( [r_1+1] (k_1^v k_2^p e_1^{r_1+1} e_2 )^{*} + q^{r_1}  (k_1^v k_2^p e_1^{r_1} e_3 )^{*}
). \]

	We evaluate $\alpha_{e_1}^{r_1} \alpha_{e_2} \alpha_{k_1}^{i_1} \alpha_{k_2}^{j_1}$.
	\[ \alpha_{e_1}^{r_1} \alpha_{e_2} (k_1^{i} k_2^{j} e_1^{r} e_3^{h} e_2^{t}) = [r=r_1,h=0,t=1] (-1)^{r_1}
(q-q^{-1})^{-r_1-1} q^{r_1(2i-j)+\frac{r_1(r_1-1)}{2}-i} [r_1]! + \]
	\[ + [r=r_1-1,h=1,t=0] (q-q^{-1}) (-1)^{r_1} (q-q^{-1})^{-r_1-1} q^{-1+r_1(2i-j)+\frac{r_1(r_1-1)}{2}-i} [r_1]!, \]
	\[ \alpha_{e_1}^{r_1} \alpha_{e_2} \alpha_{k_1}^{i_1} \alpha_{k_2}^{j_1} (k_1^{i} k_2^{j} e_1^{r} e_3^{h} e_2^{t}) = \]
	\[ = [r=r_1,h=0,t=1] (-1)^{r_1} (q-q^{-1})^{-r_1-1} q^{r_1(2i-j)+\frac{r_1(r_1-1)}{2}-i + i_1(-2i+j) + j_1i} [r_1]! + \]
	\[ + [r=r_1-1,h=1,t=0] (q-q^{-1}) (-1)^{r_1} (q-q^{-1})^{-r_1-1} q^{-1+r_1(2i-j)+\frac{r_1(r_1-1)}{2}-i+ i_1(-2i+j) + j_1i}
[r_1]! \Rightarrow \]
	\[ \alpha_{e_1}^{r_1} \alpha_{e_2} \alpha_{k_1}^{i_1} \alpha_{k_2}^{j_1} = \]
	\[ = \sum_{0 \le v,p \le d-1} ( (-1)^{r_1} (q-q^{-1})^{-r_1-1} q^{r_1(2v-p)+\frac{r_1(r_1-1)}{2} + i_1(-2v+p) + j_1v -v}
[r_1]! (k_1^{v} k_2^{p} e_1^{r_1} e_2 )^{*} + \]
	\[ + (-1)^{r_1} (q-q^{-1})^{-r_1} q^{r_1(2v-p)+\frac{r_1(r_1-1)}{2}+ i_1(-2v+p) + j_1v -v -1} [r_1]! (k_1^{v} k_2^{p}
e_1^{r_1-1} e_3 )^{*} ) = \]
	\[ = (-1)^{r_1} (q-q^{-1})^{-r_1-1} q^{r_1(2v-p)+\frac{r_1(r_1-1)}{2} + i_1(-2v+p) + j_1v -v} [r_1]! \sum_{0 \le v,p \le
d-1} (  (k_1^{v} k_2^{p} e_1^{r_1} e_2 )^{*} + (1-q^{-2}) (k_1^{v} k_2^{p} e_1^{r_1-1} e_3 )^{*}). \]

	We evaluate $\alpha_{e_1}^{r_1} \alpha_{e_3} \alpha_{e_2} \alpha_{k_1}^{i_1} \alpha_{k_2}^{j_1}$.
	\[ \alpha_{e_3} \alpha_{e_2} = [r=0,h=t=1] - q (q-q^{-1})^{-2} q^{i-j-1-i} + [r=0,h=t=1] (q-q^{-1}) (q-q^{-1})^{-2}
q^{i-j-1-i} = \]
	\[ = [r=0,h=t=1] ( - q^{-j} (q-q^{-1})^{-2} + (q-q^{-1})^{-1} q^{-j-1} ) = [r=0,h=t=1] - q^{-j-2} (q-q^{-1})^{-2}, \]
	\[ \alpha_{e_1}^{r_1} \alpha_{e_3} \alpha_{e_2} (k_1^{i} k_2^{j} e_1^{r} e_3^{h} e_2^{t}) = [r=r_1,h=t=1] (-1)^{r_1+1}
(q-q^{-1})^{-r_1-2} q^{r_1(2i-j) + \frac{r_1(r_1-1)}{2} - j - 2} [r_1]!, \]
	\[ \alpha_{e_1}^{r_1} \alpha_{e_3} \alpha_{e_2} \alpha_{k_1}^{i_1} \alpha_{k_2}^{j_1} (k_1^{i} k_2^{j} e_1^{r} e_3^{h}
e_2^{t}) = \]
	\[ = [r=r_1,h=t=1] (-1)^{r_1+1} (q-q^{-1})^{-r_1-2} q^{r_1(2i-j) + \frac{r_1(r_1-1)}{2} - j - 2 + i_1(-2i+j)+j_1i} [r_1]!
\Rightarrow \]
	\[ \alpha_{e_1}^{r_1} \alpha_{e_3} \alpha_{e_2} \alpha_{k_1}^{i_1} \alpha_{k_2}^{j_1} = \]
	\[ = \sum_{0 \le v,p \le d-1} (-1)^{r_1+1} (q-q^{-1})^{-r_1-2} q^{r_1(2v-p) + \frac{r_1(r_1-1)}{2} + i_1(-2v+p)+j_1v - p -
2} [r_1]! (k_1^{v} k_2^{p} e_1^{r_1} e_3 e_2 )^{*}. \]
	
	Thus,
	\[ \alpha_{e_1}^{r_1} \alpha_{e_3}^{h_1} \alpha_{e_2}^{t_1} \alpha_{k_1}^{i_1} \alpha_{k_2}^{j_1}  = \]
	\[ = \sum_{0 \le v,p \le d-1} ( [h_1=t_1=0] (-1)^{r_1} (q-q^{-1})^{-r_1} q^{(r_1-i_1)(2v-p) + \frac{r_1(r_1-1)}{2} + j_1v}
[r_1]! (k_1^v k_2^p e_1^{r_1})^{*} + \]
	\[ + [h_1=1,t_1=0] (-1)^{r_1} (q-q^{-1})^{-r_1-1} q^{r_1(2v-p+1) + \frac{r_1(r_1-1)}{2} + v - p + i_1(-2v+p) + j_1v} [r_1]!
\times \]
	\[ \times ( [r_1+1] (k_1^v k_2^p e_1^{r_1+1} e_2 )^{*} + q^{r_1}  (k_1^v k_2^p e_1^{r_1} e_3 )^{*} ) + \]
	\[ + [h_1=0,t_1=1] (-1)^{r_1} (q-q^{-1})^{-r_1-1} q^{r_1(2v-p)+\frac{r_1(r_1-1)}{2} + i_1(-2v+p) + j_1v -v} [r_1]! \times
\]
	\[ \times (  (k_1^{v} k_2^{p} e_1^{r_1} e_2 )^{*} + (1-q^{-2}) (k_1^{v} k_2^{p} e_1^{r_1-1} e_3 )^{*}) + \]
	\[ + [h_1=t_1=1] (-1)^{r_1+1} (q-q^{-1})^{-r_1-2} q^{r_1(2v-p) + \frac{r_1(r_1-1)}{2} + i_1(-2v+p)+j_1v - p - 2} [r_1]!
(k_1^{v} k_2^{p} e_1^{r_1} e_3 e_2 )^{*} ). \]
	
	Note that
	\[ (k_1^{i_2} k_2^{j_2} e_1^{r})^{*} = \sum_{i_1,j_1} \mu^{i_2,j_2,r,0,0}_{i_1,j_1,r,0,0} \alpha_{e_1}^{r}
\alpha_{k_1}^{i_1} \alpha_{k_2}^{j_1}, \]
	\[ \mu^{i_2,j_2,r,0,0}_{i_1,j_1,r,0,0} = (-1)^{r} \frac{(q-q^{-1})^{r}}{d^2[r]!} q^{-r(2i_2-j_2) - \frac{r(r-1)}{2} +
i_1(2i_2-j_2) - j_1i_2}. \]
	\[ (k_1^{i_2} k_2^{j_2} e_1^{r} e_3 e_2)^{*} = \sum_{i_1,j_1} \mu^{i_2,j_2,r,1,1}_{i_1,j_1,r,1,1} \alpha_{e_1}^{r}
\alpha_{e_3} \alpha_{e_2} \alpha_{k_1}^{i_1} \alpha_{k_2}^{j_1}, \]
	\[ \mu^{i_2,j_2,r,1,1}_{i_1,j_1,r,1,1} = (-1)^{r+1} \frac{(q-q^{-1})^{r+2}}{d^2[r]!} q^{-r(2i_2-j_2) - \frac{r(r-1)}{2} +
i_1(2i_2-j_2)-j_1i_2 + j_2 + 2}. \]		
	\[ (k_1^{i_2} k_2^{j_2} e_1^{r} e_3)^{*} = \sum_{i_1,j_1} \mu^{i_2,j_2,r,1,0}_{i_1,j_1,r,1,0} \alpha_{e_1}^{r} \alpha_{e_3}
\alpha_{k_1}^{i_1} \alpha_{k_2}^{j_1} + \mu^{i_2,j_2,r,1,0}_{i_1,j_1,r+1,0,1} \alpha_{e_1}^{r+1} \alpha_{e_2} \alpha_{k_1}^{i_1}
\alpha_{k_2}^{j_1} =  \]
	\[ = \sum_{i_1,j_1} \mu^{i_2,j_2,r,1,0}_{i_1,j_1,r,1,0} (-1)^{r} (q-q^{-1})^{-r-1} q^{r(2i_2-j_2+1) + \frac{r(r-1)}{2} + i_2
- j_2 + i_1(-2i_2+j_2) + j_1i_2} [r+1]! (k_1^{i_2} k_2^{j_2} e_1^{r+1} e_2 )^{*} + \]
	\[ + \mu^{i_2,j_2,r,1,0}_{i_1,j_1,r,1,0} (-1)^{r} (q-q^{-1})^{-r-1} q^{r + r(2i_2-j_2+1) + \frac{r(r-1)}{2} + i_2 - j_2 +
i_1(-2i_2+j_2) + j_1i_2} [r]! (k_1^{i_2} k_2^{j_2} e_1^{r} e_3 )^{*} + \]
	\[ + (q-q^{-1})^{-1} \mu^{i_2,j_2,r,1,0}_{i_1,j_1,r+1,0,1} (-1)^{r+1} (q-q^{-1})^{-r-1} q^{r(2i_2-j_2+1)+\frac{r(r-1)}{2} +
i_1(-2i_2+j_2) + j_1i_2 +i_2-j_2} [r+1]! (k_1^{i_2} k_2^{j_2} e_1^{r+1} e_2 )^{*} + \]
	\[ + \mu^{i_2,j_2,r,1,0}_{i_1,j_1,r+1,0,1} (-1)^{r+1} (q-q^{-1})^{-r-1} q^{r(2i_2-j_2+1)+\frac{r(r-1)}{2}+ i_1(-2i_2+j_2) +
j_1i_2 +i_2-j_2 -1} [r+1]! (k_1^{i_2} k_2^{j_2} e_1^{r} e_3 )^{*}  = \]
	\[ = \sum_{i_1,j_1} \mu^{i_2,j_2,r,1,0}_{i_1,j_1,r,1,0} (-1)^{r} (q-q^{-1})^{-r-1} q^{r(2i_2-j_2+1) + \frac{r(r-1)}{2} +
i_1(-2i_2+j_2) + j_1i_2 + i_2 - j_2 + r} [r]! (k_1^{i_2} k_2^{j_2} e_1^{r} e_3 )^{*} + \]
	\[ + \mu^{i_2,j_2,r,1,0}_{i_1,j_1,r+1,0,1} (-1)^{r+1} (q-q^{-1})^{-r-1} q^{r(2i_2-j_2+1)+\frac{r(r-1)}{2}+ i_1(-2i_2+j_2) +
j_1i_2 +i_2-j_2 -1} [r+1]! (k_1^{i_2} k_2^{j_2} e_1^{r} e_3 )^{*} =  \]
	\[ = \sum_{i_1,j_1} \mu^{i_2,j_2,r,1,0}_{i_1,j_1,r,1,0} (-1)^{r} (q-q^{-1})^{-r-1} q^{r(2i_2-j_2+1) + \frac{r(r-1)}{2} +
i_1(-2i_2+j_2) + j_1i_2 + i_2 - j_2} [r]! ( q^{r} - q^{r} + q^{-r-2} ) (k_1^{i_2} k_2^{j_2} e_1^{r} e_3 )^{*} = \]
	\[ = \sum_{i_1,j_1} \mu^{i_2,j_2,r,1,0}_{i_1,j_1,r,1,0} (-1)^{r} (q-q^{-1})^{-r-1} q^{r(2i_2-j_2) + \frac{r(r-1)}{2} +
i_1(-2i_2+j_2) + j_1i_2 + i_2 - j_2-2} [r]! (k_1^{i_2} k_2^{j_2} e_1^{r} e_3 )^{*}. \]
	
	\begin{equation} \label{eq:linkeq1}
	\mu^{i_2,j_2,r,1,0}_{i_1,j_1,r+1,0,1} = (q-q^{-1}) \mu^{i_2,j_2,r,1,0}_{i_1,j_1,r,1,0}.
	\end{equation}
	
	\[ \mu^{i_2,j_2,r,1,0}_{i_1,j_1,r,1,0} = (-1)^{r} \frac{(q-q^{-1})^{r+1}}{d^2[r]!} q^{-r(2i_2-j_2) - \frac{r(r-1)}{2} +
i_1(2i_2-j_2) - j_1i_2 - i_2 + j_2+2}, \]
	\[ \mu^{i_2,j_2,r,1,0}_{i_1,j_1,r+1,0,1} = (-1)^{r} \frac{(q-q^{-1})^{r+2}}{d^2[r]!} q^{-r(2i_2-j_2) - \frac{r(r-1)}{2} +
i_1(2i_2-j_2) - j_1i_2 - i_2 + j_2+2}. \]
	
	\[ (k_1^{i_2} k_2^{j_2} e_1^{r} e_2)^{*} = \sum_{i_1,j_1} \mu^{i_2,j_2,r,0,1}_{i_1,j_1,r-1,1,0} \alpha_{e_1}^{r-1}
\alpha_{e_3} \alpha_{k_1}^{i_1} \alpha_{k_2}^{j_1} + \mu^{i_2,j_2,r,0,1}_{i_1,j_1,r,0,1} \alpha_{e_1}^{r} \alpha_{e_2}
\alpha_{k_1}^{i_1} \alpha_{k_2}^{j_1} = \]
	\[ = \sum_{i_1,j_1} \mu^{i_2,j_2,r,0,1}_{i_1,j_1,r-1,1,0} (-1)^{r-1} (q-q^{-1})^{-r} q^{r(2i_2-j_2+1) + \frac{r(r-1)}{2} +
i_1(-2i_2+j_2) + j_1i_2 - i_2 -r} [r]! (k_1^{i_2} k_2^{j_2} e_1^{r} e_2 )^{*} + \]
	\[  + \mu^{i_2,j_2,r,0,1}_{i_1,j_1,r-1,1,0} (-1)^{r-1} (q-q^{-1})^{-r} q^{ r(2i_2-j_2+1) + \frac{r(r-1)}{2} + i_1(-2i_2+j_2)
+ j_1i_2 - i_2 - 1} [r-1]! (k_1^{i_2} k_2^{j_2} e_1^{r-1} e_3 )^{*} + \]
	\[ + \mu^{i_2,j_2,r,0,1}_{i_1,j_1,r,0,1} (-1)^{r} (q-q^{-1})^{-r-1} q^{r_1(2i_2-j_2)+\frac{r(r-1)}{2} + i_1(-2i_2+j_2) +
j_1i_2 -i_2} [r]! (k_1^{i_2} k_2^{j_2} e_1^{r} e_2 )^{*} + \]
	\[ + \mu^{i_2,j_2,r,0,1}_{i_1,j_1,r,0,1} (-1)^{r} (q-q^{-1})^{-r} q^{r(2i_2-j_2)+\frac{r(r-1)}{2}+ i_1(-2i_2+j_2) + j_1i_2
-i_2 -1} [r]! (k_1^{i_2} k_2^{j_2} e_1^{r-1} e_3 )^{*} = \]
	\[ = \sum_{i_1,j_1} \mu^{i_2,j_2,r,0,1}_{i_1,j_1,r-1,1,0} (-1)^{r-1} (q-q^{-1})^{-r} q^{r(2i_2-j_2) + \frac{r(r-1)}{2} +
i_1(-2i_2+j_2) + j_1i_2 - i_2} [r]! (k_1^{i_2} k_2^{j_2} e_1^{r} e_2 )^{*} + \]
	\[ + \mu^{i_2,j_2,r,0,1}_{i_1,j_1,r,0,1} (-1)^{r} (q-q^{-1})^{-r-1} q^{r_1(2i_2-j_2)+\frac{r(r-1)}{2} + i_1(-2i_2+j_2) +
j_1i_2 -i_2} [r]! (k_1^{i_2} k_2^{j_2} e_1^{r} e_2 )^{*} = \]
	\[ = \sum_{i_1,j_1} \mu^{i_2,j_2,r,0,1}_{i_1,j_1,r,0,1} (-1)^{r} (q-q^{-1})^{-r-1} q^{r(2i_2-j_2) + \frac{r(r-1)}{2} +
i_1(-2i_2+j_2) + j_1i_2 - i_2} [r]! ( - [r] q^{-r} (q-q^{-1}) + 1 ) (k_1^{i_2} k_2^{j_2} e_1^{r} e_2 )^{*} = \]
	\[ = \sum_{i_1,j_1} \mu^{i_2,j_2,r,0,1}_{i_1,j_1,r,0,1} (-1)^{r} [r]! (q-q^{-1})^{-r-1} q^{r(2i_2-j_2-2) + \frac{r(r-1)}{2}
+ i_1(-2i_2+j_2) + j_1i_2 - i_2} (k_1^{i_2} k_2^{j_2} e_1^{r} e_2 )^{*}. \]
	
	\begin{equation} \label{eq:linkeq2}
	\mu^{i_2,j_2,r,0,1}_{i_1,j_1,r-1,1,0} = [r] q^{-r} \mu^{i_2,j_2,r,0,1}_{i_1,j_1,r,0,1}.
	\end{equation}
	
	\[ \mu^{i_2,j_2,r,0,1}_{i_1,j_1,r,0,1} = (-1)^{r} \frac{(q-q^{-1})^{r+1}}{d^2[r]!} q^{-r(2i_2-j_2-2) - \frac{r(r-1)}{2} +
i_1(2i_2-j_2) - j_1i_2 + i_2}, \]
	\[ \mu^{i_2,j_2,r,0,1}_{i_1,j_1,r-1,1,0} = (-1)^{r} \frac{(q-q^{-1})^{r+1}}{d^2[r-1]!} q^{-r(2i_2-j_2-1) - \frac{r(r-1)}{2}
+ i_1(2i_2-j_2) - j_1i_2 + i_2}. \]
	
	Therefore,
	\begin{equation} \label{eq:linkeq3}
	\mu^{i_2,j_2,r,1,0}_{i_1,j_1,r+1,0,1} = (q-q^{-1}) q^{-2r-2i_2+j_2+2} \mu^{i_2,j_2,r,0,1}_{i_1,j_1,r,0,1},
	\end{equation}
	\begin{equation} \label{eq:linkeq4}
	\mu^{i_2,j_2,r,1,0}_{i_1,j_1,r,1,0} = \frac{1}{[r]} q^{-r-2i_2+j_2+2} \mu^{i_2,j_2,r,0,1}_{i_1,j_1,r-1,1,0},
	\end{equation}
	\begin{equation} \label{eq:linkeq5}
	\mu^{0,0,0,0,1}_{0,0,0,0,1} = \frac{(q-q^{-1})}{d^2}.
	\end{equation}
\end{proof}

\subBF*
\begin{proof}
	From Lemma \ref{lm:mulX} we have
	\[ \alpha_{e_1}^d (k_1^i k_2^j e_1^r e_3^h e_2^t) = [r=d,h=t=0] (-1)^{d} (q-q^{-q})^{-d} q^{d(2i-j)+\frac{d(d-1)}{2}} [d]! =
0 \Rightarrow \alpha_{e_1}^d = 0, \]
	\[ \alpha_{e_3} \alpha_{e_1} (k_1^i k_2^j e_1^r e_3^h e_2^t) = - [r=2,h=0,t=1] (q-q^{-1})^{-2} [2] q^{-1+i+1-j+2i-j} - \]
	\[ - [r=1,h=1,t=0] (q-q^{-1})^{-2} q^{i+1-j+2i-j} \Rightarrow \alpha_{e_3} \alpha_{e_1} = q^{-1} \alpha_{e_1} \alpha_{e_3},
\]
	\[ \alpha_{e_2} \alpha_{e_1} (k_1^i k_2^j e_1^r e_3^h e_2^t) = [r=1,h=0,t=1] (-1) (q-q^{-1})^{-2} q^{-i-1+2i-j} \Rightarrow
\alpha_{e_2} \alpha_{e_1} = q \alpha_{e_1} \alpha_{e_2} + \alpha_{e_3}, \]
	\[ \alpha_{k_1} \alpha_{e_1} (k_1^i k_2^j e_1^r e_3^h e_2^t) = [r=1,h=t=0] (-1) (q-q^{-1})^{-1} q^{-2i-2+j+2i-j} \Rightarrow
\alpha_{k_1} \alpha_{e_1} = q^{-2} \alpha_{e_1} \alpha_{k_1}, \]
	\[ \alpha_{k_2} \alpha_{e_1} (k_1^i k_2^j e_1^r e_3^h e_2^t) = [r=1,h=t=0] (-1) (q-q^{-1})^{-1} q^{i+1+2i-j} \Rightarrow
\alpha_{k_2} \alpha_{e_1} = q \alpha_{e_1} \alpha_{k_2}, \]
	\[ \alpha_{e_3}^2 (k_1^i k_2^j e_1^r e_3^h e_2^t) = [r=h=t=1] (q-q^{-1})^{-2} q^{2(i-j)} ( q^{-1} + q - q^{-1} - q ) = 0
\Rightarrow \alpha_{e_3}^2 = 0, \]
	\[ \alpha_{e_2} \alpha_{e_3} (k_1^i k_2^j e_1^r e_3^h e_2^t) = [r=0,h=t=1] (q-q^{-1})^{-2} q^{-i-1+i-j} \Rightarrow
\alpha_{e_2} \alpha_{e_3} = - q \alpha_{e_3} \alpha_{e_2}, \]
	\[ \alpha_{k_1} \alpha_{e_3} (k_1^i k_2^j e_1^r e_3^h e_2^t) = [r=1,h=0,t=1] (q-q^{-1})^{-1} q^{-2i-2+j+1+i-j} + \]
	\[ + [r=0,h=1,t=0] (q-q^{-1})^{-1} q^{-2i-2+j+1+i-j} \Rightarrow \alpha_{k_1} \alpha_{e_3} = q^{-1} \alpha_{e_3}
\alpha_{k_1}, \]
	\[ \alpha_{k_2} \alpha_{e_3} (k_1^i k_2^j e_1^r e_3^h e_2^t) = [r=1,h=0,t=1] (q-q^{-1})^{-1} q^{i+1+i-j} + \]
	\[ + [r=0,h=1,t=0] (q-q^{-1})^{-1} q^{i+1+i-j} \Rightarrow \alpha_{k_2} \alpha_{e_3} = q \alpha_{e_3} \alpha_{k_2}, \]
	\[ \alpha_{e_2}^2 = 0, \]
	\[ \alpha_{k_1} \alpha_{e_2} (k_1^i k_2^j e_1^r e_3^h e_2^t) = [r=h=0,t=1] (q-q^{-1})^{-1} q^{-2i+j+1-i} \Rightarrow
\alpha_{k_1} \alpha_{e_2} = q \alpha_{e_2} \alpha_{k_1}, \]
	\[ \alpha_{k_2} \alpha_{e_2} (k_1^i k_2^j e_1^r e_3^h e_2^t) = [r=h=0,t=1] (q-q^{-1})^{-1} q^{-i-i} \Rightarrow \alpha_{k_2}
\alpha_{e_2} = \alpha_{e_2} \alpha_{k_2}, \]
	\[ \alpha_{k_1}^d (k_1^i k_2^j e_1^r e_3^h e_2^t) = [r=h=t=0] q^{-2di+dj} =  [r=h=t=0] 1_{B_q^+} = \epsilon_{B_q^+} (k_1^i
k_2^j e_1^r e_3^h e_2^t) \Rightarrow \alpha_{k_1}^d=\epsilon_{B_q^+}, \]
	\[ \alpha_{k_2}^d (k_1^i k_2^j e_1^r e_3^h e_2^t) = [r=h=t=0] q^{di} =  [r=h=t=0] 1_{B_q^+} = \epsilon_{B_q^+} (k_1^i k_2^j
e_1^r e_3^h e_2^t) \Rightarrow \alpha_{k_2}^d=\epsilon_{B_q^+}, \]
	\[ \alpha_{k_2} \alpha_{k_1} (k_1^i k_2^j e_1^r e_3^h e_2^t) = [r=h=t=0] q^{i+2i-j} \Rightarrow \alpha_{k_2} \alpha_{k_1} =
\alpha_{k_1} \alpha_{k_2}. \]
	
	Let $\beta$ be a linear $\mathbb{Z}_2$-graded map on $B_q^+$. Then the comultiplication of $\beta$ in $X$ is given by	
	\[ \lambda_{B_q^{+},B_q^{+}}(\Delta_{(H^{op})^{*}}(\beta)) (k_1^{i_1}k_2^{j_1}e_1^{r_1}e_3^{h_1}e_2^{t_1} \otimes
k_1^{i_2}k_2^{j_2}e_1^{r_2}e_3^{h_2}e_2^{t_2}) =   \]
	\[ = (-1)^{|e_3^{h_1}e_2^{t_1}||e_3^{h_2}e_2^{t_2}|} \beta(k_1^{i_2}k_2^{j_2}e_1^{r_2}e_3^{h_2}e_2^{t_2}
k_1^{i_1}k_2^{j_1}e_1^{r_1}e_3^{h_1}e_2^{t_1}) = \]
	\[ = (-1)^{h_2(h_1+t_1)+t_1t_2} q^{(t_2-h_2-2r_2)i_1+(h_2+r_2)j_1-h_2r_1+(r_1+h_1)t_2} \beta(
k_1^{i_1+i_2}k_2^{j_1+j_2}e_1^{r_1+r_2} e_3^{h_1+h_2} e_2^{t_1+t_2} ) - \]
	\[ - [h_1=h_2=0, t_2=1] (-1)^{t_1} q^{(1-2r_2)i_1+r_2j_1+1} [r_1] \beta (k_1^{i_1+i_2}k_2^{j_1+j_2}e_1^{r_1 +r_2-1} e_3
e_2^{t_1}). \]
	
	Similarly, we have
	\[ \lambda_{B_q^{+},B_q^{+}} \circ \Delta_{X} ( \alpha_{k_v} ) (k_1^{i_1}k_2^{j_1}e_1^{r_1}e_3^{h_1}e_2^{t_1} \otimes
k_1^{i_2}k_2^{j_2}e_1^{r_2}e_3^{h_2}e_2^{t_2}) =  \]
	\[ = [r_1=r_2=h_1=h_2=t_1=t_2=0] q^{-2(i_1+i_2)+(j_1+j_2)} \Rightarrow \Delta_{X} ( \alpha_{k_v} ) = \alpha_{k_v} \otimes
\alpha_{k_v}, \]
	\[ \lambda_{B_q^{+},B_q^{+}} \circ \Delta_{X} ( \alpha_{e_1} ) (k_1^{i_1}k_2^{j_1}e_1^{r_1}e_3^{h_1}e_2^{t_1} \otimes
k_1^{i_2}k_2^{j_2}e_1^{r_2}e_3^{h_2}e_2^{t_2}) = \]
	\[ = [r_1 \vee r_2 = 1, h_1=h_2=t_1=t_2=0] - (q-q^{-1})^{-1} q^{2(i_1+i_2)-j_1-j_2 -2r_2 i_1 + r_2 j_1} = \]
	\[ = - (q-q^{-1})^{-1} ( [r_1=1,r_2=h_1=h_2=t_1=t_2=0] q^{2(i_1+i_2)-j_1-j_2} + \]
	\[ + [r_1=0,r_2=1,h_1=h_2=t_1=t_2=0] q^{2i_2-j_2 } ) \Rightarrow \]
	\[ \Rightarrow \Delta_{X} ( \alpha_{e_1} ) = \alpha_{e_1} \otimes \alpha_{k_1}^{-1} + 1_X \otimes \alpha_{e_1}, \]
	\[\lambda_{B_q^{+},B_q^{+}} \circ \Delta_X(\alpha_{e_2})( k_1^{i_1}k_2^{j_1}e_1^{r_1}e_3^{h_1}e_2^{t_1} \otimes
k_1^{i_2}k_2^{j_2}e_1^{r_2}e_3^{h_2}e_2^{t_2} ) = \]
	\[ = (q-q^{-1})^{-1} ([r_1=r_2=0,t_1=1,t_2=h_1=h_2=0] q^{-i_1-i_2} + [r_1=r_2=t_1=0,t_2=1,h_1=h_2=0] q^{-i_2} \Rightarrow
\]
	\[ \Rightarrow \Delta_{X} ( \alpha_{e_2} ) = \alpha_{e_2} \otimes \alpha_{k_2}^{-1} + 1_X \otimes \alpha_{e_2} , \]
	\[ \lambda_{B_q^{+},B_q^{+}} \circ \Delta_X(\alpha_{e_3})( k_1^{i_1}k_2^{j_1}e_1^{r_1}e_3^{h_1}e_2^{t_1} \otimes
k_1^{i_2}k_2^{j_2}e_1^{r_2}e_3^{h_2}e_2^{t_2} ) =  \]
	\[ = (q-q^{-1})^{-1} ( [r_1 \vee r_2 = 1, h_1=h_2=0, t_1 \vee t_2 = 1] q^{(t_2-2r_2)i_1 + r_2 j_1 + r_1 t_2 + i_1 + i_2 -
j_1 - j_2} + \]
	\[ + [r_1=r_2=0,h_1 \vee h_2 = 1,t_1=t_2=0] q^{-h_2i_1 + h_2j_1 + i_1 + i_2 - j_1 - j_2} - \]
	\[ - [r_1=1,r_2=h_1=h_2=t_1=0,t_2=1] q^{i_1+1 + i_1 + i_2 - j_1 - j_2} ) = \]
	\[ = (q-q^{-1})^{-1} ( [r_1=0,r_2=1,h_1=h_2=t_1=0,t_2=1] q^{ i_2  - j_2} + \]
	\[+ [r_1=1,r_2=h_1=h_2=t_1=0,t_2=1] q^{2i_1 + i_2 - j_1 - j_2 + 1} + \]
	\[+ [r_1=0,r_2=1,h_1=h_2=0,t_1=1,t_2=0] q^{-i_1 + i_2 - j_2} + \]
	\[+ [r_1=1,r_2=h_1=h_2=0,t_1=1,t_2=0] q^{ i_1 + i_2 - j_1 - j_2} + \]
	\[ + [r_1=r_2=0,h_1=1,h_2=t_1=t_2=0] q^{ i_1 + i_2 - j_1 - j_2} + [r_1=r_2=h_1=0,h_2=1,t_1=t_2=0] q^{i_2 - j_2} - \]
	\[ - [r_1=1,r_2=h_1=h_2=t_1=0,t_2=1] q^{2i_1 + i_2 - j_1 - j_2 +1} ) \Rightarrow \]
	\[ \Rightarrow \Delta_{X} ( \alpha_{e_3} ) = 1_X \otimes \alpha_{e_3} + \alpha_{e_3} \otimes \alpha_{k_1}^{-1}
\alpha_{k_2}^{-1} + (q^{-1}-q) \alpha_{e_2} \otimes \alpha_{e_1} \alpha_{k_2}^{-1}. \]
	
	We have for the counit
	\[ \epsilon_{X}(\alpha_{k_1}) = \alpha_{k_1}(1_{B^+_q}) = 1_F, \; \epsilon_{X}(\alpha_{k_2}) = \alpha_{k_1}(1_{B^+_q}) =
1_F, \]
	\[ \epsilon_{X}(\alpha_{e_1}) = \alpha_{e_1}(1_{B_q^{+}}) = 0, \; \epsilon_X(\alpha_{e_2}) = \alpha_{e_2}(1_{B_q^{+}}) = 0,
\]
	\[ \epsilon_{X}(\alpha_{e_3}) = \alpha_{e_3}(1_{B_q^{+}}) = 0. \]
	
	We evaluate values of the antipode on particular elements:
	\[ S_{X}(\alpha_{k_1})(k_1^i k_2^j e_1^r e_3^h e_2^t) = [r=h=t=0] \alpha_{k_1}(k_1^{-i} k_2^{-j}) = [r=h=t=0] q^{2i-j} =
\alpha_{k_1}^{d-1} (k_1^i k_2^j e_1^r e_3^h e_2^t), \]
	\[ S_{X}(\alpha_{k_2})(k_1^i k_2^j e_1^r e_3^h e_2^t) = [r=h=t=0] \alpha_{k_2}(k_1^{-i} k_2^{-j}) = [r=h=t=0] q^{-i} =
\alpha_{k_2}^{d-1} (k_1^i k_2^j e_1^r e_3^h e_2^t), \]
	\[ S_{X}(\alpha_{e_1})(k_1^i k_2^j e_1^r e_3^h e_2^t) = - [r=1,h=t=0] q^{2+2i-j} \alpha_{e_1}(k_1^{-i-1} k_2^{-j} e_1) = \]
	\[ = [r=1,h=t=0] (q-q^{-1})^{-1} q^{2+2i-j-2i-2+j} = - \alpha_{e_1} \alpha_{k_1} (k_1^i k_2^j e_1^r e_3^h e_2^t), \]
	\[S_X(\alpha_{e_2})(k_1^i k_2^j e_1^r e_3^h e_2^t) = \alpha_{e_2} \circ S^{-1}_{B_q^+}(k_1^i k_2^j e_1^r e_3^h e_2^t) = \]
	\[ =  - [r=h=0,t=1] q^{-i} \alpha_{e_2}(k_1^{-i} k_2^{-j-1} e_2) = - [r=h=0,t=1] (q-q^{-1})^{-1} q^{-i+i} = - \alpha_{e_2}
\alpha_{k_2} (k_1^i k_2^j e_1^r e_3^h e_2^t), \]
	\[S_X(\alpha_{e_3})(k_1^i k_2^j e_1^r e_3^h e_2^t) = \alpha_{e_3} \circ S^{-1}_{B_q^+}(k_1^i k_2^j e_1^r e_3^h e_2^t) =  \]
	\[ = (q-q^{-1})^{-1} ( [r=1,h=0,t=1] q^{2} + [r=0,h=1,t=0] q (q-q^{-1}) - [r=1,h=0,t=1] q^{2} - \]
	\[ - [r=0,h=1,t=0] q^{2} ) = (- (q-q^{-1}) q^2 \alpha_{e_1} \alpha_{e_2} \alpha_{k_1} \alpha_{k_2} - q^{2} \alpha_{e_3}
\alpha_{k_1} \alpha_{k_2})(k_1^i k_2^j e_1^r e_3^h e_2^t) = \]
	\[ = ((q-q^3) \alpha_{e_1} \alpha_{e_2} \alpha_{k_1} \alpha_{k_2} - q^{2} \alpha_{e_3} \alpha_{k_1} \alpha_{k_2})(k_1^i
k_2^j e_1^r e_3^h e_2^t). \]
	
	We now construct a basis. Since the dimension of Hopf superalgebra $X$ is equal to $4d^{3}$, it is sufficient to prove that
elements of the set
	\[ \{ \alpha_{e_1}^{w} \alpha_{e_3}^{s} \alpha_{e_2}^{l} \alpha_{k_1}^{i} \alpha_{k_2}^{j} | \; 0 \le w,i,j \le d-1, \; s,l
\in \{0,1\} \} \]
	are linear independent. Thus, we consider a superaglebra morphism $\rho : X \to B_q^{-}$:
	\[ \rho( \alpha_{e_1} ) = f_1, \; \rho( \alpha_{e_3} ) = f_3, \; \rho( \alpha_{e_2} ) = f_2, \; \rho( \alpha_{k_1} ) = k_1,
\; \rho( \alpha_{k_2} ) = k_2.  \]
	In fact, we have already proved that $\rho$ is the superaglebra morphism. Therefore,
	\[ \rho( \sum_{0 \le i_1,j_1,r_1 \le d-1, 0 \le h_1, t_1 \le 1} \eta_{r_1 h_1 t_1 i_1 j_1} \alpha_{e_1}^{r_1}
\alpha_{e_3}^{h_1} \alpha_{e_2}^{t_1} \alpha_{k_1}^{i_1} \alpha_{k_2}^{j_1}  ) =  \]
	\[ = \sum_{r_1,h_1,t_1,i_1,j_1 \in I} \eta_{r_1 h_1 t_1 i_1 j_1} \rho( \alpha_{e_1}^{r_1} \alpha_{e_3}^{h_1}
\alpha_{e_2}^{t_1} \alpha_{k_1}^{i_1} \alpha_{k_2}^{j_1} ) =  \]
	\[ = \sum_{r_1,h_1,t_1,i_1,j_1 \in I} \eta_{r_1 h_1 t_1 i_1 j_1} \rho( \alpha_{e_1})^{r_1} \rho(\alpha_{e_3})^{h_1}
\rho(\alpha_{e_2})^{t_1} \rho(\alpha_{k_1})^{i_1} \rho(\alpha_{k_2})^{j_1} = \]
	\[ = \sum_{r_1,h_1,t_1,i_1,j_1 \in I} \eta_{r_1 h_1 t_1 i_1 j_1} f_1^{r_1} f_3^{h_1} f_2^{t_1} k_1^{i_1} k_2^{j_1} = 0. \]
	From linear independence of basis elements it follows that $\eta_{r_1 h_1 t_1 i_1 j_1}=0$ for all $0 \le i_1,j_1,r_1 \le
d-1$, $0 \le h_1, t_1 \le 1$.		
\end{proof}

\Xeq*
\begin{proof}
	We have for all $v \in \{1,2\}$
	
	1.
	\[ k_1^{-1} k_1^i k_2^j e_1^r e_3^h e_2^t k_1 = q^{t-h-2r} k_1^i k_2^j e_1^r e_3^h e_2^t, \]
	\[ \alpha_{k_v}(k_1^{-1} ? k_1) = \alpha_{k_v}, \; \alpha_{e_1}(k_1^{-1} ? k_1) = q^{-2} \alpha_{e_1}, \]
	\[ \alpha_{e_2}(k_1^{-1} ? k_1) = q \alpha_{e_2}, \; \alpha_{e_3}(k_1^{-1} ? k_1) = q^{-1} \alpha_{e_3}. \]
	
	2.
	\[ k_2^{-1} k_1^i k_2^j e_1^r e_3^h e_2^t k_2 = q^{h+r} k_1^i k_2^j e_1^r e_3^h e_2^t, \]
	\[ \alpha_{k_v}(k_2^{-1} ? k_2) = \alpha_{k_v},\; \alpha_{e_1}(k_2^{-1} ? k_2) = q \alpha_{e_1}, \]
	\[ \alpha_{e_2}(k_2^{-1} ? k_2) = \alpha_{e_2}, \; \alpha_{e_3}(k_2^{-1} ? k_2) = q \alpha_{e_3}. \]
	
	3.
	\[ k_1^i k_2^j e_1^r e_3^h e_2^t e_1 = q^{t-h} k_1^i k_2^j e_1^{r+1} e_3^h e_2^t - [t=1] q k_1^i k_2^j e_1^r e_3^{h+1} = 0,
\]
	\[ \alpha_{k_v}(?e_1) = 0, \; \alpha_{e_1}(?e_1) = -(q-q^{-1})^{-1} \alpha_{k_1}^{-1}, \]
	\[ \alpha_{e_2}(?e_1) = 0, \; \alpha_{e_3}(?e_1) = 0. \]
	
	4.
	\[ k_1^i k_2^j e_1^r e_3^h e_2^t k_1 = q^{t-h-2r} k_1^{i+1} k_2^j e_1^r e_3^h e_2^t, \]
	\[ \alpha_{k_v}(?k_1) = ([v=1]q^{-2} + [v=2]q) \alpha_{k_v}, \; \alpha_{e_1}(?k_1) = \alpha_{e_1}, \]
	\[ \alpha_{e_2}(?k_1) = \alpha_{e_2}, \; \alpha_{e_3}(?k_1) = \alpha_{e_3}. \]
	
	5.
	\[ k_1^{-1} e_1 k_1^i k_2^j e_1^r e_3^h e_2^t k_1 = q^{t-h-2r-2i-2+j} k_1^i k_2^j e_1^{r+1} e_3^h e_2^t, \]
	\[ \alpha_{k_v}(k_1^{-1} e_1 ? k_1) = 0, \alpha_{e_1}(k_1^{-1} e_1 ? k_1) = - (q-q^{-1})^{-1} q^{-2} 1_X, \]
	\[ \alpha_{e_2}(k_1^{-1} e_1 ? k_1) = 0, \alpha_{e_3}(k_1^{-1} e_1 ? k_1) = q^{-1} \alpha_{e_2}. \]
	
	6.
	\[ k_1^i k_2^j e_1^r e_3^h e_2^te_2= k_1^i k_2^j e_1^r e_3^h e_2^{t+1} = 0, \]
	\[ \alpha_{k_v}(?e_2) = 0, \; \alpha_{e_1}(?e_2) = 0, \]
	\[ \alpha_{e_2}(?e_2) = (q-q^{-1})^{-1} \alpha_{k_2}^{-1}, \; \alpha_{e_3}(?e_2) = - \alpha_{e_1} \alpha_{k_2}^{-1}. \]
	
	7.
	\[ k_1^i k_2^j e_1^r e_3^h e_2^t k_2 = q^{h+r} k_1^i k_2^{j+1} e_1^r e_3^h e_2^t, \]
	\[ \alpha_{k_v}(?k_2) = q^{[v=1]} \alpha_{k_v}, \; \alpha_{e_1}(?k_2) = \alpha_{e_1}, \]
	\[ \alpha_{e_2}(?k_2) = \alpha_{e_2}, \; \alpha_{e_3}(?k_2) = \alpha_{e_3}. \]
	
	8.
	\[ k_2^{-1} e_2 k_1^i k_2^j e_1^r e_3^h e_2^t k_2 = (-1)^h q^{2h+2r+i} k_1^i k_2^j e_1^r e_3^h e_2^{t+1} - q^{h+r+i+1} [r]
k_1^i k_2^j e_1^{r-1} e_3^{h+1} e_2^t, \]
	\[ \alpha_{k_v}(k_2^{-1} e_2 ? k_2) = 0, \; \alpha_{e_1}(k_2^{-1} e_2 ? k_2) = 0, \]
	\[ \alpha_{e_2}(k_2^{-1} e_2 ? k_2) = (q-q^{-1})^{-1} 1_X, \; \alpha_{e_3}(k_2^{-1} e_2 ? k_2) = 0. \]
	
	9.
	\[ k_2^{-1}e_2 k_1^i k_2^j e_1^r e_3^h e_2^t k_2e_1 = (-1)^h q^{2h+2r+i} k_1^i k_2^j e_1^r e_3^h e_2^{t+1} e_1 - q^{h+r+i+1}
[r] k_1^i k_2^j e_1^{r-1} e_3^{h+1} e_2^t e_1 =  \]
	\[ = (-1)^h q^{h+2r+i+t+1} k_1^i k_2^j e_1^{r+1} e_3^h e_2^{t+1} - [t=0] (-1)^h q^{2h+2r+i+1} k_1^i k_2^j e_1^r e_3^{h+1} -
\]
	\[ - q^{r+i+t} [r] k_1^i k_2^j e_1^{r} e_3^{h+1} e_2^t, \]
	\[ \alpha_{k_v}(k_2^{-1}e_2 ?k_2e_1) = 0, \; \alpha_{e_1}(k_2^{-1}e_2 ?k_2e_1) = 0, \]
	\[ \alpha_{e_2}(k_2^{-1}e_2 ?k_2e_1) = 0, \; \alpha_{e_3}(k_2^{-1}e_2 ?k_2e_1) = 0. \]
	
	10.
	\[ k_2^{-1}e_2 k_1^i k_2^j e_1^r e_3^h e_2^t k_1k_2 = q^{h+r+t-h-2r+i+1} k_1^{i+1} k_2^{j} e_2 e_1^r e_3^h e_2^t= \]
	\[ = (-1)^h q^{h+t+i+1} k_1^{i+1} k_2^{j} e_1^r e_3^h e_2^{t+1} - q^{t-r+i+2} [r] k_1^{i+1} k_2^{j} e_1^{r-1} e_3^{h+1}
e_2^t,  \]
	\[ \alpha_{k_v}(k_2^{-1}e_2 ?k_1k_2) = 0, \; \alpha_{e_1}(k_2^{-1}e_2 ?k_1k_2) = 0, \]
	\[ \alpha_{e_2}(k_2^{-1}e_2 ?k_1k_2) = (q-q^{-1})^{-1} 1_X, \; \alpha_{e_3}(k_2^{-1}e_2 ?k_1k_2) = 0. \]
	
	11.
	\[ k_1^i k_2^j e_1^r e_3^h e_2^tk_2 e_1 = q^{h+r+t-h} k_1^i k_2^{j+1} e_1^{r+1} e_3^h e_2^t - [t=1] q^{h+r+1} k_1^i
k_2^{j+1} e_1^r e_3^{h+1}, \]
	\[ \alpha_{k_v}(?k_2e_1) = 0, \; \alpha_{e_1}(?k_2e_1) = - (q-q^{-1})^{-1} q^{-1} \alpha_{k_1}^{-1}, \]
	\[ \alpha_{e_2}(?k_2e_1) = 0, \; \alpha_{e_3}(?k_2e_1) = 0. \]
	
	12.
	\[ k_1^i k_2^j e_1^r e_3^h e_2^te_3 = (-1)^{t} q^t k_1^i k_2^j e_1^r e_3^{h+1} e_2^t, \]
	\[ \alpha_{k_v}(?e_3) = 0, \; \alpha_{e_1}(?e_3) = 0, \]
	\[ \alpha_{e_2}(?e_3) = 0, \; \alpha_{e_3}(?e_3) = (q-q^{-1})^{-1} \alpha_{k_1}^{-1} \alpha_{k_2}^{-1}. \]
	
	13.
	\[ k_1^i k_2^j e_1^r e_3^h e_2^t k_1k_2 = q^{t-h-2r+h+r} k_1^{i+1} k_2^{j+1} e_1^r e_3^h e_2^t, \]
	\[ \alpha_{k_v}(?k_1k_2) = ([v=1] q^{-1} + [v=2] q) \alpha_{k_v}, \; \alpha_{e_1}(?k_1k_2) = \alpha_{e_1}, \]
	\[ \alpha_{e_2}(?k_1k_2) = \alpha_{e_2}, \; \alpha_{e_3}(?k_1k_2) = \alpha_{e_3}. \]
	
	14.
	\[ k_1^{-1} k_2^{-1} e_1 e_2 k_1^i k_2^j e_1^r e_3^h e_2^t k_1k_2 = q^{t-h-2r-1+h+r+1+i-2i+j} k_1^i k_2^j e_1 e_2 e_1^r
e_3^h e_2^t =  \]
	\[ = (-1)^h q^{t+h-i+j} k_1^i k_2^j e_1^{r+1} e_3^h e_2^{t+1} - q^{t-r-i+j+1} [r] k_1^i k_2^j e_1^{r} e_3^{h+1} e_2^t, \]
	\[ \alpha_{k_v}(k_1^{-1} k_2^{-1} e_1 e_2?k_1k_2) = 0, \; \alpha_{e_1}(k_1^{-1} k_2^{-1} e_1 e_2?k_1k_2) = 0, \]
	\[ \alpha_{e_2}(k_1^{-1} k_2^{-1} e_1 e_2?k_1k_2) = 0, \; \alpha_{e_3}(k_1^{-1} k_2^{-1} e_1 e_2?k_1k_2) = (q-q^{-1})^{-1}
1_X. \]
	
	15.
	\[ k_1^{-1} k_2^{-1} e_3 k_1^i k_2^j e_1^r e_3^h e_2^t k_1k_2 = q^{t-h-2r+h+r-i-1+j+1-r} k_1^{i} k_2^{j} e_1^r e_3^{h+1}
e_2^t, \]
	\[ \alpha_{k_v}(k_1^{-1} k_2^{-1} e_3?k_1k_2) = 0, \; \alpha_{e_1}(k_1^{-1} k_2^{-1} e_3?k_1k_2) = 0, \]
	\[ \alpha_{e_2}(k_1^{-1} k_2^{-1} e_3?k_1k_2) = 0, \; \alpha_{e_3}(k_1^{-1} k_2^{-1} e_3?k_1k_2) = (q-q^{-1})^{-1} 1_X. \]
\end{proof}

\DF*
\begin{proof}
	By \ref{lm:MultQDr} the product $(1_X \otimes a) (\beta \otimes 1_{B_q^+})$ in $D$ for $a \in B_q^+$ and $\beta \in X$ is
given by
	\[ (1_X \otimes a) (\beta \otimes 1_{B_q^+}) = \sum_{(a),(a^{'})} (-1)^{|\beta||a| + |a^{''}||a^{'}| + |?||(a^{'})^{'}|}
\beta(S^{-1}( a^{''} ) ? (a^{'})^{'}) \otimes (a^{'})^{''}. \]
	
	Let us apply this formula to the generators. For any element $\beta \in X$ we have
	\[ (1_X \otimes k_i) (\beta \otimes 1_{B_q^+}) = \beta(k_i^{-1} ? k_i) \otimes k_i, \]
	\[ (1_X \otimes e_1) (\beta \otimes 1_{B_q^+}) = \beta(?e_1) \otimes 1_{B_q^+} + \beta(?k_1) \otimes e_1 - q^2
\beta(k_1^{-1} e_1 ? k_1) \otimes k_1, \]
	\[ (1_X \otimes e_2) (\beta \otimes 1_{B_q^+}) = (-1)^{|\beta|+|?|} \beta(?e_2) \otimes 1_{B_q^+} + (-1)^{|\beta|}
\beta(?k_2) \otimes e_2 + (-1)^{|\beta|+1} \beta(k_2^{-1} e_2 ? k_2) \otimes k_2, \]
	\[ (1_X \otimes e_3) (\beta \otimes 1_{B_q^+}) = (-1)^{1+|\beta|} (q-q^{-1}) \beta(k_2^{-1}e_2 ?k_2e_1) \otimes k_2 +\]
	\[ + (-1)^{1+|\beta|} (q-q^{-1}) \beta(k_2^{-1}e_2 ?k_1k_2) \otimes k_2 e_1 + (-1)^{|\beta|} (q-q^{-1}) \beta(?k_2e_1)
\otimes e_2 + (-1)^{|\beta|+|?|} \beta(?e_3) \otimes 1_{B_q^+} + \]
	\[ + (-1)^{|\beta|} \beta(?k_1k_2) \otimes e_3 + (-1)^{|\beta|} (q^2-1) \beta(k_1^{-1} k_2^{-1} e_1 e_2?k_1k_2) \otimes k_1
k_2 + (-1)^{1+|\beta|} q^2 \beta(k_1^{-1} k_2^{-1} e_3?k_1k_2) \otimes k_1 k_2. \]
	
	Note that
	\[ ( \Delta_{B_q^+} \otimes id_{B_q^+} ) \circ \Delta_{B_q^+} (e_i) =   \]
	\[ = ( \Delta_{B_q^+} \otimes id_{B_q^+} ) ( e_i \otimes 1 + k_i \otimes e_i ) = e_i \otimes 1 \otimes 1 + k_i \otimes e_i
\otimes 1 + k_i \otimes k_i \otimes e_i, \]
	\[ ( \Delta_{B_q^+} \otimes id_{B_q^+} ) \circ \Delta_{B_q^+} (e_3) =  \]
	\[ = ( \Delta_{B_q^+} \otimes id_{B_q^+} )( (q-q^{-1}) k_2 e_1 \otimes e_2 + e_3 \otimes 1 + k_1 k_2 \otimes e_3 ) = \]
	\[ = (q-q^{-1}) k_2 e_1 \otimes k_2 \otimes e_2 + (q-q^{-1}) k_1 k_2 \otimes k_2 e_1 \otimes e_2 + \]
	\[ + (q-q^{-1}) k_2 e_1 \otimes e_2 \otimes 1 + e_3 \otimes 1 \otimes 1 + k_1 k_2 \otimes e_3 \otimes 1 + \]
	\[ + k_1 k_2 \otimes k_1 k_2 \otimes e_3. \]
	
	Hence for $\alpha_{k_1}$
	\[ (1_X \otimes k_v) (\alpha_{k_1} \otimes 1_{B_q^+}) = \alpha_{k_1}(k_v^{-1} ? k_v) \otimes k_v = \alpha_{k_1} \otimes k_v
= (\alpha_{k_1} \otimes 1_{B_q^+}) (1_X \otimes k_v), \]
	\[ (1_X \otimes e_1) (\alpha_{k_1} \otimes 1_{B_q^+}) = \alpha_{k_1}(?k_1) \otimes e_1 = q^{-2} \alpha_{k_1} \otimes e_1 =
q^{-2} (\alpha_{k_1} \otimes 1_{B_q^+}) (1_X \otimes e_1), \]
	\[ (1_X \otimes e_2) (\alpha_{k_1} \otimes 1_{B_q^+}) = \alpha_{k_1}(?k_2) \otimes e_2 = q \alpha_{k_1} \otimes e_2 = q
(\alpha_{k_1} \otimes 1_{B_q^+}) (1_X \otimes e_2), \]
	\[ (1_X \otimes e_3) (\alpha_{k_1} \otimes 1_{B_q^+}) =  \alpha_{k_1}(?k_1k_2) \otimes e_3 = q^{-1} \alpha_{k_1} \otimes e_3
= q^{-1} (\alpha_{k_1} \otimes 1_{B_q^+}) (1_X \otimes e_3). \]
	
	For $\alpha_{k_2}$
	\[ (1_X \otimes k_v) (\alpha_{k_2} \otimes 1_{B_q^+}) = \alpha_{k_2}(k_v^{-1} ? k_v) \otimes k_v = \alpha_{k_2} \otimes k_v
= (\alpha_{k_2} \otimes 1_{B_q^+})(1_X \otimes k_v), \]
	\[ (1_X \otimes e_1) (\alpha_{k_2} \otimes 1_{B_q^+}) = \alpha_{k_2}(?k_1) \otimes e_1 = q \alpha_{k_2} \otimes e_1 = q
(\alpha_{k_2} \otimes 1_{B_q^+}) (1_X \otimes e_1), \]
	\[ (1_X \otimes e_2) (\alpha_{k_2} \otimes 1_{B_q^+}) = \alpha_{k_2}(?k_2) \otimes e_2 = \alpha_{k_2} \otimes e_2 =
(\alpha_{k_2} \otimes 1_{B_q^+})(1_X \otimes e_2), \]
	\[ (1_X \otimes e_3) (\alpha_{k_2} \otimes 1_{B_q^+}) =  \alpha_{k_2}(?k_1k_2) \otimes e_3 = q \alpha_{k_2} \otimes e_3 = q
(\alpha_{k_2} \otimes 1_{B_q^+}) (1_X \otimes e_3). \]
	
	For $\alpha_{e_1}$
	\[ (1_X \otimes k_1) (\alpha_{e_1} \otimes 1_{B_q^+}) = \alpha_{e_1}(k_1^{-1} ? k_1) \otimes k_1 = q^{-2} \alpha_{e_1}
\otimes k_1, \]
	\[ (1_X \otimes k_2) (\alpha_{e_1} \otimes 1_{B_q^+}) = \alpha_{e_1}(k_2^{-1} ? k_2) \otimes k_2 = q \alpha_{e_1} \otimes
k_2, \]
	\[ (1_X \otimes e_1) (\alpha_{e_1} \otimes 1_{B_q^+}) = \alpha_{e_1}(?e_1) \otimes 1_{B_q^+} + \alpha_{e_1}(?k_1) \otimes
e_1 - q^2 \alpha_{e_1}(k_1^{-1} e_1 ? k_1) \otimes k_1 = \]
	\[ = - (q-q^{-1})^{-1} \alpha_{k_1}^{-1} \otimes 1_{B_q^+} + \alpha_{e_1} \otimes e_1 + (q-q^{-1})^{-1} 1_{X} \otimes k_1,
\]
	\[ (1_X \otimes e_2) (\alpha_{e_1} \otimes 1_{B_q^+}) = \alpha_{e_1}(?k_2) \otimes e_2  = \alpha_{e_1} \otimes e_2, \]
	\[ (1_X \otimes e_3) (\alpha_{e_1} \otimes 1_{B_q^+}) = \]
	\[ = (q-q^{-1}) \alpha_{e_1}(?k_2e_1) \otimes e_2 + \alpha_{e_1}(?k_1k_2) \otimes e_3 = - q^{-1} \alpha_{k_1}^{-1} \otimes
e_2 + \alpha_{e_1} \otimes e_3. \]
	
	For $\alpha_{e_2}$
	\[ (1_X \otimes k_1) (\alpha_{e_2} \otimes 1_{B^{+}_q}) = \alpha_{e_2}(k_1^{-1} ? k_1) \otimes k_1 = q \alpha_{e_2} \otimes
k_1, \]
	\[(1_X \otimes k_2)(\alpha_{e_2} \otimes 1_{B_q^{+}}) = \alpha_{e_2}(k_2^{-1}?k_2) \otimes k_2 = \alpha_{e_2} \otimes
k_2,\]
	\[(1_X \otimes e_1)(\alpha_{e_2} \otimes 1_{B_q^{+}}) = \alpha_{e_2}(?k_1) \otimes e_1 = \alpha_{e_2} \otimes e_1,\]
	\[(1_X \otimes e_2)(\alpha_{e_2} \otimes 1_{B_q^{+}}) = (-1)^{1+|?|} \alpha_{e_2}(?e_2) \otimes 1_{B_q^{+}} -
\alpha_{e_2}(?k_2) \otimes e_2 + \alpha_{e_2}(k_2^{-1} e_2 ? k_2) \otimes k_2 =\]
	\[= - (q-q^{-1})^{-1} \alpha_{k_2}^{-1} \otimes 1_{B_q^{+}} - \alpha_{e_2} \otimes e_2 + (q-q^{-1})^{-1} 1_{X} \otimes
k_2,\]
	\[(1_X \otimes e_3)(\alpha_{e_2} \otimes 1_{B_q^{+}}) = (q-q^{-1}) \alpha_{e_2}(k_2^{-1} e_2 ? k_1 k_2) \otimes k_2 e_1 -
\alpha_{? k_1 k_2} \otimes e_3 = \]
	\[= 1_{X} \otimes k_2 e_1 - \alpha_{e_2} \otimes e_3. \]
	
	For $\alpha_{e_3}$
	\[(1_X \otimes k_1) (\alpha_{e_3} \otimes 1_{B_q^{+}}) = q^{-1} \alpha_{e_3} \otimes k_1,\]
	\[(1_X \otimes k_2) (\alpha_{e_3} \otimes 1_{B_q^{+}}) = q \alpha_{e_3} \otimes k_2,\]
	\[(1_X \otimes e_1) (\alpha_{e_3} \otimes 1_{B_q^{+}}) = \alpha_{e_3} \otimes e_1 - q \alpha_{e_2} \otimes k_1,\]
	\[(1_X \otimes e_2) (\alpha_{e_3} \otimes 1_{B_q^{+}}) = \alpha_{e_1} \alpha_{k_2}^{-1} \otimes 1_{B_q^{+}} - \alpha_{e_3}
\otimes e_2,\]		
	\[ (1_X \otimes e_3) (\alpha_{e_3} \otimes 1_{B_q^+}) = \]
	\[ = - (q-q^{-1})^{-1} \alpha_{k_1}^{-1} \alpha_{k_2}^{-1} \otimes 1_{B_q^{+}} - \alpha_{e_3} \otimes e_3 - (q^2-1)
(q-q^{-1})^{-1} 1_X  \otimes k_1 k_2 + q^2 (q-q^{-1})^{-1} 1_X \otimes k_1 k_2 = \]
	\[ = - (q-q^{-1})^{-1} \alpha_{k_1}^{-1} \alpha_{k_2}^{-1} \otimes 1_{B_q^{+}} - \alpha_{e_3} \otimes e_3 + (q-q^{-1})^{-1}
1_X  \otimes k_1 k_2. \]
\end{proof}

\braided*
\begin{proof}
	The Hopf superalgebra $D=D(B^{+}_q)$ is braided by Theorem \ref{RQDth}. Let $R_{D} \in D \otimes D$ be its universal
$R$-matrix. Define the invertible even element $\bar{R}$  of $\bar{U} \otimes \bar{U}$ by
	\[ \bar{R} = (\chi \otimes \chi)(R_{D}). \]
	Since $\chi$ is surjective Hopf superalgebra morphism we have $x=\chi(y)$ for every $x$ in $\bar{U}_q$ for some $y \in D$.
Let us show that $\bar{R}$ is universal $R$-matrix.
	\[ \bar{R} \Delta_{\bar{U}_q}(x) = (\chi \otimes \chi)(R_{D}) \Delta_{\bar{U}_q}(x) = (\chi \otimes \chi)(R_{D})
\Delta_{\bar{U}_q}(\chi(y)) =  \]
	\[ = (\chi \otimes \chi)(R_{D}) ((\chi \otimes \chi) \circ \Delta_{D})(y) = (\chi \otimes \chi)( R_{D} \Delta_{D}(y) ) = \]
	\[ = (\chi \otimes \chi)( \Delta_{D}^{op}(y) R_{D} ) = (\chi \otimes \chi)(\Delta_{D}^{op}(y)) (\chi \otimes \chi)(R_{D}) =
\]
	\[ = \Delta_{\bar{U}_q}^{op} ( \chi(y) ) ( \chi \otimes \chi)(R_{D}) = \Delta_{\bar{U}_q}^{op}(x) \bar{R}. \]
	Let us check that $\bar{U}_q$ is braided.
	\[ \bar{R}_{13} \bar{R}_{23} = (\chi \otimes \chi)(R_{D})_{13} (\chi \otimes \chi)(R_{D})_{23} = (\chi \otimes \chi \otimes
\chi) ( (R_{D})_{13} ) (\chi \otimes \chi \otimes \chi) ( (R_{D})_{23} ) = \]
	\[ = (\chi \otimes \chi \otimes \chi) ((R_{D})_{13} (R_{D})_{23}) = (\chi \otimes \chi \otimes \chi) ((\Delta_{D} \otimes
id_{D} ) (R_{D})) =  \]
	\[ = ((\Delta_{\bar{U}_{q}} \circ \chi) \otimes \chi )(R_{D}) =  (\Delta_{\bar{U}_{q}} \otimes id_{\bar{U}_q}) ( \bar{R} ),
\]		
	\[ \bar{R}_{13} \bar{R}_{12} = (\chi \otimes \chi)(R_{D})_{13} (\chi \otimes \chi)(R_{D})_{12} = (\chi \otimes \chi \otimes
\chi) ( (R_{D})_{13} ) (\chi \otimes \chi \otimes \chi) ( (R_{D})_{12} ) = \]
	\[ = (\chi \otimes \chi \otimes \chi) ((R_{D})_{13} (R_{D})_{12}) = (\chi \otimes \chi \otimes \chi) ((id_{D} \otimes
\Delta_{D} ) (R_{D})) =  \]
	\[ = ( \chi \otimes (\Delta_{\bar{U}_{q}} \circ \chi) )(R_{D}) =  (id_{\bar{U}_q} \otimes \Delta_{\bar{U}_{q}}) ( \bar{R} ).
\]
\end{proof}

\lmbimo*
\begin{proof}
	We give a proof by induction on $n$. When $n=3$ the claim of lemma holds by \ref{eq:l3decomp}. Similary for $n=4$ by
\ref{eq:l4decomp}. Let $W = \sum_{i=0}^{2} L_{n-1,\mu} g_{n-1}^{i} L_{n-1,\mu} + L_{n-3,\mu} g_{n-1} g_{n-2}^{2} g_{n-1}$. Now
we need to proof that $W$ is a $L_{n-1,\mu}$-bimodule. Since $L_{n-3,\mu}$ commutes with $g_{n-1} g_{n-2}^{2} g_{n-1}$, it is
sufficient to proof that elements $g_{i_1}^{j_1} g_{n-1} g_{n-2}^{2} g_{n-1} g_{i_2}^{j_2} \in W$, where $i_1,i_2 \in
\{n-3,n-2\}, \; j_1,j_2 \in \{ 1,2 \}$. This elements are conjugates of elements whose image lie in $L_{3,\mu}$ and $L_{4,\mu}$,
so the result follows from our basis construction for $L_{3,\mu}$ and $L_{4,\mu}$.
\end{proof}

\bibliographystyle{elsarticle-num}

\end{document}